\documentclass[11pt]{amsart}
\usepackage{amsmath, amssymb}
\usepackage{amsfonts}
\usepackage{mathrsfs}
\usepackage[arrow,matrix,curve,cmtip,ps]{xy}

\usepackage{amsthm}
\usepackage{mathtools} 
\usepackage{array}
\usepackage{mathdots}

\usepackage[top= 2.25cm, bottom=2cm, left = 2cm, right= 2cm]{geometry}
\usepackage{hyperref}
\usepackage{setspace}

\allowdisplaybreaks

\numberwithin{equation}{section}

\newtheorem{theorem}{Theorem}[section]
\newtheorem{lemma}[theorem]{Lemma}
\newtheorem{proposition}[theorem]{Proposition}
\newtheorem{corollary}[theorem]{Corollary}

\newtheorem{conjecture}[theorem]{Conjecture}
\newtheorem{question}[theorem]{Question}
\newtheorem*{theorem*}{Theorem}
\theoremstyle{remark}
\newtheorem{remark}[theorem]{Remark}
\newtheorem{definition}[theorem]{Definition}

\numberwithin{equation}{section}

\newcommand{\lif}[1]{\widetilde{#1}}  

\newcommand{\com}[2][\theta_{0}]{#2^{#1}}

\newcommand{\End}[2]{\mathcal{E}_{#1}(#2)}

\newcommand{\tEnd}[3][\omega]{\mathcal{E}_{#2}(#3,#1)}

\newcommand{\uP}[1]{\Phi^{+}_{unit}(#1)}

\renewcommand{\P}[1]{\Phi(#1)}

\newcommand{\Q}[1]{\Psi(#1)}

\newcommand{\Pbd}[1]{\Phi_{bdd}(#1)}

\newcommand{\Psm}[1]{\Phi_{sim}(#1)}

\newcommand{\Pdt}[1]{\Phi_{2}(#1)}

\newcommand{\cuP}[1]{\bar{\Phi}^{+}_{unit}(#1)}

\newcommand{\cP}[1]{\bar{\Phi}(#1)}

\newcommand{\cPbd}[1]{\bar{\Phi}_{bdd}(#1)}

\newcommand{\cPdt}[1]{\bar{\Phi}_{2}(#1)}

\newcommand{\cPsm}[1]{\bar{\Phi}_{sim}(#1)}

\newcommand{\cPel}[1]{\bar{\Phi}_{ell}(#1)}

\newcommand{\p}{\phi}

\newcommand{\lp}{\tilde{\phi}}

\renewcommand{\lq}{\tilde{\psi}}

\newcommand{\Pkt}[1]{\Pi_{#1}}

\newcommand{\cPkt}[1]{\bar{\Pi}_{#1}}

\newcommand{\clPkt}[1]{\tilde{\bar{\Pi}}_{#1}}

\renewcommand{\r}{\pi}

\newcommand{\lr}{\tilde{\pi}}

\renewcommand{\H}{\mathcal{H}}

\newcommand{\sH}{\bar{\mathcal{H}}}

\newcommand{\Idt}[2]{I_{disc #2}^{#1}}

\newcommand{\tIdt}[2]{I_{disc #2}^{(#1, \omega)}}

\newcommand{\Rdt}[2]{R_{disc #2}^{#1}}

\newcommand{\tRdt}[2]{R_{disc #2}^{(#1, \omega)}}

\newcommand{\Sdt}[2]{S_{disc #2}^{#1}} 

\newcommand{\cS}[1]{\bar{S}_{#1}}

\renewcommand{\S}[1]{\mathcal{S}_{#1}}

\newcommand{\+}{\oplus}

\renewcommand{\#}{\boxplus}

\renewcommand{\c}{\lambda}

\newcommand{\lG}{\widetilde{G}}

\newcommand{\lM}{\widetilde{M}}

\newcommand{\lP}{\widetilde{P}}

\newcommand{\x}{\omega}

\renewcommand{\L}[1]{{}^L#1}

\newcommand{\D}[1]{\widehat{#1}}

\newcommand{\Gal}[1]{\Gamma_{#1}}

\renewcommand{\a}{\alpha}

\renewcommand{\Im}{\text{Im}\,}

\newcommand{\lZ}{Z_{\widetilde{G}}}

\newcommand{\Z}{Z_{G}}

\newcommand{\iG}[1]{{G^{1}_{#1}}}

\newcommand{\ir}{\pi^{1}}

\newcommand{\Res}{\text{Res}}

\newcommand{\Ind}{\text{Ind}}

\newcommand{\Cent}{\text{Cent}}

\newcommand{\Two}{\mathbb{Z}/2\mathbb{Z}}

\newcommand{\C}{\mathbb{C}}

\newcommand{\R}{\mathbb{R}}

\newcommand{\A}{\mathbb{A}}

\newcommand{\ep}{\phi^{\mathcal{E}}}

\newcommand{\ePbd}[1]{\Phi_{bdd}^{\mathcal{E}}(#1)}

\newcommand{\ePsm}[1]{\Phi_{sim}^{\mathcal{E}}(#1)}

\newcommand{\cePbd}[1]{\bar{\Phi}_{bdd}^{\mathcal{E}}(#1)}

\newcommand{\cePsm}[1]{\bar{\Phi}_{sim}^{\mathcal{E}}(#1)}

\newcommand{\lf}{\tilde{f}}

\newcommand{\Norm}{\text{Norm}}

\newcommand{\Nm}{\text{Nm}}

\newcommand{\cN}[1]{\bar{N}_{#1}}

\newcommand{\N}[1]{\frak{N}_{#1}}

\newcommand{\cT}[1]{\bar{T}_{#1}}

\newcommand{\Cl}{\mathcal{C}}

\newcommand{\cCl}{\bar{\mathcal{C}}}

\newcommand{\m}{\bar{m}}

\newcommand{\Aut}{\text{Aut}}

\newcommand{\Int}{\text{Int}}

\newcommand{\Out}{\text{Out}}

\newcommand{\Hom}{\text{Hom}}

\newcommand{\Ker}{\text{Ker}}

\newcommand{\e}{\varepsilon}

\begin{document}
\title{L-packets of quasisplit $GSp(2n)$ and $GO(2n)$}

\author{Bin Xu}

\address{Department of Mathematics and Statistics \\  University of Calgary, Canada }
\email{bin.xu2@ucalgary.ca}

\subjclass[2010]{22E50 (primary); 11F70 (secondary)}
\keywords{similitude group, twisted endoscopic transfer, L-packet, stabilized twisted trace formula}


\maketitle

\begin{abstract}
In his monograph (2013) Arthur characterizes the L-packets of quasisplit symplectic groups and orthogonal groups. By extending his work, we characterize the L-packets for the corresponding similitude groups with desired properties. In particular, we show these packets satisfy the conjectural endoscopic character identities. 
\end{abstract}

\tableofcontents

\section{Introduction}
\label{sec: introduction}

Let $F$ be a local field of characteristic zero and $W_{F}$ be the Weil group, then the local Langlands group is defined as follows 
\[
L_{F} = 
\begin{cases}
W_{F} & \text{if $F$ is archimedean}, \\ 
W_{F} \times SL(2, \C) & \text{if $F$ is nonarchimedean}.
\end{cases}
\] 
Let $G$ be a quasisplit connected reductive group over $F$ and $\D{G}$ be its complex dual group, the Langlands dual group $\L{G}$ is a semidirect product $\D{G} \rtimes W_{F}$, where the action of $W_{F}$ on $\D{G}$ factors through the absolute Galois group $\Gal{F} = \text{Gal}(\bar{F}/F)$. A local Langlands parameter $\p$ is a $\D{G}$-conjugacy class of admissible homomorphisms from $L_{F}$ to $\L{G}$ (see \cite{Borel:1979}). In particular, it respects the projections on $W_{F}$ from both $L_{F}$ and $\L{G}$. We denote a representative of $\p$ by $\underline{\p}: L_{F} \rightarrow \L{G}$. Let $\P{G}$ be the set of local Langlands parameters and $\Pkt{}(G(F))$ be the set of isomorphism classes of irreducible admissible representations of $G(F)$. The local Langlands conjecture asserts a correspondence between $\P{G}$ and $\Pkt{}(G(F))$. The correspondence is not necessarily a bijection. In fact, it is conjectured that each $\p \in\P{G}$ is associated with a finite set $\Pkt{\p}$ of $\Pkt{}(G(F))$, such that they give a partition of $\Pkt{}(G(F))$
\begin{align*}
\Pkt{}(G(F)) = \bigsqcup_{\p \in \P{G}} \Pkt{\p}.
\end{align*}
Such sets $\Pkt{\p}$ are called L-packets. The local Langands conjecture has been proved for $GL(N)$ by Harris-Taylor \cite{HarrisTaylor:2001}, Henniart \cite{Henniart:2000} and Scholze \cite{Scholze:2013}, in which case one does get a bijection. Arthur \cite{Arthur:2013} extended their results to $Sp(N)$ and $SO(N)$ through the theory of twisted endoscopy, and in his case the packets are not always singletons. By the Langlands classification of irreducible admissible representations of $G(F)$, one can reduce this correspondence to the tempered case, namely one can replace $\Pkt{}(G(F))$ by the subset $\Pkt{temp}(G(F))$ of tempered representations, and $\P{G}$ by the subset $\Pbd{G}$ of bounded parameters (i.e., the closure of the image of $\underline{\p}|_{W_{F}}$ is compact). The tempered L-packets can be characterized by ``stability". To explain this concept, we need to introduce the Harish-Chandra characters. For any $\r \in \Pkt{}(G(F))$, the associated Harish-Chandra character is a distribution on $G(F)$ defined by 
\[
f_{G}(\r) := trace \int_{G(F)} f(g)\r(g) dg
\]
for $f \in C^{\infty}_{c}(G(F))$. Harish-Chandra showed this distribution can be represented by a $G(F)$-conjugate invariant locally integrable function $\Theta_{\r}$ over $G(F)$. Moreover, $\Theta_{\r}$ is smooth over the strongly regular semisimple elements $G_{reg}(F)$. Later on, we will simply call them characters. We say a finite linear combination $\Theta$ of Harish-Chandra characters is {\bf stable} if it is $G(\bar{F})$-conjugate invariant over $G_{reg}(F)$, namely $\Theta(\gamma) = \Theta(\gamma')$ for any $\gamma, \gamma' \in G_{reg}(F)$ such that $\gamma = g^{-1} \gamma' g$ for some $g \in G(\bar{F})$. Then the tempered L-packets are conjectured to be the minimal subsets of irreducible tempered representations, within which some linear combination of the Harish-Chandra characters is stable (cf. Conjecture 9.2, \cite{Shahidi:1990}). 

Let $D$ be a torus and $\lG$ be a quasisplit connected reductive group over $F$, which is an extension of $D$ by $G$
\begin{align*}
\xymatrix{1 \ar[r] & G \ar[r] & \lG \ar[r]^{\c}  & D \ar[r] & 1. }
\end{align*}
Dual to this exact sequence, we have
\begin{align*}
\xymatrix{1 \ar[r] & \D{D} \ar[r] & \D{\lG} \ar[r]^{\bold{p}}  & \D{G} \ar[r] & 1. }
\end{align*}
The projection $\bold{p}: \D{\lG} \rightarrow \D{G}$ can be extended to an L-homomorphism, so it induces a map $\Pbd{\lG} \rightarrow \Pbd{G}$. Labesse (\cite{Labesse:1985}, Theorem 8.1) showed this map is in fact surjective. For $\lp \in \Pbd{\lG}$ and $\p = \bold{p} \circ \lp$, it is believed that the restriction $\Pkt{\lp}|_{G} = \Pkt{\p}$. Motivated by this, we want to construct the L-packets of $\lG$ from that of $G$, when $G = Sp(2n)$ (resp. $SO(2n)$) and $\lG = GSp(2n)$ (resp. $GSO(2n)$). In fact, one can also consider the case when $G = SO(2n+1)$ and $\lG = GO(2n+1)$. Note $GO(2n+1)$ is connected. Since $GO(2n+1) \cong SO(2n+1) \times \mathbb{G}_{m}$, this case would be trivial.
To give the precise statement of our result, we need to first recall Arthur's results about $G$. We fix an outer automorphism $\theta_{0}$ of $G$, such that it is trivial when $G = Sp(2n)$, and it is induced from the conjugate action of $O(2n)$ when $G = SO(2n)$. Let $\Sigma_{0} = <\theta_{0}>$, then $\Sigma_{0}$ acts on $\Pkt{}(G(F))$. Note $\theta_{0}$ induces a dual automorphism $\D{\theta}_{0}$ on $\D{G}$, so $\Sigma_{0}$ also acts on $\P{G}$ through the action of $\D{\theta}_{0}$ on $\D{G}$. We denote the set of $\Sigma_{0}$-orbits in $\Pkt{temp}(G(F))$ by $\cPkt{temp}(G(F))$ and the set of $\Sigma_{0}$-orbits in $\Pbd{G}$ by $\cPbd{G}$. The action of $\Sigma_{0}$ can be extended to $\lG$, so we can also define the analogues of these sets for $\lG$.

\begin{theorem}[Arthur]
\label{thm: L-packet for G}
\begin{enumerate}
\item
There is a canonical way to associate any $\p \in \cPbd{G}$ with a finite subset $\cPkt{\p}$ of $\cPkt{temp}(G(F))$ such that 
\begin{align*}
\cPkt{temp}(G(F)) = \bigsqcup_{\p \in \cPbd{G}} \cPkt{\p}.
\end{align*}
\item
For $\p \in \cPbd{G}$, 
\[
\bar{\Theta}_{\p} := \frac{1}{2} \sum_{[\r] \in \cPkt{\p}} (\Theta_{\r} + \Theta_{\r^{\theta_{0}}})
\]
is stable.
\end{enumerate}
\end{theorem}

When $G = SO(2n)$, we let $\Pkt{\p}^{\Sigma_{0}}$ be the set of all irreducible representations of $O(2n)$, whose restriction to $SO(2n)$ have irreducible constituents contained in $\cPkt{\p}$, and we call $\Pkt{\p}^{\Sigma_{0}}$ an L-packet of $O(2n)$. In this sense, the sets $\cPkt{\p}$ really determine the L-packets of $Sp(2n)$ and $O(2n)$. But for simplicity, we will still call the sets $\cPkt{\p}$ L-packets of $G$ in this paper.

Suppose $\lp \in \cPbd{\lG}$ and $\p = \bold{p} \circ \lp$. Since $\cPkt{\p}$ admits a stable linear combination of Harish-Chandra characters, $\lG(F)$ acts on $\cPkt{\p}$ by conjugation. We fix a character $\lif{\zeta}$ of the centre $\lZ(F)$ of $\lG(F)$, such that its restriction to $\Z(F) = \lZ(F) \cap G(F)$ is the central character of $\cPkt{\p}$. Let $\clPkt{\p, \lif{\zeta}}$ be the subset of representations of $\cPkt{temp}(\lG(F))$ with central character $\lif{\zeta}$, whose restriction to $G(F)$ have irreducible constituents contained in $\cPkt{\p}$.  Let $X = \Hom(\lG(F)/\lZ(F)G(F), \C^{\times})$. Note $X$ acts on $\clPkt{\p, \lif{\zeta}}$ by twisting. In Corollary~\ref{cor: theta twisting character} we show there exists a subgroup $\a(\S{\p}^{\Sigma_{0}})$ of $X$ such that for any $[\lr] \in \clPkt{\p, \lif{\zeta}}$, 
\(
\lr \otimes \x \cong \lr^{\theta} \text{ for some $\theta \in \Sigma_{0}$ if and only if } \x \in \a(\S{\p}^{\Sigma_{0}}).
\)
Now we can state our main result.

\begin{theorem}
\label{thm: main local theorem}
Suppose $\p \in \cPbd{G}$, there exists a subset $\cPkt{\lp}$ of $\clPkt{\p, \lif{\zeta}}$ unique up to twisting by $X$, such that it satisfies the following properties:

\begin{enumerate}
\item
\[
\clPkt{\p, \lif{\zeta}} = \bigsqcup_{\x \in X / \a(\S{\p}^{\Sigma_{0}})} \cPkt{\lp} \otimes \x.
\]

\item 
\[
\bar{\Theta}_{\lp} := \frac{1}{2} \sum_{[\lr] \in \cPkt{\lp}} (\Theta_{\lr} + \Theta_{\lr^{\theta_{0}}})
\]
is stable.
\end{enumerate}

\end{theorem}

In this paper, we call the sets $\cPkt{\lp}$ in this theorem L-packets of $\lG$, although they really determine the L-packets of $GSp(2n)$ and $GO(2n)$ for the same reason as we have discussed above. When $F$ is archimedean, this theorem is known due to Langlands \cite{Langlands:1989} and Shelstad \cite{Shelstad:1979}. In fact this case could also follow from Theorem~\ref{thm: L-packet for G} directly. So in this paper, we will focus on the case when $F$ is nonarchimedean. Note if $\cPkt{\p}$ is a singleton, then $\cPkt{\lp}$ is also a singleton by part (1) of the theorem, but it is still by no means clear that part (2) will hold for such $\cPkt{\lp}$. Our proof of this theorem is by global means, and it is certainly interesting to know if one can establish it by purely local methods.

The main idea of the proof is to realize the L-packet as the local component of some global L-packet. To describe the global picture, we let $F$ be a number field and $\A_{F}$ be the adele ring of $F$. We define the automorphic representations of $G$ to be the irreducible constituents of the regular representation of $G(\A_{F})$ on $L^{2}(G(F) \backslash G(\A_{F}))$. If $\r$ is an irreducible admissible representation of $G(\A_{F})$, it can be decomposed as a restricted tensor product 
\[
\r = \otimes_{v}'\r_{v}
\]
of irreducible admissible representations $\r_{v}$ of $G(F_{v})$ over all the places $v$. These local representations $\r_{v}$ are unramified for almost all places, which is the necessary condition to form the restricted tensor product. We assume the global Langlands group $L_{F}$ exists and it is equipped with embeddings $L_{F_{v}} \rightarrow L_{F}$ for all places $v$. Then we can define the global Langlands parameters as in the local case. We denote the set of $\Sigma_{0}$-orbits of bounded global Langlands parameters by $\cP{G}$, for this is the set relevant in the classification of automorphic representations of $G$. For any $\p \in \cP{G}$, we can associate a family of local Langlands parameter $\p_{v} \in \cPbd{G_{v}}$ for all places by the following diagram
\[
\xymatrix{L_{F_{v}}  \ar[d]  \ar[rr]^{\,\, \p_{v}} && \L{G_{v}} \ar[d] \\
                 L_{F}  \ar[rr]^{\,\, \p} && \L{G}.}
\]
So one can define the global L-packet to be the restricted tensor product of the local L-packets
\[
\cPkt{\p} := \otimes'_{v} \cPkt{\p_{v}}.
\]

\begin{theorem}[Arthur]
\label{thm: global L-packet for G}
There exist automorphic representations in $\cPkt{\p}$.
\end{theorem}

For any irreducible admissible representation $\r$ of $G(\A_{F})$, one can associate a family of Satake parameters $c(\r) = \{c(\r_{v})\}$ for all unramified places of $\r$. If we define an equivalence relation on the families of Satake parameters attached to irreducible admissible representations of $G(\A_{F})$ by requiring $c(\r) \sim c(\r')$ if $c(\r_{v})$ is $\Sigma_{0}$-conjugate to $c(\r'_{v})$ for almost all places, then another way of characterizing $\cPkt{\p}$ is through the equivalence class $c(\p)$ of family of Satake parameters associated with the representations in $\cPkt{\p}$. If we take the standard embedding $\xi:\L{G} \rightarrow GL(N, \C)$, where $N = 2n + 1$ (resp. $2n$) if $G = Sp(2n)$ (resp. $G = SO(2n)$), then $\xi(c(\p))$ defines a family of Satake parameters for irreducible admissible representations of $GL(N, \A_{F})$. By the conjectural Langlands principle of functoriality and strong multiplicity one for automorphic representations of $GL(N)$, $\xi(c(\p))$ determines a unique automorphic representation of $GL(N)$. In practice, Arthur gets around the assumption on global Langlands group by reversing our discussion here. To be more precise, he substituted for $\cP{G}$ by the subset of self-dual automorphic representations of $GL(N)$, which are induced from cuspidal automorphic representations of the Levi subgroups of $GL(N)$. Then $\p_{v}$ will correspond to the representations of $GL(N, F_{v})$. Since we do not have the generalized Ramanujan conjecture, now we can only conclude $\p_{v} \in \cuP{G} \supseteq \cPbd{G}$ (see Proposition~\ref{prop: local constituents of cuspidal representations}). Nonetheless, the local packet $\cPkt{\p}$ can still be defined in this case. In this way, Theorem~\ref{thm: global L-packet for G} should really be viewed as a statement about Langlands principle of functoriality with respect to the embedding $\xi$. To summarize, the global L-packet $\cPkt{\p}$ can be uniquely characterized by either an equivalence class of family of Satake parameters $c(\p)$ or an automorphic representation of $GL(N)$ associated with $\xi(c(\p))$. We call this the strong multiplicity one property for the global L-packets of $G$.

The main tool in our proof is the stabilized twisted Arthur-Selberg trace formula. The ordinary stable trace formula has been established by Arthur in 
\cite{Arthur:2001}\cite{Arthur:2002}\cite{Arthur:2003}. 
The twisted case results from a long project of M{\oe}glin and Waldspurger \cite{MW:2016}
which has been finished recently. All of these also rest upon Ngo's celebrated proof \cite{Ngo:2010} of the Fundamental Lemma. To give some ideas of the proof of our theorem, we would like to briefly describe two typical kinds of trace formulas used in this paper. Let $\lif{\zeta}$ be a character of $\lZ(F) \backslash \lZ(\A_{F})$. The space of $\lif{\zeta}$-equivariant $L^{2}$-functions over $\lG(F) \backslash \lG(\A_{F})$ can be decomposed into a discrete part and a continuous part:
\[
L^{2}(\lG(F) \backslash \lG(\A_{F}), \lif{\zeta}) = L^{2}_{disc}(\lG, \lif{\zeta}) \+ L^{2}_{cont}(\lG, \lif{\zeta}).
\]
If we take a $\lif{\zeta}^{-1}$-equivariant smooth compactly supported function $\lf = \otimes_{v} \lf_{v}$ over $\lG(\A)$, then we can define an operator on $L^{2}_{disc}(\lG, \lif{\zeta})$ by 
\[
(R^{\lG}_{disc}(\lf) \varphi)(x) = \int_{\lZ(\A_{F}) \backslash \lG(\A_{F})} \lf(y) \varphi(xy) dy, \,\,\,\,\, \varphi \in L^{2}_{disc}(\lG, \lif{\zeta}).
\]
M{\"u}ller \cite{Muller:1989} showed this operator $R^{\lG}_{disc}(\lf)$ is of trace class, so we can write
\[
tr R^{\lG}_{disc}(\lf) = \sum_{\lr} m(\lr) \lf_{\lG}(\lr). 
\]
where the sum is over all irreducible admissible representations of $\lG(\A_{F})$ and $m(\lr)$ is the multiplicity of $\lr$ in $L^{2}_{disc}(\lG, \lif{\zeta})$. We define the discrete part of the trace formula to be the following distribution 
\[
I^{\lG}_{disc}(\lf) = tr R^{\lG}_{disc}(\lf) + ``\text{ symmetric part in } L^{2}_{cont}(\lG, \lif{\zeta})",
\]
where the symmetry on the continuous spectrum is given by the action of the regular elements of the relative Weyl groups (see Section~\ref{subsec: stable trace formula}). The stable trace formula gives a stabilization of this distribution $I^{\lG}_{disc}(\lf)$, and it relates the ``error terms" to the stable distributions on some smaller groups, i.e, elliptic endoscopic groups of $\lG$. We state it in the following theorem.

\begin{theorem}[Arthur]
By induction, one can define a stable distribution
\begin{align}
\label{eq: stable trace formula}
S^{\lG}_{disc}(\lf) = I^{\lG}_{disc}(\lf) - \sum_{\lG'} \iota(\lG, \lG')S^{\lG'}_{disc}(\lf^{\lG'}),
\end{align}
where the sum is over elliptic endoscopic groups $\lG' \neq \lG$ of $\lG$, $\iota(\lG, \lG')$ are some constants (see \eqref{formula: endoscopic coefficient}), and $\lf \rightarrow \lf^{\lG'}$ is the Langlands-Shelstad-Kottwitz transfer.
\end{theorem}

The relation between $S^{\lG}_{disc}(\lf)$ and L-packets can be described in the following conjecture.

\begin{conjecture}[Stable Multiplicity Formula]
\label{conj: global conjecture}
\[
S^{\lG}_{disc}(\lf) = \sum_{\lq \in \Q{\lG}} a_{\lq} S^{\lG}_{\lq}(\lf),
\]
and 
\[
S^{\lG}_{\lq}(\lf) = \prod_{v} \lf_{v}(\lq_{v}),
\]
where $\lf_{v}(\lq_{v})$ is a linear combination of Harish-Chandra characters in some finite subset $\Pkt{\lq_{v}}$ of $\Pkt{}(\lG(F_{v}))$, which defines a stable distribution on $\lG(F_{v})$. Moreover, there is an explicit formula for the constants $a_{\lq}$.
\end{conjecture}

In this conjecture, $\Q{\lG}$ is the set of so-called global Arthur parameters of $\lG$, which generalizes the set $\P{\lG}$ of bounded global Langlands parameters. The global packet $\Pkt{\lq} = \bigotimes'_{v}\Pkt{\lq_{v}}$ associated with the stable distribution $S^{\lG}_{\lq}(\lf)$ is called a global Arthur packet. One can view the global L-packets as a special case of global Arthur packets, and the local L-packets that we are looking for will be the local components of some global L-packets, which contribute to $S^{\lG}_{disc}(\lf)$. Before we can talk about how to isolate a global L-packet from $S^{\lG}_{disc}(\lf)$, we want to introduce the twisted version of \eqref{eq: stable trace formula} first. 

Let $\x$ be a character of $\lG(\A_{F})/\lG(F)G(\A_{F})$ and $\theta \in \Sigma_{0}$, we define the discrete part of the $(\theta, \x)$-twisted trace formula to be
\[
I^{(\lG^{\theta}, \x)}_{disc}(\lf) = tr (R(\theta)^{-1} \circ R(\x) \circ R^{\lG}_{disc}(\lf)) + ``\text{ $(\theta, \x)$-twisted symmetric part in } L^{2}_{cont}(\lG, \lif{\zeta})",
\]
where $R(\theta)$ is induced by action on $\lG(\A_{F})$ by $\theta$, and $R(\x)$ is induced by multiplication on $L^{2}_{disc}(\lG, \lif{\zeta})$ by $\x$. Then the stabilization of $I^{(\lG^{\theta}, \x)}_{disc}(\lf)$ is given by the following theorem.

\begin{theorem}[Moeglin and Waldspurger]
\begin{align}
\label{eq: twisted stable trace formula}
I^{(\lG^{\theta}, \x)}_{disc}(\lf) = \sum_{\lG'} \iota(\lG, \lG')S^{\lG'}_{disc}(\lf^{\lG'}),
\end{align}
where the sum is over $(\theta, \x)$-twisted elliptic endoscopic groups $\lG'$ of $\lG$.
\end{theorem}

One application of \eqref{eq: twisted stable trace formula} is it gives a multiplicity formula for the automorphic representations of $\lG$. Let $X = \Hom(\lG(\A_{F}) / \lZ(\A_{F})G(\A_{F}), \C^{\times})$ and  $Y = \Hom(\lG(\A_{F}) / \lG(F) \lZ(\A_{F})G(\A_{F}), \C^{\times})$. If $\lr$ is an irreducible admissible representation of $\lG(\A_{F})$, we write $Y(\lr) = \{\x \in Y: \lr \cong \lr \otimes \x\}$, which is finite.

\begin{proposition}
\label{prop: global result 1}
\begin{enumerate}
\item
Suppose $\lr$ is a discrete automorphic representation of $\lG$, and $\r$ is an irreducible constituent of $\lr$ restricted to $G(\A_{F})$. If $[\r] \in \cPkt{\p}$ for $\p \in \cP{G}$, then
\begin{align}
m(\lr) = m_{\lp} |Y(\lr) /\a(\S{\p})|,   
\end{align}
where $\a(\S{\p})$ (see \eqref{eq: global twisted endoscopic sequence}) is a subgroup of $Y(\lr)$,  $m_{\lp} = 1 \text{ or } 2$. Moreover, $m_{\lp} = 2$ only when $G$ is special even orthogonal, $\p \notin \P{G^{\theta_{0}}}$ (see Section~\ref{subsec: substitute Langlands parameter}), and $\lr \otimes \x \cong \lr^{\theta_{0}}$ for some $\x \in Y$.
\item
Suppose $\lr$ and $\lr'$ are discrete automorphic representations of $\lG$, and there exists $\x \in X$ such that $\lr_{v}$ is $\Sigma_{0}$-conjugate to $\lr'_{v} \otimes \x_{v}$ for all places. If $\r$ is an irreducible constituent of $\lr$ restricted to $G(\A_{F})$ and $[\r] \in \cPkt{\p}$ for $\p \in \cP{G}$, then there exists some $\x' \in Y$ and $\theta \in \Sigma_{0}$ such that $\lr' \cong \lr^{\theta} \otimes \x'$.

\end{enumerate}

\end{proposition}

Back to the proof of Theorem~\ref{thm: main local theorem}, a key step is to isolate the global L-packets from the stable distribution $S^{\lG}_{disc}(\lf)$ for $\lf = \otimes_{v}\lf_{v}$ such that $\lf_{v}$ is $\Sigma_{0}$-invariant. By the theory of multipliers, one can isolate the parts associated with different equivalence classes of families of Satake parameters. For $\p \in \cP{G}$, the equivalence class $c(\p)$ determines the packet $\cPkt{\p}$ of $G$ uniquely from our previous discussion. But this may not be the case for $\lG$. In view of part (2) of Proposition~\ref{prop: global result 1}, this means if the global L-packet $\cPkt{\lp}$ exists for $\lp \in \cP{\lG}$, there might exist $\x \in Y$ such that $\cPkt{\lp} \neq \cPkt{\lp}\otimes \x$, whereas $\cPkt{\lp_{v}} = \cPkt{\lp_{v}} \otimes \x_{v}$ for almost all places. For our proof, we only need something weaker, that is we will fix a nonarchimeadan place $u$, and we require if $\cPkt{\lp_{v}} = \cPkt{\lp_{v}} \otimes \x_{v}$ for all places $v \neq u$, then $\cPkt{\lp} = \cPkt{\lp} \otimes \x$. Such global parameters can be constructed using the result of Sug-Woo Shin on automorphic plancherel density \cite{Shin:2012}. At last, we prove the following global result, which is parallel with Theorem~\ref{thm: global L-packet for G}.

\begin{theorem}
\label{thm: global result 2}
For $\p \in \cP{G}$ satisfying $\S{\lp} = 1$ (see Section~\ref{subsec: Langlands parameters}), there exists a global L-packet  
\[
\cPkt{\lp} = \otimes'_{v} \cPkt{\lp_{v}}
\]
of $\lG$ unique up to twisting by $Y$, such that if $\lr$ is an automorphic representation of $\lG$ whose irreducible constituents in the restriction to $G(\A_{F})$ are contained in $\cPkt{\p}$, then $[\lr]$ is contained in $\cPkt{\lp} \otimes \x$ for some $\x \in Y$.
\end{theorem}

The local and global results of this paper will be proved together by a complicated induction argument. For the purpose of giving a clear proof of the local results, we have minimized the global assumptions needed in our induction arguments by imposing very restrictive conditions on the global results (like in Theorem~\ref{thm: global result 2}). In a sequel to this paper, we will prove the global results of this paper in a more general setting. 

A full description of the discrete spectrum of $\lG$ will also require the Arthur packets. Unfortunately, the technique in this paper will not be sufficient for that. This is somehow reflected by the fact that the Arthur packets of $G$ can have more complicated structure than its L-packets. To be able to construct the Arthur packets of $\lG$ in the nonarchimedean case, one will need to extend the works of M{\oe}glin \cite{Moeglin:2006} \cite{Moeglin:2009} on explicit construction of the Arthur packets of $G$. The global case could be even more challenging, because that would require certain description of the residue spectrum for both $G$ and $\lG$. So we would like to keep that as a project for the future.

This paper is organized as follows. In Section~\ref{sec: preliminary}, we discuss various group theoretic properties about $G$ and $\lG$. We introduce their Levi subgroups and twisted endoscopic groups. We also discuss the relation between $\cP{G}$ and $\cP{\lG}$ both in the local and global cases. We recall some known results about restricting the local representations of $\lG$ to $G$, in particular we have restriction multiplicity one in this case. In Section~\ref{sec: Arthur's theory}, we review Arthur's endoscopic classification theory for $G$ in the tempered case. In the local theory, we describe the $\theta$-twisted endoscopic character identities (or character relations) for $G$ and $\theta \in \Sigma_{0}$. In the global theory, we give Arthur's multiplicity formula for automorphic representations of $G$. In Section~\ref{sec: coarse L-packet}, we state our main local theorem (Theorem~\ref{thm: refined L-packet}). In this theorem, we formulate the $(\theta, \x)$-twisted endoscopic character identities for $\lG$ and $\x \in X$, which are natural extensions of the $\theta$-twisted endoscopic character identities for $G$. Similarly we also formulate the natural extensions of the twisted local intertwining relations from $G$ to $\lG$. In Section~\ref{sec: multiplicity formula}, we introduce various stable trace formulas used in this paper. We prove Proposition~\ref{prop: global result 1} as an application of the twisted stable trace formula. We also state some global conjectures, whose special versions have to be proved together with our main local theorem. In particular, we give the precise statement of Conjecture~\ref{conj: global conjecture} in the tempered case. In the end of this section, we make a comparison of both sides of the twisted stable trace formulas for $\lG$, which is analogous to what Arthur did for $G$. In the final section, we give the proofs of our main local theorem together with all the global theorems by an induction argument. In particular, we address the issue of lack of strong multiplicity one as we mentioned above.

{\bf Some standard notations}: If $G$ is a reductive group over a field $F$, let $G^{0}$ be the identity component, $G_{der}$ be the derived group of $G^{0}$, $G_{sc}$ be the simply connected cover of $G_{der}$, and $G_{ad}$ be the adjoint group of $G_{der}$. We denote the centre of $G$ by $Z_{G}$ or $Z(G)$, the split connected component of $Z_{G}$ by $A_{G}$. If $G$ is connected, let $X^{*}(G)$ be the group of algebraic characters of $G$ over $F$ and $\mathfrak{a}_{G} = \Hom_{\mathbb{Z}}(X^{*}(G), \mathbb{R})$. If $F$ is a local field, there is a homomorphism $H_{G}: G(F) \rightarrow \mathfrak{a}_{G}$ defined by $e^{<H_{G}(g), \chi>} = |\chi(g)|_{F}$ for $g \in G(F)$ and $\chi \in X^{*}(G)$. If $G$ is abelian and $\theta$ is an automorphism of $G$, let $G^{\theta}$ be the $\theta$-invariant subgroup of $G$, and $G_{\theta}$ be the $\theta$-coinvariant group of $G$, i.e., $G_{\theta} = G / (\theta - 1)G$. If $A$ is a locally compact abelian group, we denote its Pontryagin dual by $A^{*}$. 

{\bf Acknowledgements}: The author wants to thank his thesis advisor James Arthur for his generous support and constant encouragement when this work was carried out. He also wants to thank the hospitality of the Institute for Advanced Study, where he finished writing up the current version. During his stay at IAS, he was supported by the National Science Foundation No. DMS-1128155 and DMS-1252158. At last, the author wants to thank the referee for many helpful comments and suggestions.

\section{Preliminary}
\label{sec: preliminary}

\subsection{Groups}
\label{subsec: groups}

\subsubsection{Similitude groups}
\label{subsubsec: notations}

Let $F$ be a local or global field of characteristic zero and $\bar{F}$ be its algebraic closure. When $F$ is global, let us denote the adele ring over $F$ by $\A_{F}$, and the id\`ele group by $I_{F}$. The absolute Galois group over $F$ is written as $\Gal{F}$ or $\Gal{}$ for abbreviation. Let $G$ be a quasisplit connected reductive group over $F$ and $D$ be a torus. We denote by $\lG$ an extension of $D$ by $G$
\begin{align}
\label{eq: extension}
\xymatrix{1 \ar[r] & G \ar[r] & \lG \ar[r]^{\c}  & D \ar[r] & 1. }
\end{align}
Let us denote the centres of $G$ and $\lG$ by $\Z$ and $\lZ$ respectively. Sometimes we need to distinguish $\c$ for different groups, so we will also write $\c_{G} = \c$. The primary example that we are going to consider in this paper is when $G$ is a special even orthogonal group or a symplectic group, and $\lG$ is the corresponding similitude group, in which case $\c$ is called the similitude character.

A split general symplectic group (or symplectic similitude group) is defined as follows
$$GSp(2n) = \{ g \in GL(2n) : g  
\begin{pmatrix} 
      0 & -J_{n} \\
      J_{n} &  0 
\end{pmatrix}
{}^tg = \c(g) 
\begin{pmatrix}
      0 & -J_{n} \\
      J_{n} &  0 
\end{pmatrix}
\},$$
where 
\(
J_{n} = 
\begin{pmatrix}
        &&&1\\
        &&1&\\
        &\iddots&&\\
        1&&&
\end{pmatrix}
\)
 and $\c(g)$ is a scalar. It is connected as an algebraic group. A split general even orthogonal group (or orthogonal similitude group) is defined by
$$GO(2n) = \{ g \in GL(2n) : g  
\begin{pmatrix} 
      0 &  J_{n} \\
      J_{n} &  0 
\end{pmatrix}
{}^tg = \c(g)
\begin{pmatrix} 
      0 &  J_{n} \\
      J_{n} &  0 
\end{pmatrix}
\}.$$
Since $(\det g)^{2} = \c(g)^{2n}$, it has two connected components depending on whether $\det g / \c(g)^{n}$ being $1$ or $-1$. Let us denote the identity component by $GSO(2n)$, and we call it the connected general even orthogonal group. Because $SO(2n)$ (resp.$GSO(2n)$) has an outer automorphism from the conjugate action by $O(2n)$ (resp.$GO(2n)$), let us denote an outer twist of $SO(2n)$ (resp.$GSO(2n)$) with respect to this outer automorphism and an arbitrary quadratic extension $E / F$ by $SO(2n, \eta)$ (resp.$GSO(2n, \eta)$), where $\eta$ is the quadratic (id\`ele class) character associated to $E / F$ by the local (global) class field theory. We would like to allow $E = F$ and $\eta = 1$, in which case this is the split group. If $G = SO(2n, \eta)$, we define $\eta_{G} = \eta$. If $G = Sp(2n)$, we define $\eta_{G} = 1$. These groups that we have defined above give all the quasisplit general symplectic groups and quasisplit connected general even orthogonal groups. 

Another description of quasisplit general symplectic groups and quasisplit connected general even orthogonal groups is given by 
\[
GSp(2n) = (\mathbb{G}_{m} \times Sp(2n)) / (\Two) \,\, \text{ and } \,\, GSO(2n, \eta) = (\mathbb{G}_{m} \times SO(2n, \eta)) / (\Two),
\] 
where $\Two$ is embedded diagonally into the centre of each factor. The similitude character $\c$ is square on $\mathbb{G}_{m}$ and trivial on the other factor. More generally we can define 
\begin{align}
\label{eq: similitude}
G(Sp(2n_{1}) \times \cdots \times Sp(2n_{s}) \times SO(2n_{s+1}, \eta_{1}) \times \cdots \times SO(2n_{s+t}, \eta_{t}))
\end{align}
to be 

\[
(\mathbb{G}_{m} \times Sp(2n_{1}) \times \cdots \times Sp(2n_{s}) \times SO(2n_{s+1}, \eta_{1}) \times \cdots \times SO(2n_{s+t}, \eta_{t})) / (\Two),
\]
where $\Two$ is again embedded diagonally. We can also generalize the similitude character $\c$ to these groups such that it is square on $\mathbb{G}_{m}$ and trivial on all the other factors.
At last let us write $GSp(0) = GSO(0) = \mathbb{G}_{m}$ and set $\c = id$ in this case.  

For any quasisplit connected reductive group $G$ defined over $F$, we denote by $\D{G}$ its complex dual group, by $Z(\D{G})$ the centre of $\D{G}$, and by $\L{G}$ its $L$-group, which is a semidirect product of $\D{G}$ with the Weil group $W_{F}$, i.e., $\D{G} \rtimes W_{F}$. Then dual to the extension \eqref{eq: extension}, we have 
\begin{align*}
\xymatrix{1 \ar[r] & \D{D} \ar[r] & \D{\lG} \ar[r]^{\bold{p}}  & \D{G} \ar[r] & 1,}
\end{align*}
where all the homomorphisms can be extended to $L$-homomorphisms of $L$-groups. If $\lG$ is $GSp(2n)$ or $GSO(2n, \eta)$, then $\D{\lG}$ is the general Spin group
\[
GSpin(2n+1, \C) = (\C^{\times} \times Spin(2n+1, \C)) / (\Two) \,\, \text{ or } \,\, GSpin(2n, \C) = (\C^{\times} \times Spin(2n, \C)) / (\Two),
\] 
where $\Two$ is embedded diagonally to the centre of each factor. Here the embedding needs to be specified. Note that the Spin group is an extension of the special orthogonal group by $\Two$. If we denote the generator of this $\Two$ by $z$, then in defining the general Spin group we want $\Two$ to be embedded to $<z>$ for the Spin factor. In fact,
$Z(Spin(2n + 1, \C)) = <z>$, and for $Z(Spin(2n, \C))$ there is an exact sequence 
\[
\xymatrix{1 \ar[r] & <z> \ar[r] & Z(Spin(2n, \C)) \ar[r] & Z(SO(2n, \C)) \ar[r] & 1}.
\]
We take a preimage of the generator of $Z(SO(2n, \C)) \cong \Two$ in $Z(Spin(2n, \C))$ and denote it by $w$, then it is well-known that $w^{2} = 1$ if $n$ is even and $w^{2} = z$ if $n$ is odd. On the other hand, $Z(GSpin(2n+1, \C)) \cong \C^{\times}$, and $Z(GSpin(2n, \C)) \cong \C^{\times} \times \Two$. This is because when $n$ is even $u = (1, w)$  (resp. $u = (\sqrt{-1}, w)$ when $n$ is odd) splits the exact sequence
\[
\xymatrix{1 \ar[r] & \C^{\times} \ar[r] & Z(GSpin(2n, \C)) \ar[r] & Z(SO(2n, \C)) \ar[r] & 1.}
\]
$\L{\lG}$ is $GSpin(2n+1, \C) \times W_{F}$ or $GSpin(2n, \C) \rtimes W_{F}$ where the action of $W_{F}$ on $GSpin(2n, \C)$ factors through the Galois group $\Gal{E/F}$ of the quadratic extension $E/F$ associated with $\eta$, and it acts trivially on $\C^{\times}$. It is interesting to see its action on the centre of $GSpin(2n, \C)$. If $\tau$ is the nontrivial element in $\Gal{E/F}$, then $\tau$ is trivial on the factor $\C^{\times}$ and 
\begin{align}
\label{eq: Galois action}
\tau(u) = (-1) \cdot u,  \text{ for } -1 \in \C^{\times}.
\end{align}
If $\lG$ is type \eqref{eq: similitude}, then $\D{\lG}$ is 
\[
(\C^{\times} \times Spin(2n_{1}+1, \C) \times \cdots \times Spin(2n_{s}+1, \C) \times Spin(2n_{s+1}, \C) \times \cdots \times Spin(2n_{s+t}, \C)) / (\Two)^{s+t},
\]
where $(\Two)^{s+t}$ is embedded as the subgroup generated by
\[
(-1, \overbrace{1, \cdots, 1}^{k-1}, z, \overbrace{1, \cdots, 1}^{s+t-k} \,)
\]
for $1 \leqslant k \leqslant s+t$. For $\L{\lG}$, the action of $W_{F}$ on $\D{\lG}$ factors through the Galois group $\Gal{E'/F}$, where $E'$ is the composite field $E_{1} E_{2} \cdots E_{t}$ for the quadratic extensions $E_{i}/F$ associated with $\eta_{i}$, and it acts on each factor as in the previous case.

\begin{lemma}
\label{lemma: similitude character}
The image of $\c$ on GSp(2n,F), GSO(2n,F) is $F^\times$, and on $GSO(2n, \eta)(F)$ is $Nm_{E/F}E^\times$, where $E/F$ is the quadratic extension associated to $\eta$.
\end{lemma}

\begin{proof}
The cases of $GSp(2n)$ and $GSO(2n)$ are obvious, so we will only consider the case that $\lG = GSO(2n, \eta)$. If $n=1$, $GSO(2, \eta)$ can be embedded into $GL(2)$ and $\c$ is given by the determinant map. Since $GSO(2, \eta)(F) = E^{\times}$, it is easy to see that the determinant map becomes the norm map on $E^{\times}$, and the image is $Nm_{E/F}E^\times$. For general $n$, we can take a Borel subgroup $\lif{B}$ of $GSO(2n, \eta)$ with a maximal torus $\lif{T}$ and unipotent radical $\lif{N}$. By the Bruhat decomposition, 
\[
\lG(F) = \bigsqcup_{w \in W(\lif{T}(F), \lG(F))} \lif{B}(F) \dot{w} \lif{B}(F),
\] 
where $\dot{w}$ are representatives of $w$ in $\lG(F)$. Since $W(\lif{T}(F), \lG(F)) \cong W(T(F), G(F))$ for $T = G \cap \lif{T}$, one can take $\dot{w}$ in $G(F)$. Moreover, $\lif{N} = N$ for $N = G \cap \lif{N}$. Therefore, $\c(\lG(F)) = \c(\lif{B}(F)) = \c(\lif{T}(F))$. Let us write $\lG = (\mathbb{G}_{m} \times G) / (\Two)$, and choose $\lif{T}(\bar{F})$ such that it consists of $(x, g)$ modulo $\Two$, where $x \in \bar{F}^{\times}$ and 
\[
g = \text{diag}\{z_{1}, \cdots, z_{n-1}, y, y^{-1}, z_{n-1}^{-1}, \cdots, z_{1}^{-1} \} \in G(\bar{F})
\]
with $z_{i}, y \in \bar{F}^{\times}$. If $(x, g) \in \lif{T}(F)$, then $(x, y) \in GSO(2, \eta)(F)$, and $\c(x, g) = \c_{SO(2, \eta)}(x, y) = x^{2}$. On the other hand, if $(x, y) \in GSO(2, \eta)(F)$, then by letting $z_{i} = x$ for $1 \leqslant i \leqslant n-1$, we have $(x, g) \in \lif{T}(F)$. This shows $\c(\lif{T}(F)) = \c_{SO(2, \eta)}(GSO(2, \eta)(F)) = Nm_{E/F}E^\times$.

\end{proof}

This lemma can be easily generalized to groups of type \eqref{eq: similitude}.

\begin{lemma}
\label{lemma: generalized similitude character}
Suppose $\lG$ is of type \eqref{eq: similitude}, the image of $\c$ on $\lG(F)$ is 
\[
F^{\times} \cap Nm_{E_{1}/F}E_{1}^{\times} \cap \cdots \cap Nm_{E_{t}/F}E_{t}^{\times},
\]
where $E_{i} / F$ is the quadratic extension associated to $\eta_{i}$ for $1 \leqslant i \leqslant t$.
\end{lemma}

\begin{proof}
Let us denote by $\lif{\lG}$ the product 
\begin{align}
\label{eq: product group}
GSp(2n_{1}) \times GSp(2n_{2}) \times \cdots \times GSp(2n_{s}) \times GSO(2n_{s+1}, \eta_{1}) \times \cdots \times GSO(2n_{s+t}, \eta_{t}),
\end{align}
then $\lG$ is the subgroup of $\lif{\lG}$ characterized by $\c_{1}(g_{1}) = \cdots = \c_{s+t}(g_{s+t})$ for $(g_{1}, \cdots, g_{s+t}) \in \lif{\lG}$. In particular, $\c(g) = \c_{1}(g_{1})$ for $g \in \lG \subseteq \lif{\lG}$. Then the lemma follows immediately from Lemma~\ref{lemma: similitude character}.
\end{proof}

When $F$ is global, we have the following corollary, whose proof is obvious.

\begin{corollary}
\label{cor: similitude character}
Suppose $\lG$ is of type \eqref{eq: similitude}, the image of $\c$ on $\lG(\A_{F})$ is 
\[
I_{F} \cap Nm_{E_{1}/F}I_{E_{1}} \cap \cdots \cap Nm_{E_{t}/F}I_{E_{t}},
\]
where $E_{i} / F$ is the quadratic extension associated to $\eta_{i}$ for $1 \leqslant i \leqslant t$.
\end{corollary}

\begin{corollary}
\label{cor: relative Hasse principle}
Suppose $\lG$ is of type \eqref{eq: similitude}, then $\c(\lG(\A_{F})) \cap F^{\times} = \c(\lG(F))$ and $\c(\lZ(\A_{F})) \cap F^{\times} = \c(\lZ(F))$.
\end{corollary}

\begin{proof}
For the first equality, by Lemma~\ref{lemma: generalized similitude character} and Corollary~\ref{cor: similitude character} it suffices to show $Nm_{E_{i}/F}I_{E_{i}} \cap F^{\times} = Nm_{E_{i}/F} E_{i}^{\times}$ for all $1 \leqslant i \leqslant t$, and this is a consequence of Hasse norm theorem (see \cite{Neukirch:1999}, Corollary VI.4.5). For the second equality, note $\c(\lZ(\A_{F})) = I_{F}^{2}$ and $\c(\lZ(F)) = F^{\times^{2}}$. So we need to show $F^{\times} \cap I_{F}^{2} = F^{\times^{2}}$, and this follows from Grunwald-Wang theorem.
\end{proof}

\subsubsection{Levi subgroups}
\label{subsubsec: Levi subgroups}

Let $\lG$, $G$, $D$ and $\c$ be defined as in Section~\ref{subsubsec: notations}. If we restrict $\c$ to a Levi subgroup $\lM$ of $\lG$, then its kernel will be a Levi subgroup $M$ of $G$, and we have 
\[
\label{eq: Levi subgroup}
\xymatrix{1 \ar[r] & M \ar[r] & \lM \ar[r]^{\c}  & D \ar[r] & 1.}
\]
It is easy to see that this induces a bijection between Levi subgroups of $\lG$ and $G$. 

Suppose $\lG$ is a general symplectic group or a connected general even orthogonal group of semisimple rank $n$, then $\lM$ is isomorphic to
\begin{align}
\label{eq: levi}
GL(n_{1}) \times \cdots \times GL(n_{r}) \times \lG_{-},
\end{align}
where $\lG_{-}$ is of the same type as $\lG$ with semisimple rank $n_{-} \geqslant 0$ and $n = \sum_{i = 1}^{r} n_{i} + n_{-}$. Throughout this paper we fix a Borel subgroup $\lif{B}$ of $\lG$ consisting of upper-triangular matrices and we choose $\lM$ to be contained in the group 
\[
\begin{pmatrix}
GL(n_{1})&&&&&&0 \\
&\ddots &&&&& \\
&& GL(n_{r})&&&&\\
&&&\lG_{-} &&&\\
&&&&GL(n_{r})&& \\
&&&&&\ddots&\\
0&&&&&&GL(n_{1})
\end{pmatrix}.
\]
In fact this gives all the standard Levi subgroups if $\lG$ is $GSp(2n)$ or $GSO(2n, \eta)$ ($\eta \neq 1$), and $GO(2n)$-conjugacy classes of standard Levi subgroups if $\lG$ is $GSO(2n)$. We fix an isomorphism from $\eqref{eq: levi}$ to $\lif{M}$ as follows
\[
(g_{1}, \cdots g_{r}, g) \longrightarrow \text{diag}\{g_{1}, \cdots, g_{r}, g, \c(g){}_tg^{-1}_{r}, \cdots, \c(g){}_tg^{-1}_{1}\}
\]
if $n_{-} > 0$, and
\[
(g_{1}, \cdots g_{r}, g) \longrightarrow \text{diag}\{g_{1}, \cdots, g_{r}, \c(g){}_tg^{-1}_{r}, \cdots, \c(g){}_tg^{-1}_{1}\}
\]
if $n_{-} = 0$.
Here ${}_tg_{i} = J_{n_{i}}{}^tg_{i}J^{-1}_{n_{i}}$ for $1 \leqslant i \leqslant r$. Under this isomorphism, the Weyl group $W(\lM) = \Norm(A_{\lM}, G)/\lM$ acts on $\lM$ by permuting the general linear factors and changing some $g_{i}$ to $\c(g){}_tg^{-1}_{i}$ (also compositions of these). Finally, note $M \cong GL(n_{1}) \times \cdots \times GL(n_{r}) \times G_{-}$ and $W(M) \cong W(\lM)$.

\subsubsection{Twisted endoscopic groups}
\label{subsubsec: twisted endoscopy}
Let $G$ be a quasisplit connected reductive group over $F$. When $F$ is local, we have an isomorphism 
\begin{align}
\label{eq: local characters}
H^{1}(W_{F}, Z(\D{G})) \longrightarrow \Hom(G(F), \C^{\times}).
\end{align}
When $F$ is global, we have a homomorphism
\begin{align}
\label{eq: global characters}
H^{1}(W_{F}, Z(\D{G})) \longrightarrow \Hom(G(\A_{F})/G(F), \C^{\times}),
\end{align}
where $\Hom(G(\A_{F})/G(F), \C^{\times})$ denotes the quasicharacter of $G(\A_{F})$ trivial on $G(F)$. If we let $F_{v}$ be the localization of $F$ at place $v$, then there is a commutative diagram
\begin{align*}
\xymatrix{  H^{1}(W_{F}, Z(\D{G})) \ar[d] \ar[r] &  \Hom(G(\A_{F})/ G(F), \C^{\times})  \ar[d] \\
                  H^{1}(W_{F_{v}}, Z(\D{G}_{v})) \ar[r]  & \Hom(G(F_{v}), \C^{\times}). }     
\end{align*}
Let 
\[
\Ker^{1}(W_{F}, Z(\D{G})) :=  \bigcap_{v} \Ker \{H^{1}(W_{F}, Z(\D{G})) \rightarrow H^{1}(W_{F_{v}}, Z(\D{G}_{v}))\},
\]
it is finite and gives the kernel of \eqref{eq: global characters}. Suppose $\lG$, $G$, $D$ and $\c$ are defined as in Section~\ref{subsubsec: notations}, then we have the following fact.

\begin{lemma}
Suppose $Z(\D{G})$ is $\Gal{F}$-invariant and $D$ is split, then $\Ker^{1}(W_{F}, Z(\D{\lG})) = 1$.
\end{lemma}

\begin{proof}
It is a consequence of Chebotarev's density theorem (see \cite{Neukirch:1999}, Corollary VII.13.10) that 
\[
\Ker^{1}(W_{F}, Z(\D{G})) = \Ker^{1}(W_{F}, \D{D}) = 1.
\]
Then we consider the exact sequence 
\[
\xymatrix{1 \ar[r] &  \D{D} \ar[r]   & Z(\D{\lG})  \ar[r]  & Z(\D{G})  \ar[r]  &1,}
\]
it induces a commutative diagram
\[
\xymatrix{\pi_{0}(Z(\D{G})^{\Gal{}}) \ar[d]^{\simeq} \ar[r] & H^{1}(W_{F}, \D{D}) \ar[d] \ar[r] & H^{1}(W_{F}, Z(\D{\lG})) \ar[d] \ar[r] & H^{1}(W_{F}, Z(\D{G})) \ar[d] \\
\pi_{0}(Z(\D{G}_{v})^{\Gal{v}}) \ar[r] & H^{1}(W_{F_{v}}, \D{D}_{v}) \ar[r] & H^{1}(W_{F_{v}}, Z(\D{\lG}_{v}))  \ar[r] & H^{1}(W_{F_{v}}, Z(\D{G}_{v})),}
\]
with both the top and bottom rows being exact. Suppose $u \in \Ker^{1}(W_{F}, Z(\D{\lG}))$, then by the commutativity of the right square and $\Ker^{1}(W_{F}, Z(\D{G})) = 1$, $u$ has a preimage $w$ in $H^{1}(W_{F}, \D{D})$. Since $\Ker^{1}(W_{F}, \D{D}) = 1$, the Langlands correspondence for tori allows us to identify $H^{1}(W_{F}, \D{D})$ with $\Hom(D(\A_{F})/D(F), \C^{\times})$. Now without loss of generality, we can assume $w \neq 1$. By the commutativity of the left square and the fact that the left end vertical map is an isomorphism, the localization of $w$ is determined by the localizations of those in the image of $\pi_{0}(Z(\D{G})^{\Gal{}})$ in $H^{1}(W_{F}, \D{D})$. Finally, we use Chebotarev's density theorem again to conclude that $w$ has to lie in the image of $\pi_{0}(Z(\D{G})^{\Gal{}})$, and hence $u = 1$.

\end{proof}

\begin{corollary}
\label{cor: ker1}
Suppose $\lG$ is of type \eqref{eq: similitude} and $\c$ is the generalized similitude character, then
\[
\Ker^{1}(W_{F}, Z(\D{\lG})) = \Ker^{1}(W_{F}, Z(\D{G})) = 1.
\]
\end{corollary}

\begin{proof}
One just needs to observe that in this case $\Gal{F}$ acts trivially on $Z(\D{G})$, and $D = \mathbb{G}_{m}$.
\end{proof}

Let $\theta$ be an automorphism of $G$, and $\x$ be a quasicharacter of $G(F)$ if $F$ is local, or a quasicharacter of $G(\A_{F})$ trivial on $G(F)$ if $F$ is global. We define a twisted endoscopic datum for $(G, \theta, \x)$ to be a triple $(H, s, \xi)$, where $H$ is a quasisplit connected reductive group over $F$, $s$ is a semisimple element in $\D{G} \rtimes \D{\theta}$, and $\xi$ is an $L$-embedding from $\L{H}$ to $\L{G}$ satisfying the following conditions:
\begin{enumerate}
\item 
\(
\Int(s) \circ \xi = {\bold a} \cdot \xi, 
\)
for a $1$-cocycle ${\bold a}$ of $W_{F}$ in $Z(\D{G})$ which is mapped to $\x$ under \eqref{eq: local characters} or \eqref{eq: global characters};
\item $\D{H} \cong \Cent(s, \D{G})^{0}$ through $\xi$.
\end{enumerate}
Here $H$ is called a twisted endoscopic group of $G$ and for abbreviation we will denote $(H, s, \xi)$ by $H$. In this definition, we have required $\xi$ to be an $L$-embedding of $\L{H}$. But in general, $\xi$ can be an embedding of certain extension group of $\D{H}$ by $W_{F}$, which may not necessarily be isomorphic to $\L{H}$. In that case, one has to consider $z$-pairs (see \cite{KottwitzShelstad:1999}, 2.2). Since we do not need to deal with this general situation in this paper, we are content with the current definition. 

Two twisted endoscopic data $(H, s, \xi)$ and $(H', s', \xi')$ are called isomorphic if there exists an element $g \in \D{G}$ such that $g\xi(\L{H})g^{-1} = \xi'(\L{H'})$ and $gsg^{-1} \in s'Z(\D{G})$. Here such $g$ is called an isomorphism. We denote by $\tEnd{}{G^{\theta}}$ the set of isomorphism classes of twisted endoscopic data for $(G, \theta, \x)$. When $\theta = id$ and $\x =1$, we get the ordinary endoscopic data, and we abbreviate $\End{}{G^{\theta}, \x}$ to $\End{}{G}$.  A twisted endoscopic datum $(H, s, \xi)$ is called {\bf elliptic} if $\xi(Z(\D{H})^{\Gal{F}})^{0} \subseteq Z(\D{G})$, and we denote by $\End{ell}{G^{\theta}, \x}$ the set of isomorphism classes of twisted elliptic endoscopic data for $(G, \theta, \x)$. When $G = GL(N)$, we write $\End{ell}{N^{\theta}}$ for $\End{ell}{GL(N)^{\theta}}$. One can see from the definition that a twisted endoscopic group for $G$ can be viewed as an elliptic endoscopic group of some $\theta$-stable Levi subgroup $M$ (which also admits a $\theta$-stable parabolic subgroup $P \supseteq M$) of $G$. On the other hand, one can obtain all the twisted endoscopic groups of $G$ by taking the Levi subgroups of the twisted elliptic endoscopic groups of $G$. 

If $(H, s, \xi)$ is a twisted endoscopic datum for $(G, \theta, \x)$, we denote the automorphism group of this twisted endoscopic datum by $\Aut_{G}(H)$. By our definition, it is a subgroup of $\D{G}$. We define the inner automorphism group $\Int_{G}(H)$ of this twisted endoscopic datum to be $\D{H}Z(\D{G})^{\Gal{F}}$, and the outer automorphism group to be
\[
\Out_{G}(H) = \Aut_{G}(H) / \Int_{G}(H).
\]
By fixing an $F$-splitting for $H$, we get a homomorphism from $\Out_{G}(H)$ to the outer automorphism group $\Out(H)$ of $H$. Let us denote the image by $\Out(H, G)$, and define 
\[
C := \{z \in Z(\D{G}): \sigma (z) z^{-1} \in Z(\D{G}) \cap \D{H} \text{ for $\sigma \in \Gal{F}$ }\}.
\] 
Then there is an exact sequence
\begin{align}
\label{eq: automorphism of endoscopic datum}
\xymatrix{1 \ar[r] &  C/C \cap \D{H}Z(\D{G})^{\Gal{F}}   \ar[r]   & \Out_{G}(H)  \ar[r]  & \Out(H, G) \ar[r]  &1.}
\end{align}
When $F$ is local, there is an action of $\Out_{G}(H)$ on $\H(H)$ (or equivalently $C^{\infty}_{c}(H(F))$). For $g \in \Out_{G}(H)$, let us  denote its image in $\Out(H, G)$ by $\tau_{g}$ and choose a representative $\dot{g}$ in $\Aut_{G}(H)$ such that $\Int(\dot{g})$ preserves a $\Gal{F}$-splitting of $\D{H}$. Then $b_{g}(w) = \dot{g}^{-1} \xi(1 \rtimes w) \dot{g} \xi(1 \rtimes w)^{-1}$ defines a $1$-cocycle of $W_{F}$ in $Z(\D{H})$ and it induces a quasicharacter $w_{g}$ of $H(F)$ by \eqref{eq: local characters}. So the action of $\Out_{G}(H)$ on $\H(H)$ sends $f(h)$ to $^{g}f(h) = f(\tau_{g}(h)) \x_{g}(h)^{-1}$. In all the cases that we will be considering in this paper, one can always split the exact sequence \eqref{eq: automorphism of endoscopic datum} and get $\Out_{G}(H) \cong \Out(H, G) \times (C/C \cap \D{H}Z(\D{G})^{\Gal{F}})$ such that $\Out(H, G)$ acts on $\H(H)$ through its action on $H(F)$. When $G$ is a product of symplectic groups and special even orthogonal groups, $\Out_{G}(H) \cong \Out(H,G)$. When $G = GL(N)$, we write $\Out_{N}(H)$ for $\Out_{GL(N)}(H)$.

Suppose $\lG$, $G$, $D$, $\c$ are defined as Section~\ref{subsec: groups}. Let $\theta$ be an automorphism of $\lG$, and we assume {\bf $\c$ is $\theta$-invariant}, then $\theta$ also induces an automorphism of $G$. If $\x_{\lG}$ is a quasicharacter associated with $\lG$ as in the setup of twisted endoscopic datum, let us write $\x_{G}$ for the restriction of $\x_{\lG}$ to $G(F)$ if $F$ is local or to $G(\A_{F})/G(F)$ if $F$ is global. The following lemma describes the relation for twisted endoscopic data between $\lG$ and $G$.

\begin{proposition}
\label{prop: lifting endoscopic group}
There is a one to one correspondence between $\End{}{G^{\theta}, \x_{G}}$ and 
\[
\bigsqcup_{\x_{\lG}|_{G} = \x_{G}} \End{}{\lG^{\theta}, \x_{\lG}},
\] 
such that if $G'$ corresponds to $\lG'$, then there exists an exact sequence 
\[
\xymatrix{1 \ar[r] & G' \ar[r]  & \lG' \ar[r]^{\c'} & D \ar[r] & 1 }.
\]
Moreover, the same is true for twisted elliptic endoscopic data.
\end{proposition}

\begin{proof}
We say a twisted endoscopic datum $(\lG', \lif{s}, \lif{\xi})$ for $(\lG, \theta, \x_{\lG})$ corresponds to a twisted endoscopic datum $(G', s, \xi)$ for $(G, \theta, \x_{G})$ if $\bold{p}(\lif{s}) = s$ and they satisfy the following diagram
\[
\xymatrix{ \L{\lG'} \ar[r]^{\lif{\xi}}     \ar[d]^{\bold{p}}   & \L{\lG}   \ar[d]^{\bold{p}} \\
    \L{G'}    \ar[r]^{\xi}     &    \L{G}.}
\]
In \cite{Xu:2016} we have shown this gives a one to one correspondence between isomorphism classes of twisted endoscopic data when $F$ is local. In fact, our proof also works in the global case except that we need to use a global lifting result of Labesse for homomorphisms from the global Weil group to $L$-groups (see the paragraph after Theorem~\ref{thm: lifting parameter}). Moreover, it is easy to see $\lif{\xi}(Z(\D{\lG'})^{\Gal{F}})^{0} \subseteq Z(\D{\lG})$ if and only if $\xi(Z(\D{G}')^{\Gal{F}})^{0} \subseteq Z(\D{G})$, so $(\lG', \lif{s}, \lif{\xi})$ is elliptic if and only if $(G', s, \xi)$ is elliptic.
\end{proof}

\begin{remark}
\label{rk: lifting endoscopic group}

The most important case for us is when $\x_{G} = 1$. Then Proposition~\ref{prop: lifting endoscopic group} shows there is a one to one correspondence between the $\theta$-twisted endoscopic data $\End{}{G^{\theta}}$ and the $(\theta, \x)$-twisted endoscopic data
\[
\bigsqcup_{\x} \End{}{\lG^{\theta}, \x},
\]
where $\x$ runs through quasicharacters of $\lG(F)/G(F)$ if $F$ is local and quasicharacters of $\lG(\A_{F})/\lG(F)G(\A_{F})$ if $F$ is global. The same is true for elliptic endoscopic data.

\end{remark}

As our most important examples, let us consider the general symplectic groups and connected general even orthogonal groups with trivial automorphisms, and we have the following table (cf. \cite{Arthur:2013}, 1.2 and \cite{Morel:2011}, 2.1): let $n = n_{1} + n_{2}$.

\begin{spacing}{1.5}
\begin{center}
\begin{tabular}{| c | m{5cm}  || c | m{5cm} | }
     \hline
      $G$                        &      $\End{ell}{G}$                                                                &     $\lG$              &       $\tEnd{ell}{\lG}$ \\
     \hline
      $Sp(2n)$               &      $Sp(2n_1) \times SO(2n_2, \eta) $                           &     $GSp(2n)$         &      $G(Sp(2n_1) \times SO(2n_2, \eta))$ \newline  $ \omega = \eta \circ \c$ \\
     \hline
      $SO(2n)$              &      $SO(2n_1, \eta) \times SO(2n_2, \eta)$                    &     $GSO(2n)$        &      $G(SO(2n_1, \eta) \times SO(2n_2, \eta))$ \newline  $ \omega = \eta \circ \c$ \\
     \hline
      $SO(2n, \eta)$  &    $SO(2n_1, \eta_1) \times SO(2n_2, \eta_1\eta) $ & $GSO(2n, \eta')$  & $G(SO(2n_1, \eta) \times SO(2n_2, \eta \eta'))$ \newline $\omega = \eta \circ \c$ \\
     \hline
\end{tabular}
\end{center}
\end{spacing}
 
Note in the cases above the isomorphism classes of twisted endoscopic data are completely determined by the twisted endoscopic groups. But that is not the case in general. For example, in the case of connected general even orthogonal groups, if we let $\theta_{0}$ be the outer automorphism induced by the conjugate action of the full orthogonal group, then the isomorphism classes of $\theta_{0}$-twisted elliptic endoscopic data of $SO(2n, \eta')$ are classified by $\theta_{0}$-twisted elliptic endoscopic groups $Sp(2n_1) \times Sp(2n_2)$ $(n = n_{1} + n_{2} +1)$ with a pair of quadratic characters $(\eta,\eta\eta')$. Correspondingly, the isomorphism classes of $(\theta_{0},\omega)$-twisted elliptic endoscopic data of $GSO(2n, \eta')$ are classified by $(\theta_{0}, \x)$-twisted elliptic endoscopic groups $G(Sp(2n_1) \times Sp(2n_2))$ $(n = n_{1} + n_{2} +1)$ with a pair of quadratic characters $(\eta,\eta\eta')$ and $\omega =\eta \circ \c$. 


\subsection{Langlands parameters}
\label{subsec: Langlands parameters}

Suppose $G$ is a quasisplit connected reductive group over $F$, a Langlands parameter of $G$ is a $\D{G}$-conjugacy class of admissible homomorphisms from the Langlands group $L_{F}$ to the $L$-group of $G$ (cf. \cite{Borel:1979}). We denote the set of Langlands parameters of $G$ by $\P{G}$, and for any $\p \in \P{G}$ we denote a representative by $\underline{\p}: L_{F} \rightarrow \L{G}$. If $F$ is local, the Langlands group is defined as follows, 
\[
L_{F} = 
\begin{cases}
W_{F} & \text{if $F$ is archimedean}, \\ 
W_{F} \times SL(2, \C) & \text{if $F$ is nonarchimedean}.
\end{cases}
\] 
If $F$ is global, the existence of Langlands group is still conjectural. Let $\lG$, $G$, $D$ and $\c$ be defined as in Section~\ref{subsubsec: notations}. The following theorem shows the relation for local Langlands parameters between $G$ and $\lG$.

\begin{theorem}[Labesse]
\label{thm: lifting parameter}
Suppose $F$ is a local, every Langlands parameter $\p$ of $G$ can be lifted to a Langlands parameter $\lp$ of $\lG$ in the sense that the following diagram commutes
\begin{displaymath}
 \xymatrix{ L_{F} \ar[rr]^{\underline{\lp}}  \ar[drr]_{\underline{\p}} & & \L{\lG} \ar[d]^{\bold{p}}  \\
& & \L{G}.}
\end{displaymath}
\end{theorem}

In fact the global analogue of this theorem is also true if one uses the the global Weil group $W_{F}$ instead of the global Langlands group $L_{F}$. Both the local and global cases are proved in (\cite{Labesse:1985}, Theorem 8.1). In the global case, since we do not have the global Langlands group yet, this kind of lifting theorem for global Langlands parameters is unavailable. However let us assume the existence of global Langlands group and also the same kind of lifting theorem at this moment, so that we can investigate the relation for both local and global Langlands parameters between $G$ and $\lG$ in a uniform way. Moreover, the consequences of this investigation will serve as motivations for the later definitions (see Section~\ref{sec: Arthur's theory}) that complement the lack of global Langlands group. 

To further simplify our discussion, we are going to assume 
\begin{align}
\label{eq: group assumption}
\Ker^{1}(W_{F}, Z(\D{\lG})) = \Ker^{1}(W_{F}, Z(\D{G})) = \Ker^{1}(W_{F}, \D{D})= 1
\end{align}
when $F$ is global. This assumption allows us to treat the local and global cases at the same time, and it also suffices for our purpose in view of Corollary~\ref{cor: ker1}. Let $\Sigma$ be a finite abelian group of automorphisms of $\lG$ preserving an $F$-splitting of $\lG$, and we assume $\c$ is {\bf $\Sigma$-invariant}, so $\Sigma$ also acts on $G$. We denote the dual automorphisms by $\D{\Sigma}$ and form the semidirect products $\D{\lG}^{\Sigma} := \D{\lG} \rtimes \D{\Sigma}$ and $\D{G}^{\Sigma} := \D{G} \rtimes \D{\Sigma}$. Let $\Sigma$ act on $\P{\lG}$ and $\P{G}$ through the action of $\D{\Sigma}$ on $\D{\lG}$ and $\D{G}$ respectively. For $\theta \in \Sigma$, we denote by $\P{G^{\theta}}$ the set of $\p \in \P{G}$ such that $\p^{\theta} = \p$.

Suppose $F$ is either local or global, for any $\p \in \P{G}$ we choose a representative $\underline{\p}$. Let $L_{F}$ act on $\D{D}$, $\D{G}^{\Sigma}$, and $\D{\lG}^{\Sigma}$ by conjugation through $\underline{\p}$. We denote the corresponding group cohomology by $H^{*}_{\underline{\p}}(L_{F}, \cdot)$. Note $H^{0}_{\underline{\p}}(L_{F}, \D{D}) = \D{D}^{\Gal{}}$, $H^{1}_{\underline{\p}}(L_{F}, \D{D}) = H^{1}(W_{F}, \D{D})$, and 
\[
S^{\Sigma}_{\underline{\p}}: = \Cent(\Im \underline{\p}, \D{G}^{\Sigma}) = H^{0}_{\underline{\p}}(L_{F}, \D{G}^{\Sigma}), 
\]
\[
S_{\underline{\lp}}^{\Sigma} := \Cent(\Im \underline{\lp}, \D{\lG}^{\Sigma}) = H^{0}_{\underline{\p}}(L_{F}, \D{\lG}^{\Sigma}).
\]
The short exact sequence 
\[
\xymatrix{1 \ar[r] & \D{D} \ar[r] & \D{\lG}^{\Sigma} \ar[r]  & \D{G}^{\Sigma} \ar[r] & 1}
\]
induces a long exact sequence
\begin{align*}
\xymatrix{1 \ar[r] &  \D{D}^{\Gamma} \ar[r] & S_{\underline{\lp}}^{\Sigma} \ar[r] & S_{\underline{\p}}^{\Sigma} \ar[r]^{\delta \quad \quad} & H^{1}(W_{F}, \D{D}),}
\end{align*}
and hence
\begin{align}
\label{eq: old twisted endoscopic sequence}
\xymatrix{1 \ar[r] &  S_{\underline{\lp}}^{\Sigma}/\D{D}^{\Gal{}} \ar[r]^{\quad \iota} & S_{\underline{\p}}^{\Sigma} \ar[r]^{\delta \quad \quad} & H^{1}(W_{F}, \D{D}).}
\end{align}

To describe $\delta$, we can identify 
\[
S_{\underline{\p}}^{\Sigma} = \{ \lif{s} \in \D{\lG}^{\Sigma}: \lif{s} \underline{\lp}(u) \lif{s}^{-1}\underline{\lp}(u)^{-1} \in \D{D}, \text{ for all } u \in L_{F}\} /  \D{D}. 
\]
Then $\delta(s) : u \longmapsto \lif{s} \underline{\lp}(u) \lif{s}^{-1}\underline{\lp}(u)^{-1}$, where $\lif{s}$ is a preimage of $s$ in $\D{\lG}^{\Sigma}$, and $\delta(s)$ factors through $W_{F}$. About \eqref{eq: old twisted endoscopic sequence} we have the following lemma.

\begin{lemma}
\label{lemma: centralizer}
The image of $\delta$ consists of ${\bold a} \in H^{1}(W_{F}, \D{D})$ such that 
\[
\lp^{\theta} = \lp \otimes {\bold a}
\]
for some $\theta \in \Sigma$, and in particular it is finite.
\end{lemma}

\begin{proof}
We have shown the lemma when $F$ is local in \cite{Xu:2016}. In particular the same argument applies to the global case except for the finiteness of $\Im \delta$. When $F$ is global, we need to use the commutative diagram
\[
\xymatrix{
1 \ar[r] & S^{\Sigma}_{\lp} / \D{D}^{\Gal{}} \ar[r]    \ar@{^{(}->}[d]    &   S^{\Sigma}_{\p} \ar[r]^{\delta \quad \quad}   \ar@{^{(}->}[d]   &   H^{1}(W_{F}, \D{D}) \ar[d] \\
1 \ar[r] & S^{\Sigma}_{\lp_{v}} / \D{D}_{v}^{\Gal{v}} \ar[r]    &   S^{\Sigma}_{\p_{v}}  \ar[r]^{\delta_{v} \quad \quad}  &  H^{1}(W_{F_{v}}, \D{D}_{v}).}
\]
Since $\com[0]{S_{\p}}$ is mapped into $\com[0]{S_{\p_{v}}}$ and $\delta_{v}$ is trivial on $\com[0]{(S^{\Sigma}_{\lp_{v}}/ \D{D}^{\Gal{v}})} = \com[0]{S_{\p_{v}}}$, then $\delta = \prod_{v} \delta_{v}$ is trivial on $\com[0]{S_{\p}}$. This implies the image of $\delta$ is finite.
\end{proof}

Take
\[
\xymatrix{1 \ar[r] & \D{D} \ar[r] & Z(\D{\lG}) \ar[r]  & Z(\D{G}) \ar[r] & 1}
\]
and it induces
\[
\xymatrix{1 \ar[r] &  \D{D}^{\Gamma} \ar[r] & Z(\D{\lG})^{\Gal{}} \ar[r] & Z(\D{G})^{\Gal{}} \ar[r]^{\delta \quad} & H^{1}(W_{F}, \D{D}) \ar[r] & H^{1}(W_{F}, Z(\D{\lG})). }
\]
So $\Ker \delta|_{Z(\D{G})^{\Gal{}}} = Z(\D{\lG})^{\Gal{}} / \D{D}^{\Gamma}$. Let 
\(
\bar{H}^{1}(W_{F}, \D{D}) : = H^{1}(W_{F}, \D{D}) / \delta(Z(\D{G})^{\Gal{}}).
\) 
Taking the quotient of \eqref{eq: old twisted endoscopic sequence} by $Z(\D{G})^{\Gal{}}$, we get
\begin{align}
\label{eq: twisted endoscopic sequence mod center}
\xymatrix{1 \ar[r] &  \cS{\underline{\lp}}^{\Sigma} \ar[r]^{\iota} & \cS{\underline{\p}}^{\Sigma} \ar[r]^{\bar{\delta} \quad \quad}  & \bar{H}^{1}(W_{F}, \D{D}),}                 
\end{align}
where $\cS{\underline{\lp}}^{\Sigma} =  S_{\underline{\lp}}^{\Sigma}/Z(\D{\lG})^{\Gal{}}$ and $\cS{\underline{\p}}^{\Sigma} =  S_{\underline{\p}}^{\Sigma}/Z(\D{G})^{\Gal{}}$. Since $\Im \delta$ is finite, we have $\cS{\underline{\lp}}^{0} = \cS{\underline{\p}}^{0}$. After taking the quotient of \eqref{eq: twisted endoscopic sequence mod center} by the identity component, we get
\begin{align}
\label{eq: twisted endoscopic sequence}
\xymatrix{1 \ar[r] &  \S{\underline{\lp}}^{\Sigma} \ar[r]^{\iota} & \S{\underline{\p}}^{\Sigma} \ar[r]^{\bar{\delta} \quad \quad}  & \bar{H}^{1}(W_{F}, \D{D}),}                
\end{align}
where $\S{\underline{\lp}}^{\Sigma} =  \cS{\underline{\lp}}^{\Sigma} / \cS{\underline{\lp}}^{0}$ and $\S{\underline{\p}}^{\Sigma} =  \cS{\underline{\p}}^{\Sigma} / \cS{\underline{\p}}^{0}$.  There are natural maps from $S^{\Sigma}_{\underline{\p}}, \cS{\underline{\p}}^{\Sigma}$, and $\S{\underline{\p}}^{\Sigma}$ to $\D{\Sigma}$, and for $\theta \in \Sigma$, we denote the preimages of $\D{\theta} \in \D{\Sigma}$ by $S^{\theta}_{\underline{\p}}, \cS{\underline{\p}}^{\theta}$ and $\S{\underline{\p}}^{\theta}$ respectively.

By the Langlands correspondence for tori and the assumption $\Ker^{1}(W_{F}, \D{D})= 1$, we can identify $H^{1}(W_{F}, \D{D})$ with $\Hom(D(F), \C^{\times})$ if $F$ is local or $\Hom(D(\A_{F})/D(F), \C^{\times})$ if $F$ is global. Then we can compose with $\c$ to get a homomorphism from $H^{1}(W_{F}, \D{D})$ to $\Hom(\lG(F)/G(F), \C^{\times})$ if $F$ is local or $\Hom(\lG(\A_{F})/ \lG(F)G(\A_{F}), \C^{\times})$ if $F$ is global. Since $\delta(Z(\D{G})^{\Gal{}})$ is trivial in $H^{1}(W_{F}, Z(\D{\lG}))$, it induces the trivial character on $\lG(F)$ if $F$ is local or $\lG(\A_{F})$ if $F$ is global. So we have a homomorphism 
\[
r: \bar{H}^{1}(W_{F}, \D{D}) \rightarrow \Hom(\lG(F)/G(F), \C^{\times})
\]
if $F$ is local, and 
\[
r: \bar{H}^{1}(W_{F}, \D{D}) \rightarrow \Hom(\lG(\A_{F})/ \lG(F)G(\A_{F}), \C^{\times})
\]
if $F$ is global. In the local case, $r$ is an isomorphism due to the fact that $\eqref{eq: local characters}$ is an isomorphism. For the global case, we have the following lemma.

\begin{lemma}
\label{lemma: identification}
If $F$ is global and $\lG$ is of type \eqref{eq: similitude}, then $r$ is an isomorphism.
\end{lemma}

\begin{proof}
First we consider the following diagram
\begin{align}
\label{diagram: global Langlands correspondence for characters}
\xymatrix{   \Hom(\lG(\A_{F}) / \lG(F), \C^{\times})  &  \Hom(D(\A_{F})/D(F), \C^{\times}) \ar[l]_{\c^{*}} & \\
                     H^{1}(W_{F}, Z(\D{\lG})) \ar[u]  & H^{1}(W_{F}, \D{D}) \ar[u]_{\simeq} \ar[l] & \pi_{0}(Z(\D{G})^{\Gal{}}). \ar[l]_{\delta} }
\end{align}
By Corollary~\ref{cor: relative Hasse principle}, we have $\Im \c^{*} = \Hom(\lG(\A_{F})/ \lG(F)G(\A_{F}), \C^{\times})$, and hence $r$ is surjective. On the other hand, the kernel of 
\begin{align*}
H^{1}(W_{F}, Z(\D{\lG})) \longrightarrow \Hom(\lG(\A_{F})/ \lG(F), \C^{\times}) 
\end{align*}
is 
\(
\Ker^{1}(W_{F},Z(\D{\lG})) = 1
\)
by Corollary~\ref{cor: ker1}. Therefore $r$ is also an inclusion.

\end{proof}

Let us denote the composition $r \circ \bar{\delta}$ by $\a$, then we can rewrite \eqref{eq: twisted endoscopic sequence} as 
\begin{align}
\label{eq: local twisted endoscopic sequence}
\xymatrix{1 \ar[r] &  \S{\underline{\lp}}^{\Sigma} \ar[r]^{\iota} & \S{\underline{\p}}^{\Sigma} \ar[r]^{\a \quad \quad \quad \quad}  & \Hom(\lG(F)/G(F), \C^{\times})}               
\end{align}
if $F$ is local, and
\begin{align}
\label{eq: global twisted endoscopic sequence}
\xymatrix{1 \ar[r] &  \S{\underline{\lp}}^{\Sigma} \ar[r]^{\iota} & \S{\underline{\p}}^{\Sigma} \ar[r]^{\a \quad \quad \quad \quad \quad \quad}  & \Hom(\lG(\A_{F})/ \lG(F)G(\A_{F}), \C^{\times})}                
\end{align}
if $F$ is global. Note in the global case, we only know it is exact when $\lG$ is of type \eqref{eq: similitude} according to the previous lemma. Sometimes, we want to distinguish the map $\a$ for different groups, so we will also write $\a^{G} = \a$.

Next we want to discuss the relation between lifting Langlands parameters (see Theorem~\ref{thm: lifting parameter}) and lifting twisted endoscopic groups (see Proposition~\ref{prop: lifting endoscopic group}). Suppose $F$ is local or global and $\p \in \P{G}$. For any semisimple element $s \in \cS{\underline{\p}}^{\theta}$, let $\D{G}'  := \Cent(s, \D{G})^{0}$ and it can be equipped with a Galois action given by $\underline{\p}$. This determines a quasisplit connected reductive group $G'$, and $\underline{\p}$ will factor through $\L{G'}$ for some $\theta$-twisted endoscopic datum $(G', s, \xi)$ of $G$, and hence we get a parameter $\p' \in \P{G'}$. In this way, we call $(G', \p')$ corresponds to $(\p, s)$, and we denote it by $(G', \p') \rightarrow (\p, s)$. 

By Proposition~\ref{prop: lifting endoscopic group}, $(G', s, \xi)$ can be lifted to a $(\theta, \x)$-twisted endoscopic datum $(\lG', \lif{s}, \lif{\xi})$ of $\lG$ for some character $\x$ of $\lG(F)/G(F)$ if $F$ is local or $\lG(\A_{F})/ \lG(F)G(\A_{F})$ if $F$ is global. Then by theorem~\ref{thm: lifting parameter} and the global assumption that we made after, we can have a lift $\lp'$ of $\p'$ in $\P{\lG'}$. All of these can be summarized in the diagram below
\begin{displaymath}
\xymatrix{L_{F}    \ar[r]^{\underline{\lp}'}    \ar[dr]_{\underline{\p}'}  & \L{\lG'} \ar[r]^{\lif{\xi}}     \ar[d]   & \L{\lG}   \ar[d] \\
&    \L{G'}    \ar[r]^{\xi}     &    \L{G}.}
\end{displaymath}
Then we have the following lemma.

\begin{lemma}
\label{lemma: twisted character}
$\a(s) = \x$.
\end{lemma}

\begin{proof}
It has been shown for the local case in \cite{Xu:2016}, and the proof for the global case is the same.
\end{proof}

\begin{remark}
\label{rk: twisted character}
From this lemma we see the character $\x$ associated with the twisted endoscopic datum $\lG'$ only depends on the image of $s$ in $\S{\p}^{\theta}$. In the global case, lifting Langlands parameter is not available due to the lack of the global Langlands group. However one can always lift twisted endoscopic groups in both local and global cases, so this lemma is behind the idea of our later definition of the map $\a$ (see \eqref{eq: global twisted endoscopic sequence}) in the global case.
\end{remark}

\subsection{Representations}
\label{subsec: representations}

Let us assume $F$ is a local field, and $G$, $\lG$, $D$, $\c$ are defined as in Section~\ref{subsubsec: notations}. In this section, we would like to recall some results about the restriction map $\Pkt{}(\lG(F)) \rightarrow \Pkt{}(G(F))$ from (\cite{Xu:2016}, Section 6.1).

\begin{lemma}
\label{lemma: finite restriction}
If $\lr$ is an irreducible admissible representation of $\lG(F)$, then the restriction of $\lr$ to G(F) is a direct sum of finitely many irreducible admissible representations.
\end{lemma}

\begin{theorem}[J.D. Adler and D. Prasad, \cite{AdlerPrasad:2006}]
\label{thm: restriction multiplicity one}
Suppose $\lG$ is a quasisplit general symplectic group or connected general even orthogonal group, and $\lr$ is an irreducible admissible representation of $\lG(F)$, then the restriction of $\lr$ to $G(F)$ is multiplicity free.
\end{theorem}

\begin{remark}
\label{rk: restriction multiplicity one}
This theorem can be easily extended to the groups $\lG$ of type \eqref{eq: similitude}. To do so, we can first extend a representation of $\lG(F)$ to $\lif{\lG}(F)$ (see \eqref{eq: product group}), and then restrict it to $G(F)$.
\end{remark}

\begin{lemma}
\label{lemma: existence}
If $\r$ is an irreducible admissible representation of G(F), then there exists a unique irreducible admissible representation $\lr$ of $\lG(F)$ up to twisting by $\Hom(\lG(F)/G(F), \C^{\times})$, such that it contains $\r$ in its restriction to $G(F)$. 
\end{lemma}

If $\r$ is an irreducible admissible representation of $G(F)$, let us denote
\[
\lG(\r) = \{ g \in \lG(F) : \r^g \cong \r \}. 
\] 
If $\lr$ is an irreducible admissible representation of $\lG(F)$, let us denote 
\[
X(\lr) = \{ \x \in (\lG(F)/\lZ(F)G(F))^{*} : \lr \otimes \x \cong \lr \}.
\]

\begin{proposition}
\label{prop: restriction multiplicity one}
Suppose $\lG$ is of type \eqref{eq: similitude}, $\lr$ is an irreducible admissible representation of $\lG(F)$ and $\r$ is contained in its restriction to $G(F)$, then for $\x \in (\lG(F)/\lZ(F)G(F))^{*}$, $\x$ is in $X(\lr)$ if and only if $\x$ is trivial on $\lG(\r)$. Moreover, the restriction of $\lr$ contains $|X(\lr)|$ irreducible admissible representations of G(F).
\end{proposition}

\begin{lemma}
\label{lemma: restrict tempered character}
Suppose $\lr$ is an irreducible admissible unitary representation of $\lG(F)$, then $\lr$ is an essentially discrete series representation of $\lG(F)$ if and only if its restriction to $G(F)$ is an essentially discrete series representation. The same is true of the tempered representations.
\end{lemma}

\subsection{Langlands-Shelstad-Kottwitz transfer}
\label{subsec: the transfer map}

Let $F$ be a local field of characteristic zero and $G$ be a quasisplit connected reductive group over $F$. Suppose $\theta$ is an automorphism of $G$ preserving an $F$-splitting and $\x_{G}$ is a quasicharacter of $G(F)$. We choose a quasicharacter $\chi$ on a closed subgroup $Z_{F}$ of $\Z(F)$, and define $\H(G, \chi)$ to be the space $\chi^{-1}$-equivariant smooth compactly supported functions over $G(F)$ (i.e., equivariant Hecke algebra of $G$). Let $\delta$ be a strongly $\theta$-regular $\theta$-semisimple element of $G(F)$ such that $\x_{G}$ is trivial on the $\theta$-twisted centralizer group $G^{\theta}_{\delta}(F)$ of $\delta$. We choose Haar measures on $G(F)$ and $G^{\theta}_{\delta}(F)$, and they induce a $G(F)$-invariant measure on $G^{\theta}_{\delta}(F) \backslash G(F)$. Then we can form the $(\theta, \x_{G})$-twisted orbital integral of $f \in \H(G, \chi)$ over $\delta$ as
\[
O^{\theta, \x_{G}}_{G}(f, \delta) := \int_{G_{\delta}(F) \backslash G(F)} \x_{G}(g)f(g^{-1}\delta \theta(g)) dg.
\]
We also form the $(\theta, \x_{G})$-twisted stable orbital integral over $\delta$ as
\[
SO^{\theta, \x_{G}}_{G}(f, \delta) := \sum_{\{\delta'\}_{G(F)}^{\theta} \thicksim_{st} \{\delta\}_{G(F)}^{\theta}} O_{G}^{\theta, \x_{G}}(f, \delta'), 
\]
where the sum is over $\theta$-twisted conjugacy classes $\{\delta'\}^{\theta}_{G(F)}$ in the $\theta$-twisted stable conjugacy class of $\delta$ (i.e., $\delta' = g^{-1} \delta \theta(g)$ for some  $g \in G(\bar{F})$), and the Haar measure on $G^{\theta}_{\delta'}(F)$ is translated from that on $G^{\theta}_{\delta}(F)$ by conjugation. Let $\mathcal{I}(G^{\theta, \x_{G}}, \chi)$ ($\mathcal{SI}(G^{\theta, \x_{G}}, \chi)$) be the space of $(\theta, \x_{G})$-twisted (stable) orbital integrals of $\H(G, \chi)$ over the set $G^{\theta}_{reg}(F)$ of strongly $\theta$-regular $\theta$-semisimple elements of $G(F)$, then by definition we have projections 
\[
\xymatrix{\H(G, \chi) \ar@{->>}[r] & \mathcal{I}(G^{\theta, \x_{G}}, \chi) \ar@{->>}[r] & \mathcal{SI}(G^{\theta, \x_{G}}, \chi).}
\]

Suppose $\r$ is an irreducible admissible representation of $G(F)$ and $\chi_{\r}$ is the central character of $\r$. Let $\chi = \chi_{\r}|_{Z_{F}}$. Suppose $\r^{\theta} \cong \r \otimes \x_{G}$, let $A_{\r}(\theta, \x_{G})$ be the intertwining operator between $\r \otimes \x_{G}$ and $\r^{\theta}$ (this is uniquely determined up to a scalar), we then define the $(\theta, \x_{G})$-twisted character of $\r$ to be the distribution 
\begin{align}
\label{eq: twisted character}
f_{G^{\theta}}(\r, \x_{G}) :=  trace (\int_{Z_{F} \backslash G(F)} f(g)\r(g) dg \circ A_{\r}(\theta, \x_{G})),
\end{align}
for $f \in \H(G, \chi)$. By results of Harish-Chandra \cite{H-C:1963} \cite{H-C:1999} in the non-twisted case, Bouaziz \cite{Bouaziz:1987}, Clozel \cite{Clozel:1987} and Lemaire \cite{Lemaire:2016} in the twisted case, there exists a locally integrable function $\Theta^{G^{\theta}, \x_{G}}_{\r}$ on $G(F)$ such that for $x \in G^{\theta}_{reg}(F), g \in G(F)$
\[
\Theta^{G^{\theta}, \x_{G}}_{\r} (g^{-1} x \theta(g)) = \x_{G}(g) \Theta^{G^{\theta}, \x_{G}}_{\r}(x),
\]
and
\[
f_{G^{\theta}}(\r, \x_{G}) = \int_{Z_{F} \backslash G(F)} f(g)\Theta^{G^{\theta}, \x_{G}}_{\r}(g) dg.
\]
By the twisted Weyl integration formula, one can show this character defines a linear functional on $\mathcal{I}(G^{\theta, \x_{G}}, \chi)$. A linear functional on $\mathcal{I}(G^{\theta, \x_{G}}, \chi)$ is called {\bf stable} if it factors through $\mathcal{SI}(G^{\theta, \x_{G}}, \chi)$. 

For a $(\theta, \x_{G})$-twisted endoscopic datum $(H, s, \xi)$ of $G$, there is a map defined over $F$ from the semisimple conjugacy classes of $H(\bar{F})$ to the $\theta$-twisted conjugacy classes of $\theta$-semisimple elements in $G(\bar{F})$. We call a strongly regular semisimple element $\gamma \in H(\bar{F})$ is strongly $G$-regular if its associated $H(\bar{F})$-conjugacy class maps to a $\theta$-twisted $G(\bar{F})$-conjugacy class of strongly $\theta$-regular $\theta$-semisimple elements in $G(\bar{F})$. We denote the set of strongly $G$-regular semisimple elements of $H(F)$ by $H_{G-reg}(F)$. The transfer factor defined in \cite{KottwitzShelstad:1999} is a function
\[
\Delta_{G, H}(\cdot, \cdot): H_{G-reg}(F) \times G^{\theta}_{reg}(F) \rightarrow \C,
\]
which is nonzero only when $\gamma \in H_{G-reg}(F)$ is a norm of $\delta \in G^{\theta}_{reg}(F)$, i.e., the $H(\bar{F})$-conjugacy class of $\gamma$ maps to the $\theta$-twisted $G(\bar{F})$-conjugacy class of $\delta$. Note if $\delta \in G^{\theta}_{reg}(F)$ has a norm $\gamma \in H_{G-reg}(F)$, then $\x_{G}$ is trivial on $G^{\theta}_{\delta}(F)$ (see Lemma 4.4.C, \cite{KottwitzShelstad:1999}). In this paper, we always normalize the transfer factor with respect to some fixed $\theta$-stable Whittaker datum $(B, \Lambda)$, and we also assume the Haar measure is preserved for any admissible embedding $T_{H} \xrightarrow{\simeq} T_{\theta}$, where $T_{H}$ is a maximal torus of $H$, $T$ is a $\theta$-stable maximal torus of $G$ and $T_{\theta} = T/(\theta - 1)T$.

There is a canonical inclusion $(Z_{G})_{\theta} \hookrightarrow Z_{H}$. Let us denote the image of $Z_{F}$ in $Z_{H}(F)$ by $Z'_{F}$, then one can associate a quasicharacter $\chi'$ of $Z'_{F}$, depending only on $\chi$ and the twisted endoscopic embedding $\xi$. The Langlands-Kottwitz-Shelstad transfer map (or twisted endoscopic transfer) is a correspondence from $f \in \H(G, \chi)$ to $f^{H} \in \H(H, \chi')$ such that
\begin{align}
\label{eq: geometric transfer}
SO_{H}(f^{H}, \gamma) = \sum_{\{\delta'\}_{G(F)}^{\theta} \thicksim_{st} \{\delta\}_{G(F)}^{\theta}} \Delta_{G, H}(\gamma, \delta') O_{G}^{\theta, \x_{G}}(f, \delta')
\end{align}
where the sum is over $\theta$-twisted conjugacy classes $\{\delta'\}_{G(F)}^{\theta}$ in the $\theta$-twisted stable conjugacy class of $\delta$. In particular, it descends to a surjection 
\[
\mathcal{I}(G^{\theta, \x_{G}}, \chi) \longrightarrow \mathcal{SI}(H, \chi')^{\Out_{G}(H)},
\]
where the action of $\Out_{G}(H)$ on $\mathcal{SI}(H, \chi')$ is independent of the choice of $F$-splitting for $H$ (see Section~\ref{subsubsec: twisted endoscopy}). The existence of such a transfer has been a long standing problem. In the real case, it is now a theorem of Shelstad \cite{Shelstad:2012}. In the nonarchimedean case, the main obstacle is the Fundamental Lemma, which has been finally resolved by Ngo \cite{Ngo:2010}. And the proof of the transfer conjecture in this case was completed by Waldspurger \cite{Waldspurger:2008}.

Now let us assume $G$, $\lG$, $D$ and $\c$ are defined as in Section~\ref{subsubsec: notations}. Let $\theta$ be an automorphism of $\lG$ preserving an $F$-splitting and $\c$ is $\theta$-invariant. Let $\lif{Z}_{F}$ be a closed subgroup of $Z_{\lG}(F)$ such that $\lif{Z}_{F} \rightarrow (Z_{\lG})_{\theta}(F)$ is injective and $D(F) / \c(\lif{Z}_{F})$ is finite (this is possible because we assume $\c$ is $\theta$-invariant). Let $Z_{F} = \lif{Z}_{F} \cap G(F)$. We choose Haar measures on $\lif{Z}_{F}$ and $Z_{F}$ such that the measure on $Z_{F} \backslash G(F)$ is the restriction of that on $\lif{Z}_{F} \backslash \lG(F)$. Let $\lif{\chi}$ be a quasicharacter of $\lif{Z}_{F}$ and we denote its restriction to $Z_{F}$ by $\chi$. For every $f \in \H(G, \chi)$,
it can be extended to ${\lG}(F)$ through $\lif{Z}_{F}$ by $\lif{\chi}$, and the extension lies in $\H(\lG, \lif{\chi})$, supported on $\lif{Z}_{F}G(F)$. Hence we get an inclusion map
\begin{align}
\label{map: inclusion of Hecke algebra}
\xymatrix{\H(G, \chi) \, \ar@{^{(}->}[r]   &  \H(\lG, \lif{\chi})   \\
f \ar@{|->}[r]   & \lif{f} },  
\end{align}
and we can identify $\H(G, \chi)$ with its image. Let $\x_{\lG}$ be a quasicharacter of $\lG(F)$ and $\x_{G} = \x_{\lG}|_{G}$. For any strongly $\theta$-regular $\theta$-semisimple element $\delta$ of $G(F)$ such that $\x_{G}$ is trivial on the $G^{\theta}_{\delta}(F)$, we fix the Haar measure on $\lG^{\theta}_{\delta}(F) \backslash \lG(F)$, which determines the Haar measure on $G^{\theta}_{\delta}(F) \backslash G(F)$ by restriction. Then for $f \in \H(G, \chi)$, and $\lf \in \H(\lG, \lif{\chi})$ being its extension, we have
\[
SO_{\lG}(\lf, \delta) = SO_{G}(f, \delta),
\]
and
\[
O_{\lG}^{\theta, \x_{\lG}}(\lf, \delta) = \sum_{\{\delta'\}_{G(F)}^{\theta} \thicksim_{\lG(F)} \{\delta\}_{G(F)}^{\theta}} O_{G}^{\theta, \x_{G}}(f, \delta')\x_{\lG}(g)
\]
where the sum is over $\theta$-twisted $G(F)$-conjugacy classes $\{\delta'\}_{G(F)}^{\theta}$ in the $\theta$-twisted $\lG(F)$-conjugacy classes $\{\delta\}_{\lG(F)}^{\theta}$ with $\delta' = g^{-1} \delta g$ for $g \in \lG(F)$, and the Haar measure on $G^{\theta}_{\delta'}(F)$ is translated from that on $G^{\theta}_{\delta}$ by conjugation. Because $\lif{Z}_{F}G(F)$ is $\theta$-conjugate invariant under $\lG(F)$, the map \eqref{map: inclusion of Hecke algebra} induces a map from $\mathcal{I}(G^{\theta, \x_{G}}, \chi)$ to $\mathcal{I}(\lG^{\theta, \x_{\lG}}, \lif{\chi})$. Moreover $\lif{Z}_{F}G(\bar{F})$ is conjugate invariant under $\lG(\bar{F})$, so it also induces a map from $\mathcal{SI}(G, \chi)$ to $\mathcal{SI}(\lG, \lif{\chi})$.

Suppose $\lG' \in \End{}{\lG^{\theta}, \x_{\lG}}$ and $G' \in \End{}{G^{\theta}, \x_{G}}$ correspond to each other according to Proposition~\ref{prop: lifting endoscopic group}. The natural inclusion $(\lZ)_{\theta} \rightarrow Z_{\lG'}$ induces an inclusion on $\lif{Z}_{F}$. So we can define $\lif{Z}'_{F} \subseteq Z_{\lG'}(F)$ to be the image of $\lif{Z}_{F}$ and $Z'_{F} = \lif{Z}'_{F} \cap Z_{G'}(F)$. The twisted endoscopic transfer sends $\H(\lG, \lif{\chi})$ to $\H(\lG', \lif{\chi}')$, where $\chi'$ is a quasicharacter of $\lif{Z}'_{F}$, depending only on $\lif{\chi}$ and the twisted endoscopic embedding. Let $\chi'$ be the restriction of $\lif{\chi}'$ to $Z'_{F}$. Then we have 
\begin{align*}
\xymatrix{\H(G', \chi') \, \ar@{^{(}->}[r]   &  \H(\lG', \lif{\chi}')   \\
f \ar@{|->}[r]   & \lif{f} },  
\end{align*}
The following lemma shows these inclusion maps are compatible with the twisted endoscopic transfers.

\begin{lemma}[\cite{Xu:2016}, Lemma 3.8]
\label{lemma: twisted endoscopic transfer}
Suppose $f \in \H(G, \chi)$, then the $(\theta, \x_{\lG})$-twisted endoscopic transfer of the extension $\lf$ of $f$ is equal to the extension of $(\theta, \x_{G})$-twisted endoscopic transfer $f^{G'}$ of $f$ as elements in $\mathcal{SI}({\lG'}, \lif{\chi}')$, i.e.
\begin{align}
\lf^{\lG'} = \lif{(f^{G'})} \label{eq: twisted endoscopic transfer}
\end{align}
\end{lemma}

\begin{remark}
The inclusion map \eqref{map: inclusion of Hecke algebra} of Hecke algebras induces a restriction map of distributions in the opposite direction. Moreover the restriction of an invariant distribution is again invariant, and the restriction of a stable distribution is again stable. In particular, the restriction of the character of a representation is compatible with the restriction of the representation in the usual sense.
\end{remark}

\begin{corollary}
\label{cor: twisted endoscopic transfer}
Suppose $S^{\lG'}(\cdot)$ is a stable distribution on $\lG'$, then the restriction of the pull-back of $S^{\lG'}(\cdot)$ is equal to the pull-back of the restriction of $S^{\lG'}(\cdot)$, i.e.
\[
S^{\lG'}(\lif{f}^{\lG'}) = S^{\lG'}(\lif{f^{G'}})
\]
\end{corollary}

\begin{proof}
One just need to substitute \eqref{eq: twisted endoscopic transfer} into $S^{\lG'}(\cdot)$.
\end{proof}


\section{Arthur's Classification Theory: tempered case}
\label{sec: Arthur's theory}

In this section we will review Arthur's classification theory for the tempered representations of quasi-split symplectic groups and special even orthogonal groups (cf. \cite{Arthur:2013}). So throughout this section, $G$ will always be a quasisplit symplectic group or special even orthogonal group over $F$ (if not specified). We fix an outer automorphism $\theta_{0}$ of $G$, and a nontrivial automorphism $\theta_{N}$ of $GL(N)$, so that they preserve an $F$-splitting respectively. When $G$ is symplectic, $\theta_{0}$ is trivial. When $G$ is special even orthogonal, we require $\theta_{0}$ to be the unique outer automorphism induced from the conjugation of the full orthogonal group. We denote $\Sigma_{0} = <\theta_{0}>$.  When $F$ is local, $\Sigma_{0}$ acts on $\Pkt{}(G(F))$ and we denote the set of $\Sigma_{0}$-orbits in $\Pkt{}(G(F)$ by $\cPkt{}(G(F))$. We denote by $\sH(G, \chi)$ the subspace of $\Sigma_{0}$-invariant functions in $\H(G, \chi)$, and we abbreviate $\H(GL(N))$ to $\H(N)$. We also denote the corresponding space of (stable) twisted orbital integrals by $\bar{\mathcal{I}}(G^{\theta}, \x_{G})$ ($\bar{\mathcal{SI}}(G^{\theta}, \x_{G})$) for $\theta \in \Sigma_{0}$ and $\x_{G} \in \Hom(G(F), \C^{\times})$.

\subsection{Substitute Langlands parameter}
\label{subsec: substitute Langlands parameter}

First let $F$ be a global field, we define the sets of substitute global generic (or tempered) Langlands parameters as follows,
\begin{align*}
\Psm{N} & := \{ \text{ isomorphism classes of irreducible unitary cuspidal automorphic representations of } \, GL(N) \}, \\
\Psm{N^{\theta_{N}}} & := \{ \p \in \Psm{N} : \p = \p^{\vee} \}, \\
\P{N^{\theta_{N}}} & := \{ \p = l_{1}\p_{1} \# \cdots \# l_{r}\p_{r} : \p = \p^{\vee} , \p_{i} \in \Psm{N_{i}}, \text{ with } \sum_{i=1}^{r} l_{i}N_{i} = N \}.
\end{align*}
Here $\p^{\vee}$ denotes the dual (or contragredient) of $\p$ if $\p \in \Psm{N}$, and
\[
\p^{\vee} := l_{1}\p_{1}^{\vee} \# \cdots \# l_{r}\p_{r}^{\vee}
\]
if $\p \in \P{N^{\theta_{N}}}$. Note that $\P{N^{\theta_{N}}}$ is just a set of formal sums of irreducible unitary cuspidal automorphic representations, and for every parameter $\p \in \P{N^{\theta_{N}}}$ we can assign a family of semisimple conjugacy classes in $GL(N, \mathbb{C})$ by 
\[
c(\p_{v}) :=  \underbrace{c(\p_{1, v}) \+ \cdots \+ c(\p_{1, v})}_{l_{1}} \+ \cdots \+ \underbrace{c(\p_{r, v}) \+ \cdots \+ c(\p_{r, v})}_{l_{r}}
\]
for unramified places $v$ of $\p$, where $c(\p_{i, v})$ is the Satake parameter of the local component $\p_{i, v}$. Inside $\Psm{N^{\theta_{N}}}$ there are two types of parameters, we call $\p$ is of {\bf orthogonal type} if the symmetric square $L$-function $L(s, \p, S^2)$ has a pole at $s = 1$; we call $\p$ is of {\bf symplectic type} if the skew-symmetric square $L$-function $L(s, \p, \wedge^2)$ has a pole at $s = 1$. In fact every $\p \in \Psm{N^{\theta_{N}}}$ will always be either one of these two types due to the fact that the Rankin-Selberg L-function
\[
L(s, \p \otimes \p) = L(s, \p, S^2)L(s, \p, \wedge^2)
\] 
has a simple pole at $s = 1$. Moreover when $N$ is odd, $\p$ is always of orthogonal type. The following theorem proved in (\cite{Arthur:2013}, Theorem 1.4.1 and Theorem 1.5.3) shows how automorphic representations of $GL(N)$ are related to that of orthogonal groups and symplectic groups. If $\r$ is an automorphic representation of $G$, we denote by $c(\r) = \{c(\r_{v})\}$ the family of Satake parameters of $\r_{v}$ at the unramified places.

\begin{theorem}
\label{thm: global functorial lifting}
Suppose $\p \in \Psm{N^{\theta_{N}}}$, then there is a unique class of elliptic endoscopic data $(G_{\p}, s_{\p}, \xi_{\p})$ in $\End{ell}{N^{\theta_{N}}}$ such that 
\(
c(\p_{v}) = \xi_{\p}(c(\r_{v}))
\)
for some discrete automorphic representation $\r$ of $G_{\p}$ at almost all places. Moreover if $\p$ is of orthogonal type, $\D{G}_{\p} = SO(2n+1, \C)$ when $N = 2n+1$, or $SO(2n, \C)$ when $N = 2n$; if $\p$ is of symplectic type, $\D{G}_{\p} = Sp(2n, \C)$ with $N = 2n$.
\end{theorem}

For $\p = l_{1}\p_{1} \# \cdots \# l_{r}\p_{r} \in \P{N^{\theta_{N}}}$, since $\p = \p^{\vee}$, one gets an involution on the indices by letting $\p_{i^{\vee}} = \p_{i}^{\vee}$, and consequently one has $l_{i} = l_{i^{\vee}}$. This gives a disjoint decomposition of these indices 
\[
I_{\p} \sqcup J_{\p} \sqcup J_{\p}^{\vee},
\]
where $I_{\p}$ indexes the set of self-dual simple parameters. Let $K_{\p} = I_{\p} \sqcup J_{\p}$, and let $I_{\p, O}$ ($I_{\p, S}$) indexes the self-dual simple parameters of orthogonal (symplectic) type. By Theorem~\ref{thm: global functorial lifting}, for each $\p_{i}$ with $i \in I_{\p}$, we have a twisted elliptic endoscopic group $G_{i}$ of $GL(N_{i})$ and we fix the twisted endoscopic embedding $\xi_{i} : \L{G_{i}} \longrightarrow \L{GL(N_{i})}$. 
For $\p_{j}$ with $j \in J_{\p}$, let us just take $G_{j}$ to be $GL(N_{j})$ and define an $L$-embedding $\xi_{j} : \L{G_{j}} \longrightarrow \L{GL(2N_{j})}$ by sending $g \rtimes w$ to $\text{diag}\{g, {}^tg^{-1}\} \times w$. Then Arthur defines a substitute global Langlands group by taking the fibre product
\[
\mathcal{L}_{\p} := \prod_{k \in K_{\p}} \{ \L{G_{k}} \longrightarrow W_{F} \},
\]
and he also defines an $L$-homomorphism $\p^{\mathcal{E}} : \mathcal{L}_{\p} \longrightarrow \L{GL(N)}$, where
\[
\p^{\mathcal{E}} := \bigoplus_{k \in K_{\p}} l_{k}\xi_{k}.
\] 
By viewing $G$ as in $\End{ell}{N^{\theta_{N}}}$, we can define the set of substitute global parameters of $G$ as follows
\[
\cP{G} := \{ \p \in \P{N^{\theta_{N}}} : \ep \text{ factors through } \L{G} \}.
\]
As a simple exercise, one can show for $\p = l_{1}\p_{1} \# \cdots \# l_{r}\p_{r} \in \P{N^{\theta_{N}}}$, $\p$ is in $\cP{G}$ if and only if  $l_{i}$ is even for all $i \in I_{\p, S}$. Since $\Out_{N}(G) \cong \Sigma_{0}$, the above set is really the analogue of the set of $\Sigma_{0}$-conjugacy classes of global Langlands parameters for $G$. For $\p \in \cP{G}$ and $\Sigma \subseteq \Sigma_{0}$, one can define
\begin{align*}
S^{\Sigma}_{\p} &= \Cent(\Im \ep, \D{G}^{\Sigma}), \\
\cS{\p}^{\Sigma} &= S^{\Sigma}_{\p} / Z(\D{G})^{\Gal{F}}, \\
\S{\p}^{\Sigma} &= \cS{\p}^{\Sigma} / \com[0]{\cS{\p}},
\end{align*}
and from here one can also define the following important subsets of $\cP{G}$
\begin{align*}
\cPsm{G} &= \{ \p \in \cP{G} : \cS{\p}^{\Sigma_{0}} = 1\}, \\
\cPdt{G} &= \{ \p \in \cP{G} : |\cS{\p}| < \infty \}, \\
\cP{G^{\theta}} & = \{ \p \in \cP{G}: S_{\p}^{\theta} \neq \emptyset\}, \\
\cPel{G^{\theta}} &= \{ \p \in \cP{G^{\theta}} : |\cS{\p, s}^{0}| < \infty \, \text{for some semisimple} \, s \in \cS{\p}^{\theta} \},
\end{align*}
where $\theta \in \Sigma_{0}$. In fact, one can even compute $S_{\p}$ very explicitly (see \cite[(1.4.8)]{Arthur:2013})
\begin{align}
\label{formula: centralizer}
S_{\p} = (\prod_{i \in I_{\p, O}} O(l_{i}, \C))_{\p}^{+} \times (\prod_{i \in I_{\p, S}} Sp(l_{i}, \C)) \times (\prod_{j \in J_{\p}} GL(l_{j}, \C)), 
\end{align}
where $(\prod_{i \in I_{\p, O}} O(l_{i}, \C))_{\p}^{+}$ is the kernel of the character 
\[
\e_{\p}^{+} : \prod_{i} g_{i} \longrightarrow \prod_{i} (\det \, g_{i})^{N_{i}}, \,\,\,\,\, g_{i} \in O(l_{i}, \C), i \in I_{\p, O}.
\]
Note $G$ is symplectic or speical even orthogonal here, so we have $I_{\p, O} = I^{+}_{\p}$ and $I_{\p, S} = I^{-}_{\p}$ in Arthur's original formula of $S_{\p}$. When $G$ is special even orthogonal, 
\begin{align}
\label{formula: plus centralizer}
S^{\Sigma_{0}}_{\p} = (\prod_{i \in I_{\p, O}} O(l_{i}, \C)) \times (\prod_{i \in I_{\p, S}} Sp(l_{i}, \C)) \times (\prod_{j \in J_{\p}} GL(l_{j}, \C)).
\end{align}
As a consequence, one has the following description of those subsets of $\cP{G}$.

\begin{lemma}
\label{lemma: discrete parameter}
\begin{enumerate}
\item $\cPsm{G} = \Psm{N} \cap \cP{G}$.
\item Suppose $\p \in \cP{G}$, then $\p$ is in $\cPdt{G}$ if and only if $K_{\p} = I_{\p, O}$ and $l_{i} =1$ for all $i \in K_{\p}$.
\item Suppose $\p$ is in $\cPel{G^{\theta}}$ for $\theta \in \Sigma_{0}$, then $K_{\p} = I_{\p, O}$ and $l_{i} \leqslant 2$ for all $i \in K_{\p}$.
\item Suppose $G$ is special even orthogonal and $\p \in \cP{G}$, then $\p$ is in $\cP{G^{\theta_{0}}}$ if and only if there exists $i \in I_{\p, O}$ such that $N_{i}$ is odd.
\end{enumerate}
\end{lemma}

The proof is a direct application of formulas~\eqref{formula: centralizer} and \eqref{formula: plus centralizer}.

Now let $F$ be a local field, we also define the substitute local generic (or tempered) Langlands parameters similarly as follows,
\begin{align*}
\ePsm{N} & := \{ \text{ isomorphism classes of irreducible essentially discrete series representations of } \, GL(N) \}, \\
\ePsm{N^{\theta_{N}}} & := \{ \ep \in \ePsm{N} : \ep = ( \ep ) ^{\vee} \}, \\
\ePbd{N^{\theta_{N}}} & := \{ \ep = l_{1}\ep_{1} \+ \cdots \+ l_{r}\ep_{r} : \ep = ( \ep )^{\vee} , \ep_{i} \in \ePsm{N_{i}} \text{ with } \sum_{i=1}^{r} l_{i}N_{i} = N \}.
\end{align*}
Suppose $\ep \in \ePsm{\com{N}}$, we call $\ep$ is of {\bf orthogonal type} if the local symmetric square $L$-function $L(s, \ep, S^2)$ has a pole at $s =0$;  we call $\ep$ is of {\bf symplectic type} if the local skew-symmetric square $L$-function $L(s, \ep, \wedge^2)$ has a pole at $s=0$. As in the global case, every $\ep \in \ePsm{N^{\theta_{N}}}$ will be either of orthogonal type or symplectic type. We would like to state a local version of Theorem~\ref{thm: global functorial lifting}, which is proved in \cite[Theorem 6.1.1 and Corollary 6.8.1]{Arthur:2013}. For $\ep \in \ePsm{N^{\theta_{N}}}$, let $\r_{\ep}$ be the self-dual essentially discrete series representation of $GL(N)$ defined by $\ep$. We write
\begin{align}
\label{eq: twisted character of GL(N)}
f_{N^{\theta_{N}}}(\ep) := f_{GL(N)^{\theta_{N}}}(\r_{\ep}) \,\,\,\,\,\,\,\,\,\,\, f \in \H(N),
\end{align}
with respect to some intertwining operator $A_{\r_{\ep}}(\theta_{N})$.

\begin{theorem}
\label{thm: local functorial lifting}
Suppose $\ep \in \ePsm{N^{\theta_{N}}}$, then there is a unique class of elliptic endoscopic data $(G_{\ep}, s_{\ep}, \xi_{\ep})$ in $\End{ell}{N^{\theta_{N}}}$ such that 
\[
f_{N^{\theta_{N}}}(\ep) = f^{G_{\ep}}(\ep), \text {\, \, for all \,}f \in \H(N)
\]
for some stable distribution $f(\ep)$ on $G_{\ep}$, where $f^{G_{\ep}}$ is the twisted endoscopic transfer of $f$. Moreover if $\ep$ is of orthogonal type, $\D{G}_{\ep} = SO(2n+1, \C)$ when $N = 2n+1$, or $SO(2n, \C)$ when $N = 2n$; if $\ep$ is of symplectic type, $\D{G}_{\ep} = Sp(2n, \C)$ with $N = 2n$.
\end{theorem}

Note $\Out_{N}(G) \cong \Sigma_{0}$, so the twisted endoscopic transfer $f^{G_{\ep}}$ lies in $\sH(G)$. As in the global case, one can define the substitute local Langlands group $\mathcal{L}_{\ep}$ and substitute local parameter $\ep$. One can also define the set $\cePbd{G}$ of substitute parameters for $G$, various centralizer groups of parameter $\ep$ in $\D{G}$, and various subsets in $\cePbd{G}$. Moreover, the formula~\eqref{formula: centralizer}, \eqref{formula: plus centralizer} and Lemma~\ref{lemma: discrete parameter} still hold in the local case.

The link between these substitute local parameters and the genuine local Langlands parameters is through the local Langlands correspondence for $GL(N)$ proved by Harris-Taylor \cite{HarrisTaylor:2001}, Henniart \cite{Henniart:2000} and Scholze \cite{Scholze:2013}. The local Langlands correspondence for $GL(N)$ gives a bijection between $\ePsm{N}$ and the set $\Pdt{N}$ of equivalence classes of $N$-dimensional irreducible unitary representations of $L_{F}$, which also induces a bijection for the self-dual ones. Later in the paper, we will identify them and use the notation $\Psm{N}$ as in the global case. The following theorem proved in (\cite{Arthur:2013}, Corollary 6.8.1) shows the substitute local parameters of $G$ correspond to the genuine local Langlands parameters of $G$ under this bijection. 

\begin{theorem}
\label{thm: parameter identification}
Suppose $\ep \in \ePsm{N^{\theta_{N}}}$, then $\ep$ is in $\cePsm{G}$ if and only if its corresponding Langlands parameter $\p$ factors through $\L{G}$.
\end{theorem}

Since the elements in $\ePbd{N^{\theta_{N}}}$ corresponds to self-dual tempered representations of $GL(N)$ by the parabolic induction, the local Langlands correspondence also gives a bijection between $\Pbd{N^{\theta_{N}}}$ and $\ePbd{N^{\theta_{N}}}$. And we have the following corollary.

\begin{corollary}
\label{cor: parameter identification}
\begin{enumerate}
\item Suppose $\ep \in \ePbd{N^{\theta_{N}}}$, then $\ep$ is in $\cePbd{G}$ if and only if its corresponding Langlands parameter $\p$ factors through $\L{G}$.\\
\item Suppose $\ep \in \cePbd{G}$ corresponds to $\p \in \cPbd{G}$, then $S_{\ep} \cong S_{\underline{\p}}$ for any representative $\underline{\p}$ of $\p$.
\end{enumerate}
\end{corollary}

For the proof, one just needs to notice for $\p \in \cPbd{G}$, there is a decomposition through the twisted endoscopic embedding to $GL(N, \C)$  
\[
\p = l_{1}\p_{1} \+ \cdots \+ l_{r} \p_{r},
\]
where $\p_{i} \in \Pdt{N_{i}}$ is irreducible.

As a consequence of Theorem~\ref{thm: parameter identification} and Corollary~\ref{cor: parameter identification}, one can identify $\cePbd{G}$ with $\cPbd{G}$ through the twisted endoscopic embedding $\xi: \L{G} \rightarrow \L{GL(N)}$. And we also denote $S_{\ep}$ by $S_{\p}$. 

\subsection{Local theory}
\label{subsec: local theory}

Now we can state the main local result of Arthur's theory (\cite{Arthur:2013}, Theorem 1.5.1 and Theorem 2.2.1) for quasisplit symplectic groups and special even orthogonal groups in the tempered case. Let $F$ be local. We fix a $\theta_{N}$-stable Whittaker datum $(B_{N}, \Lambda)$ for $GL(N)$. We also fix the twisted endoscopic embedding $\xi: \L{G} \rightarrow \L{GL(N)}$.

\begin{theorem}
\label{thm: LLC}

For every $\p \in \cPbd{G}$, one can associate it with a finite set $\cPkt{\p}$ of $\cPkt{temp}(G(F))$ such that it satisfies the following properties:

\begin{enumerate}

\item 
The distribution 
\begin{align}
\label{eq: stable distribution}
f(\p) := \sum_{[\r] \in \cPkt{\p}} f_{G}(\r), \, \, \,  f \in \sH(G)
\end{align}
is stable.

\item 
If we normalize the intertwining operator $A_{\r_{\p}}(\theta_{N})$ such that it preserves the Whittaker functional on $\r_{\p}$, then 
\begin{align}
\label{eq: twisted character identity GL(N)}
f_{N^{\theta_{N}}}(\p) = f^{G}(\p)
\end{align}
for $f \in \H(N)$ and the twisted endoscopic transfer $f^{G} \in \sH(G)$.

\item
There is a disjoint decomposition
\[
\cPkt{temp}(G(F)) = \bigsqcup_{\p \in \cPbd{G}} \cPkt{\p}.
\]

\end{enumerate}

\end{theorem}

Since the transfer map $\mathcal{I}(N^{\theta_{N}}) \rightarrow \mathcal{SI}(G)^{\Out_{N}(G)}$ is surjective (see Section~\ref{subsec: the transfer map}), $\p$ determines the stable distribution \eqref{eq: stable distribution} on $G(F)$ through \eqref{eq: twisted character identity GL(N)}. In this way, $\p$ determines the L-packet $\cPkt{\p}$. If $G$ is a product of symplectic groups and special even orthogonal groups, we define a group of automorphisms of $G$ by taking the product of $\Sigma_{0}$ on each factors, and we denote this group again by $\Sigma_{0}$. We denote the set of $\Sigma_{0}$-orbits in $\Pkt{temp}(G(F))$ by $\cPkt{temp}(G(F))$ and the set of $\Sigma_{0}$-orbits in $\Pbd{G}$ by $\cPbd{G}$. Let $\sH(G)$ be the $\Sigma_{0}$-invariant functions in $\H(G)$. Then part (1) and (3) of this theorem can also be generalized to this case, in particular, the L-packets of $G$ are formed by taking tensor products of those of each factor. If $G' \in \End{}{G^{\theta}}$ for $\theta \in \Sigma_{0}$, then $G' \cong M_{l} \times G'_{-}$, where $M_{l}$ is a product of general linear groups, and $G'_{-}$ is again a product of symplectic groups and special even orthogonal groups. We can extend the action of $\Sigma_{0}$ to $G'$ by letting it act trivially on $M_{l}$. Then we can define $\cPbd{G'}$ and $\cPkt{temp}(G'(F))$ similarly. Part (1) and (3) of this theorem can again be extended further to this case.

\begin{theorem}
\label{thm: character identity}
\begin{enumerate}
\item
For $\p \in \cPbd{G}$, there is a canonical pairing between $\cPkt{\p}$ and $\S{\p}$, which induces an inclusion
\begin{align}
\label{eq: canonical pairing for G}
[\r] \longrightarrow <\cdot, \r>, \,\,\,\, [\r] \in \cPkt{\p},
\end{align}
from $\cPkt{\p}$ into the group $\D{\S{\p}}$ of characters on $\S{\p}$, such that $<\cdot, \r> = 1$ if $G$ and $\r$ are unramified. It becomes a bijection when $F$ is nonarchimedean. 
\item
Suppose $s$ is a semisimple element in $\cS{\p}$ and $(G', \p') \longrightarrow (\p, s)$ with $G' \in \End{}{G}$ and $\p' \in \cPbd{G'}$. The packet $\cPkt{\p'}$ can be defined by the generalization of the previous theorem. If $x$ is the image of $s$ in $\S{\p}$, then
\begin{align}
\label{eq: character relation for G}
f^{G'}(\p') = \sum_{[\r] \in \cPkt{\p}} <x, \r> f_{G}(\r) , \,\,\,\,  f \in \sH(G).       
\end{align}
\end{enumerate}
\end{theorem}

When $G$ is special even orthogonal, one could further characterize those $\theta_{0}$-stable tempered representation.
It is a theorem proved in (\cite{Arthur:2013}, Theorem 2.2.3).

\begin{theorem}
\label{thm: twisted LLC}
Suppose $G$ is a special even orthogonal group and $\p \in \cPbd{\com{G}}$.
\begin{enumerate}
\item For any $[\r] \in \cPkt{\p}$, $\r$ is a $\theta_{0}$-stable representation of $G(F)$ and hence has an extension $\r^{+}$ to $G^{+}(F) = G(F) \rtimes <\theta_{0}>$. 
\\
\item Suppose $s$ is a semisimple element in $\com{\cS{\p}}$ and $(G', \p') \longrightarrow (\p, s)$ with $G' \in \End{}{\com{G}}, \p' \in \cPbd{G'}$, then 
\begin{align}
\label{eq: theta twisted character relation for G}
f'(\p') = \sum_{[\r] \in \cPkt{\p}} <x, \r^{+}>f_{G^{\theta_{0}}}(\r), \,\,\,\, f \in \sH(G),  
\end{align}
where $x$ is the image of $s$ in $\com{\S{\p}}$, $<\cdot, \r^{+}>$ is an extension of the character $<\cdot, \r>$ to $\S{\p}^{+} = <\S{\p}^{\theta_{0}}>$, and the intertwining operator $A_{\r}(\theta_{0}) = \r^{+}(\theta_{0})$.

\end{enumerate}
\end{theorem}

\begin{remark}
\label{rk: twisted LLC}
In the $\theta_{0}$-twisted character relation \eqref{eq: theta twisted character relation for G}, both the extensions of representation $\r^{+}$ and character $<\cdot, \r^{+}>$ are not uniquely determined, but the product $<\cdot, \r^{+}>f_{G^{\theta_{0}}}(\r)$ is determined and depends only on $\r$. Moreover, \eqref{eq: twisted character identity GL(N)} is the analogue of \eqref{eq: theta twisted character relation for G} for general linear groups, where we have fixed the extension of $\r_{\p}$ using the Whittaker datum and taken the extended character $<\cdot, \r^{+}>$ to be trivial. It is not hard to see how to generalize both Theorem~\ref{thm: LLC} and Theorem~\ref{thm: twisted LLC} to products of symplectic groups and special even orthogonal groups. 
\end{remark}

Because one does not know whether all the local constituents of a unitary cuspidal automorphic representation of $GL(N)$ are tempered, i.e. the generalized Ramanujan conjecture, one has to deal with a more general set of parameters $\uP{N}$, which is defined as follows. Let $\nu^{a}$ denote the map $| \cdot |_{F}^{a}$ of $W_{F}$ for $a \in \R$. Then,
\begin{align*}
\uP{N} := \{& \p = \p_{1} \+ \cdots \+ \p_{r} \+ ( \nu^{a_{1}}\p_{r+1} \+ \nu^{-a_{1}}\p_{r+1} ) \cdots \+ (\nu^{a_{s}}\p_{r+s} \+ \nu^{-a_{s}}\p_{r+s}) : \\ 
& \p_{i} \in \Psm{N_{i}} \text{ for } 1 \leqslant i \leqslant r+s \text{ and } 0 < a_{j} < 1/2 \text{ for } 1 \leqslant j \leqslant s \}.
\end{align*}
From the classification of the unitary dual of $GL(N)$ (cf. \cite{Tadic:2009} \cite{Vogan:1986} archimedean case, \cite {Tadic:1986} nonarchimedean case), we know the associated irreducible admissible representation $\r_{\p}$ for any $\p \in \uP{N}$ is unitary. And we have the following fact.

\begin{proposition}
\label{prop: local constituents of cuspidal representations}
If $F$ is global and $\p \in \Psm{N}$, then $\p_{v} \in \uP{N_{v}}$.
\end{proposition}  

Correspondingly, we can define
\[
\cuP{G} : = \cP{G} \cap \uP{N}.
\]
Theorem~\ref{thm: LLC},  Theorem~\ref{thm: character identity} and Theorem~\ref{thm: twisted LLC} can be extended to the case $\p \in \cuP{G}$ except for the constituents of $\cPkt{\p}$ may be non-tempered. In fact, for any $\p \in \cuP{G^{\theta}}$ with $\theta \in \Sigma_{0}$, $\p$ can be regarded as $\p_{M, \lambda} := \p_{M} \otimes (\lambda \circ |\cdot|_{F})$ for some $\theta$-stable Levi subgroup $M$ (which also admits a $\theta$-stable parabolic subgroup $P \supseteq M$), where $\p_{M} \in \cPbd{M^{\theta}}$ and $\lambda \in \mathfrak{a}^{*}_{M}$ lies in the open chamber determined by $P$. Let $\cPkt{\p_{M, \lambda}} := \cPkt{\p_{M}} \otimes e^{<H_{M}(\cdot), \lambda>}$. Then one can just define $\cPkt{\p}$ to be the irreducible constituents of the parabolic induction $\mathcal{I}_{P}( \cPkt{\p_{M, \lambda}} )$. Note that the $\theta$-twisted endoscopy transfer is compatible with this parabolic induction, and also $\S{\p_{M}} \cong \S{\p}$, then it is enough to know the following proposition.

\begin{proposition}
\label{prop: irreducibility of non-unitary induced representation}
Suppose $F$ is local, $\p \in \cuP{G}$, and $\p$ can be regarded as $\p_{M, \lambda}$ where $\p_{M} \in \cPbd{M}$ and $\lambda \in \mathfrak{a}^{*}_{M}$ lies in some open chamber determined by $P \supseteq M$. Then for any $[\r_{M}] \in \cPkt{\p_{M}}$, the induced representation $\mathcal{I}_{P} (\r_{M, \lambda}) $ is irreducible.
\end{proposition}

Proposition~\ref{prop: local constituents of cuspidal representations} and Proposition~\ref{prop: irreducibility of non-unitary induced representation} are well known to experts, but for the convenience of the readers we will give their proofs in Appendix~\ref{sec: irreducibility}. 

\subsection{Global theory}
\label{subsec: global theory}

Now let us assume $F$ is global, and we fix the twisted endoscopic embedding $\xi: \L{G} \rightarrow \L{GL(N)}$. The global parameters and local parameters are related by the following theorem.

\begin{theorem}
\label{thm: local-global compatibility}
Suppose $\p \in \cPsm{G}$, then $\p_{v}$ factors through $\L{G_{v}}$ for all places $v$, i.e. $\p_{v} \in \cuP{G_{v}}$.
\end{theorem}

This theorem is proved in (\cite{Arthur:2013}, Theorem 1.4.2).  So for $\p \in \cP{G}$, one has a commutative diagram
\[
\xymatrix{L_{F_{v}} \ar[r]^{\p_{v}}   \ar[d]    &  \L{G_{v}} \ar[d] \\
\mathcal{L}_{\p} \ar[r]^{\ep}    &  \L{G},}
\]
where $L_{F_{v}} \rightarrow \mathcal{L}_{\p}$ is defined by $\p_{v}$. It gives rise to an inclusion $S_{\p} \hookrightarrow S_{\p_{v}}$ for any place $v$, which induces a homomorphism $\S{\p} \rightarrow \S{\p_{v}}$. One can define the global $L$-packet by taking the restricted tensor product
\[
\cPkt{\p} := \sideset{}{'} \bigotimes_{v}  \cPkt{\p_{v}}
\]
and define the global pairing by
\[
<x, \r> := \prod_{v} <x_{v}, \r_{v}>.
\]
Note that for almost all places $v$, $<\cdot, \r_{v}> = 1$ by Theorem~\ref{thm: LLC}, so this product is well-defined. The main global result of Arthur's theory is to give a description of the discrete spectrum of automorphic representations of $G$. Here we only state it for those discrete automorphic representations parametrized by $\cPdt{G}$, i.e. for $\p \in \cPdt{G}$, we want to describe $L^2_{disc, \p} (G(F) \backslash G(\A_{F}))$ which consists of discrete automorphic representations $\r$ such that the Satake parameters satisfy $\xi(c(\r_{v})) = c(\p_{v})$ for almost all places. Let $\sH(G) = \otimes'_{v} \sH(G_{v})$.

\begin{theorem}
\label{thm: discrete spectrum}
Suppose $\p \in \cPdt{G}$, there is a decomposition as $\sH(G)$-modules
\[
L^2_{disc, \p} (G(F) \backslash G(\A_{F}))= m_{\p} \sum_{\substack{[\r] \in \cPkt{\p}  \\  <\cdot, \r> = 1}} \r
\]
where $m_{\p} =1 \text{ or } 2$, and $m_{\p} = 2$ only when $G$ is special even orthogonal and $\p \notin \cP{\com{G}}$. Moreover, 
\[
L^2_{disc, \p} (G(F) \backslash G(\A_{F})) = 0
\]
for $\p \in \cP{G} - \cPdt{G}$. 
\end{theorem}

\begin{remark}
\label{rk: discrete spectrum}
This theorem is a special case of (\cite{Arthur:2013}, Theorem 1.5.2). By Arthur's complete description of the discrete spectrum for orthogonal and symplectic groups, one can see $\cPdt{G}$ only contributes to the discrete spectrum of $G$. It is not hard to extend this to products of symplectic and special even orthogonal groups. In fact if $G = G_{1} \times G_{2} \times \cdots \times G_{q}$, then we can define $\cP{G}$ to be consisting of $\p := \p_{1} \times \p_{2} \times \cdots \times \p_{q}$ such that $\p_{i} \in \cP{G_{i}}$ for $1 \leqslant i \leqslant q$. Moreover, we can define $\mathcal{L}_{\p} := \prod_{i=1}^{q}\mathcal{L}_{\p_{i}}$, then $S_{\p} = \prod_{i=1}^{q} S_{\p_{i}}$. And we let $m_{\p} = \prod_{i=1}^{q} m_{\p_{i}}$.
\end{remark}

For $\p \in \cP{G}$ and any subgroup $\Sigma \subseteq \Sigma_{0}$, let $\mathcal{L}_{\p}$ act on $\D{D}$, $\D{\lG}^{\Sigma}$ and $\D{G}^{\Sigma}$ by conjugation through $\ep$. We denote the corresponding group cohomology by $H^{*}_{\ep}(\mathcal{L}_{\p}, \cdot)$. Note $H^{0}_{\ep}(\mathcal{L}_{\p}, \D{D}) = \D{D}^{\Gal{}}, H^{0}_{\ep}(\mathcal{L}_{\p}, \D{G}^{\Sigma}) = S^{\Sigma}_{\p}$ and $H^{1}_{\ep}(\mathcal{L}_{\p}, \D{D}) = H^{1}(W_{F}, \D{D})$. We define $S_{\lp}^{\Sigma} := H^{0}_{\ep}(\mathcal{L}_{\p}, \D{\lG}^{\Sigma})$. Then we have the following diagram 

\[
\xymatrix{
1 \ar[r] & S^{\Sigma}_{\lp} / \D{D}^{\Gal{}} \ar[r]    \ar@{^{(}->}[d]    &   S^{\Sigma}_{\p} \ar[r]^{\delta \quad \quad}   \ar@{^{(}->}[d]   &   H^{1}(W_{F}, \D{D}) \ar[d] \\
1 \ar[r] & S^{\Sigma}_{\lp_{v}} / \D{D}_{v}^{\Gal{v}} \ar[r]    &   S^{\Sigma}_{\p_{v}}  \ar[r]^{\delta_{v} \quad \quad}  &  H^{1}(W_{F_{v}}, \D{D}_{v}).}
\]
Then Lemma~\ref{lemma: centralizer} is still valid, and we again have the following exact sequence as in Section~\ref{subsec: Langlands parameters}
\begin{align*}
\xymatrix{1 \ar[r] &  \S{\lp}^{\Sigma} \ar[r]^{\iota} & \S{\p}^{\Sigma} \ar[r]^{\a \quad \quad \quad \quad \quad \quad}  & \Hom(\lG(\A_{F})/ \lG(F)G(\A_{F}), \C^{\times}).}                
\end{align*}

\section{Coarse L-packet}
\label{sec: coarse L-packet}

\subsection{Statement of main local theorem}
\label{subsec: statement of main local theorem}

Now we assume $\lG$ is of type \eqref{eq: similitude}, and $\c$ is the generalized similitude character. In this case $G$ is a product of symplectic groups and special even orthogonal groups. We also assume $\theta \in \Sigma_{0}$.


\begin{lemma}[\cite{Xu:2016}, Lemma 3.13]
\label{lemma: twisting character}
Suppose $\p \in \cPbd{G}$ and $[\r] \in \cPkt{\p}$ then
\begin{align}
<x, (\r^{+})^{g}> = \x_{x}(g)<x, \r^{+}>  \label{eq: theta twisting character}
\end{align}
for any $g \in \lG(F)$ and $x \in \S{\p}^{\theta}$, where $\x_{x} = \a(x)$ and $\r^{+}$ is an extension of $\r$ to $G^{+}(F) = G(F) \times <\theta>$. 
\end{lemma}

\begin{corollary}[\cite{Xu:2016}, Proposition 6.28]
\label{cor: theta twisting character}
Suppose $\p \in \cPbd{G}$ and $[\r] \in \cPkt{\p}$. If $\lr$ is an irreducible admissible representation of $\lG(F)$ whose restriction to $G(F)$ contains $\r$, then $\lr^{\theta} \cong \lr \otimes \x$ if and only if $\x \in \a(\S{\p}^{\theta})$. In particular,
\[
X(\lr) = \a(\S{\p}).
\]
\end{corollary}

\begin{remark}
In view of Proposition~\ref{prop: irreducibility of non-unitary induced representation}, Lemma~\ref{lemma: twisting character} and Corollary~\ref{cor: theta twisting character} also hold for parameters in $\cuP{G}$, and the proofs are the same.
\end{remark}

For $\p \in \cPbd{G}$, let us fix a character $\lif{\zeta}$ of $\lZ(F)$ such that $\lif{\zeta}|_{\Z(F)}$ is the central character of $\cPkt{\p}$. Then we define $\clPkt{\p, \lif{\zeta}}$ to be the subset of $\cPkt{temp}(\lG)$ with central character $\lif{\zeta}$, whose restriction to $G(F)$ have irreducible constituents contained in $\cPkt{\p}$. Let 
\(
X = \Hom(\lG(F)/\lZ(F)G(F), \C^{\times}),
\)
so $X$ acts on $\clPkt{\p, \lif{\zeta}}$ by twisting. We call $\clPkt{\p, \lif{\zeta}}$ a coarse $L$-packet of $\lG$, and its structure can be described in the following proposition.

\begin{proposition}[\cite{Xu:2016}, Proposition 6.29]
\label{prop: coarse L-packet}
Suppose $\p \in \cPbd{G}$ and $\lif{\zeta}$ is chosen as above.
\begin{enumerate}
\item The orbits in $\cPkt{\p}$ under the conjugate action of $\lG(F)$ all have size $|\S{\p} / \S{\lp}|$. If $F$ is nonarchimedean, there are exactly $|\S{\lp}|$ orbits. \\
\item There is a natural fibration 
\[
\xymatrix{ X / \a(\S{\p}^{\Sigma_{0}})  \ar[r]   &   \clPkt{\p, \lif{\zeta}}    \ar[r]^{Res \quad}     &  \cPkt{\p} / \lG(F)}
\]
\item There is a pairing 
\[
\lr \longrightarrow <\cdot, \lr>
\]
from $\clPkt{\p, \lif{\zeta}} / X$ into $\D{\S{\lp}}$. It is uniquely characterized by 
\[
<x, \lr> = <\iota(x), \r> ,
\]
where $\iota: \S{\lp} \hookrightarrow \S{\p}$ and $\r$ is any irreducible representation of $G(F)$ in the restriction of $\lr$. Suppose $\lG$ and $\lr$ are unramified, then $<\cdot, \lr> = 1$. Moreover, this mapping from $\clPkt{\p, \lif{\zeta}} / X$ to $\D{\S{\lp}}$ is injective and when $F$ is nonarchimedean it is in fact a bijection.
\end{enumerate} 
\end{proposition}

This proposition is also true for $\cuP{G}$. Now it is natural to ask the following question.

\begin{question}
\label{que: refined L-packet}
For any lift $\lp$ of $\p \in \cPbd{G}$, can one assign a packet $\cPkt{\lp}$ of representations of $\lG(F)$ which gives a section of $\Res: \clPkt{\p, \lif{\zeta}} \rightarrow \cPkt{\p} / \lG(F)$, and also a stable distribution?
\end{question}

The answer to this question can be formulated in the following theorem, which is our main local result.

\begin{theorem}
\label{thm: refined L-packet}
Suppose $\p \in \cPbd{G}$, and $\lif{\zeta}$ is a character of $\lZ(F)$ whose restriction to $\Z(F)$ is the central character of $\cPkt{\p}$. Let $\lif{\chi} = \lif{\zeta}|_{\lif{Z}_{F}}$. Then there exists a subset $\cPkt{\lp}$ of $\clPkt{\p, \lif{\zeta}}$ unique up to twisting by $X$, and it is characterized by the following properties:
\begin{enumerate}
\item 
\[
\clPkt{\p, \lif{\zeta}} = \bigsqcup_{\x \in X / \a(\S{\p}^{\Sigma_{0}})} \cPkt{\lp} \otimes \x. 
\]  

\item For $\lf \in \sH(\lG, \lif{\chi})$, the distribution
\[
\lf(\lp) := \sum_{[\lr] \in \cPkt{\lp}} \lf_{\lG}(\lr)
\]
is stable.

\item Suppose $s$ is a semisimple element in $\cS{\p}^{\theta}$ with $\x = \a(s)$ and  $(G', \p') \longrightarrow (\p, s)$. Fix a packet $\cPkt{\lp'}$ defined by part (1) and local Langlands correspondence for $GL(n)$, then we can choose $\cPkt{\lp}$ such that
\begin{align}
\label{eq: theta twisted character relation}
\lf'(\lp') = \sum_{[\lr] \in \cPkt{\lp}} \lf_{\lG^{\theta}}(\lr, \x), \,\,\,\,\,\, \lf \in \sH(\lG, \lif{\chi})  
\end{align}
where $\lf_{\lG^{\theta}}(\lr, \x) = tr (\lr(\lf) \circ A_{\lr}(\theta, \x))$, and $A_{\lr}(\theta, \x)$ is an intertwining operator between $\lr \otimes \x$ and $\lr^{\theta}$, which is normalized in a way so that if $f$ is the restriction of $\lf$ on $G(F)$, then 
\begin{align}
\label{eq: theta twisted intertwining operator}
(\lf|_{\lif{Z}_{F}G(F)})_{\lG^{\theta}}(\lr, \x) = \sum_{\r \subseteq \lr|_{G}} <x, \r^+> f_{G^{\theta}}(\r)  
\end{align}
where $x$ is the image of $s$ in $\S{\p}^{\theta}$, $\r^{+}$ is an extension of $\r$ to $G^{+}(F) = G(F) \rtimes <\theta>$ such that $\r^{+}(\theta) = A_{\r}(\theta)$.
\end{enumerate}
\end{theorem}

\begin{remark}
\label{rk: refined L-packet}
\begin{enumerate}

\item In the notation of $\cPkt{\lp}$, one can think of $\lp$ as some parameter of $\lG$ lifted from $\p$. Since $\cPkt{\lp}$ is only defined up to twisting by $X$, one can also take $\lp$ as a formal symbol built in the notation of $\cPkt{\lp}$. In this paper, we will take the second point of view.

\item The normalizations in \eqref{eq: theta twisted intertwining operator} is a consequence of \eqref{eq: theta twisting character}. When $\theta = id$, $\x = 1$ and $x \in \S{\lp}$, $A_{\lr}(id, 1)$ becomes a scalar and is equal to $<x, \lr>$ by \eqref{eq: theta twisted intertwining operator}. So we obtain the character relation from \eqref{eq: theta twisted character relation}
\begin{align}
\label{eq: character relation}
\lf'(\lp') = \sum_{[\r] \in \cPkt{\lp}} \lf_{\lG}(\lr, 1) = \sum_{[\r] \in \cPkt{\lp}} <x, \lr> \lf_{\lG}(\lr).
\end{align}

\item If $F$ is archimedean, $\cPkt{\lp}$ is defined by Langlands \cite{Langlands:1989}. In fact, we have $\cPkt{\lp} = \clPkt{\p, \lif{\zeta}}$, if $\cPkt{\p}$ is not a singleton (see Proposition~\ref{prop: refined L-packet Archimedean case} and also \cite{H-C:1975}, Theorem 27.1). Moreover, Part (2) and (3) can be directly reduced to \eqref{eq: stable distribution}, \eqref{eq: character relation for G} and \eqref{eq: theta twisted character relation for G} (see \cite{Xu:2016}, Remark 6.32).

\item For $\p \in \cuP{G}$, since $\cPkt{\p} = \mathcal{I}_{P}(\cPkt{\p_{M, \lambda}})$, we can define $\cPkt{\lp} := \mathcal{I}_{\lP}(\cPkt{\lp_{M, \lambda}})$, and this theorem can be easily extended to this case.
\end{enumerate}
\end{remark}

Let us call this subset $\cPkt{\lp}$ the refined $L$-packet of $\lG$, and it is the genuine $L$-packet that one would expect modulo the action of $\Sigma_{0}$. As we can see from this theorem, the refined $L$-packet is uniquely determined by the character relation \eqref{eq: character relation} up to twisting by $X$. As a consequence of that, we can give another characterization of the refined L-packet.

\begin{corollary}
\label{cor: refined L-packet}
In the setup of the previous theorem, any stable linear combination in $\clPkt{\p, \lif{\zeta}}$ is given by a linear combination of $\lf(\lp \otimes \x) := (\lf \otimes \x)(\lp)$ for $\x \in X$.
\end{corollary}

\begin{proof}
In the archimedean case, one can deduce this from (\cite{Shelstad:1979}, Lemma 5.3).
So we will assume $F$ is nonarchimedean, and we fix a refined L-packet $\cPkt{\lp} = \{ [\lr_{i}] \}_{i = 1}^{r}$. Suppose
\[
\lf(\lp^{1}) := \sum_{i, j} a_{i j} \lf_{\lG}(\lr_{i} \otimes \x_{j})
\]
is also stable for distinct $\x_{j} \in X / \a(\S{\p}^{\Sigma_{0}})$ and $a_{ij} \in \C$. Here $\lp^{1}$ is just a formal symbol for denoting another stable distribution. Since the map $[\lr] \longrightarrow <\cdot, \lr>$ is a bijection in the nonarchimedean case, we have 
\[
\sum_{x \in \S{\lp}} <x, \lr_{i}><x, \lr_{j}> = r \cdot \delta_{ij}.
\]
By inverting the formula of character relation \eqref{eq: character relation} we get
\[
\lf_{\lG}(\lr_{i}) = \sum_{x \in \S{\lp}} c(\lr_{i}, x) \lf' (\lp, x),
\]
where $\lf' (\lp, x) = \lf'(\lp')$ for some semisimple element $s \in \cS{\lp}$ whose image in $\S{\lp}$ is $x$, and some $\cPkt{\lp'}$ with $(G', \p') \longrightarrow (\p, s)$, and 
\[
c(\lr_{i}, x) = \frac{1}{r} <x, \lr_{i}>.
\]
Therefore
\[
\lf^{\lG}(\lp^{1}) = \sum_{i, j} a_{ij} c(\lr_{i}, x) \lf' (\lp\otimes \x_{j}, x).
\]
If we separate those terms with $x = 1$ from the right hand side, and move them to the left hand side, we get
\begin{align}
\label{eq: stable distribution identity}
\lf^{\lG}(\lp^{1}) - \frac{1}{r} \sum_{i, j} a_{ij }\lf(\lp\otimes \x_{j}) = \sum_{i, j} \sum_{x \neq 1} c(\lr_{i}, x) \lf' (\lp\otimes \x_{j}, x)  
\end{align}
Now let us consider the endoscopic transfer map
\[
\xymatrix{\mathcal{T}^{\varepsilon}: \sH(\lG) \ar[r] & \bigoplus_{\lG' \in \End{ell}{\lG}} \bar{\mathcal{SI}}(\lG') \\ 
\lf \ar@{|->}[r]  &\bigoplus_{\lG' \in \End{ell}{\lG}} \lf^{\lG'}. }
\]
The left hand side of \eqref{eq: stable distribution identity} can be viewed as value of $\mathcal{T}^{\varepsilon}(\lf)$ on some stable distribution of $\lG$, and similarly the right hand side of \eqref{eq: stable distribution identity} can be viewed as values of $\mathcal{T}^{\varepsilon}(\lf)$ on stable distributions of elliptic endoscopic groups $\lG' \in \End{ell}{\lG} - \{\lG\}$. It is a consequence of the main results in \cite{Arthur:1996} that the image of $\mathcal{T}^{\varepsilon}$ can be characterized as families of functions $(\lf^{\lG'})_{\lG' \in \End{ell}{\lG}}$ such that for any $\lG'_{1}, \lG'_{2} \in \End{ell}{\lG}$ the parabolic descent of any two functions $\lf^{\lG'_{1}}$ and  $\lf^{\lG'_{2}}$ to their common Levi subgroups $\lM'$ of $\lG'_{1}$ and $\lG'_{2}$ coincide. Since $\p'$ does not factor though any proper Levi subgroups of $\L{G}$ for $x \neq 1$, then the stable distribution associated to $\p'$ is not supported on any proper Levi subgroups of $G$. The same is true for the stable distribution associated with $\cPkt{\lp'}$. So the right hand side of $\eqref{eq: stable distribution identity}$ is not valued on any stable distributions supported on the Levi subgroups of $\lG$. Since $\eqref{eq: stable distribution identity}$ holds for all functions in $\sH(\lG)$, then both sides of $\eqref{eq: stable distribution identity}$ must be zero. Therefore, 
\[
0 = \lf^{\lG}(\lp^{1}) - \frac{1}{r} \sum_{i, j} a_{ij }\lf(\lp\otimes \x_{j})  =  \sum_{k, j} (a_{kj} - \frac{1}{r} \sum_{i} a_{ij}) \lf_{\lG}(\lr_{k} \otimes \x_{j}).
\]
By the linear independence of characters, we have 
\[
a_{kj} - \frac{1}{r} \sum_{i} a_{ij} = 0
\]
for any $k, j$. As we fix $j$ and vary $k$, we get a system of linear equations. The solutions of this system are $a_{ij} = a_{1j}$ for $1 \leqslant i \leqslant r$. Since this is also valid when we vary $j$, so we can conclude
\[
\lf(\lp^{1}) = \sum_{ j} a_{1j} \lf(\lp \otimes \x_{j}).
\]
\end{proof}

This corollary is also valid for $\p \in \cuP{G}$, and the proof is the same.


\subsection{Local twisted intertwining relation}
\label{subsec: local twisted intertwining relation}

The proof of our main local theorem (Theorem~\ref{thm: refined L-packet}) requires global methods, and the existence of refined $L$-packet needs to be proved together with the character relations. Before we proceed to prove the theorem, let us first consider another form of character relation, called intertwining relation. The intertwining relation in its global form comes up naturally in the trace formula and it plays an important role in Arthur's work \cite{Arthur:2013}. Here we need a twisted version of the intertwining relation, which in its local form is related to the twisted character relation \eqref{eq: theta twisted character relation}. In this section, we again assume $\lG$ is of type \eqref{eq: similitude}, and $\theta \in \Sigma_{0}$.

Suppose $\p \in \cPbd{G}$, and we assume $\p$ factors through $\p_{M} \in \cPbd{M}$ for some Levi subgroup $M$ of $G$. Let us define
\[
\cT{\p}(G, M) = A_{\D{M}} Z(\D{G})^{\Gal{F}} / Z(\D{G})^{\Gal{F}},
\]
where $A_{\D{M}}$ is the maximal split central torus in $\D{M}$. It is a torus in $\com[0]{\cS{\p}}$. Then we can define its normalizer in $\cS{\p}$
\[
\cN{\p}(G, M)  = \Norm(\cT{\p}(G, M), \cS{\p}),
\]
and the group of its connected components
\begin{align*}
\N{\p}(G, M) & = \cN{\p}(G, M) / \cN{\p}(G, M)^{0} \\
& = \Norm(\cT{\p}(G, M), \cS{\p}) / \Cent(\cT{\p}(G, M), \com[0]{\cS{\p}})^{0}.
\end{align*}
Notice $\S{\p}(M) := \S{\p_{M}}$ is a normal subgroup of $\N{\p}(G, M)$. The quotient $\N{\p}(G, M) / \S{\p}(M)$ is the Weyl group 
\[
W_{\p}(G, M) 
= W(\cS{\p}, \cT{\p}(G, M)).
\] 
We write $\com[0]{W_{\p}}(G, M)$ to be the normal subgroup of automorphisms in $W_{\p}(G, M)$ that are induced from the connected component $\com[0]{\cS{\p}}$, and let
\[
R_{\p}(G, M) = W_{\p}(G, M) / \com[0]{W_{\p}}(G, M)
\]
Moreover, $\com[0]{W_{\p}}(G, M)$ is a normal subgroup of $\N{\p}(G, M)$, and we denote their quotient by $\S{\p}(G, M)$, which is a subgroup of $\S{\p}$. 
Suppose $\lM$ is the Levi subgroup of $\lG$ containing $M$, then similarly we can define
\[
\cT{\lp} (G, M) = A_{\D{\lM}} Z(\D{\lG})^{\Gal{F}} / Z(\D{\lG})^{\Gal{F}}
\]
which is a torus of $\com[0]{\cS{\lp}}$. Since $A_{\D{\lM}} / \D{D} = A_{\D{M}}$, we have $\cT{\lp}(G, M) = \cT{\p}(G, M)$. We can also define
\[
\cN{\lp}(G, M)  = \Norm(\cT{\lp}(G, M), \cS{\lp}) \subseteq \cN{\p}(G, M),
\]
and the group of its connected components
\begin{align*}
\N{\lp}(G, M)  & = \cN{\lp}(G, M) / \cN{\lp}(G, M)^{0} \\
& = \Norm(\cT{\lp}(G, M), \cS{\lp}) / \Cent(\cT{\lp}(G, M), \com[0]{\cS{\lp}})^{0} \subseteq \N{\p}(G, M).
\end{align*}
Again $\S{\lp}(M) := \S{\lp_{M}}$ is a normal subgroup of $\N{\lp}(G, M)$. The quotient $\N{\lp}(G, M) / \S{\lp}(M)$ is the Weyl group
\[
W_{\lp}(G, M) = W(\cS{\lp}, \cT{\lp}(G, M)).
\]
Let us write $\com[0]{W_{\lp}}(G, M)$ to be the normal subgroup of automorphism in $W_{\lp}(G, M)$ that are induced from the connected component $\com[0]{\cS{\lp}}$. Since $\com[0]{\cS{\p}} = \com[0]{\cS{\lp}}$, we have $\com[0]{W_{\lp}}(G, M) = \com[0]{W_{\p}}(G, M)$. So 
\[
R_{\lp}(G, M) = W_{\lp}(G, M) / \com[0]{W_{\lp}}(G, M) \subseteq R_{\p}(G, M).
\]
At last, $\com[0]{W_{\lp}}(G, M)$ is a normal subgroup of $\N{\lp}(G, M)$, and we denote their quotient by $\S{\lp}(G, M)$, which is a subgroup of $\S{\lp}$. If $\p_{M} \in \cPdt{M}$, then $\cT{\p}(G, M) = \cT{\p}$ is a maximal torus in $\com[0]{\cS{\p}}$, and hence $\S{\p}(G, M) = \S{\p}$, $\S{\lp}(G, M) = \S{\lp}$. So in this case let us also write 
\[
\N{\p}(G, M) = \N{\p}, \,\,\,\,\,\,\,\,\,\,\,  \N{\lp}(G, M) = \N{\lp},
\]
\[
W_{\p}(G, M) = W_{\p}, \,\,\,\,\,\,\,\,\,\,\,  W_{\lp}(G, M) = W_{\lp},
\]
\[
\com[0]{W_{\p}}(G, M) = \com[0]{W_{\p}}, \,\,\,\,\,\,\,\,\,\,\,  \com[0]{W_{\lp}}(G, M) = \com[0]{W_{\lp}}.
\]
To summarize all these relations, we have the following commutative diagram.
\begin{align} 
\label{eq: twisted intertwining relation diagram}
\xymatrix @C=0.5cm @R=0.5cm{&&&&& 1 \ar[dd] && 1 \ar[dd] && \\
&&&& 1     \ar[dd]   && 1    \ar[dd]       &&& \\
&&&&& \com[0]{W_{\lp}}(G, M)                \ar[dd]           \ar@{=}[rr]              && \com[0]{W_{\lp}}(G, M)          \ar[dd]&& \\
&&&& \com[0]{W_{\p}}(G, M)           \ar@{=}[ur]             \ar[dd]          \ar@{=}[rr]  && \com[0]{W_{\p}}(G, M)                  \ar@{=}[ur]                \ar[dd] \\
& 1 \ar[rr] &&           \S{\lp}(M)           \ar@{_{(}->}[dl]            \ar@{=}[dd]            \ar[rr]  && \N{\lp}(G, M)             \ar@{_{(}->}[dl]            \ar[dd]  \ar[rr]   && W_{\lp}(G, M)      \ar@{_{(}->}[dl]  \ar [dd]      \ar[rr] && 1\\  
1 \ar[rr]  && \S{\p}(M)           \ar@{=}[dd]              \ar[rr]  &&          \N{\p}(G, M)        \ar[dd]             \ar[rr] && W_{\p}(G, M)        \ar[dd]            \ar[rr] && 1 & \\
& 1 \ar[rr] && \S{\lp}(M)       \ar@{_{(}->}[dl]            \ar[rr]  && \S{\lp}(G, M)      \ar[dd]     \ar@{_{(}->}[dl]    \ar[rr]   && R_{\lp}(G, M)       \ar[dd]       \ar@{_{(}->}[dl]  \ar[rr] && 1\\
1 \ar[rr]  && \S{\p}(M)        \ar[rr]  && \S{\p}(G, M)         \ar[dd]           \ar[rr] && R_{\p}(G, M)      \ar[dd]          \ar[rr] && 1 & \\
&&&&& 1 && 1  && \\
&&&& 1      && 1      &&& }  
\end{align}

Suppose $u \in \N{\p}(G, M)$, we write $w_{u}$ for the image of $u$ in $W_{\p}(G, M)$ and $x_{u}$ for the image of $u$ in $\S{\p}(G, M)$. Since $w_{u}$ normalizes $A_{\D{M}}$, it also normalizes $\D{M}$, and therefore can be treated as an element in $W(\D{M})$. The standard parabolic subgroup $P$ containing $M$ allows us to identify $w_{u}$ with an element in $W(M) \cong W(\lM)$. We choose a representative $\theta_{u}$ of $w_{u}$ in $G(F)$ preserving the $F$-splitting of $M$.
Then $\p_{M} \in \cPbd{M^{\theta_{u}}}$ and $u$ defines an element in $\S{\p}(M)^{\theta_{u}} := \S{\p_{M}}^{\theta_{u}}$. Note that 
\[
M \cong GL(N_{1}) \times \cdots \times GL(N_{q}) \times G_{-},
\]
and 
\[
\lM \cong GL(N_{1}) \times \cdots \times GL(N_{q}) \times \lG_{-},
\]
where $G_{-}$ (resp. $\lG_{-}$) is of the same type as $G$ (resp. $\lG$) with smaller rank.
Suppose 
\[
\p_{M} = \p_{1} \times \cdots \times \p_{q} \times \p_{-},
\]
where $\p_{i} \in \Pbd{N_{i}}$ and $\p_{-} \in \cPbd{G_{-}}$. Then we can define 
\[
\cPkt{\p_{M}} = \r_{\p_{1}} \otimes \cdots \otimes \r_{\p_{q}} \otimes \cPkt{\p_{-}},
\]
where $\r_{\p_{i}}$ is associated to $\p_{i}$. And any representation in this packet can be written as
\begin{align*}
\r_{M} & = \r_{\p_{1}} \otimes \cdots \otimes \r_{\p_{q}} \otimes \r_{-} \\
& = \r_{GL} \otimes \r_{-}.
\end{align*}
Since $\S{\p_{M}} \cong \S{\p_{-}}$, we can define a pairing between $\cPkt{\p_{M}}$ and $\S{\p_{M}}$ by 
\[
<\cdot, \r_{M}> : = <\cdot, \r_{GL}>  <\cdot, \r_{-}>,
\]
where $<\cdot, \r_{GL}>$ is in fact trivial. 
By Theorem~\ref{thm: twisted LLC} and the local Langlands correspondence for $GL(n)$, we know $\r_{M}^{\theta_{u}} \cong \r_{M}$. As usual, we can take the intertwining operator $\r_{M}^{+}(\theta_{u}) = \r_{GL}^{+}(\theta_{u}) \otimes \r_{-}^{+}(\theta_{u})$, which preserves the Whittaker models on those general linear components, then the extension $< \cdot, \r_{GL}^{+} >$ of $<\cdot, \r_{GL}>$ to $\S{\p}(M)^{\theta_{u}}$ is trivial (see Theorem~\ref{thm: LLC}). So the extension $<\cdot, \r_{-}^{+}>$ defined in Theorem~\ref{thm: twisted LLC} (see also Remark~\ref{rk: twisted LLC}) determines an extension $<\cdot, \r_{M}^{+}>$. Now we define
\begin{align}
\label{eq: induced character}
f_{G}(\p, u) = \sum_{[\r_{M}] \in \cPkt{\p_{M}}} <u, \r_{M}^{+}>tr(R_{P}(\theta_{u}, \r_{M}^{+}, \p)\mathcal{I}_{P}(\r_{M}, f)), \,\,\,\,\,\, f \in \sH(G, \chi),   
\end{align}
where $R_{P}(\theta_{u}, \r_{M}^{+}, \p)$ is the normalized self-intertwining operator on the space $\mathcal{H}_{P}(\r_{M})$ of normalized induced representation $\mathcal{I}_{P}(\r_{M})$ (see \cite{Arthur:2013}, Section 2.4). 
If we assume the existence of refined $L$-packet $\cPkt{\lp_{-}}$ for $\lp_{-}$ as defined in Theorem~\ref{thm: refined L-packet}, then $\cPkt{\lp_{M}}$ can be defined in the same way as $\cPkt{\p_{M}}$. And we can also define
\begin{align}
\label{eq: induced twisted character}
\lif{f}_{\lG}(\lp, u) = \sum_{[\lr_{M}] \in \cPkt{\lp_{M}}} tr(R_{\lP}(u, \lr_{M}, \lp) \mathcal{I}^{\x}_{P}(\lr_{M} \otimes \x^{-1}, \lif{f})), \,\,\,\,\,\, \lif{f} \in \sH(\lG, \lif{\chi}).   
\end{align}
Here we need to give some explanations of this distribution. Firstly, $\x = \a^{M}(u) = \a^{G}(x_{u})$. If 
\[
\lr_{M} = \r_{\p_{1}} \otimes \cdots \otimes \r_{\p_{q}} \otimes \lr_{-}
\]
contains $\r_{M}$ in its restriction to $M(F)$, then it follows from Corollary~\ref{cor: theta twisting character} that $\lr_{M}^{\theta_{u}} \cong \lr_{M} \otimes \x$, and we let $A_{\lr_{M}}(\theta_{u}, \x)$ be the intertwining operator. Secondly, the automorphism $\theta_{u}$ on $\lM$ is a composition of permutations of the general linear factors and automorphisms sending $g_{i}$ to
\(
\theta_{N_{i}}(g_{i}) \cdot \c(g_{-}),
\) 
where $g = g_{1} \times \cdots \times g_{q} \times g_{-} \in \lM$. However the effect on general linear components of $\lr_{M}$ is the same as for $M$, so we can use the same intertwining operator for the general linear components of $\lr_{M}$. In view of \eqref{eq: theta twisted intertwining operator}, the pairing inside \eqref{eq: induced character} is built into $A_{\lr_{M}}(\theta_{u}, \x)$ and hence into the operator $R_{\lP}(u, \lr_{M}, \lp)$. Thirdly,
\[
\mathcal{I}^{\x}_{P}(\lr_{M} \otimes \x^{-1}, \lif{f}) = R(\x) \circ \mathcal{I}_{P}(\lr_{M} \otimes \x^{-1}, \lif{f})
\] 
where $R(\x)$ is multiplication by $\x$, and $R_{\lP}(u, \lr_{M}, \lp)$ is the normalized intertwining operator between $\mathcal{H}_{\lP}(\lr_{M})$ and $\mathcal{H}_{\lP}(\lr_{M} \otimes \x^{-1})$. The last thing is about this normalization. Let us recall the formulation of the normalized intertwining operator
\[
R_{P}(\theta_{u}, \r_{M}^{+}, \p) := \r_{M}^{+}(\theta_{u}) \circ (r_{P}(w_{u}, \p_{M})^{-1} J_{P}(\theta_{u}, \r_{M})).
\]
Here $r_{P}(w_{u}, \p_{M})$ is the normalizing factor, and $J_{P}(\theta_{u}, \r_{M})$ is the unnormalized intertwining operator between $\mathcal{H}_{P}(\r_{M})$ and $\mathcal{H}_{P}(\r_{M}^{\theta_{u}^{-1}})$, which is defined by an integral over
\[
N_{P} \cap w_{u} N_{P} w_{u}^{-1} \backslash N_{P},
\]
where $N_{P}$ is the unipotent radical of $P$. The key point is to notice that 
\[
\Res^{\lG(F)}_{G(F)} \mathcal{I}_{\lP} ( \lr_{M} ) \cong  \mathcal{I}_{P} ( \Res^{\lM(F)}_{M(F)} \lr_{M} ).
\]
So we obtain isomorphisms between the following spaces as $\sH(G, \chi)$-modules
\begin{align*}
\mathcal{H}_{\lP}(\lr_{M}) & \cong \bigoplus_{\r_{M} \subseteq \Res \lr_{M}} \mathcal{H}_{P}(\r_{M}) \cong \mathcal{H}_{\lP}(\lr_{M} \otimes \x^{-1}),  \\
\mathcal{H}_{\lP}(\lr_{M}^{\theta_{u}^{-1}}) & \cong \bigoplus_{\r_{M} \subseteq \Res \lr_{M}} \mathcal{H}_{P}(\r_{M}^{\theta_{u}^{-1}}).
\end{align*}
Under these identifications, we can easily see from the definition of unnormalized intertwining operators that
\[
J_{\lP}(\theta_{u}, \lr_{M}) = \bigoplus_{\r_{M} \subseteq \Res \lr_{M}} J_{P}(\theta_{u}, \r_{M}).
\]
Let $\r_{\p_{M}} = \r_{\p_{1}} \otimes \cdots \otimes \r_{\p_{q}} \otimes \r_{\p_{-}}$. The normalizing factor $r_{P}(w_{u}, \p_{M})$ for $J_{P}(\theta_{u}, \r_{M})$ is equal to the product of $\lambda$-factor $\lambda(w_{u})$ (see \cite{Arthur:2013}, (2.3.19)) and 
\begin{align}
\label{formula: normalizing factor}
L(0, \r_{\p_{M}}, \rho^{\vee}_{w_{u}^{-1}P | P} ) \varepsilon(0, \r_{\p_{M}}, \rho^{\vee}_{w_{u}^{-1}P | P}, \psi_{F} )^{-1} L(1, \r_{\p_{M}}, \rho^{\vee}_{w_{u}^{-1}P | P})^{-1}
\end{align}
where the $L$-functions involved here are either Rankin-Selberg $L$-functions or (skew)-symmetric square $L$-functions. We can set $r_{\lP}(w_{u}, \lp_{M}) := r_{P}(w_{u}, \p_{M})$ for $J_{\lP}(\theta_{u}, \lr_{M})$. In fact this is what one would expect according to Langlands' conjectural formula for the normalizing factors. By his conjecture, \eqref{formula: normalizing factor} could be replaced by
\[
L(0, \rho^{\vee}_{w_{u}^{-1}P | P} \circ \p_{M} ) \varepsilon(0, \rho^{\vee}_{w_{u}^{-1}P | P} \circ \p_{M}, \psi_{F} )^{-1} L(1, \rho^{\vee}_{w_{u}^{-1}P | P} \circ \p_{M} )^{-1},
\]
where $\rho^{\vee}_{w_{u}^{-1}P | P}$ is the contragredient of the adjoint representation of $\L{M}$ over $\D{\mathfrak{n}}_{w_{u}^{-1}P} \cap \D{\mathfrak{n}}_{P} \backslash \D{\mathfrak{n}}_{w_{u}^{-1}P}$, where $\D{\mathfrak{n}}_{P}$ is the Lie algebra of $\D{N}_{P}$. Since
\[
\rho^{\vee}_{w_{u}^{-1}\lP | \lP} \circ \lp_{M} = \rho^{\vee}_{w_{u}^{-1}P | P} \circ \p_{M},
\]
then the conjectural formulas for  $r_{\lP}(w_{u}, \lp_{M})$ and $r_{P}(w_{u}, \p_{M})$ are the same. Finally we can normalize $A_{\lr_{M}}(\theta_{u}, \x)$ according to \eqref{eq: theta twisted intertwining operator}, so after composing with this operator we get 
\begin{align}
\label{eq: intertwining operator identity}
R_{\lP}(u, \lr_{M}, \lp)  := A_{\lr_{M}}(\theta_{u}, \x) \circ (r_{\lP}(w_{u}, \lp_{M})^{-1} J_{\lP}(\theta_{u}, \lr_{M})) 
                                        = \bigoplus_{\r_{M} \subseteq \Res \lr_{M}} <u, \r_{M}^+> R_{P}(\theta_{u}, \r_{M}^{+}, \p).     
\end{align}
As a result, if $\lf \in \sH(\lG, \lif{\chi})$ is supported on $\lif{Z}_{F}G(F)$ and $f$ is its restriction to $G(F)$, then 
\[
\lf_{\lG}(\lp, u)= f_{G}(\p, u).
\]

Suppose $s$ is a semisimple element in $\cS{\p}$, and $(G', \p') \longrightarrow (\p, s)$. For any lift $\cPkt{\lp'}$, let us define
\[
\lf'_{\lG}(\lp, s) = \lf^{\lG'}(\lp'), \,\,\,\,\,\, \lf \in \sH(\lG, \lif{\chi}).
\]
Now we can state the local $\x$-twisted intertwining relation.

\begin{theorem}
\label{thm: twisted intertwining relation}
Suppose $\p \in \cPbd{G}$, for semisimple $s \in \cS{\p}$ with $(G', \p') \longrightarrow (\p, s)$, the following identity holds for some lifts $\cPkt{\lp}$ and $\cPkt{\lp'}$ that
\begin{align}
\label{eq: twisted intertwining relation}
\lf_{\lG}(\lp, u) = \lf'_{\lG}(\lp, s),  \,\,\,\,\,\,\, \lf \in \sH(\lG, \lif{\chi}),  
\end{align} 
where $u \in \N{\p}(G, M)$ and $s \in \cS{\p}$ have the same image in $\S{\p}$.
\end{theorem}

The next lemma shows that for $\lG$, the $\x$-twisted intertwining relation \eqref{eq: twisted intertwining relation} is equivalent to the $\x$-twisted character relation (see \eqref{eq: theta twisted character relation} when $\theta$ = id), if one has the local intertwining relation for $G$ (\cite{Arthur:2013}, Theorem 2.4.1).

\begin{lemma}
\label{lemma: twisted intertwining relation 1}
For $\p \in \cPbd{G}$, we assume $\p$ factors through $\p_{M} \in \cPdt{M}$ and $\cPkt{\lp_{M}}$ exists. We define $\cPkt{\lp} := \mathcal{I}_{P}(\cPkt{\lp_{M}})$. Suppose $u \in \N{\p}(G, M)$ and semisimple $s \in \cS{\p}$ have the same image $x$ in $\S{\p}$. Then 
\[
\lf_{\lG}(\lp, u) = \lf'_{\lG}(\lp, s)  
\] 
for some $\cPkt{\lp}$ and $\cPkt{\lp'}$ if and only if 
\[ 
\lf'_{\lG}(\lp, s)= \sum_{[\lr] \in \cPkt{\lp}} \lf_{\lG}(\lr, \x),   
\]
where $\x = \a(x)$ and $\lf \in \sH(\lG, \lif{\chi})$.
\end{lemma}

\begin{proof}
By Corollary~\ref{cor: theta twisting character}, we have $\cPkt{\lp} = \cPkt{\lp} \otimes \x$. Since
\[
\lf_{\lG}(\lp, u)  = \sum_{[\lr_{M}] \in \cPkt{\lp_{M}}} tr(R_{\lP}(u, \lr_{M}, \lp)\mathcal{I}^{\x}_{P}(\lr_{M} \otimes \x^{-1}, \lif{f})),
\]
we can assume 
\[
\lf_{\lG}(\lp, u) =  \sum_{[\lr] \in \cPkt{\lp}} \lf_{\lG}(\lr, \x)',
\]
where $\lf_{\lG}(\lr, \x)' = tr(\lr(\lf) \circ A_{\lr}(\x)')$ for some $A_{\lr}(\x)'$ intertwining $\lr \otimes \x$ and $\lr$. Note that if $f$ is the restriction of $\lf$ on $G(F)$, then $(\lf|_{\lif{Z}_{F}G(F)})_{\lG}(\lp, u) = f_{G}(\p, u)$. By the local intertwining relation for $G$ (\cite{Arthur:2013}, Theorem 2.4.1), we have 
\[
(\lf|_{\lif{Z}_{F}G(F)})_{\lG}(\lp, u) = f_{G}(\p, u) = f'_{G}(\p, s) = \sum_{[\r] \in \cPkt{\p}} <x, \r>f_{G}(\r).
\]
So $\lf_{\lG}(\lr, \x)' =  \lf_{\lG}(\lr, \x)$ for $\lf$ supported on $\lif{Z}_{F}G(F)$. This means for $[\lr] \in \cPkt{\lp}$, $A_{\lr}(\x)' = A_{\lr}(\x)$ as defined by \eqref{eq: theta twisted intertwining operator}. Thus $\lf_{\lG}(\lr, \x)' =  \lf_{\lG}(\lr, \x)$ for all $\lf \in \sH(\lG, \lif{\chi})$, and the lemma is clear.
\end{proof}

As we can see from the proof of this lemma, $\lf_{\lG}(\lp, u)$ only depends on the image of $u$ in $\S{\p}$. And one should expect $\lf'_{\lG}(\lp, s)$ only depends on the image of $s$ in $\S{\p}$ as well either from the $\x$-twisted character relation or the $\x$-twisted intertwining relation. But there is a little ambiguity here for $\lf'_{\lG}(\lp, s)$ depends on the choice of lift $\cPkt{\lp'}$. The next lemma resolves the ambiguity and establishes this property directly.

\begin{lemma}
\label{lemma: induced twisted character}
For $\p \in \cPbd{G}$ and $x \in \S{\p}$, there is a natural way to get a family of lifts $\cPkt{\lp'}$ for all semisimple $s \in \cS{\p}$ with image $x$ in $\S{\p}$ and $(G', \p') \rightarrow (\p, s)$. And for these lifts $\lf'_{\lG}(\lp, s)$ are the same.
\end{lemma}

\begin{proof}
The proof is essentially the same as for $f'_{G}(\p, s)$ in  (\cite{Arthur:2013}, Section 4.5) except for one does not have any ambiguity in that case. Since it is important to clarify the ambiguity here, we will review the original proof and show how one can get rid of this ambiguity. Suppose semisimple $s \in \cS{\p}$ has image $x$ in $\S{\p}$, if $s$ is replaced by an $\com[0]{\cS{\p}}$-conjugate $s_{1}$, then the corresponding pair $(G'_{1}, \p'_{1})$ is isomorphic to $(G', \p')$ under $\com[0]{\cS{\p}}$-conjugation. And this extends to an isomorphism between $\lG'_{1}$ and $\lG'$ for $\com[0]{\cS{\lp}} = \com[0]{\cS{\p}}$. So we can simply take the lifts $\cPkt{\lp'} \cong \cPkt{\lp'_{1}}$ and it is clear that $\lf'_{\lG}(\lp, s) = \lf'_{\lG}(\lp, s_{1})$. 

Now if we fix a maximal torus $\bar{T}_{\p}$ of $\com[0]{\cS{\p}}$ and a Borel subgroup $\bar{B}_{\p}$ containing it, any automorphism of the complex reductive group $\com[0]{\cS{\p}}$ stabilizes a conjugate of $(\bar{T}_{\p}, \bar{B}_{\p})$. So we can choose a semisimple representative $s_{x}$ of $x$ in $\cS{\p}$ so that $\Int(s_{x})$ stabilizes $(\bar{T}_{\p}, \bar{B}_{\p})$, and such representatives are determined up to a $\bar{T}_{\p}$-translate. Moreover the complex torus
\[
\bar{T}_{\p, x} = \Cent(s_{x}, \bar{T}_{\p})^{0}
\]
in $\bar{T}_{\p}$ is uniquely determined by $x$. Note that $\bar{T}_{\p, x}$ is the connected component of the kernel of the following morphism
\[
\xymatrix{ \bar{T}_{\p} \ar[r]  & \bar{T}_{\p} \\
t \ar@{|->}[r]  & s_{x}^{-1}ts_{x}t^{-1}}
\]
So any point of $\bar{T}_{\p}$ can be written as $(s_{x}^{-1}ts_{x}t^{-1}) t_{x}$ for $t \in \bar{T}_{\p}$ and $t_{x} \in \bar{T}_{\p, x}$ (see \cite{Springer:2009}, Corollary 5.4.5), and hence any point in $s_{x} \bar{T}_{\p}$ can be written as
\[
s_{x} (s_{x}^{-1}ts_{x}t^{-1}) t_{x} = t s_{x} t^{-1} t_{x} = t s_{x} t_{x}  t^{-1}, \,\,\,\, t \in \bar{T}_{\p}, \,\,\, t_{x} \in \bar{T}_{\p, x}.
\]
This means it is enough to consider the $\bar{T}_{\p, x}$-translates of $s_{x}$. Finally, we can take the centralizer $\D{M}$ of $\bar{T}_{\p, x}$ in $\D{G}$ which is a Levi subgroup of $\D{G}$, and it is dual to a Levi subgroup $M_{x}$ of $G$. So $(\p, s_{x})$ is the image of a pair
\[
(\p_{M_{x}}, s_{M_{x}}), \,\,\,\, \p_{M_{x}} \in \cPbd{M_{x}}, \,\, s_{M_{x}} \in S_{\p_{M_{x}}},
\] 
attached to $M_{x}$ under the $L$-embedding $\L{M_{x}} \subseteq \L{G}$. This pair is in turn the image of an endoscopic pair $(M'_{x}, \p'_{M_{x}})$, and in particular, $\p'_{M_{x}} \in \cPdt{M'_{x}}$. Note that for all $\bar{T}_{\p, x}$-translates $s_{x, 1}$of $s_{x}$, the corresponding $\p'$ also factors through $\L{M'_{x}}$. And we have
\[
\lf'_{\lG}(\lp, s_{x}) = \lf^{M'_{x}}(\lp'_{M_{x}}) = \lf'_{\lG}(\lp, s_{x, 1}).
\]
Now if we reverse our argument, we see any lift $\cPkt{\lp'_{M_{x}}}$ will give rise to a family of lifts $\cPkt{\lp'}$ for all semisimple $s \in \cS{\p}$ with image $x$ in $\S{\p}$, such that $\lf'_{\lG}(\lp, s)$ are the same. This finishes the proof.

\end{proof}

In fact our discussion of the $\x$-twisted intertwining relation for $\lG$ can be extended to that for $\lG \rtimes \theta$. For $\p \in \cPbd{G^{\theta}}$, suppose it factors through $\p_{M} \in \cPbd{M}$ for some Levi subgroup $M$ of $G$. Let us define
\[
\xymatrix{\N{\p}^{\theta}(G, M) = \Norm(\cT{\p}(G, M), \cS{\p}^{\theta})/ \Cent(\cT{\p}(G, M), \com[0]{\cS{\p}})^{0},  & W_{\p}^{\theta}(G, M) = W(\cS{\p}^{\theta}, \cT{\p}(G, M)), }
\]
\[
\xymatrix{\N{\p}^{+}(G, M) = \Norm(\cT{\p}(G, M), \cS{\p}^{+})/ \Cent(\cT{\p}(G, M), \com[0]{\cS{\p}})^{0},  & W_{\p}^{+}(G, M) = W(\cS{\p}^{+}, \cT{\p}(G, M)), }
\] 
and 
\[
\xymatrix{\N{\lp}^{+}(G, M) = \Norm(\cT{\lp}(G, M), \cS{\lp}^{+})/ \Cent(\cT{\lp}(G, M), \com[0]{\cS{\lp}})^{0},  & W_{\lp}^{+}(G, M) = W(\cS{\lp}^{+}, \cT{\lp}(G, M)).}
\]
Then we can have a commutative diagram analogous to \eqref{eq: twisted intertwining relation diagram}.
\begin{align} 
\label{eq: theta twisted intertwining relation diagram}
\xymatrix @C=0.5cm @R=0.5cm{&&&&& 1 \ar[dd] && 1 \ar[dd] && \\
&&&& 1     \ar[dd]   && 1    \ar[dd]       &&& \\
&&&&& \com[0]{W_{\lp}}(G, M)                \ar[dd]           \ar@{=}[rr]              && \com[0]{W_{\lp}}(G, M)          \ar[dd]&& \\
&&&& \com[0]{W_{\p}}(G, M)           \ar@{=}[ur]             \ar[dd]          \ar@{=}[rr]  && \com[0]{W_{\p}}(G, M)                  \ar@{=}[ur]                \ar[dd] \\
& 1 \ar[rr] &&           \S{\lp}(M)           \ar@{_{(}->}[dl]            \ar@{=}[dd]            \ar[rr]  && \N{\lp}^{+}(G, M)             \ar@{_{(}->}[dl]            \ar[dd]  \ar[rr]   && W_{\lp}^{+}(G, M)      \ar@{_{(}->}[dl]  \ar [dd]      \ar[rr] && 1\\  
1 \ar[rr]  && \S{\p}(M)           \ar@{=}[dd]              \ar[rr]  &&          \N{\p}^{+}(G, M)        \ar[dd]             \ar[rr] && W_{\p}^{+}(G, M)        \ar[dd]            \ar[rr] && 1 & \\
& 1 \ar[rr] && \S{\lp}(M)       \ar@{_{(}->}[dl]            \ar[rr]  && \S{\lp}^{+}(G, M)      \ar[dd]     \ar@{_{(}->}[dl]    \ar[rr]   && R_{\lp}^{+}(G, M)       \ar[dd]       \ar@{_{(}->}[dl]  \ar[rr] && 1\\
1 \ar[rr]  && \S{\p}(M)        \ar[rr]  && \S{\p}^{+}(G, M)         \ar[dd]           \ar[rr] && R_{\p}^{+}(G, M)      \ar[dd]          \ar[rr] && 1 & \\
&&&&& 1 && 1  && \\
&&&& 1      && 1      &&& }  
\end{align}

For $u \in \N{\p}^{\theta}(G, M)$, we again write $w_{u}$ for the image of $u$ in $W^{\theta}_{\p}(G, M)$ and $x_{u}$ for the image of $u$ in $\S{\p}^{\theta}(G, M)$. Since $w_{u}$ normalizes $A_{\D{M}}$, it also normalizes $\D{M}$, and therefore can be treated as an element in the Weyl set $W(\D{G} \rtimes \D{\theta}, \D{M})$. The standard parabolic subgroup $P$ containing $M$ allows us to identify $w_{u}$ with an element in the Weyl set $W(G \rtimes \theta, M) \cong W(\lG \rtimes \theta, \lM)$. We choose a representative $\theta_{u}$ of $w_{u}$ in $G(F) \rtimes \theta$ preserving the $F$-splitting of $M$. Then $\p_{M} \in \cPbd{M^{\theta_{u}}}$ and $u$ defines an element in $\S{\p_{M}}^{\theta_{u}}$.

As in the previous case, we define
\begin{align}
\label{eq: theta induced character}
f_{G^{\theta}}(\p, u) = \sum_{\r_{M} \in \cPkt{\p_{M}}} <u, \r_{M}^{+}>tr(R_{P|\theta P}(\theta_{u}, \r_{M}^{+}, \p) \mathcal{I}_{P}^{\theta}(\r_{M}, f)), \,\,\,\,\,\, f \in \sH(G, \chi).  
\end{align}
Here $R_{P|\theta P}(\theta_{u}, \r_{M}^{+}, \p)$ is the normalized intertwining operator between $\mathcal{H}_{\theta P}(\r^{\theta^{-1}}_{M})$ and $\mathcal{H}_{P}(\r_{M})$, and 
\[
\mathcal{I}_{P}^{\theta}(\r_{M}, f) = R(\theta)^{-1} \circ \mathcal{I}_{P}(\r_{M}, f),
\]
where $R(\theta)$ is induced from the action of $\theta$ on $G(F)$.
We can also define
\begin{align}
\label{eq: theta induced twisted character}
\lif{f}_{\lG^{\theta}}(\lp, u) = \sum_{[\lr_{M}] \in \cPkt{\lp_{M}}} tr(R_{\lP | \theta \lP}(u, \lr_{M}, \lp) \mathcal{I}^{\theta, \x}_{P}(\lr_{M} \otimes \x^{-1}, \lif{f})), \,\,\,\,\,\, \lif{f} \in \sH(\lG, \lif{\chi}).   
\end{align}
Here $\x = \a^{M}(u) = \a^{G}(x_{u})$, $R_{\lP | \theta \lP}(u, \lr_{M}, \lp)$ is the normalized intertwining operator between $\mathcal{H}_{\theta \lP}(\lr_{M}^{\theta^{-1}})$ and $\mathcal{H}_{\lP}(\lr_{M} \otimes \x^{-1})$, and 
\[
\mathcal{I}_{\lP}^{\theta, \x}(\lr_{M} \otimes \x^{-1}, \lf) = R(\theta)^{-1} \circ \mathcal{I}^{\x}_{\lP}(\lr_{M} \otimes \x^{-1}, \lf).
\]
As before, we can identify 
\[
\mathcal{H}_{\theta \lP}(\lr_{M}^{\theta^{-1}}) \cong \bigoplus_{\r_{M} \subseteq \Res \lr_{M}} \mathcal{H}_{\theta P}(\r_{M}^{\theta^{-1}}),
\] 
and 
\[
\mathcal{H}_{\lP}(\lr_{M} \otimes \x^{-1}) \cong \bigoplus_{\r_{M} \subseteq \Res \lr_{M}} \mathcal{H}_{P}(\r_{M}).
\]
Then under these identifications, we have 
\[
R_{\lP | \theta \lP}(u, \lr_{M}, \lp) = \bigoplus_{\r_{M} \subseteq \Res \lr_{M}} <u, \r_{M}^{+}> R_{P|\theta P}(\theta_{u}, \r_{M}^{+}, \p).
\]
Therefore, if $\lf \in \sH(\lG, \lif{\chi})$ is supported on $\lif{Z}_{F}G(F)$ and $f$ is the restriction of $\lf$ to $G(F)$, then 
\[
\lf_{\lG^{\theta}}(\lp, u) = f_{G^{\theta}}(\p, u).
\]

Suppose $s$ is a semisimple element in $\cS{\p}^{\theta}$, and $(G', \p') \longrightarrow (\p, s)$. For any lift $\cPkt{\lp'}$, let us define
\[
\lf'_{\lG^{\theta}}(\lp, s) = \lf^{\lG'}(\lp'), \,\,\,\,\,\, \lf \in \sH(\lG, \lif{\chi}).
\]
Now we can state the $(\theta, \x)$-twisted intertwining relation.

\begin{theorem}
\label{thm: theta twisted intertwining relation}
Suppose $\p \in \cPbd{G}$, for semisimple $s \in \cS{\p}^{\theta}$ with $(G', \p') \longrightarrow (\p, s)$, the following identity holds for some lifts $\cPkt{\lp}$ and $\cPkt{\lp'}$ that 
\begin{align}
\label{eq: theta twisted intertwining relation}
\lf_{\lG^{\theta}}(\lp, u) = \lf'_{\lG^{\theta}}(\lp, s),  \,\,\,\,\,\,\, \lf \in \sH(\lG, \lif{\chi}), 
\end{align} 
where $u \in \N{\p}^{\theta}(G, M)$ and $s \in \cS{\p}^{\theta}$ have the same image in $\S{\p}^{\theta}$.
\end{theorem}

Finally, it is easy to see that Lemma~\ref{lemma: twisted intertwining relation 1} and Lemma~\ref{lemma: induced twisted character} can be extended to this case too. Moreover, the discussion of this section can be carried out for $\p \in \cuP{G}$ as well, and the corresponding twisted intertwining relation will follow from the tempered case.

\section{Stable Trace Formula and Multiplicity Formula} 
\label{sec: multiplicity formula}

\subsection{General setup and stable trace formula}
\label{subsec: stable trace formula}

In this section, we will set up the means to prove the main local theorem (Theorem~\ref{thm: refined L-packet}). The method is global and we are going to use various types of stabilized trace formulas. To be more precise it is the discrete part of the trace formula and its stabilized form that we are going to use. A detailed discussion of this can be found in (\cite{Arthur:2013}, Chapter 3). The stabilization of the ordinary trace formula has been established by Arthur in \cite{Arthur:2001} \cite{Arthur:2002} \cite{Arthur:2003}, and it also rests upon Ngo's proof \cite{Ngo:2010} of the Fundamental Lemma. In the twisted case, this results from the long project of Moeglin and Waldspurger \cite{MW:2016}.

Let us assume $F$ is global, and let $G$, $\lG$, $D$ and $\c$ be defined as in Section~\ref{subsubsec: notations}. Let $\theta$ be an automorphism of $\lG$ preserving an $F$-splitting and we assume $\c$ is $\theta$-invariant. Let $\lif{Z}_{\A_{F}} = \prod'_{v} \lif{Z}_{F_{v}}$ be a closed subgroup of $\lZ(\A_{F})$, such that $\lif{Z}_{F_{v}}$ satisfy the conditions in Section~\ref{subsec: the transfer map}. We also require $\lif{Z}_{\A_{F}} \Z(\A_{F}) = \lZ(\A_{F})$ and $\lif{Z}_{\A_{F}} \lZ(F)$ is mapped to a closed and cocompact subgroup in $\lZ(\A_{F})_{\theta}$. Let $\lif{Z}_{F} = \lif{Z}_{\A_{F}} \cap \lZ(F)$ and $\lif{\chi}$ be a character of $\lif{Z}_{\A_{F}} / \lif{Z}_{F}$. Let $Z_{{\A_{F}}} = \lif{Z}_{\A_{F}} \cap \Z(\A_{F})$ and $Z_{F} = \lif{Z}_{F} \cap \Z(F)$. We denote the restriction of $\lif{\chi}$ to $Z_{\A_{F}}$ by $\chi$. First we consider the discrete part of the $\theta$-twisted trace formula for $G$. For any nonnegative real number $t$ and $f \in \H(G, \chi)$, it is a distribution defined as follows
\begin{align}
\label{eq: spectral side}
\Idt{G^{\theta}}{, t}(f) = \sum_{\{ M \}} |W(M)|^{-1} \sum_{w \in W^{\theta}(M)_{reg}} |\det(w-1)_{\mathfrak{a}^{G^{\theta}}_{M}}|^{-1} tr(M_{P|\theta P, t}(w, \chi)I^{\theta}_{P, t}(\chi, f)).     
\end{align}
Here we need to give some explanations of this formula. The  outer sum is taken over $G(F)$-conjugacy classes of Levi subgroups $M$ of $G$, and the inner sum is taken over elements $w$ in the Weyl set
\[
W^{\theta}(M) = \Norm(A_{M}, G \rtimes \theta) / M
\]
such that $|\det(w-1)_{\mathfrak{a}^{G^{\theta}}_{M}}|^{-1} \neq 0$, where $\mathfrak{a}^{G^{\theta}}_{M}$ is the kernel of the canonical projection of 
\[
\mathfrak{a}_{M} \rightarrow \mathfrak{a}_{G} \rightarrow \mathfrak{a}_{G^{\theta}},
\] 
and 
\[
\mathfrak{a}_{G^{\theta}} = \mathfrak{a}_{G} / \{ X - \theta(X): X \in \mathfrak{a}_{G}\}.
\]
For any Levi subgroup $M$ of $G$, we can take the direct sum of 
\[
L^{2}_{disc, t}(M(F) \backslash M(\A_{F}), \zeta_{M}) \subseteq L^{2}_{disc}(M(F) \backslash M(\A_{F}), \zeta_{M})
\] 
such that the central character $\zeta_{M}$ extends $\chi$ and is invariant under some element of $W^{\theta}(M)_{reg}$, and the archimedean infinitesimal characters of the irreducible constituents have norm $t$ on their imaginary parts, then 
\[
I_{P, t}(\chi, f) = \int_{Z_{\A_{F}} \backslash G(\A_{F})} f(g) I_{P, t}(\chi, g) dg
\]
defines an operator on the space $\H_{P, t}(\chi)$ of the corresponding normalized induced representations. The operator $I^{\theta}_{P, t}(\chi, f)$ is the composition $R(\theta)^{-1} \circ I_{P, t}(\chi, f)$, and $M_{P|\theta P, t}(w, \chi)$ is the standard intertwining operator between $\H_{\theta P, t}(\chi)$ and $\H_{P, t}(\chi)$. Note when $\theta = id$, the term for $M = G$ in \eqref{eq: spectral side} is given by trace of $R^{G}_{disc, t}(f) := I_{G, t}(\chi, f)$ on the corresponding part of the discrete spectrum of $G$.

The discrete part of the trace formula \eqref{eq: spectral side} can be stabilized, and we get the following formula,
\begin{align}
\label{eq: endoscopic side}
\Idt{G^{\theta}}{, t}(f) = \sum_{G' \in \End{ell}{G^{\theta}}} \iota(G, G') \Sdt{G'}{, t}(f^{G'}).
\end{align}
Here $\Sdt{G'}{, t}(f')$ are stable distributions on $G'$, and they are defined by induction from the stabilized formula for $\Idt{G'}{, t}(f')$. If we denote the image of $Z_{\A_{F}}$ under the inclusion $(\Z)_{\theta} \rightarrow Z_{G'}$ by $Z'_{\A_{F}}$ and let $Z'_{F} = Z'_{\A_{F}} \cap Z_{G'}(\A_{F})$, then $\Idt{G'}{, t}(f')$ is defined with respect to a character $\chi'$ of $Z'_{\A_{F}}/Z'_{F}$ determined by $\chi$ and the twisted endoscopic embedding $\L{G}' \rightarrow \L{G}$, and $f' \in \H(G', \chi')$. The coefficients $\iota(G, G')$ are given by the Kottwitz's formula,
\begin{align}
\label{formula: endoscopic coefficient}
\iota(G, G')= \frac{|\pi_{0}(Z(\D{G})^{\Gal{}})|}{|\pi_{0}(Z(\D{G'})^{\Gal{})}|} \cdot  \frac{|\ker^{1}(F, Z(\D{G'}))|}{|\ker^{1}(F, Z(\D{G}))|} \cdot |\pi_{0}(\Aut_{G}(G'))|^{-1} \cdot |\kappa_{G^{\theta}} / \kappa_{G^{\theta}} \cap Z(\D{G}')|^{-1}.
\end{align}
where $\kappa_{G^{\theta}} = A_{\D{G}}^{\D{\theta}}$. When $\theta = id$, the term $|\kappa_{G^{\theta}} / \kappa_{G^{\theta}} \cap Z(\D{G}')|^{-1} =1$.

We can also write down the same trace formulas for $\lG$, but in this case we also need to consider the $\x$-twisted version of these trace formulas. Let $\x$ be a character of $\lG(\A_{F})/\lG(F)$ and $\lf \in \H(\lG, \lif{\chi})$, the discrete part of the $(\theta, \x)$-twisted trace formula for $\lG$ takes the form
\begin{align}
\label{eq: twisted spectral side}
\tIdt{\lG^{\theta}}{, t}(\lf) = \sum_{ \{ \lM \} } |W(\lM)|^{-1} \sum_{w \in W^{\theta}(\lM)_{reg}} |\det(w-1)_{\mathfrak{a}^{\lG^{\theta}}_{\lM}}|^{-1} tr(M_{\lP|\theta \lP, t}(w, \lif{\chi}) I^{\theta, \x}_{\lP, t}(\lif{\chi}, \lf)),     
\end{align}
where the operator $I^{\theta, \x}_{\lP, t}(\lif{\chi}, \lf)$ is the composition $R(\theta)^{-1} \circ R(\x) \circ I_{\lP, t}(\lif{\chi}, \lf)$. For the term corresponding to $\lM = \lG$, we let $R^{\lG}_{disc, t}(\lf) := I_{\lG, t}(\lif{\chi}, \lf)$ and denote $R(\theta)^{-1} \circ R(\x) \circ R^{\lG}_{disc, t}(\lf)$ by $R^{(\lG^{\theta}, \x)}_{disc, t}(\lf)$. After stabilization, \eqref{eq: twisted spectral side} becomes
\begin{align}
\label{eq: twisted endoscopic side}
\tIdt{\lG^{\theta}}{, t}(\lf) = \sum_{\lG' \in \tEnd{ell}{\lG^{\theta}}} \iota(\lG, \lG') \Sdt{\lG'}{, t}(\lf^{\lG'}),          
\end{align}
where the coefficients $\iota(\lG, \lG')$ are given by the same kind of formula as in ~\eqref{formula: endoscopic coefficient}.

We denote by $\mathcal{A}(\lG)$ (resp. $\mathcal{A}_{2}(\lG)$) the set of (resp. discrete) automorphic representations of $\lG$, and by $\Cl_{aut}(\lG)$ the set of families of Satake parameters of automorphic representations of $\lG$, modulo the equivalence relation that $c = c' \in \Cl_{aut}(\lG)$ if and only if $c_{v} = c'_{v}$ for almost all places $v$. More generally, we extend this notion to admissible representations of $\lG(\A_{F})$, and we denote the corresponding set by $\Cl_{\A}(\lG)$. For $\lif{c} \in \Cl_{\A}(\lG)$ and its projection $c$ on $\L{G}$, we can write $\Idt{G^{\theta}}{, t, c}(f)$ (resp. $\tIdt{\lG^{\theta}}{, t, \lif{c}}(\lf)$ and $R^{(\lG^{\theta}, \x)}_{disc, t, \lif{c}}(\lf)$) for the part of $\Idt{G^{\theta}}{, t}(f)$ (resp. $\tIdt{\lG^{\theta}}{, t}(\lf)$ and $R^{(\lG^{\theta}, \x)}_{disc, t}(\lf)$), which is contributed from automorphic representations $\r$ (resp. $\lr$) satisfying $c(\r) = c$ (resp. $c(\lr) = \lif{c}$). Then $\Sdt{G}{, t, c}(f)$ can be defined inductively using \eqref{eq: endoscopic side} for $\theta = id$. To be more precise, let
\[
\Sdt{G'}{, t, c}(f') = \sum_{c' \rightarrow c} \Sdt{G'}{, t, c'}(f'),
\]
and the sum is over the preimages $c' $of $c$ in $\mathcal{C}_{\A}(G')$ under the twisted endoscopic embedding $\L{G}' \rightarrow \L{G}$. Then we define
\begin{align*}
\Sdt{G}{, t, c}(f) = \Idt{G}{, t, c}(f)  - \sum_{G' \in \End{ell}{G} - \{G\}} \iota(G, G') \Sdt{G'}{, t, c}(f^{G'}).
\end{align*}
Similarly, we can define $\Sdt{\lG}{, t, \lif{c}}(\lf)$. The next lemma shows that $\Sdt{G}{, t, c}(f)$ (resp. $\Sdt{\lG}{, t, \lif{c}}(\lf)$) is stable and we get a decomposition for \eqref{eq: endoscopic side} (resp. \eqref{eq: twisted endoscopic side}).

\begin{lemma}
\label{lemma: trace formula component}
\begin{enumerate}
\item $\Sdt{G}{, t, c}(f)$ (resp. $\Sdt{\lG}{, t, \lif{c}}(\lf)$) is stable. 
\item The stabilization of the twisted trace formula \eqref{eq: endoscopic side} (resp. \eqref{eq: twisted endoscopic side}) can be decomposed according to $c \in \Cl_{\A}(G)$ (resp. $\lif{c} \in \Cl_{\A}(\lG)$), i.e.
\[
\Idt{G^{\theta}}{, t, c}(f) = \sum_{G' \in \End{ell}{G^{\theta}}} \iota(G, G') \Sdt{G'}{, t, c}(f^{G'}).
\]
resp.
\begin{align}
\label{eq: trace formula component}
\tIdt{\lG^{\theta}}{, t, \lif{c}}(\lf) = \sum_{\lG' \in \tEnd{ell}{\lG^{\theta}}} \iota(\lG, \lG') \Sdt{\lG'}{, t, \lif{c}}(\lf^{\lG'}).          
\end{align}

\end{enumerate}
\end{lemma}

The lemma is an application of the theory of multipliers, and the proof is the same as in (\cite{Arthur:2013}, Lemma 3.3.1). 

As in the local case where we study the relation of representations between $G$ and $\lG$, here we want to discuss the relation of $\Idt{G}{, t, c}(f)$  (resp. $\Sdt{G}{, t, c}(f)$) with $\Idt{\lG}{, t, \lif{c}}(\lf)$ (resp. $\Sdt{\lG}{, t, \lif{c}}(\lf)$). The next lemma is the first step of studying this relation.

\begin{lemma}
\label{lemma: Hecke eigenvalue correspondence}
Suppose $\lr$ is an irreducible admissible representation of $\lG(\A_{F})$ and $\r$ is an irreducible constituent of $\lr$ restricted to $G(\A_{F})$, then the set of Satake parameters $c(\lr)$ is mapped to $c(\r)$ under the projection $p : \L{\lG} \longrightarrow \L{G}$.
\end{lemma}

\begin{proof}
This lemma is essentially local, and it suffices to show for any place $v$ of $F$, if both $\r_{v}$ and $\lr_{v}$ are unramified and $\r_{v}$ is contained in the restriction of $\lr_{v}$ to $G(F_{v})$, then $c(\lr_{v})$ is mapped to $c(\r_{v})$. If $\lG_{v}$ is a torus, this follows from the Langlands correspondence for tori. In general, $\lr_{v}$ is an irreducible constituent of 
\(
\mathcal{I}_{\lif{B}_{v}} (\lif{\chi}_{v})
\)
for some unramified character $\lif{\chi}_{v}$ on the maximal torus $\lif{T}_{v}$ with Borel subgroup $\lif{B}_{v} \supseteq \lif{T}_{v}$, and one has $c(\lr_{v}) = c(\lif{\chi}_{v})$. Since
\[
\Res^{\lG_{v}}_{G_{v}} \, \, \mathcal{I}_{\lif{B}_{v}} \, (\lif{\chi}_{v}) \cong \mathcal{I}_{B_{v}} \, ( \Res^{\lif{T}_{v}}_{T_{v}} \, \lif{\chi}_{v}),
\]
then $c(\r_{v}) = c(\lif{\chi}_{v}|_{T_{v}})$. So again by the Langlands correspondence for tori one has $c(\lr_{v})$ mapped to $c(\r_{v})$.
\end{proof}


Now we assume $\lG$ is of type \eqref{eq: similitude}. By Corollary~\ref{cor: relative Hasse principle}, $\c(\lZ(\A_{F})) \cap D(F) = \c(\lZ(F))$. So we have $\c(\lZ(\A_{F})) \cap \c(\lG(F)) = \c(\lZ(F))$, which is equivalent to
\[
G(\A_{F}) \cap \lG(F) \lZ(\A_{F}) = G(F) \Z(\A_{F}).
\]
Therefore, 
\[
G(F)\Z(\A_{F}) \backslash G(\A_{F}) = \lG(F) \lZ(\A_{F}) \backslash \lG(F) \lZ(\A_{F}) G(\A_{F}).
\]
Let $\lif{\zeta}$ be a character of $\lZ(\A_{F})/\lZ(F)$ and $\zeta$ be the restriction of $\lif{\zeta}$ to $\Z(\A_{F})$, then we have
\[
L^{2}_{disc}( G(F) \backslash G(\A_{F}), \zeta) = L^{2}_{disc}(\lG(F) \backslash \lG(F) \lZ(\A_{F}) G(\A_{F}) , \lif{\zeta}).
\]
Note that right multiplication by $\lG(F) \lZ(\A_{F}) G(\A_{F})$ on the right hand side induces an action on the left hand side. In fact the action by $\lG(F)$ on the left hand side is given by conjugation on $G(F) \backslash G(\A_{F})$ and the action by $\lZ(\A_{F})$ is through the central character $\lif{\zeta}$. The following lemma shows that the $L^{2}$-discrete spectrum of $\lG(\A_{F})$ is essentially induced from the $L^{2}$-discrete spectrum of $G(\A_{F})$.

\begin{lemma}
\label{lemma: induced discrete spectrum}
$Ind^{\, \lG(\A_{F})}_{\, \lG(F)\lZ(\A_{F})G(\A_{F})} L^{2}_{disc}( G(F) \backslash G(\A_{F}), \zeta) \cong L^{2}_{disc}(\lG(F) \backslash \lG(\A_{F}), \lif{\zeta})$.
\end{lemma}

\begin{proof}
First of all,  there is a natural $\lG(\A_{F})$-equivariant isomorphism 
\[
\lG(F)\lZ(\A_{F}) \backslash \lG(\A_{F}) \cong G(F)\Z(\A_{F}) \backslash G(\A_{F}) \times_{\lG(F)\lZ(\A_{F})G(\A_{F})} \lG(\A_{F}).
\]
Here we can view 
\(
\lG(F)\lZ(\A_{F})G(\A_{F}) \backslash \lG(\A_{F}) 
\)
as a closed subgroup of $\c(\lZ(\A_{F}))D(F) \backslash D(\A_{F})$. Since 
\[
\c(\lZ(\A_{F}))D(F) \backslash D(\A_{F})
\] 
is compact (see \cite{Neukirch:1999}, Theorem 6.1.6), then $\lG(F)\lZ(\A_{F})G(\A_{F}) \backslash \lG(\A_{F})$ is also compact. So one can define an inner product on the space of $\lG(F)\lZ(\A_{F})G(\A_{F})$-equivariant continuous functions from $\lG(\A_{F})$ to $L^{2}( G(F) \backslash G(\A_{F}), \zeta )$, i.e.,
\[
[L^{2}( G(F) \backslash G(\A_{F}), \zeta ) \otimes C (\lG(\A_{F}))]^{\lG(F)\lZ(\A_{F})G(\A_{F})}
\]
by integrating over $\lG(F)\lZ(\A_{F})G(\A_{F}) \backslash \lG(\A_{F})$. Moreover, one can normalize its Haar measure such that
\[
L^{2}(\lG(F) \backslash \lG(\A_{F}), \lif{\zeta}) \cong \text{ completion of }[L^{2}( G(F) \backslash G(\A_{F}), \zeta ) \otimes C (\lG(\A_{F}))]^{\lG(F)\lZ(\A_{F})G(\A_{F})},
\]
which is compatible with $\lG(\A_{F})$-action. Note that the right hand side is nothing but 
\[
\Ind^{\lG(\A_{F})}_{\lG(F)\lZ(\A_{F})G(\A_{F})} L^{2}( G(F) \backslash G(\A_{F}), \zeta).
\]
Finally, since $\lG(F)\lZ(\A_{F})G(\A_{F}) \backslash \lG(\A_{F})$ is compact, one must have
\[
\Ind^{\lG(\A_{F})}_{\lG(F)\lZ(\A_{F})G(\A_{F})} L^{2}_{disc}( G(F) \backslash G(\A_{F}), \zeta) \cong L^{2}_{disc}(\lG(F) \backslash \lG(\A_{F}), \lif{\zeta}).
\]

\end{proof}

Let $X$ be the set of characters of $\lG(\A_{F}) / \lZ(\A_{F})G(\A_{F})$, and let $Y$ be the set of characters of $\lG(\A_{F}) / \lG(F) \lZ(\A_{F})G(\A_{F})$. If $\r$ is an irreducible admissible representation of $G(\A_{F})$, and $\lr$ is an irreducible admissible representation of $\lG(\A_{F})$, let us define
\begin{align*}
\lG(\r) &= \{g \in \lG(\A_{F}): \r^{g} \cong \r\} \\
X(\lr) &= \{ \x \in X: \lr \cong \lr \otimes \x \} \\
Y(\lr) &= Y \cap X(\lr).
\end{align*}
By \cite[Lemma 4.11]{HiragaSaito:2012}, we know $Y(\lr) = (\lG(\A_{F})/\lG(\r)\lG(F))^{*}$ is finite. The following lemma is inspired by \cite[Lemma 6.2]{LabesseLanglands:1979}.

\begin{lemma}
\label{lemma: multiplicity relation}
Suppose $\lr$ is an irreducible admissible representation of $\lG(\A_{F})$, and $\r$ is an irreducible constituent of $\lr$ restricted to $G(\A_{F})$. Then the multiplicities of $\lr$ and $\r$ in the discrete spectrum are related by the following formula
\begin{align}
\label{eq: multiplicity relation}
\sum_{\x \in X/ YX(\lr)} m(\lr \otimes \x) = \sum_{g \in \lG(\A_{F}) / \lG(\r)\lG(F)} m(\r^{g}).   
\end{align}
\end{lemma}

\begin{proof}

By Lemma~\ref{lemma: induced discrete spectrum},  
\begin{align}
\label{eq: multiplicity relation 0}
L^{2}_{disc}(\lG(F) \backslash \lG(\A_{F}), \lif{\zeta}) \cong \Ind^{\lG(\A_{F})}_{\lG(F)\lZ(\A_{F})G(\A_{F})} L^{2}_{disc}( G(F) \backslash G(\A_{F}), \zeta)
\end{align}
and we would like to expand the right hand side. First, we need to decompose $L^{2}_{disc}( G(F) \backslash G(\A_{F}), \zeta)$ as a representation of $\lG(F)\lZ(\A_{F})G(\A_{F})$. Recall $\lZ(\A_{F})$ acts through $\lif{\zeta}$ and $\lG(F)$ acts by conjugation on $G(F) \backslash G(\A_{F})$. Let $\r$ be any constituent in $L^{2}_{disc}(G(F) \backslash G(\A_{F}), \zeta )$ and
\[
G_{1}(\r) = \lG(F)\lZ(\A_{F})G(\A_{F}) \cap \lG(\r).
\]
Then $G_{1}(\r)$ will act on the $\r$-isotypic component $I(\r)$ and we get 
\[
I(\r) = \bigoplus^{m(\r)}_{\x_{1}} \r_{1} \otimes \x_{1},
\] 
where $\r_{1}$ is an extension of $\r$ to $G_{1}(\r)$ and the sum is over $m(\r)$ characters $\x_{1}$ of $G_{1}(\r) / \lZ(\A_{F})G(\A_{F})$, which depend on the extension $\r_{1}$ and can have multiplicities. Since $m(\r) = m(\r^{g})$ for $g \in \lG(F)$, we have the following decomposition 
\[
L^{2}_{disc}( G(F) \backslash G(\A_{F}), \zeta) = \bigoplus_{ \{ \r \} } \Ind^{\lG(F)\lZ(\A_{F})G(\A_{F})}_{G_{1}(\r)} (\bigoplus^{m(\r)}_{\x_{1}} \r_{1} \otimes \x_{1}),
\]
where the outer sum is taken over equivalence classes $\{ \r \}$ of constituents in $L^{2}_{disc}(G(F) \backslash G(\A_{F}), \zeta )$ under the action by $\lG(F)$. Taking this expression into \eqref{eq: multiplicity relation 0}, we get 
\begin{align*}
L^{2}_{disc}(\lG(F) \backslash \lG(\A_{F}), \lif{\zeta}) & \cong \bigoplus_{ \{ \r \} } \Ind^{\lG(\A_{F})}_{\lG(F)\lZ(\A_{F})G(\A_{F})} \Ind^{\lG(F)\lZ(\A_{F})G(\A_{F})}_{G_{1}(\r)} (\bigoplus^{m(\r)}_{\x_{1}} \r_{1} \otimes \x_{1}) \\
& \cong \bigoplus_{ \{ \r \} } \Ind^{\lG(\A_{F})}_{G_{1}(\r)} (\bigoplus^{m(\r)}_{\x_{1}} \r_{1} \otimes \x_{1}). 
\end{align*}
Moreover, 
\begin{align*}
L^{2}_{disc}(\lG(F) \backslash \lG(\A_{F}), \lif{\zeta}) & \cong \bigoplus_{ \{ \r \} } \Ind^{\lG(\A_{F})}_{\lG(\r)} \Ind^{\lG(\r)}_{G_{1}(\r)} (\bigoplus^{m(\r)}_{\x_{1}} \r_{1} \otimes \x_{1}) \\
& \cong  \bigoplus_{ \{ \r \} } \Ind^{\lG(\A_{F})}_{\lG(\r)} \bigoplus^{m(\r)}_{\x_{1}} ( \bigoplus_{\x \in (\lG(\r) / G_{1}(\r))^{*}} \lr_{1} \otimes \x \,\,\,\,) \otimes \x_{1} \\
& \cong \bigoplus_{ \{ \r \} }  \bigoplus^{m(\r)}_{\x_{1}} \bigoplus_{\x \in (\lG(\r) / G_{1}(\r))^{*}} \Ind^{\lG(\A_{F})}_{\lG(\r)} \lr_{1} \otimes \x \otimes \x_{1},
\end{align*}
where $\lr_{1}$ is an extension of $\r_{1}$ to $\lG(\r)$ and $\x_{1}$ is extended to $\lG(\r)$. Suppose $\r' = \r^{g}$ for some $g \in \lG(\A_{F}) / \lG(F)\lG(\r)$, we have 
\begin{align*}
G_{1}(\r') = G_{1}(\r^{g}) = \lG(F)\lZ(\A_{F})G(\A_{F}) \cap \lG(\r^{g}) = \lG(F)\lZ(\A_{F})G(\A_{F}) \cap \lG(\r)^{g}.
\end{align*}
Since $\lG(\r)^{g} \cong \lG(\r)$, then $G_{1}(\r') = G_{1}(\r)$. Hence $\r_{1}' \cong \r_{1}^{g} \otimes \x_{g}$ for some character $\x_{g}$ of $G_{1}(\r) / \lZ(\A_{F})G(\A_{F})$. Similarly one can show $\lr_{1}' \cong \lr_{1}^{g} \otimes \x_{g}$ for some extension of $\x_{g}$ to $\lG(\r)$. So
\[
\Ind^{\lG(\r)}_{G_{1}(\r)} \r_{1}' \cong \bigoplus_{\x \in (\lG(\r) / G_{1}(\r))^{*}} \lr_{1}^{g} \otimes \x \otimes \x_{g},
\]
and 
\begin{align*}
\Ind^{\lG(\A_{F})}_{\lG(\r)} \Ind^{\lG(\r)}_{G_{1}(\r)} \r_{1}' & \cong  \bigoplus_{\x \in (\lG(\r) / G_{1}(\r))^{*}} \Ind^{\lG(\A_{F})}_{\lG(\r)} \lr_{1}^{g} \otimes \x \otimes \x_{g} \\
& \cong \bigoplus_{\x \in (\lG(\r) / G_{1}(\r))^{*}} \Ind^{\lG(\A_{F})}_{\lG(\r)} \lr_{1} \otimes \x \otimes \x_{g}.
\end{align*}
Therefore 
\[
L^{2}_{disc}(\lG(F) \backslash \lG(\A_{F}), \lif{\zeta}) \cong \bigoplus_{ \{ \r \}^{\sim} } \bigoplus_{g \in \lG(\A_{F}) / \lG(F)\lG(\r)}  \bigoplus^{m(\r^{g})}_{\x_{1}} \bigoplus_{\x \in (\lG(\r) / G_{1}(\r))^{*}}  \Ind^{\lG(\A_{F})}_{\lG(\r)} \lr_{1} \otimes \x \otimes \x_{1} \otimes \x_{g},
\]
where the outer sum is taken over equivalence classes $\{ \r \}^{\sim}$ of constituents in $L^{2}(G(F) \backslash G(\A_{F}), \zeta )$ under the action by $\lG(\A_{F})$. Note the characters $\x_{1}$ in this formula depend on $\r^{g}$. By our definition of $G_{1}(\r)$, the characters of $\lG(\r) / G_{1}(\r)$ can be extended to that of $\lG(\A_{F}) / \lG(F)\lZ(\A_{F})G(\A_{F})$. If we let
\[
\lr =  \Ind^{\lG(\A_{F})}_{\lG(\r)} \lr_{1},
\]
then from the above formula one can see easily that
\[
\sum_{\x \in X/ YX(\lr)} m(\lr \otimes \x) = \sum_{g \in \lG(\A_{F}) / \lG(\r)\lG(F)} m(\r^{g}).   
\]
\end{proof}

\subsection{Multiplicity formula}
\label{subsec: multiplicity formula}

It is natural to apply Arthur's multiplicity formula (cf. Theorem~\ref{thm: discrete spectrum}) to the right hand side of \eqref{eq: multiplicity relation} for those representations parametrized by $\p \in \cPdt{G}$. But we can not apply that formula directly since it only gives the multiplicity of $\r$ as an $\sH(G)$-module for $[\r] \in \cPkt{\p}$. So let us define
\[
\m(\r) = \sum_{\r' \sim \r} m(\r'),
\]
where $\r' \cong \r$ as $\sH(G)$-modules. Then the multiplicity formula for $[\r] \in \cPkt{\p}$ asserts that
\begin{align}
\label{eq: multiplicity formula for classical group}
\m(\r) = m_{\p} |\S{\p}|^{-1} \sum_{x \in \S{\p}} <x , \r> ,   
\end{align}
where $m_{\p}$ is defined in Theorem~\ref{thm: discrete spectrum} and Remark~\ref{rk: discrete spectrum}. For any irreducible admissible representation $\lr$ of $\lG(\A_{F})$, whose restriction to $G(\A_{F})$ contains $\r$, let us also write
\[
\m(\lr) = \sum_{\{\lr' \sim_{X} \lr\} / X} \quad \sum_{\x \in X/ YX(\lr')} m(\lr' \otimes \x),            
\]
where $\lr' \cong \lr \otimes \x'$ as $\sH(\lG)$-modules for some $\x' \in X$, and we take such $\lr'$ modulo twists by $X$ in the sum. Then we can rewrite the formula \eqref{eq: multiplicity relation} as 
\begin{align}
\label{eq: multiplicity relation 1}
\m(\lr) = \sum_{g \in \lG(\A_{F}) / \lG(\r)\lG(F)} \m(\r^{g}).   
\end{align}
Now we can apply Arthur's multiplicity formula \eqref{eq: multiplicity formula for classical group} to the right hand side of \eqref{eq: multiplicity relation 1} to get the following result.

\begin{lemma}
\label{lemma: coarse multiplicity formula}
Suppose $\lr$ is an irreducible admissible representation of $\lG(\A_{F})$, and $\r$ is an irreducible constituent of $\lr$ restricted to $G(\A_{F})$. If $[\r] \in \cPkt{\p}$ for $\p \in \cPdt{G}$, then 
\begin{align}
\label{eq: coarse multiplicity formula}
\m(\lr) = m_{\p} \frac{|Y(\lr)|}{|\a(\S{\p})|}  \cdot |\S{\lp}|^{-1} \sum_{x \in \S{\lp}}<x , \r>. 
\end{align}
\end{lemma}

\begin{proof}
First we want to rewrite the right hand side of \eqref{eq: multiplicity relation 1} as an integral over $ \lG(\A_{F}) / \lG(F)\lZ(\A_{F})G(\A_{F}) $. Consider the integral
\begin{align*}
\int_{ \lG(\A_{F}) / \lG(F)\lZ(\A_{F})G(\A_{F}) } \m(\r^{g}) \, \mathrm{d} g  & = \sum_{g \in \lG(\A_{F}) / \lG(F)\lG(\r)} \int_{\lG(\r) / \lG(\r) \cap \lG(F)\lZ(\A_{F})G(\A_{F})} \m(\r^{hg}) \, \mathrm{d} h \\
& = \sum_{g \in \lG(\A_{F}) / \lG(F)\lG(\r)} \m(\r^{g}) \cdot \, \mathrm{vol}\{ \lG(\r) / \lG(\r) \cap \lG(F)\lZ(\A_{F})G(\A_{F}) \}.
\end{align*}
Since  
\begin{align*}
\mathrm{vol}\{ \lG(\r) / \lG(\r) \cap \lG(F)\lZ(\A_{F})G(\A_{F}) \}  = \frac{\mathrm{vol}\{ \lG(\A_{F}) / \lG(F)\lZ(\A_{F})G(\A_{F}) \}}{ |\lG(\A_{F}) / \lG(F)\lG(\r)|}
\end{align*}
and 
\[
|\lG(\A_{F}) / \lG(F)\lG(\r)| = |Y(\lr)|,
\]
then
\begin{align}
\m(\lr) =  \frac{|Y(\lr)|}{\mathrm{vol}\{ \lG(\A_{F}) / \lG(F)\lZ(\A_{F})G(\A_{F}) \}} \cdot \int_{ \lG(\A_{F}) / \lG(F)\lZ(\A_{F})G(\A_{F}) } \m(\r^{g}) \, \mathrm{d} g.  \label{eq: coarse multiplicity formula 1}
\end{align}
Combining the multiplicity formula \eqref{eq: multiplicity formula for classical group} and also our local formula \eqref{eq: theta twisting character}, we can compute the integral on the  right hand side of \eqref{eq: coarse multiplicity formula 1} as follows,
\begin{align*}
& \int_{ \lG(\A_{F}) / \lG(F)\lZ(\A_{F})G(\A_{F}) } \m(\r^{g}) \, \mathrm{d} g  = \int_{ \lG(\A_{F}) / \lG(F)\lZ(\A_{F})G(\A_{F}) } m_{\p} |\S{\p}|^{-1} \sum_{x \in \S{\p}} <x , \r^{g}> \, \mathrm{d} g \\
& = m_{\p} |\S{\p}|^{-1} \sum_{x \in \S{\p}} <x , \r> \cdot \int_{ \lG(\A_{F}) / \lG(F)\lZ(\A_{F})G(\A_{F}) } \x_{x}(g) \, \mathrm{d} g \\
& = m_{\p} |\S{\p}|^{-1} \sum_{x \in \S{\lp}} <x, \r> \cdot  \mathrm{vol}\{ \lG(\A_{F}) / \lG(F)\lZ(\A_{F})G(\A_{F}) \} \\
& = m_{\p} |\S{\p} / \S{\lp}|^{-1} |\S{\lp}|^{-1} \sum_{x \in \S{\lp}} <x , \r>  \cdot  \mathrm{vol}\{ \lG(\A_{F}) / \lG(F)\lZ(\A_{F})G(\A_{F}) \}.  \\
\end{align*}
Substitute this into \eqref{eq: coarse multiplicity formula 1}, one gets 
\[
\m(\lr) = m_{\p} \frac{|Y(\lr)|}{|\a(\S{\p})|} \cdot |\S{\lp}|^{-1} \sum_{x \in \S{\lp}} <x , \r>.
\]
\end{proof}

Although this lemma does not give a multiplicity formula for $\lG$, it has a very interesting consequence.

\begin{corollary}
\label{cor: coarse multiplicity formula}
Suppose $\lr$ is an irreducible admissible representation of $\lG(\A_{F})$, and $\r$ is an irreducible constituent of $\lr$ restricted to $G(\A_{F})$. If $[\r] \in \cPkt{\p}$ for $\p \in \cPdt{G}$, then there exists $\x \in X$ such that $\lr \otimes \x$ is isomorphic to a discrete automorphic representation as $\sH(\lG)$-module if and only if $<\cdot, \lr> = 1$. In particular,  if $\S{\lp} =1$ such character always exists.
\end{corollary}

\begin{proof}
Since $<x, \lr> = <x, \r>$ for $x \in \S{\lp}$, it follows from the formula \eqref{eq: coarse multiplicity formula} that 
\[
\m(\lr) = \begin{cases}
                                                                         m_{\p} \frac{|Y(\lr)|}{|\a(\S{\p})|} &\text{ if } <\cdot, \lr> = 1, \\
                                                                         0 &\text{ otherwise }.                                                                        
                                                                         \end{cases}  
\]
So the first part of this corollary is clear. Next if $\S{\lp} = 1$, then we always have
\[
\m(\lr) = m_{\p} \frac{|Y(\lr)|}{|\a(\S{\p})|},
\]
and this shows the second part.
\end{proof}

In fact, we can refine the result of Lemma~\ref{lemma: coarse multiplicity formula} to get a multiplicity formula for $\lG$ by applying the stabilized twisted trace formulas. First, we need to define an equivalence relation on $\Cl_{\A}(G)$ such that $c \sim c' \in \Cl_{\A}(G)$ if and only if $c_{v}$ is $\Sigma_{0}$-conjugate to $c'_{v}$ for almost all places, and we denote the set of equivalence classes by $\cCl_{\A}(G)$. Let $\cCl_{aut}(G)$ be the subset of $\cCl_{\A}(G)$ consisting of equivalence classes of $\Cl_{aut}(G)$. 

\begin{lemma}
\label{lemma: vanishing}
Suppose $\lif{c} \in \Cl_{\A}(\lG)$, then 
\[
\Idt{\lG}{, t, \lif{c}} (\lf) = \Sdt{\lG}{, t, \lif{c}} (\lf ) = 0
\]
for $\lf \in \H(\lG, \lif{\chi})$, unless the projection of $\lif{c}$ under $\bold{p}: \L{\lG} \rightarrow \L{G}$ belongs to the set $\cCl_{aut}(G)$.
\end{lemma}

\begin{proof}
It follows from Lemma~\ref{lemma: Hecke eigenvalue correspondence} and Lemma~\ref{lemma: induced discrete spectrum} that 
\begin{align}
\label{eq: vanishing}
tr\Rdt{\lG}{, t, \lif{c}} (\lf )= 0,    
\end{align}
unless $\lif{c}$ projects to $c \in \cCl_{aut}(G)$. Suppose the projection of $\lif{c}$ in $\cCl_{\A}(G)$ does not belong to $\cCl_{aut}(G)$, then by the principle of functoriality (which results from Arthur's theory \cite{Arthur:2013}), it neither belongs to $\cCl_{aut}(M)$ for any Levi subgroup $M$ of $G$, nor to $\cCl_{aut}(G')$ for any endoscopic group $G'$ of $G$. Then for the same reason as \eqref{eq: vanishing}, one gets $tr\Rdt{\lM}{, t, \lif{c}} (\lf_{M}) = 0$. So it follows from the definition (see \eqref{eq: spectral side}) that 
\[
\Idt{\lG}{, t, \lif{c}} (\lf) = 0.
\]
Since
\[
\Sdt{\lG}{, t, \lif{c}} (\lf) = \Idt{\lG}{, t, \lif{c}} (\lf) - ( \sum_{\lG' \in \End{ell}{\lG} - \{\lG\}} \iota(\lG, \lG') \Sdt{\lG'}{, t, \lif{c}}(\lf^{\lG'}) \,\,\, ),
\]  
we can assume $\Sdt{\lG'}{, t, \lif{c}}(\lf^{\lG'}) = 0$ by induction, then
\[
 \Sdt{\lG}{, t, \lif{c}} (\lf) = 0.
 \]
\end{proof}

For $\p \in \cP{G}$,  Arthur (cf. \cite{Arthur:2013}, Section 3.3) defines the $\p$-component of the discrete part of the twisted trace formula for $G$ and its stabilized form. Note that $c(\p)$ defines an element in $\cCl_{\A}(G)$, and $\p$ also determines the norm of the imaginary part of archimedean infinitesimal character, which can be denoted by $t(\p)$, so we can write 
\[
\Idt{G^{\theta}}{, \p}(f) =  \sum_{c \rightarrow c(\p)} \Idt{G^{\theta}}{, t(\p), c}(f),
\]
and
\[
\Sdt{G}{, \p}(f) = \sum_{c \rightarrow c(\p)} \Sdt{G}{, t(\p), c}(f),
\]
where these sums are all over preimages $c$ of $c(\p)$ in $\Cl_{\A}(G)$. Then the stabilization of the $\p$-component of the twisted trace formula for $G$ is 
\begin{align*}
\Idt{G^{\theta}}{, \p}(f) = \sum_{G' \in \End{ell}{G^{\theta}}} \iota(G, G') \Sdt{G'}{, \p}(f^{G'}),    
\end{align*}
where 
\[
\Sdt{G'}{, \p}(f^{G'}) = \sum_{c' \rightarrow c(\p)} \Sdt{G'}{, t(\p), c'}(f^{G'}).
\]
Here we want to define the $\p$-component of the discrete part of the twisted trace formula for $\lG$. Let us write
\begin{align*}
\tIdt{\lG^{\theta}}{, \p}(\lf) &= \sum_{\lif{c} \rightarrow c(\p)} \tIdt{\lG^{\theta}}{, t(\p), \lif{c}}(\lf),
\end{align*}
and 
\[
\Sdt{\lG}{, \p}(\lf) = \sum_{\lif{c} \rightarrow c(\p)} \Sdt{\lG}{, t(\p), \lif{c}}(\lf).
\]
Then the stabilization of the $\p$-component of the twisted stable trace formula for $\lG$ is 
\begin{align}
\label{eq: endoscopic side component}
\tIdt{\lG^{\theta}}{, \p}(\lf) = \sum_{\lG' \in \End{ell}{\lG^{\theta}, \x}} \iota(\lG, \lG') \Sdt{\lG'}{, \p}(\lf^{\lG'}),    
\end{align}
where 
\[
\Sdt{\lG'}{, \p}(\lf^{\lG'}) = \sum_{\lif{c}' \rightarrow c(\p)} \Sdt{\lG'}{, t(\p), \lif{c}'}(\lf^{\lG'}).
\]
Similarly, we can also define $R_{disc, \p}^{(\lG^{\theta}, \x)}(\lf)$. For $\p \in \cPdt{G}$, it only contributes to the discrete spectrum of $G$ (cf. Remark~\ref{rk: discrete spectrum}), and by Lemma~\ref{lemma: induced discrete spectrum} it also only contributes to the discrete spectrum of $\lG$. So we have
\begin{align}
\label{eq: discrete part vs discrete spectrum}
\tIdt{\lG^{\theta}}{, \p}(\lf) = tr R_{disc, \p}^{(\lG^{\theta}, \x)}(\lf).
\end{align}

Now we can give our multiplicity formula for $\lG$, and we will start with the simplest case, i.e., $\lG = GSp(2n)$ or $GSO(2n, \eta)$.

\begin{proposition}
\label{prop: multiplicity formula}
Suppose $\lG = GSp(2n)$ or $GSO(2n, \eta)$, $\lr$ is a discrete automorphic representation of $\lG$, and $\r$ is an irreducible constituent of $\lr$ restricted to $G(\A_{F})$. If $[\r] \in \cPkt{\p}$ for $\p \in \cPdt{G}$, then 
\begin{align}
\label{eq: multiplicity formula}
m(\lr) = m_{\lp} \frac{|Y(\lr)|}{|\a(\S{\p})|},   
\end{align}
where $m_{\lp} = 1 \text { or } 2$, and $m_{\lp} = 2$ only when $G$ is special even orthogonal, $\p \notin \cP{\com{G}}$, and $\lr \cong \lr^{\theta_{0}} \otimes \x$ for some $\x \in Y$.
\end{proposition}

\begin{proof}
Since $\lr$ is automorphic, we can take $\r$ to be automorphic as well by Lemma~\ref{lemma: induced discrete spectrum}, and hence $<\cdot, \r> = 1$. It follows from Lemma~\ref{lemma: coarse multiplicity formula} that 
\begin{align}
\label{eq: multiplicity formula 0}
\m(\lr) = m_{\p} \frac{|Y(\lr)|}{|\a(\S{\p})|}.
\end{align}
Since $\theta_{0}$ acts on $\{\lr' \sim_{X} \lr\}$, we can write 
\[
\m_{0}(\lr) = \sum_{\{\lr' \sim_{X} \lr\} / X, \theta_{0}} \quad \sum_{\x \in X/YX(\lr')} m(\lr' \otimes \x),
\]
where the sum modulo twists by $X$ and $\theta_{0}$. If $\r \cong \r^{\theta_{0}}$, then 
\[
\m_{0}(\lr) = \m(\lr) = \sum_{\x \in X/YX(\lr)} m(\lr \otimes \x).
\]
If $\r \ncong \r^{\theta_{0}}$, $\m_{0}(\lr) = \frac{1}{2} \m(\lr)$. Therefore, we have
\begin{align}
\label{eq: multiplicity formula 1}
\m_{0}(\lr) =  \begin{cases}                                                                                           
                                                                          m_{\p} \frac{|Y(\lr)|}{|\a(\S{\p})|} &     \text{ if } \r \cong \r^{\theta_{0}}, \\
                                                                          \frac{|Y(\lr)|}{|\a(\S{\p})|}              &     \text{ if } \r \ncong \r^{\theta_{0}}.                                                                
                                                                          \end{cases}                        
\end{align}
Note that $\a(\S{\p}) \subseteq Y(\lr)$, so 
\begin{align*}
\frac{|Y(\lr)|}{|\a(\S{\p})|} = |Y(\lr) / \a(\S{\p})|.  
\end{align*}
In particular, $Y(\lr) / \a(\S{\p})$ is a two-group. We can fix a subgroup of representatives in $Y(\lr)$ and denote it again by $Y(\lr) / \a(\S{\p})$.

Let us first consider the case $\r \ncong \r^{\theta_{0}}$. If $Y(\lr) / \a(\S{\p}) = 1$, then the lemma becomes obvious. So let us assume $1 \neq \x \in Y(\lr) / \a(\S{\p})$, and by the stabilized $\x$-twisted trace formula \eqref{eq: endoscopic side component}, one gets
\[
\tIdt{\lG}{, \p}(\lf) = \sum_{\lG' \in \tEnd{ell}{\lG}} \iota(\lG, \lG') \Sdt{\lG'}{, \p}(\lf^{\lG'}),       
\]
for $\lf \in \H(\lG, \lif{\chi})$. Since $\x$ is not in $\a(\S{\p})$, $\p$ can not factor through $\L{G'}$ for any $G' \in \End{ell}{G}$ such that $\lG' \in \tEnd{ell}{\lG}$. Then by Lemma~\ref{lemma: vanishing}, $\Sdt{\lG'}{, \p}(\lf^{\lG'}) = 0$ for all $\lG' \in \tEnd{ell}{\lG}$, and hence
\[
\tIdt{\lG}{, \p}(\lf) = 0.
\]
In particular,
\begin{align}
\label{eq: multiplicity formula 2}
tr \tRdt{\lG}{, \p}(\lf) = \tIdt{\lG}{, \p}(\lf) = 0,   
\end{align}
as $\p \in \cPdt{G}$. This is true for all nontrivial $\x \in Y(\lr) / \a(\S{\p})$. Let $I(\lr)$ be the $\lr$-isotypic component in $\Rdt{\lG}{, \p}$, and one observes $Y(\lr) / \a(\S{\p})$ will act on $I(\lr)$ by multiplication. The action of $Y(\lr) / \a(\S{\p})$ does not commute with that of $\lG(\A_{F})$, but one can take
\[
\iG{\A_{F}} = \{g \in \lG(\A_{F}) : \x(g) = 1 \text{ for all } \x \in Y(\lr) / \a(\S{\p}) \},
\]
which is of finite index in $\lG(\A_{F})$, and then it will commute with the action of $\iG{\A_{F}}$. In fact one can have a decomposition 
\[
I(\lr) = \bigoplus_{g \in \lG(\A_{F}) / \iG{\A_{F}}} I((\ir)^{g})
\]
by restricting to $\iG{\A_{F}}$.  The point is each summand is invariant under $Y(\lr) / \a(\S{\p})$ and has the same multiplicity as $\lr$. By \eqref{eq: multiplicity formula 2}, one has 
\[
tr ( R(\lf) \circ R(\x) )|_{I(\lr)} = 0,
\]
where $R(\x)$ denotes the multiplication by $\x$. In particular, one can restrict to those $\lf$ supported on $\iG{\A_{F}}$, then one has 
\[
tr (R(\lf) \circ R(\x) ) |_{I((\ir)^{g})} = 0
\] 
for all $g \in \lG(\A_{F}) / \iG{\A_{F}}$. We can view $I(\ir)$ as a representation of $\H(\iG{\A_{F}}) \times Y(\lr) / \a(\S{\p})$ and write it as $ \ir \otimes W$, then 
\[
tr (R(\lf) \circ R(\x) )|_{I(\ir)} = tr \ir(\lf) \cdot tr \ir_{W}(\x) = 0,
\]
where $\ir_{W}$ is the corresponding representation of $Y(\lr) / \a(\S{\p})$ on $W$. Therefore,
\[
tr \ir_{W}(\x) = 0
\] 
for $1 \neq \x \in Y(\lr) / \a(\S{\p})$.
We claim 
\begin{align}
\label{eq: divisibility}
|Y(\lr) / \a(\S{\p})| \text{ divides } \dim (W).
\end{align}
If that is the case, by noticing $m(\lr) = \dim (W)$ and comparing with \eqref{eq: multiplicity formula 1} one must have 
\[
|Y(\lr) / \a(\S{\p})| = \dim (W),
\]
hence $m(\lr) = |Y(\lr) / \a(\S{\p})|$.

To prove the claim \eqref{eq: divisibility}, one just needs to show the following general statement. Suppose $V$ is a finite dimensional representation over the complex numbers of a finite group $A$ such that the trace of each nontrivial element of $A$ is zero, then the order of $A$ must divide the dimension of $V$. To see this, let $\chi_{V}$ and $\chi_{\text{triv}}$ be the characters of $V$ and the trivial representation of $A$ respectively, then the multiplicity of the trivial representation in $V$ can be given by
\[
m = \langle \chi_{V}, \chi_{\text{triv}} \rangle = \dim(V)/|A|,
\]
which is an integer. Hence $|A|$ divides $\dim(V)$.


For the case $\r \cong \r^{\theta_{0}}$ and $m_{\p} = 1$,  the proof is the same. So we are left with the case $\r \cong \r^{\theta_{0}}$ and $m_{\p} = 2$. In this case, we have $\lr^{\theta_{0}} \cong \lr \otimes \x$ for some $\x \in X$. Let 
\[
X_{0}(\lr) = \{ \x \in X : \lr \cong \lr \otimes \x \text{ or } \lr^{\theta_{0}} \cong \lr \otimes \x \}.
\] 
If $\lr \otimes \x \ncong \lr^{\theta_{0}}$ for any $\x \in Y$, then 
\[
\sum_{\x \in X/ YX_{0}(\lr)} m(\lr \otimes \x) = \frac{|Y(\lr)|}{|\a(\S{\p})|}
\]
and the rest of the proof is again the same. If $\lr \otimes \x_{1} \cong \lr^{\theta_{0}}$ for some $\x_{1} \in Y$, we need to consider the action of the two-group
\begin{align}
\label{eq: theta character group}
<(\theta_{0}, \x_{1})> \times ( Y(\lr) / \a(\S{\p}) )
\end{align}
on $I(\lr)$, where $(\theta_{0}, \x_{1})$ acts by $R(\theta_{0})^{-1} \circ R(\x_{1})$. It commutes with the action of $(\theta_{0}, \x_{1})$-invariant functions in $\H(\iG{\A_{F}})$, i.e. $\lf^{\theta_{0}} \otimes \x_{1} = \lf$. And as a module of such space of functions, we have
\[
I(\lr) \cong ( \bigoplus_{g \in \lG(\A_{F}) / \iG{\A_{F}}} I((\ir_{+})^{g}) ) \bigoplus ( \bigoplus_{g \in \lG(\A_{F}) / \iG{\A_{F}}} I((\ir_{-})^{g}) )
\]
where the sign is according to the eigenvalues $\{\pm 1\}$ of any fixed intertwining operator between $\lr \otimes \x_{1}$ and $\lr^{\theta_{0}}$ after we identify $I(\lr) \cong m(\lr) \, \lr$. Note that the multiplicity of $\lr$ is the same as that of irreducible modules in $I(\ir_{+})$ and $I(\ir_{-})$ of the subspace of functions described above, and then the rest of the argument proceeds in the same way as before by using the stabilized $(\theta, \x)$-twisted trace formula \eqref{eq: endoscopic side component} for $(\theta_{0}, \x)$ in \eqref{eq: theta character group}.
 

\end{proof}

\begin{corollary}
\label{cor: modular character}
Suppose $\lr$ and $\lr'$ are discrete automorphic representations of $\lG$, such that $\lr \cong \lr' \otimes \x$ as $\sH(\lG)$-modules for some $\x \in X$. If $\r$ is an irreducible constituent in the restriction of $\lr$ to $G(\A_{F})$ and $[\r] \in \cPkt{\p}$ for $\p \in \cPdt{G}$, then there exists some $\x' \in Y$ and $\theta \in \Sigma_{0}$ such that $\lr' \cong \lr^{\theta} \otimes \x'$.
\end{corollary}

\begin{proof}
If $\lG = GSp(2n)$ or $GSO(2n, \eta)$, this can be seen easily by comparing \eqref{eq: multiplicity formula} with \eqref{eq: multiplicity formula 1}. In general, we can first go to the product group $\lif{\lG}$ of general symplectic groups and connected general even orthogonal groups (see \eqref{eq: product group}), and it is clear this corollary holds in that case. Then by restricting to $\lG$ we get the result.
\end{proof}

To generalize Proposition~\ref{prop: multiplicity formula}, for $[\r] \in \cPkt{\p}$ with $\p \in \cPdt{G}$, we denote by $\Sigma_{0}(\r, Y)$ the subgroup of $\Sigma_{0}$ consisting of $\theta$ such that $\lr  \otimes \x \cong \lr^{\theta}$ for some $\x \in Y$, where $\lr$ is an irreducible admissible representation of $\lG(\A_{F})$ containing $\r$ in its restriction to $G(\A_{F})$. If we write $\Sigma_{Y}(\r)$ for the quotient of $\Sigma_{0}$ by $\Sigma_{0}(\r, Y)$, then we have an exact sequence 
\begin{align}
\xymatrix{1 \ar[r] & \Sigma_{0}(\r, Y) \ar[r] & \Sigma_{0} \ar[r] & \Sigma_{Y}(\r) \ar[r] & 1},
\end{align}
where all these groups are two-groups. We can also choose a splitting of this sequence and write $\Sigma_{0} \cong \Sigma_{0}(\r, Y) \times \Sigma_{Y}(\r)$. 

\begin{corollary}
\label{cor: multiplicity formula}
Suppose $\lr$ is a discrete automorphic representation of $\lG$, and $\r$ is an irreducible constituent of $\lr$ restricted to $G(\A_{F})$. If $[\r] \in \cPkt{\p}$ for $\p \in \cPdt{G}$, then 
\begin{align}
\label{eq: generalized multiplicity formula}
m(\lr) = m_{\lp} \frac{|Y(\lr)|}{|\a(\S{\p})|},   
\end{align}
where $m_{\lp} = m_{\p} / |\Sigma_{Y}(\r)|$.
\end{corollary}

\begin{proof}
We will use the formula~\eqref{eq: multiplicity formula 0}
\[
\m(\lr) = m_{\p} \frac{|Y(\lr)|}{|\a(\S{\p})|}.
\]
It follows from Corollary~\ref{cor: modular character} that 
\[
\m(\lr) = \sum_{\theta \in \Sigma_{Y}(\r)} m(\lr^{\theta}).
\]
Since $m(\lr^{\theta}) = m(\lr)$ for $\theta \in \Sigma_{0}$, then we get
\[
|\Sigma_{Y}(\r)| \cdot m(\lr) = m_{\p} \frac{|Y(\lr)|}{|\a(\S{\p})|}.
\]
So by writing $m_{\lp} = m_{\p} / |\Sigma_{Y}(\r)|$, we have proved the formula \eqref{eq: generalized multiplicity formula}.
\end{proof}

Suppose $\p \in \cPdt{G}$, let $\lif{\zeta}$ be a character of $\lZ(\A_{F}) / \lZ(F)$ such that $\zeta = \lif{\zeta}|_{\Z}$ is the central character of $\cPkt{\p}$. If we denote by $\clPkt{\p, \lif{\zeta}}$ all equivalence classes of irreducible admissible representations of $\lG(\A_{F})$ as $\sH(\lG)$-modules with central character $\lif{\zeta}$, whose restriction to $G(\A_{F})$ have irreducible constituents contained in $\cPkt{\p}$, then by Corollary~\ref{cor: coarse multiplicity formula} we can always choose a representative for $[\lr] \in \clPkt{\p, \lif{\zeta}} / X$ with $<\cdot, \lr> = 1$ in the discrete spectrum of $\lG$. The following proposition gives a decomposition of the $\p$-component of the discrete spectrum of $\lG$.

\begin{proposition}
\label{prop: discrete spectrum}
Suppose $\p \in \cPdt{G}$, we have the following decomposition as $\sH(\lG)$-modules 
\begin{align}
\label{eq: discrete spectrum}
L^2_{disc, \p} (\lG(F) \backslash \lG(\A_{F}), \lif{\zeta}) = m_{\p} \sum_{\x \in Y / \a(\S{\p})} \quad \sum_{\substack{[\lr] \in \clPkt{\p, \lif{\zeta}} / X  \\  <\cdot, \lr> = 1}} \lr \otimes \x,
\end{align}
where $\lr$ are taken to be the representatives of $\clPkt{\p, \lif{\zeta}} / X$ in the discrete automorphic spectrum. Moreover,
\[
L^2_{disc, \p} (\lG(F) \backslash \lG(\A_{F}), \lif{\zeta}) = 0
\]
for $\p \in \cP{G} - \cPdt{G}$.
\end{proposition}

\begin{proof}
For $\p \in \cP{G} - \cPdt{G}$, we have $L^2_{disc, \p} (G(F) \backslash G(\A_{F})) = 0$ (cf. Theorem~\ref{thm: discrete spectrum}). Then it follows from Lemma~\ref{lemma: Hecke eigenvalue correspondence} and Lemma~\ref{lemma: induced discrete spectrum} that 
\(
L^2_{disc, \p} (\lG(F) \backslash \lG(\A_{F}), \lif{\zeta}) = 0.
\) 
Next we assume $\p \in \cPdt{G}$. By Lemma~\ref{lemma: induced discrete spectrum}, $L^2_{disc, \p} (\lG(F) \backslash \lG(\A_{F}), \lif{\zeta})$ consists of discrete automorphic representations in $\clPkt{\p, \lif{\zeta}}$. Then for any automorphic representation $\lr'$ in $L^2_{disc, \p} (\lG(F) \backslash \lG(\A_{F}), \lif{\zeta})$, there exists a representative $\lr$ chosen in \eqref{eq: discrete spectrum} such that $\lr \cong \lr' \otimes \x$ as $\sH(\lG)$-modules for some $\x \in X$. By Corollary~\ref{cor: modular character}, $\lr' \cong \lr^{\theta} \otimes \x'$ for $\theta \in \Sigma_{0}$ and $\x' \in Y$. In particular, $\lr' \cong \lr \otimes \x'$ as $\sH(\lG)$-modules. Therefore, it suffices to count the multiplicity of $\lr$ as $\sH(\lG)$-modules in $L^2_{disc, \p} (\lG(F) \backslash \lG(\A_{F}), \lif{\zeta})$. By Corollary~\ref{cor: modular character} again,  
\[
\sum_{\lr' \sim \lr} m(\lr') = \sum_{\theta \in \Sigma_{0}, \, \x \in \bar{Y}(\lr)} m(\lr^{\theta} \otimes \x) = |\frac{\bar{Y}(\lr)}{Y(\lr)}| \cdot |\Sigma_{Y}(\r)| \cdot m(\lr),
\]
where $\lr' \cong \lr$ as $\sH(\lG)$-modules and 
\[
\bar{Y}(\lr) = \{\x \in Y : \lr \otimes \x \cong \lr \text{ as $\sH(\lG)$-modules}\}.
\]
By Corollary~\ref{cor: multiplicity formula}, we have 
\[
|\Sigma_{Y}(\r)| \cdot m(\lr) = m_{\p} \frac{|Y(\lr)|}{|\a(\S{\p})|},
\] 
so 
\[
\sum_{\lr' \sim \lr} m(\lr') = m_{\p} \frac{|\bar{Y}(\lr)|}{|\a(\S{\p})|}.
\] 
This is exactly the multiplicity we get from \eqref{eq: discrete spectrum}.

\end{proof}

Now let us get back to the multiplicity formula. Note under the assumption of Proposition~\ref{prop: multiplicity formula}, if $\lG$ is general symplectic, then the multiplicity formula \eqref{eq: multiplicity formula} becomes
\[
m(\lr) = \frac{|Y(\lr)|}{|\a(\S{\p})|}.
\]
It is an interesting question to ask when one can have multiplicity one, i.e. $|Y(\lr)| = |\a(\S{\p})|$. Since $\a(\S{\p})$ is a subgroup of $Y(\lr)$, it is the same to ask when 
\[
\a(\S{\p}) = Y(\lr).
\]
By Corollary~\ref{cor: theta twisting character} we have the following description for $Y(\lr)$. Let us define
\begin{align*}
\prod^{aut}_{v} \a(\S{\p_{v}}) &:= \{ \x \in Y :  \x_{v} \in \a(\S{\p_{v}}) \text{ for all } \, v \}, \\
\prod^{aut}_{almost \, all \, v} \a(\S{\p_{v}}) &:= \{ \x \in Y : \x_{v} \in \a(\S{\p_{v}}) \text{ for almost all} \, v \},
\end{align*}
then
\[
Y(\lr) = \prod^{aut}_{v} \a(\S{\p_{v}}).
\]
Moreover, we get a sequence of inclusions
\[
\a(\S{\p}) \subseteq \prod^{aut}_{v} \a(\S{\p_{v}}) \subseteq \prod^{aut}_{almost \, all \, v} \a(\S{\p_{v}}).
\]
Motivated by the case that $G$ is symplectic and $\p \in \cPdt{G}$, we give the following definition for both symplectic groups and special even orthogonal groups.

\begin{definition}
Suppose $\p \in \cP{G}$, we say {\bf multiplicity one} holds for $\lp$ if 
\[
\a(\S{\p}) = \prod^{aut}_{v} \a(\S{\p_{v}}).
\]
\end{definition}

\begin{definition}
Suppose $\p \in \cP{G}$, we say {\bf strong multiplicity one} holds for $\lp$ if 
\[
\prod^{aut}_{v} \a(\S{\p_{v}}) = \prod^{aut}_{almost \, all \, v} \a(\S{\p_{v}}).
\]
\end{definition}

The motivation for the first definition is now clear, while the second definition needs some explanation. But before giving the explanation, we want to give two modified definitions of the same kind. In view of Theorem~\ref{thm: discrete spectrum}, we need to deal with the group of characters $\x_{v}$ such that
\[
\lf_{v}(\lr_{v} \otimes \x_{v}) = \lf_{v}(\lr_{v}), \,\,\,\, \lf_{v} \in \sH(\lG_{v})
\]
for $[\r _{v}] \in \cPkt{\p_{v}}$. It follows from Corollary~\ref{cor: theta twisting character} that this group is isomorphic to $\a(\S{\p_{v}}^{\Sigma_{0}})$.
Then we can similarly define a sequence of inclusions
\[
\a(\S{\p}^{\Sigma_{0}}) \subseteq \prod^{aut}_{v} \a(\S{\p_{v}}^{\Sigma_{0}}) \subseteq \prod^{aut}_{almost \, all \, v} \a(\S{\p_{v}}^{\Sigma_{0}}),
\]
and define the concepts of multiplicity one and strong multiplicity one in the same way regarding these groups.

\begin{definition}
Suppose $\p \in \cP{G}$, we say {\bf $\Sigma_{0}$-multiplicity one} holds for $\lp$ if 
\[
\a(\S{\p}^{\Sigma_{0}}) = \prod^{aut}_{v} \a(\S{\p_{v}}^{\Sigma_{0}}).
\]
\end{definition}

\begin{definition}
Suppose $\p \in \cP{G}$, we say {\bf $\Sigma_{0}$-strong multiplicity one} holds for $\lp$ if 
\[
\prod^{aut}_{v} \a(\S{\p_{v}}^{\Sigma_{0}}) = \prod^{aut}_{almost \, all \, v} \a(\S{\p_{v}}^{\Sigma_{0}}).
\]
\end{definition}
Recall that we can associate a global $L$-packet $\cPkt{\p}$ to $\p \in \cP{G}$, so we can talk about strong multiplicity one for the global $L$-packet $\cPkt{\p}$, i.e. if $\r$ is automorphic, and $[\r_{v}] \in \cPkt{\p_{v}}$ for almost all places $v$, then $[\r]$ lies in $\cPkt{\p}$. As in the local case (see Theorem~\ref{thm: refined L-packet}), we can expect to lift the global $L$-packet $\cPkt{\p}$ to some global $L$-packet $\cPkt{\lp}$ for $\lG$. Obviously the lift is not unique, but as one can see from Corollary~\ref{cor: modular character}, it should be unique up to twisting by id\`ele class characters. Because we already have strong multiplicity one for $\cPkt{\p}$, so strong multiplicity one for $\cPkt{\lp}$ is equivalent to the property that for any $\x$ in $Y$ if $\cPkt{\lp_{v}} = \cPkt{\lp_{v}} \otimes \x_{v}$ for almost all places $v$, then $\cPkt{\lp} = \cPkt{\lp} \otimes \x$. And it can be easily seen that this property is equivalent to the condition of $\Sigma_{0}$-strong multiplicity one in our definition.

\subsection{Statement of global theorem} 
\label{subsec: statement of main global theorems}

After discussing the multiplicity question, we want to describe the $\p$-component of the discrete spectrum for $\lG$, which should be an analogue of Theorem~\ref{thm: discrete spectrum}. We again assume $\lG$ is of type ~\eqref{eq: similitude}.

\begin{conjecture}
\label{conj: global L-packet}
\begin{enumerate}
\item Suppose $\p \in \cP{G}$, one can associate a global packet $\cPkt{\lp}$ of $\sH(\lG)$-modules of irreducible admissible representations for $\lG(\A_{F})$ satisfying the following properties:
          \begin{enumerate}
          \item$ \cPkt{\lp} = \bigotimes'_{v} \cPkt{\lp_{v}}$ where $\cPkt{\lp_{v}}$ is some lift of $\cPkt{\p_{v}}$ defined in Theorem~\ref{thm: refined L-packet}.
          \item there exists $[\lr] \in \cPkt{\lp}$ such that that $\lr$ is isomorphic to an automorphic representation as $\sH(\lG)$-modules.
          \end{enumerate}
Moreover, $\cPkt{\lp}$ is unique up to twisting by characters of $\lG(\A_{F}) / \lG(F)G(\A_{F})$. And we can define a global character of $\S{\lp}$ by 
\[
<x, \lr> := \prod_{v}<x_{v}, \lr_{v}> \,\,\,\,\, \text{ for } \, \, \lr \in \cPkt{\lp} \text{ and } \, \, x \in \S{\lp}. 
\]
\item Suppose $\p \in \cPdt{G}$, the $\p$-component of the discrete spectrum of $\lG(\A_{F})$ as $\sH(\lG)$-module has a decomposition.
\begin{align}
\label{formula: discrete spectrum}
L^{2}_{disc, \p}(\lG(F) \backslash \lG(\A_{F}), \lif{\zeta}) = m_{\p} \bigoplus_{\x \in Y / \a(\S{\p})} \bigoplus_{\substack{ [\lr] \in \cPkt{\lp} \otimes \x \\ <\cdot, \lr> = 1}} \lr,
\end{align}
where $m_{\p}$ is defined as in Remark~\ref{rk: discrete spectrum}.
\end{enumerate}
\end{conjecture}

Along with this conjecture,  we need to prove some results about the stable multiplicity formula for $\lG$ (see Conjecture~\ref{conj: global conjecture}). This formula has been conjectured by Arthur \cite{Arthur:1990} for any quasisplit connected reductive groups, and he also proved this for special orthogonal group and symplectic group in \cite{Arthur:2013}. To state the formula, we need some preparations. Suppose $S$ is a connected complex reductive group with an automorphism $\theta$, we denote $S^{\theta} = S \rtimes \theta$, which can be viewed as a connected component of the complex reductive group $S^{+} := S \rtimes <\theta>$. We fix a maximal torus $T$ of $S$, and define the Weyl set
\[
W^{\theta}(S) = \Norm(T, S^{\theta}) / T.
\]
Let $W^{\theta}(S)_{reg}$ be the set of Weyl elements $w$ such that 
\[
\det(w-1)|_{\mathfrak{a}_{T}} \neq 0.
\] 
Moreover, let $s^{0}(w)$ denote the sign $(-1)^{n}$, where $n$ is the number of positive roots of $(S, T)$ mapped by $w$ to negative roots. Now we can assign to $S^{\theta}$ a real number
\[
i^{\theta}(S) = |W(S)|^{-1} \sum_{w \in W^{\theta}_{reg}(S)} s^{0}(w) |\det(w - 1)|_{\mathfrak{a}_{T}}^{-1},
\]
where $W(S)$ is the Weyl group of $S$. Next we want to define a constant $\sigma(S_{1})$ associated with any connected complex reductive group $S_{1}$. To define this we have to introduce some more notations. Still for the original $S^{\theta}$,  let us denote the set of semisimple elements of $S^{\theta}$ by $S^{\theta}_{ss}$. And for any $s \in S^{\theta}_{ss}$, we write
\begin{align*}
S_{s} & = \Cent(s, S).
\end{align*}
Let 
\[
S^{\theta}_{ell} = \{s \in S^{\theta}_{ss} : |Z(S_{s})| < \infty \},
\]
and $\mathcal{E}^{\theta}_{ell}(S)$ be the $S$-conjugacy classes in $S^{\theta}_{ell}$. Finally the constant $\sigma(S_{1})$ can be characterized by the following proposition (\cite{Arthur:2013}, Proposition 4.1.1).

\begin{proposition}
\label{prop: endoscopy of complex group}
There are unique constants $\sigma(S_{1})$ defined for all connected complex reductive groups $S_{1}$, such that for any connected component $S^{\theta}$ of a complex reductive group, the following number
\[
e^{\theta}(S) = \sum_{s \in \mathcal{E}^{\theta}_{ell}(S)} |\pi_{0}(S_{s})|^{-1} \sigma((S_{s})^{0})
\]
equals $i^{\theta}(S)$, and furthermore
\[
\sigma(S_{1}) = \sigma(S_{1} / Z_{1}) |Z_{1}|^{-1},
\]
for any central subgroup $Z_{1}$ of $S_{1}$.
\end{proposition}

Now we can state the stable multiplicity formula for $\lG$ as follows.

\begin{conjecture}
\label{conj: stable multiplicity formula}
Suppose $\p \in \cP{G}$, then 
\begin{align}
\label{formula: stable multiplicity}
\Sdt{\lG}{, \p}(\lf) = m_{\p} \sum_{\x \in Y / \a(\S{\p})} |\S{\lp}|^{-1} \sigma( \com[0]{\cS{\p}}) \lf^{\lG} (\lp \otimes \x), \,\,\,\,\,\,\, \lf  \in \sH(\lG, \lif{\chi}),
\end{align}
where
\[
 \lf^{\lG} (\lp \otimes \x) := \prod_{v} \lf_{v}(\lp_{v} \otimes \x_{v}),
 \]
 with respect to $\cPkt{\lp}$ defined in Conjecture~\ref{conj: global L-packet}.
\end{conjecture}



Finally, we need a twisted version of the decomposition \eqref{formula: discrete spectrum}, whose role will be clear in the next section.

\begin{conjecture}
\label{conj: compatible normalization}
Suppose $\p \in \cPdt{G}$ and $x \in \S{\p}^{\theta}$ with $\a(x) = \x$ for $\theta \in \Sigma_{0}$ and some character $\x$ of $\lG(\A_{F})/\lG(F)G(\A_{F})$. For $[\lr] \in \cPkt{\lp}$ with $<\cdot, \lr> =1$, the canonical intertwining operator 
\[
R(\theta)^{-1} \circ R(\x)
\]
restricted to the $\lr$-isotypic component $I(\lr)$ in the discrete spectrum is equal to the product of $m(\lr)$ and the local intertwining operators $A_{\lr_{v}}(\theta, \x_{v})$ normalized by $x_{v}$ (see \eqref{eq: theta twisted intertwining operator}), i.e.
\begin{align}
\label{formula: theta twisted discrete spectrum}
\tIdt{\lG^{\theta}}{, \p}(\lf) = m_{\p} \sum_{\x' \in Y / \a(\S{\p})} \sum_{\substack{[\lr] \in \cPkt{\lp} \otimes \x' \\ <\cdot, \lr> = 1}} \lf_{\lG^{\theta}}(\lr, \x),  \,\,\,\,\, \lf \in \sH(\lG, \lif{\chi}),
\end{align}
where $ \lf_{\lG^{\theta}}(\lr, \x) = \prod_{v} \lf_{\lG^{\theta}_{v}}(\lr_{v}, \x_{v})$, and it does not depend on $x$.
\end{conjecture}

\begin{remark}
\label{rk: compatible normalization}
This kind of result has been proved in the cases of special even orthogonal groups (see \cite{Arthur:2013}, Theorem 4.2.2) and general linear groups (see \cite{Arthur:2013}, Lemma 4.2.3).
\end{remark}

In this paper, we will only establish these conjectures in a special case. 

\begin{theorem}
	\label{thm: main global}
	Suppose $G = G_{1} \times G_{2} \times \cdots \times G_{q}$, such that $G_{i}$ is a symplectic group or a special even orthogonal group. For $\p = \p_{1} \times \p_{2} \times \cdots \times \p_{q} \in \cP{G}$ with $\p_{i} \in \cP{G_{i}}$, if $\S{\lp_{i}} = 1$ for all $i$, then Conjecture~\ref{conj: global L-packet}, ~\ref{conj: compatible normalization} hold. If we further assume $\p \in \cPdt{G}$, then Conjecture~\ref{conj: stable multiplicity formula} also holds.
\end{theorem}

\subsection{Comparison of trace formulas}
\label{subsec: comparison of trace formulas}

We assume $\lG$ is of type \ref{eq: similitude} and $\theta \in \Sigma_{0}$. Since we are going to prove all the theorems by induction, here we would like to take a temporary induction assumption: we assume Conjecture~\ref{conj: global L-packet}, ~\ref{conj: stable multiplicity formula}, ~\ref{conj: compatible normalization} together with our main local theorem (Theorem~\ref{thm: refined L-packet}) hold for the proper Levi subgroups and twisted endoscopic groups of $\lG$. Based on this assumption, we want to expand the $\p$-component of \eqref{eq: twisted spectral side} and \eqref{eq: twisted endoscopic side} in terms of local objects. Before we do the expansion, let us write $\P{G, \p}$ for the set of global Langlands parameters of $G$ giving rise to $\p \in \cP{G}$. It is clear that $|\P{G, \p}| = m_{\p}$. So we can write formally those formulas \eqref{formula: discrete spectrum}, \eqref{formula: stable multiplicity} and \eqref{formula: theta twisted discrete spectrum} for $\p_{G} \in \P{G, \p}$ by simply setting $m_{\p_{G}} = 1$. In fact these formal formulas do make sense when we associate to $\p_{G}$ the refined global $L$-packet (see \cite{Arthur:2013}, Section 8.4). But we do not need this refinement here, for eventually we are going to sum over $\P{G, \p}$. The benefit of working with these global Langlands parameters is one can imitate the computation in (\cite{Arthur:1990}, Section 5 and 7), where one does assume the global Langlands correspondence. 

\subsubsection{The spectral expansion}
\label{subsubsec: spectral expansion}

Let us write the $\p$-component of \eqref{eq: twisted spectral side} as 

\begin{align*}
\tIdt{\lG^{\theta}}{, \p}(\lf) = \sum_{ \{ \lM \} } |W(\lM)|^{-1} \sum_{w \in W^{\theta}(\lM)_{reg}} |\det(w-1)_{\mathfrak{a}^{\lG^{\theta}}_{\lM}}|^{-1} tr(M_{\lP|\theta \lP, \p}(w, \lif{\chi}) I^{\theta, \x}_{\lP, \p}(\lif{\chi}, \lf)).
\end{align*}
So the key is to expand
\begin{align}
\label{eq: twisted spectral expansion 1}
tr(M_{\lP|\theta \lP, \p}(w, \lif{\chi}) I^{\theta, \x}_{\lP, \p}(\lif{\chi}, \lf))
\end{align}
By definition, \eqref{eq: twisted spectral expansion 1} does not vanish only if there exists $\p_{M} \in \cPdt{M, \p}$. Moreover, the $(G(F) \rtimes \Sigma_{0})$-conjugacy class of $M$ such that $\cPdt{M, \p} \neq \emptyset$ is determined by $\p$, and the choice of $\p_{G} \in \P{G, \p}$ determines the $G(F)$-conjugacy class of $M$ such that $\Pdt{M, \p_{G}} \neq \emptyset$. So we can fix such a $G(F)$-conjugacy class of $M$ and assume $\p_{M} \in \cPdt{M, \p}$. To apply our induction assumption we also need to assume $M \neq G$, i.e., $\p \notin \cPdt{G}$.

Note the diagram \eqref{eq: theta twisted intertwining relation diagram} in our discussion of the local $(\theta, \x)$-twisted intertwining relation can be defined in the global case, and the global analogue of those groups in the diagram will map to their local counterparts. It is not hard to show, using a similar argument in Proposition~\ref{prop: multiplicity formula}, that \eqref{eq: twisted spectral expansion 1} vanishes unless there exists $u \in \N{\p}^{\theta}$ such that $w_{u} = w$ and 
\(
\x = \a(x_{u}).
\)
So $\p_{M} \in \cPdt{M^{\theta_{u}}}$. Now we can apply Conjecture~\ref{conj: global L-packet} to $\lp_{M}$. For any $[\lr_{M}] \in \cPkt{\lp_{M}}$, let us define
\[
R_{\lP|\theta \lP}(u, \lr_{M}, \lp) : = \prod_{v} R_{\lP_{v} |\theta \lP_{v}}(u_{v}, \lr_{M_{v}}, \lp_{v}).
\]
In particular, if $\lr_{M} \in \mathcal{A}_{2}(\lM)$, we can write 
\[
R_{\lP|\theta \lP}(w, \lr_{M}, \lp) := r_{P}(w, \p_{M})^{-1} M_{\lP|\theta \lP}(w, \lr_{M}), 
\]
where $r_{P}(w, \p_{M})$ is the global normalizing factor defined by 
\[
r_{P}(w, \p_{M}) = \prod_{v} r_{P}(w_{v}, \p_{M_{v}}).
\]
It follows from Conjecture~\ref{conj: compatible normalization} and analogous result for $GL(N)$ (cf. Remark~\ref{rk: compatible normalization}) 
that
\[
R_{\lP|\theta \lP}(w, \lr_{M}, \lp) = R_{\lP|\theta \lP}(u, \lr_{M}, \lp)
\]
for any $u \in \N{\p}^{\theta}(w, \x)$. Here $\N{\p}^{\theta}(w, \x)$ consists of $u \in \N{\p}^{\theta}$ such that $w_{u} = w$ and $\a(x_{u}) = \x$. Applying Conjecture~\ref{conj: global L-packet} (2) to $\lM$, we can write \eqref{eq: twisted spectral expansion 1} as a double sum over $\p_{G} \in \P{G, \p}$ and $\p_{M} \in \Pdt{M^{\theta_{u}}, \p_{G}}$ of 
\[
 \sum_{\x' \in Y / \a(\S{\p_{M}})} \sum_{[\lr_{M}] \in \cPkt{\lp_{M}} \otimes \x'} \delta_{\lp_{M} }(\lr_{M}) r_{P}(w, \p_{M}) tr ( R_{\lP|\theta \lP}(u, \lr_{M}, \lp) I^{\theta, \x}_{\lP}(\lr_{M} \otimes \x^{-1}, \lf) ),   
\]
where
\[
\delta_{\lp_{M}} (\lr_{M}) = |\S{\lp_{M}}|^{-1} \sum_{x \in \S{\lp_{M}}} <x, \lr_{M}>.
\]
Moreover, we can write
\[
\sum_{x \in \S{\lp_{M}}} <x, \lr_{M}>  R_{\lP|\theta \lP}(u, \lr_{M}, \lp) = \sum_{u \in \N{\p}^{\theta}(w, \x)} R_{\lP|\theta \lP}(u, \lr_{M}, \lp).
\]
If we switch the sum over $\lr_{M} \in \cPkt{\lp_{M}} \otimes \x'$ with $u \in \N{\p}^{\theta}(w, \x)$, and define
\[
\lf_{\lG^{\theta}}(\lp \otimes \x', u) = \sum_{[\lr_{M}] \in \cPkt{\lp_{M}} \otimes \x'} tr (R_{\lP|\theta \lP}(u, \lr_{M}, \lp) I^{\theta, \x}_{\lP}(\lr_{M} \otimes \x^{-1}, \lf)  ),
\]
then \eqref{eq: twisted spectral expansion 1} becomes a double sum over $\p_{G} \in \P{G, \p}$ and $\p_{M} \in \Pdt{M^{\theta_{u}}, \p_{G}}$ of
\begin{align}
\sum_{\x' \in Y / \a(\S{\p_{M}})}  |\S{\lp_{M}}|^{-1} \sum_{u \in \N{\p}^{\theta}(w, \x)}  r_{P}(w, \p_{M}) \lf_{\lG^{\theta}}(\lp \otimes \x', u).    \label{eq: twisted spectral expansion 2}
\end{align}
Since we are taking $\lf \in \sH(\lG, \lif{\chi})$, the contributions of $\p_{G} \in \Pdt{G, \p}$ to the $\p$-component of \eqref{eq: twisted spectral side} are the same. 
So the $\p$-component of \eqref{eq: twisted spectral side} can be written as a sum over $w \in W^{\theta}(\lif{M})_{reg}$ and $\p_{M} \in \Pdt{M^{\theta_{u}}, \p_{G}}$ of
\[
m_{\p} | W(\lif{M}) |^{-1} | det (w - 1)_{\mathfrak{a}^{\lG^{\theta}}_{\lM}} |^{-1} 
\]
multiplied with \eqref{eq: twisted spectral expansion 2}. Here we can identify $W(\lM)$ with $W(M)$, and it is easy to see
\[
| det (w - 1)_{\mathfrak{a}^{\lG^{\theta}}_{\lM}} | = | det (w - 1)_{\mathfrak{a}^{G^{\theta}}_{M}} |.
\]
Next we want to switch the order of the double sum over $w \in W^{\theta}(\lif{M})_{reg}$ and $\p_{M} \in \Pdt{M^{\theta_{u}}, \p_{G}} $ to a double sum over $\p_{M} \in \Pdt{M, \p_{G}} $ and $w \in W^{\theta}_{\p , reg}$, where $W^{\theta}_{\p , reg} = W^{\theta}(\lif{M})_{reg} \cap W^{\theta}_{\p}$. Since the nonzero contribution of each $\p_{M} \in \Pdt{M, \p_{G}}$ is the same and 
\[
| \Pdt{M, \p_{G}}  | = \frac{| W(M) |}{| W_{\p} |},
\]
then we get a single sum over $w \in W^{\theta}_{\p , reg}$ of 
\[
m_{\p} |W_{\p}|^{-1} |\S{\lp_{M}}|^{-1}  | det (w - 1)_{\mathfrak{a}^{G^{\theta}}_{M}} |^{-1}
\]
multiplied with
\[
\sum_{\x' \in Y / \a(\S{\p_{M}})}  \sum_{u \in \N{\p}^{\theta}(w, \x)} r_{P}(w, \p_{M}) \lf_{\lG^{\theta}}(\lp \otimes \x', u).
\]
Note the double sum over $ w \in W^{\theta}_{\p , reg}$ and $u \in \N{\p}^{\theta}(w, \x)$ can be rearranged as a double sum over $x \in \S{\p}^{\theta}(\x)$ and $u \in \N{\p, reg}^{\theta}(x)$, where 
\[
\S{\p}^{\theta}(\x) = \{x \in \S{\p}^{\theta} : \a(x) = \x\},
\] 
and
\[
\N{\p, reg}^{\theta}(x) = \{ u \in \N{\p}^{\theta} : x_{u} = x, w_{u} \in W^{\theta}_{\p, reg} \}.
\]
So we end up with a sum over $x \in \S{\p}^{\theta}(\x)$ of 
\[
m_{\p} |W_{\p}|^{-1} |\S{\lp_{M}}|^{-1} 
\]
multiplied with
\begin{align}
\sum_{\x' \in Y / \a(\S{\p_{M}})} \sum_{u \in \N{\p, reg}^{\theta}(x)} | det (w_{u} - 1)_{\mathfrak{a}^{G^{\theta}}_{M}} |^{-1} r_{P}(w_{u}, \p_{M}) \lf_{\lG^{\theta}}(\lp \otimes \x', u).   \label{eq: twisted spectral expansion 3}
\end{align}
If we write 
\[
r^{G}_{\p}(w_{u}) = r_{P}(w_{u}, \p_{M}),
\]
and define $s^{0}_{\p}(w_{u})$ to be $(-1)^{n}$, where $n$ is the number of positive roots of $(\cS{\p}^{0}, \bar{T}_{\p})$ mapped to negative roots by $w_{u}$, then by Arthur's sign lemma (\cite{Arthur:2013}, Lemma 4.3.1) we have
\[
r^{G}_{\p}(w_{u}) = s^{0}_{\p}(w_{u}).
\]
Moreover, by our comments after Lemma~\ref{lemma: twisted intertwining relation 1}, $\lf_{\lG^{\theta}}(\lp \otimes \x', u)$ only depends on the image of $u$ in $\S{\p}^{\theta}$, so we can write
\[
\lf_{\lG^{\theta}}(\lp \otimes \x', u) = \lf_{\lG^{\theta}}(\lp \otimes \x', x)
\]
Therefore, the term \eqref{eq: twisted spectral expansion 3} becomes
\[
\sum_{\x' \in Y / \a(\S{\p_{M}})}  ( \sum_{u \in \N{\p, reg}^{\theta}(x)} s^{0}_{\p}(w_{u})  | det (w_{u} - 1)_{\mathfrak{a}^{G^{\theta}}_{M}} |^{-1} ) \lf_{\lG^{\theta}}(\lp \otimes \x', x).
\]
For $\x' \in \a(\S{\p})$, $\cPkt{\lp} \otimes \x' = \cPkt{\lp}$ and 
\[
(\lf|_{\lif{Z}_{F}G(F)})_{\lG^{\theta}}(\lp \otimes \x', x) = f_{G^{\theta}}(\p, x) = (\lf|_{\lif{Z}_{F}G(F)})_{\lG^{\theta}}(\lp, x),
\] 
where $f$ is the restriction of $\lf$ to $G(F)$. 
So we get for $\x' \in \a(\S{\p})$
\[
\lf_{\lG^{\theta}}(\lp \otimes \x', x) =  \lf_{\lG^{\theta}}(\lp, x).
\]
Therefore we only need to take the sum over $\x' \in Y / \a( \S{\p} )$ in \eqref{eq: twisted spectral expansion 3}, and then multiply by $|\a(\S{\p}) / \a(\S{\p_{M}})|$. Since
\[
|\a(\S{\p}) /  \a(\S{\p_{M}}) | =  \frac{|\S{\p}|}{|\S{\lp}|} \cdot \frac{|\S{\lp_{M}}|}{|\S{\p_{M}}|},
\]
the resulting constant multiple is $m_{\p}$ times 
\begin{align*}
|W_{\p}|^{-1} |\S{\lp_{M}}|^{-1} \frac{|\S{\p}|}{|\S{\lp}|} \cdot \frac{|\S{\lp_{M}}|}{|\S{\p_{M}}|} & =  | W_{\p} |^{-1} | \S{\p_{M}} |^{-1} \frac{\S{\p}}{\S{\lp}} \\
& = |\N{\p}|^{-1} \frac{| \S{\p} |}{| \S{\lp} |} \\
& = | W^{0}_{\p} |^{-1} | \S{\p} |^{-1} \frac{| \S{\p} |}{| \S{\lp} |} \\
& = | W^{0}_{\p} |^{-1}| \S{\lp} |^{-1}.
\end{align*}
Let 
\[
C_{\lp} = m_{\p}| \S{\lp} |^{-1},
\]
and we define
\[
i^{\theta}_{\p}(x) = |W^{0}_{\p}|^{-1} \sum_{w \in W^{\theta}_{\p, reg}(x)} s^{0}_{\p}(w) | det (w - 1)_{\mathfrak{a}^{G^{\theta}}_{M}} |^{-1},
\]
where $W^{\theta}_{\p, reg}(x)$ is the image of $\N{\p, reg}^{\theta}(x)$ in $W^{\theta}_{\p, reg}$. Hence we have shown the following lemma.

\begin{lemma}
\label{lemma: twisted spectral expansion}
Suppose $\p \in \cP{G} - \cPdt{G}$, $\theta \in \Sigma_{0}$ and $\x \in Y$, then
\begin{align}
\label{eq: twisted spectral expansion}
\tIdt{\lG^{\theta}}{, \p} (\lf) = C_{\lp} \sum_{\x' \in Y / \a(\S{\p})} \sum_{x \in \S{\p}^{\theta}(\x)} i^{\theta}_{\p}(x) \lf_{\lG^{\theta}}(\lp \otimes \x', x),     \,\,\,\,\, \lf \in \sH(\lG, \lif{\chi}).          
\end{align}
\end{lemma}

\subsubsection{The endoscopic expansion}
\label{subsubsec: endoscopic expansion}

Parallel to this $(\theta, \x)$-twisted spectral expansion \eqref{eq: twisted spectral expansion}, we will proceed to expand the $\p$-component of \eqref{eq: twisted endoscopic side}. Note that if $\theta = id, \x = 1$, we can only expand the right hand side of 
\[
\Idt{\lG}{, \p}(\lf) - \Sdt{\lG}{, \p} (\lf) = \sum_{\lG' \in \End{ell}{\lG} - \{\lG\} } \iota( \lG, \lG' ) \Sdt{\lG'}{, \p}(\lf^{\lG'})
\]
based on our temporary induction assumption. By Corollary~\ref{cor: ker1} we know $\Ker^{1}(F, Z(\D{\lG})) = \Ker^{1}(F, Z(\D{\lG'})) = 1$, so the formula \eqref{formula: endoscopic coefficient} applied to $\iota(\lG, \lG')$ can be simplified as 
\begin{align}
\label{formula: endoscopic coefficient 1}
\iota(\lG, \lG') = | \bar{Z}(\D{\lG'})^{\Gal{}} |^{-1} | \Out_{\lG}(\lG') |^{-1} |\pi_{0}(\kappa_{\lG^{\theta}})|^{-1}
\end{align}
where 
\[
\bar{Z}(\D{\lG'})^{\Gal{}} = Z(\D{\lG'})^{\Gal{}} Z(\D{\lG})^{\Gal{}} / Z(\D{\lG})^{\Gal{}},
\]
and 
\[
\Out_{\lG}(\lG') = \Aut_{\lG}(\lG') / \D{\lG'}Z(\D{\lG})^{\Gal{}}.
\]
Note that $|\pi_{0}(\kappa_{\lG^{\theta}})| = 1$ here, and this formula \eqref{formula: endoscopic coefficient 1} is given in (\cite{Arthur:1990}, Lemma 3.2). By applying Conjecture~\ref{conj: stable multiplicity formula} to $\lG'$, we get
\[
\Sdt{\lG'}{, \p'}(\lf') = \sum_{\x' \in Y' / \a^{G'}(\S{\p'})} | \S{\lp'} |^{-1} \sigma(\com[0]{\cS{\p'}}) \lf' (\lp' \otimes \x'),   \,\,\,\,\, \lf' \in \bar{\mathcal{H}}(\lG', \lif{\chi}')
\]
for $\p' \in \P{G'}$.
By Lemma~\ref{lemma: vanishing}, the $\p$-component of \eqref{eq: twisted endoscopic side} is summed over 
\(
\p_{G} \in \P{G, \p}
\)
and
\[
\{ (\lG', \p') : \lG' \in \tEnd{ell}{\lG^{\theta}} \text{ and } \p' \in \P{G', \p_{G}} \}
\]
of distributions $\Sdt{\lG'}{, \p'}(\lf')$. Again because we are taking $\lf \in \sH(\lG, \lif{\chi})$, the contributions of $\p_{G} \in \P{G, \p}$ to the $\p$-component of \eqref{eq: twisted endoscopic side} are the same. If we fix a parameter $\p_{G} \in \P{G, \p}$, then the first sum collapses to be a constant multiple
\[
| \P{G, \p} | = m_{\p}.
\]
Now let us fix $\p_{G} := \ep$ as a homomorphism from $\mathcal{L}_{\p}$ to $\L{G}$, instead of a $\D{G}$-conjugacy class, and let $\cS{\p}^{\theta} = \Cent ( \Im \p_{G}, \D{G} \rtimes \D{\theta} ) / Z(\D{G})^{\Gal{}}$. We observe $(\lG' ,\p')$ will correspond to $(\p_{G}, s)$ for $s \in \cS{\p, ss}^{\theta}(\x)$, where
\[
\cS{\p, ss}^{\theta}(\x) = \{ s \in \cS{\p, ss}^{\theta}: \a(s) = \x\}
\]
by taking suitable $\D{G}$-conjugation, and $s$ is determined up to $\cS{\p}$-conjugation. Let us use the convention to denote the conjugacy class of $\cS{\p}$ in $\cS{\p, ss}^{\theta}(\x)$ by 
\[
\cS{\p} \backslash \cS{\p, ss}^{\theta}(\x).
\]
Then this correspondence gives us a map 
\[
\{ (\lG', \p') : \lG' \in \tEnd{ell}{\lG^{\theta}} \text{ and } \p' \in \P{G', \p_{G}} \} \longrightarrow \cS{\p} \backslash \cS{\p, ss}^{\theta}(\x)
\]
The point is the contribution of $(\lG', \p')$ only depends on its image under this map, so we want to write the double sum over $(\lG', \p')$ as a single sum over the image. To characterize the image, it is equivalent to find $s \in \cS{\p, ss}^{\theta}(\x)$ with $(\lG'_{s}, \p') \longrightarrow (\p_{G}, s)$ such that $\lG'_{s}$ is elliptic. If we define
\begin{align*}
\cS{\p, ell}^{\theta} & = \{ s \in \cS{\p, ss}^{\theta} : | Z(\com[0]{\cS{\p, s}}) | < \infty \} \\
\cS{\p, ell}^{\theta}(\x) & = \{ s \in \cS{\p, ss}^{\theta}(\x) : | Z(\com[0]{\cS{\p, s}}) | < \infty \},
\end{align*}
then for $s \in \cS{\p, ell}^{\theta}(\x)$, it is easy to see that $\lG'_{s}$ is elliptic. The converse is not true, but the contribution from pairs $(\lG'_{s}, \p')$ with $s \notin \cS{\p, ell}^{\theta}(\x)$ is zero by the stable multiplicity formula~\eqref{formula: stable multiplicity}. In fact
\[
\com[0]{\cS{\p'}} = (\cS{\p, s})^{0} \bar{Z}(\D{\lG'})^{\Gal{}} / \bar{Z}(\D{\lG'})^{\Gal{}}
\]
and $\sigma(\com[0]{\cS{\p'}}) = 0$ unless $| Z(\com[0]{\cS{\p'}}) | < \infty$. If we write
\begin{align*}
\cS{\p, ell}'^{\theta} = \begin{cases}
                             \cS{\p, ell} - \{1\} & \text{ if } \theta = id \\
                             \cS{\p, ell}^{\theta} & \text{ otherwise }
                                        \end{cases} 
\end{align*}
and
\begin{align*}
\cS{\p, ell}'^{\theta}(\x) = \begin{cases}
                             \cS{\lp, ell} - \{1\} & \text{ if } \theta = id, \, \x =1 \\
                             \cS{\p, ell}^{\theta}(\x) & \text{ otherwise }
                                        \end{cases}
\end{align*}
then the effective image of this map should be $\cS{\p} \backslash \cS{\p, ell}'^{\theta}(\x)$. The next problem is to count the fibre of this map. Since $\lG'$ is taken to be the isomorphism class of endoscopic data, the fibre containing $(\lG', \p')$ must have the endoscopic datum isomorphic to $\lG'$, and hence can be obtained by the action of $\Aut_{G}(G')$. Moreover, $\p'$ is taken to be $\D{G}'$-conjugacy classes, so the fibre should be isomorphic to 
\[
\Aut_{G}(G') / S_{\p, s} \Int_{G}(G') \cong \Out_{G}(G') / ( S_{\p, s} \Int_{G}(G') / \Int_{G}(G')),
\]
where $\Int_{G}(G') = \D{G}' Z(\D{G})^{\Gal{}}$ and $S_{\p, s}$ is the preimage of $\cS{\p, s}$ in $S_{\p}$. Moreover let us write 
\[
S_{\p, s} \Int_{G}(G') / \Int_{G}(G') \cong S_{\p, s} / S_{\p, s} \cap \D{G}' Z(\D{G})^{\Gal{}}
\]
So we can turn the $\p$-component of \eqref{eq: twisted endoscopic side} into a sum over 
\(
s \in \cS{\p} \backslash \cS{\p, ell}'^{\theta}(\x),
\) 
but multiplied with the size of each fibre
\[
| \Out_{G}(G') | | S_{\p, s} / S_{\p, s} \cap \D{G}' Z(\D{G})^{\Gal{}} |^{-1}.
\]
In fact, it is more convenient to sum over the conjugacy classes in $\cS{\p, ell}'^{\theta}(\x)$ by the group $\com[0]{\cS{\lp}}$. So let us define
\begin{align*}
\mathcal{E}_{\p, ell}'^{\theta} & = \com[0]{\cS{\lp}} \backslash \cS{\p, ell}'^{\theta}, \\
\mathcal{E}_{\p, ell}'^{\theta}(\x) & = \com[0]{\cS{\lp}} \backslash \cS{\p, ell}'^{\theta}(\x),
\end{align*}
then changing to sum over $\mathcal{E}_{\p, ell}'^{\theta}(\x)$ amounts to multiplying by
\[
|\com[0]{\cS{\lp}} / \cS{\lp, s}^{0}| |\cS{\p} / \cS{\p, s}|^{-1} = | \cS{\p, s} /  \cS{\lp, s}^{0} | |\cS{\p} / \com[0]{\cS{\lp}} |^{-1}.
\]
Finally, we get a sum over $s \in \mathcal{E}_{\p, ell}'^{\theta}(\x)$ of the product of the following three terms
\[
|\Out_{\lG}(\lG')|^{-1} |\Out_{G}(G')|,
\]
\[
| S_{\p, s} / S_{\p, s} \cap \D{G'} Z(\D{G})^{\Gal{}} |^{-1} |\S{\lp'}|^{-1} | \bar{Z}(\D{\lG'})^{\Gal{}} |^{-1}  | \cS{\p, s} /  \cS{\lp, s}^{0} |,
\]
and 
\[
m_{\p} |\cS{\p} / \com[0]{\cS{\lp}} |^{-1} \sum_{\x' \in Y' / \a^{G'}(\S{\p'})} \sigma(\com[0]{\cS{\p'}}) \lf^{\lG'} (\lp' \otimes \x')
\]
where $(G' , \p') \rightarrow (\p_{G}, s)$. Note that 
\[
\xymatrix{1 \ar[r] & \D{D} \ar[r] & \Aut_{\lG}(\lG') \ar[r]  & \Aut_{G}(G') \ar[r] & 1. }
\]
so 
\begin{align*}
|\Out_{\lG}(\lG')|^{-1} |\Out_{G}(G')| & =  |\Int_{G}(G') / (\Int_{\lG}(\lG') / \D{D})| \\
& = | \D{G}' Z(\D{G})^{\Gal{}} /  (\D{\lG'}Z(\D{\lG})^{\Gal{}} / \D{D})|^{-1}  \\
& = |Z(\D{G})^{\Gal{}} / Z(\D{G})^{\Gal{}} \cap \D{G}' (Z(\D{\lG})^{\Gal{}} / \D{D})|^{-1} \\
& = |Z(\D{G})^{\Gal{}} / (\D{G}' \cap Z(\D{G})^{\Gal{}})(Z(\D{\lG})^{\Gal{}} / \D{D}) |^{-1}.
\end{align*}
Moreover, we have $\cS{\lp, s}^{0} = \cS{\p, s}^{0}$ and $\com[0]{\cS{\lp}} = \com[0]{\cS{\p}}$, so we can rewrite the expansion of $\p$-component of \eqref{eq: twisted endoscopic side} as a sum over $s \in \mathcal{E}_{\p, ell}'^{\theta}(\x)$ of the product of the following two terms 
\begin{align}
| S_{\p, s} / S_{\p, s} \cap \D{G}' Z(\D{G})^{\Gal{}} |^{-1} |\S{\p'}|^{-1} | \bar{Z}(\D{G}')^{\Gal{}} |^{-1}  | \cS{\p, s} /  \cS{\p, s}^{0} | \sigma(\com[0]{\cS{\p'}})   \label{eq: twisted endoscopic expansion 1}
\end{align}
and 
\begin{align}
\sum_{\x' \in Y' / \a^{G'}(\S{\p'})} m_{\p} |\S{\p}|^{-1} |\S{\p'}| |\S{\lp'}|^{-1} | \bar{Z}(\D{G}')^{\Gal{}} | | \bar{Z}(\D{\lG'})^{\Gal{}} |^{-1} |Z(\D{G})^{\Gal{}} / (\D{G}' \cap Z(\D{G})^{\Gal{}})(Z(\D{\lG})^{\Gal{}} / \D{D}) |^{-1}
\lf' (\lp' \otimes \x').   \label{eq: twisted endoscopic expansion 2}
\end{align}
As one can see, \eqref{eq: twisted endoscopic expansion 1} is only relevant to $G$, and it can be simplified as in (\cite{Arthur:2013}, Section 4.4). So we will just repeat the simplification there. First we note that
\[
|\S{\p'}| = |\pi_{0}(\cS{\p'})| = |\cS{\p, s} \cap \bar{\D{G}'} / \com[0]{(\cS{\p, s})} \bar{Z}(\D{G}')^{\Gal{}}|,
\]
where $\bar{\D{G}'}$ denotes the quotient 
\(
\D{G}' Z(\D{G})^{\Gal{}} / Z(\D{G})^{\Gal{}}.
\)
Consequently,
\begin{align*}
&| S_{\p, s} / S_{\p, s} \cap \D{G}' Z(\D{G})^{\Gal{}} |^{-1} \, |\S{\p'}|^{-1} \\
 = & | \cS{\p, s} / \cS{\p, s} \cap \bar{\D{G}'} |^{-1} \,  | \cS{\p, s} \cap{\bar{\D{G}'}} / \com[0]{(\cS{\p, s})} \bar{Z}(\D{G}')^{\Gal{}} | \\
 = & |\cS{\p, s} / (\cS{\p, s})^{0} \bar{Z}(\D{G}')^{\Gal{}}|^{-1}.
\end{align*}
The product of the first four factors of \eqref{eq: twisted endoscopic expansion 1} therefore equals
\begin{align*}
& |\cS{\p, s} / \cS{\p, s}^{0} \bar{Z}(\D{G}')^{\Gal{}}|^{-1} \cdot |\cS{\p, s}^{0} \bar{Z}(\D{G}')^{\Gal{}} / (\cS{\p, s})^{0} \ \bar{Z}(\D{G}')^{\Gal{}}|^{-1} \cdot | \bar{Z}(\D{G}')^{\Gal{}} |^{-1} \cdot | \cS{\p, s} /  \cS{\p, s}^{0} | \\ \\
= & |\cS{\p, s}^{0} \bar{Z}(\D{G}')^{\Gal{}} / \cS{\p, s}^{0}| \cdot |\cS{\p, s}^{0} / (\cS{\p, s})^{0}|^{-1} \cdot |\cS{\p, s}^{0} \cap (\cS{\p, s})^{0} \bar{Z}(\D{G}')^{\Gal{}} / (\cS{\p, s})^{0}| \cdot | \bar{Z}(\D{G}')^{\Gal{}} |^{-1} \\ \\
= & |\pi_{0}(\cS{\p, s}^{0} )|^{-1} \cdot |\bar{Z}(\D{G}')^{\Gal{}} / \cS{\p, s}^{0} \cap \bar{Z}(\D{G}')^{\Gal{}}| \cdot |\cS{\p, s}^{0} \cap \bar{Z}(\D{G}')^{\Gal{}} / (\cS{\p, s})^{0} \cap \bar{Z}(\D{G}')^{\Gal{}} | \cdot | \bar{Z}(\D{G}')^{\Gal{}} |^{-1} \\ \\
= & |\pi_{0}(\cS{\p, s}^{0})|^{-1} \cdot |(\cS{\p, s})^{0} \cap \bar{Z}(\D{G}')^{\Gal{}} |^{-1}.
\end{align*}
Furthermore, we can write
\begin{align*}
\sigma(\com[0]{\cS{\p'}}) = & \sigma( (\cS{\p, s})^{0} / (\cS{\p, s})^{0} \cap \bar{Z}(\D{G}')^{\Gal{}}) \\
= & \sigma((\cS{\p, s})^{0}) |(\cS{\p, s})^{0} \cap \bar{Z}(\D{G}')^{\Gal{}}|.
\end{align*}
Hence the first term \eqref{eq: twisted endoscopic expansion 1} is equal to 
\[
|\pi_{0}(\cS{\p, s}^{0})|^{-1} \sigma((\cS{\p, s})^{0}).
\]
For the second term \eqref{eq: twisted endoscopic expansion 2}, let us denote
\[
\lf^{\lG'}(\lp' \otimes \x') = \lf'_{\lG^{\theta}}(\lp \otimes \x', s),   \,\,\,\,  \lf \in \sH(\lG, \lif{\chi}).
\]
Also notice $|\a^{G'}(\S{\p'})| = |\S{\p'} / \S{\lp'}|$ and $m_{\p} |\S{\lp}|^{-1} = C_{\lp}$, so we can write it as
\[
\sum_{\x' \in Y' / \a^{G'}(\S{\p'})} C_{\lp} |\a(\S{\p})|^{-1}|\a^{G'}(\S{\p'})| | \bar{Z}(\D{G}')^{\Gal{}} | | \bar{Z}(\D{\lG'})^{\Gal{}} |^{-1} |Z(\D{G})^{\Gal{}} / (\D{G}' \cap Z(\D{G})^{\Gal{}})(Z(\D{\lG})^{\Gal{}} / \D{Z}) |^{-1} \lf'_{\lG^{\theta}}(\lp' \otimes \x', s).
\]
In view of \eqref{eq: twisted spectral expansion}, we need to turn this into a sum over $Y / \a(\S{\p})$ instead of $Y' / \a^{G'}(\S{\p'})$. To do so we need the following two lemmas.

\begin{lemma}
\label{lemma: twisted endoscopic expansion 1}
Suppose $\p_{v} \in \cuP{G_{v}}$, and $s_{v}$ is a semisimple element of  $\cS{\p_{v}}$ with $(G'_{v}, \p'_{v}) \rightarrow (\p_{v}, s_{v})$. If we assume the main local theorem~\ref{thm: refined L-packet} for the lift $\lp'_{v}$ of $\p'_{v}$, then for any $\x'_{v} \in \a(\S{\p_{v}})$ we have 
\[
\lf'_{\lG^{\theta}_{v}}(\lp_{v} \otimes \x'_{v}, s_{v}) = \lf_{\lG^{\theta}_{v}}'(\lp_{v}, s_{v}), \,\,\,\,\,\, \lf_{v} \in \bar{\mathcal{H}}(\lG_{v}, \lif{\chi}_{v}).
\] 
\end{lemma}

\begin{proof}
Since $\lf'_{\lG^{\theta}_{v}}(\lp_{v}, s_{v})$ only depends on the image of $s_{v}$ in $\S{\p_{v}}^{\theta}$ (see Lemma~\ref{lemma: induced twisted character}), according to the formula \eqref{formula: centralizer} of $S_{\p_{v}}$, we can assume $s_{v}$ commutes with some $t_{v} \in \cS{\p_{v}}$ such that $\a(t_{v}) = \x_{v}'$. Note that $t_{v} \in \Aut_{G_{v}}(G'_{v}) / Z(\D{G}_{v})^{\Gal{v}}$. If $t_{v} \in \Int_{G_{v}}(G'_{v}) / Z(\D{G}_{v})^{\Gal{v}}$, then it is easy to see $\x_{v}' \in \a^{G'_{v}}(\S{\p'_{v}})$, so there is nothing to prove. If $t_{v} \notin \Int_{G_{v}}(G'_{v}) / Z(\D{G}_{v})^{\Gal{v}}$, we denote the inducing automorphism of $G'_{v}$ by $\theta'$, and it can be extended to $\lG'_{v}$. Then it follows from Corollary~\ref{cor: theta twisting character} that $\cPkt{\lp'_{v}}^{\theta'} = \cPkt{\lp'_{v}} \otimes \x_{v}'$. Since $\lf_{v}^{\lG'_{v}}$ is $\Out_{\lG_{v}}(\lG'_{v})$-invariant, we have $\lf_{v}^{\lG'_{v}}(\lp'_{v}) = \lf_{v}^{\lG'_{v}}((\lp'_{v})^{\theta'}) = \lf_{v}^{\lG'_{v}}(\lp'_{v} \otimes \x_{v}')$.
\end{proof}

Since $\lZ(\A_{F}) = A_{\lG}(\A_{F}) \cdot \Z(\A_{F})$, we can identify $Y$ with the quotient
\[
\frac{\Ker \{ H^{1}(W_{F}, \D{D}) \rightarrow H^{1}(W_{F}, \D{A}_{\lG})\}}{\Ker \{ H^{1}(W_{F}, \D{D}) \rightarrow \Hom(\lG(\A_{F})/ \lG(F), \C^{\times})\}},
\]
where the involved homomorphisms are from the following diagram
\begin{align*}
\label{diagram: global Langlands correspondence for characters}
\xymatrix{    H^{1}(W_{F}, \D{D}) \ar[r] \ar[d]^{\simeq} & H^{1}(W_{F}, Z(\D{\lG})) \ar[r] \ar[d] & H^{1}(W_{F}, \D{A}_{\lG}) \ar[d]^{\simeq} \\
\Hom(D(\A_{F})/D(F), \C^{\times}) \ar[r]  & \Hom( \lG(\A_{F}) / \lG(F), \C^{\times}) \ar[r]   & \Hom(A_{\lG}(\A_{F}) / A_{\lG}(F)). }
\end{align*}
In the same way, we can identify $Y'$ with
\[
\frac{\Ker \{ H^{1}(W_{F}, \D{D}) \rightarrow H^{1}(W_{F}, \D{A}_{\lG'})\}}{ \Ker \{ H^{1}(W_{F}, \D{D}) \rightarrow \Hom(\lG'(\A_{F})/ \lG'(F), \C^{\times})\}}.
\]
Note that $A_{\lG} \cong  A_{\lG'}$ under the inclusion of $\lZ = (\lZ)_{\theta}$ to $Z_{\lG'}$, and also we have
\[
\xymatrix{ H^{1}(W_{F}, \D{D}) \ar[r]   & H^{1}(W_{F}, Z(\D{\lG})) \ar[r]  & H^{1}(W_{F}, \D{A}_{\lG} ) \\
H^{1}(W_{F}, \D{D}) \ar[r]  \ar@{=}[u] & H^{1}(W_{F}, Z(\D{\lG'})) \ar[r]  & H^{1}(W_{F}, \D{A}_{\lG'}) \ar[u]^{\simeq}. }
\]
So $\Ker \{ H^{1}(W_{F}, \D{D}) \rightarrow H^{1}(W_{F}, \D{A}_{\lG})\} = \Ker \{ H^{1}(W_{F}, \D{D}) \rightarrow H^{1}(W_{F}, \D{A}_{\lG'})\}$, and we can first sum over this group. Then the rest is to determine the quotient 
\[
\frac{|\Ker \{ H^{1}(W_{F}, \D{D}) \rightarrow \Hom(\lG'(\A_{F})/ \lG'(F), \C^{\times})\}|}{|\Ker \{ H^{1}(W_{F}, \D{D}) \rightarrow \Hom(\lG(\A_{F})/ \lG(F), \C^{\times})\}|}.
\]

\begin{lemma}
\label{lemma: twisted endoscopic expansion 2}
\begin{align*}
& | \bar{Z}(\D{G}')^{\Gal{}} | | \bar{Z}(\D{\lG'})^{\Gal{}} |^{-1} |Z(\D{G})^{\Gal{}} / (\D{G}' \cap Z(\D{G})^{\Gal{}})(Z(\D{\lG})^{\Gal{}} / \D{D}) |^{-1}  \\ = \quad &  \frac{|\Ker \{ H^{1}(W_{F}, \D{D}) \rightarrow \Hom(\lG'(\A_{F})/ \lG'(F), \C^{\times})\}|}{|\Ker \{ H^{1}(W_{F}, \D{D}) \rightarrow \Hom(\lG(\A_{F})/ \lG(F), \C^{\times})\}|}.
\end{align*}
\end{lemma}

\begin{proof}
From the proof of Lemma~\ref{lemma: identification}, we see 
\[
\Ker \{ H^{1}(W_{F}, \D{D}) \rightarrow \Hom(\lG'(\A_{F})/ \lG'(F), \C^{\times})\} \cong Z(\D{G}')^{\Gal{}} / (Z(\D{\lG'})^{\Gal{}} / \D{D}),
\]
and 
\[
\Ker \{ H^{1}(W_{F}, \D{D}) \rightarrow \Hom(\lG(\A_{F})/ \lG(F), \C^{\times})\} \cong Z(\D{G})^{\Gal{}} / (Z(\D{\lG})^{\Gal{}} / \D{D}).
\]
Therefore it is enough to show 
\begin{align}
\label{eq: twisted endoscopic expansion 3}
| \bar{Z}(\D{G}')^{\Gal{}} | | \bar{Z}(\D{\lG'})^{\Gal{}} |^{-1} |Z(\D{G})^{\Gal{}} / (\D{G}' \cap Z(\D{G})^{\Gal{}})(Z(\D{\lG})^{\Gal{}} / \D{D}) |^{-1} = \frac{|Z(\D{G}')^{\Gal{}} / (Z(\D{\lG'})^{\Gal{}} / \D{D})|}{|Z(\D{G})^{\Gal{}} / (Z(\D{\lG})^{\Gal{}} / \D{D})|}.
\end{align}
We start by considering the following exact sequence
\begin{align*}
1 \longrightarrow (\D{G}' \cap Z(\D{G})^{\Gal{}}) / (\D{G}' \cap (Z(\D{\lG})^{\Gal{}} / \D{D})) \longrightarrow Z(\D{G})^{\Gal{}} / (Z(\D{\lG})^{\Gal{}} / \D{D}) \\ \longrightarrow  Z(\D{G})^{\Gal{}} / (\D{G}' \cap Z(\D{G})^{\Gal{}})(Z(\D{\lG})^{\Gal{}} / \D{D}) \longrightarrow 1.
\end{align*}
It follows 
\[
|Z(\D{G})^{\Gal{}} / (\D{G}' \cap Z(\D{G})^{\Gal{}})(Z(\D{\lG})^{\Gal{}} / \D{D})|^{-1} = \frac{|(\D{G}' \cap Z(\D{G})^{\Gal{}}) / (\D{G}' \cap (Z(\D{\lG})^{\Gal{}} / \D{D}))|}{|Z(\D{G})^{\Gal{}} / (Z(\D{\lG})^{\Gal{}} / \D{D})|}.
\]
If we write 
\[
(\D{G}' \cap Z(\D{G})^{\Gal{}}) / (\D{G}' \cap (Z(\D{\lG})^{\Gal{}} / \D{D})) = (Z(\D{G}')^{\Gal{}} \cap Z(\D{G})^{\Gal{}}) / ((Z(\D{\lG'}) \cap Z(\D{\lG})^{\Gal{}})/\D{D}),
\]
then 
\begin{align*}
&|\D{G}' \cap Z(\D{G})^{\Gal{}} / \D{G}' \cap (Z(\D{\lG})^{\Gal{}} / \D{D})| |\bar{Z}(\D{G}')^{\Gal{}}| |\bar{Z}(\D{\lG'})^{\Gal{}}|^{-1} \\ 
= \quad & |Z(\D{G}')^{\Gal{}} / ((Z(\D{\lG'}) \cap Z(\D{\lG})^{\Gal{}})/\D{D})|  |\bar{Z}(\D{\lG'})^{\Gal{}}|^{-1} \\
= \quad & |Z(\D{G}') / (Z(\D{\lG'})^{\Gal{}} / \D{D})|.
\end{align*}
Hence \eqref{eq: twisted endoscopic expansion 3} holds.
\end{proof}

As a consequence of Lemma~\ref{lemma: twisted endoscopic expansion 1} and Lemma~\ref{lemma: twisted endoscopic expansion 2}, we can sum over $Y / \a(\S{\p})$ for \eqref{eq: twisted endoscopic expansion 2} and get
\[
\sum_{\x' \in Y / \a(\S{\p})} C_{\lp} \lf'_{\lG^{\theta}} (\lp \otimes \x', s).
\]

To sum up, we have shown the $\p$-component of \eqref{eq: twisted endoscopic side} has an expansion
\begin{align*}
\sum_{s \in \mathcal{E}'^{\theta}_{\p, ell}(\x)} |\pi_{0}(\cS{\p, s}^{0})|^{-1} \sigma((\cS{\p, s})^{0}) \sum_{\x' \in Y / \a(\S{\p})} C_{\lp} \lf'_{\lG^{\theta}} (\lp \otimes \x', s) \\ 
= \sum_{\x' \in Y / \a(\S{\p})} C_{\lp} \sum_{s \in \mathcal{E}'^{\theta}_{\p, ell}(\x)} |\pi_{0}(\cS{\p, s}^{0})|^{-1} \sigma((\cS{\p, s})^{0}) \lf'_{\lG^{\theta}} (\lp \otimes \x', s) .
\end{align*}
Finally by the same argument as in the proof of Lemma~\ref{lemma: induced twisted character}, there exists a family of global lifts $\cPkt{\lp'}$ for all $s \in \cS{\p,ss}^{\theta}$ with image $x$ in $\S{\p}^{\theta}$, such that the distribution $\lf'_{\lG^{\theta}}(\lp \otimes \x', s)$ are the same. So we can write
\[
\lf'_{\lG^{\theta}}(\lp \otimes \x', s) = \lf'_{\lG^{\theta}}(\lp \otimes \x', x).
\]
Moreover, we can split the sum over $s \in \mathcal{E}'^{\theta}_{\p, ell}(\x)$ into a double sum over $x \in \S{\p}^{\theta}(\x)$ and $s \in \mathcal{E}'^{\theta}_{\p, ell}(x)$, where $ \mathcal{E}'^{\theta}_{\p, ell}(x)$ is the subset of $ \mathcal{E}'^{\theta}_{\p, ell}$ that mapped to $x$. If we define
\[
e'^{\theta}_{\p}(x) = \sum_{s \in \mathcal{E}'^{\theta}_{\p, ell}(x)} |\pi_{0}(\cS{\p, s}^{0})|^{-1} \sigma((\cS{\p, s})^{0}),
\]
then we get the following lemma.

\begin{lemma}
\label{lemma: twisted endoscopic expansion}
Suppose $\p \in \cP{G}$, $\theta \in \Sigma_{0}$ and $\x \in Y$. If $\theta =id, \x = 1$ then
\begin{align}
\label{eq: endoscopic expansion}
\Idt{\lG}{, \p}(\lf) - \Sdt{\lG}{, \p} (\lf) = C_{\lp} \sum_{\x' \in Y / \a(\S{\p})} \sum_{x \in \S{\lp}} e'_{\p}(x) \lf_{\lG}' (\lp \otimes \x', x).    
\end{align}
Otherwise,
\begin{align}
\label{eq: twisted endoscopic expansion}
\tIdt{\lG^{\theta}}{, \p} (\lf) = C_{\lp} \sum_{\x' \in Y / \a(\S{\p})} \sum_{x \in \S{\p}^{\theta}(\x)} e'^{\theta}_{\p}(x) \lf_{\lG^{\theta}}' (\lp \otimes \x', x).    
\end{align}
\end{lemma}

\begin{corollary}
\label{cor: stability}
Suppose $\p \in \cP{G}$ and $\S{\lp} = 1$, then the distribution $\Idt{\lG}{, \p}(\lf)$ is stable . 
\end{corollary}

\begin{proof}
Since $\S{\lp} = 1$, it follows from \eqref{eq: endoscopic expansion} that
\[
\Idt{\lG}{, \p}(\lf) = \Sdt{\lG}{, \p} (\lf) + C_{\lp} \sum_{\x' \in Y / \a(\S{\p})} e'_{\p}(1) \lf_{\lG}' (\lp \otimes \x', 1).    
\]
Note that $\lf_{\lG}' (\lp \otimes \x', 1)$ is defined by inducing global L-packets of Levi subgroups of $\lG$, and hence is stable. Therefore $\Idt{\lG}{, \p}(\lf)$ is stable.

\end{proof}

Later on, we will compare the formulas in Lemma~\ref{lemma: twisted spectral expansion} and Lemma~\ref{lemma: twisted endoscopic expansion}. Note it follows from Proposition~\ref{prop: endoscopy of complex group} that for $x \in \S{\p}^{\theta}$,
\[
i^{\theta}_{\p}(x) - e'^{\theta}_{\p}(x) = \begin{cases}
                                      0 & \text{ if } x \neq 1, \\
                                      \sigma(\com[0]{\cS{\p}}) & \text{ if } x = 1.
                                      \end{cases}
\]



\section{Refined L-packet}
\label{sec: refined L-packet}

\subsection{Beginning of proofs}
\label{subsec: beginning of proofs}

In the following sections, we are going to prove the main local theorem (Theorem~\ref{thm: refined L-packet}) along with the global theorem (Theorem~\ref{thm: main global}). First we need to impose our induction assumptions. It consists of a local part and a global part. Let $F$ be either local or global. We denote
\begin{align*}
G(n) & := Sp(2n), SO(2n+2, \eta), \\
\lG(n) & := GSp(2n), GSO(2n+2, \eta).
\end{align*}
Let 
\[
G = G(n_{1}) \times G(n_{2}) \times \cdots \times G(n_{q})
\]
and $\lG$ be the corresponding similitude group (see \eqref{eq: similitude}), then our induction assumptions can be stated as follows.

{\bf Local Induction Assumption:} The main local theorem (Theorem~\ref{thm: refined L-packet}) holds for $\lG$, when $n_{i} < N$ for all $1 \leqslant i \leqslant q$. 

{\bf Global Induction Assumption:} The global theorem (Theorem~\ref{thm: main global}) hold for $\lG$, when $\sum_{i = 1}^{q} n_{i} < N$. 

\begin{remark}
\label{rk: induction assumption}
When $\lG = GSp(2N)$, these assumptions imply all the local and global theorems hold for the Levi subgroups and twisted endoscopic groups of $\lG$. But this is not true when $\lG = GSO(2N + 2, \eta)$ for it can have twisted endoscopic group of the form $G(Sp(2N_{1}) \times Sp(2N_{2}))$ with $N = N_{1} + N_{2}$ and $N_{1}, N_{2} \geqslant 0$. To fix this, we will first prove the local and global theorems for $\lG$ based on our induction assumption, when $G$ does not contain any factor of $SO(2N+2, \eta)$. Then we can add those results to our induction assumptions and repeat the same arguments to prove the rest of the cases. 
\end{remark}

We will first establish the main local theorem (Theorem~\ref{thm: refined L-packet}) for $\lG = \lG(N)$, which is the most important case. In view of Remark~\ref{rk: refined L-packet}, we can further assume $F$ is nonarchimedean. Under our local induction assumption, we can prove a lot of cases of the main local theorem. The precise statement is formulated in the following lemma.

\begin{lemma}
\label{lemma: refined L-packet for non-discrete parameter}
Suppose $\p \in \cPbd{G} - \cPdt{G}$, then one can assign an L-packet $\cPkt{\lp}$ for any lift $\lp$ such that it satisfies (1) and (2) of the main local theorem (Theorem~\ref{thm: refined L-packet}). Furthermore, the $(\theta, \x)$-twisted character relation \eqref{eq: theta twisted character relation} holds for $\theta \in \Sigma_{0}$ and semisimple $s \in \cS{\p}^{\theta}$ such that $|\cS{\p, s}^{0}| = \infty$.
\end{lemma}

\begin{proof}
Suppose $\p \in \cPbd{G} - \cPdt{G}$, then $\p$ factors through $\p_{M} \in \cPdt{M}$ for some proper Levi subgroup $M$ of $G$.  Since 
\[
M \cong G(m) \times \prod_{i} GL(n_{i})
\] 
with $m < N$, by our local induction assumption we can define a refined L-packet $\cPkt{\lp_{M}}$ associated to $\lp_{M}$. Then we can take local packet $\cPkt{\lp}$ for $\lp$ to be the irreducible constituents of those induced from $\cPkt{\lp_{M}}$. Because $\cPkt{\p}$ is also obtained by induction from $\cPkt{\p_{M}}$, we can easily see that $\cPkt{\lp}$ will satisfy (1) and (2) of the main local theorem. For the $(\theta, \x)$-twisted character relation \eqref{eq: theta twisted character relation}, it will follow from the usual descent argument. For $(G', \p') \rightarrow (\p, s)$, let $T_{\p, s}$ be a maximal torus of $(S_{\p, s})^{0}$, which is nontrivial by our assumption that $|\cS{\p, s}^{0}| = \infty$. Then $\D{M}' = \Cent(T_{\p, s}, \D{G}')$ defines a proper Levi subgroup of $\D{G}'$ such that $\p'$ factors through $\p'_{M} \in \cPdt{M'}$. Moreover, $M' \in \End{ell}{M^{\theta}}$ for a proper $\theta$-stable Levi subgroup $M$ of $G$, which is determined by $\D{M} = \Cent(T_{\p, s}, \D{G})$. So
\[
\lf'(\lp') = \lf^{\lM'}(\lp'_{M}) = \sum_{[\lr_{M}] \in \cPkt{\lp_{M}}} \lf_{\lM^{\theta}}(\lr_{M}, \x) = \sum_{[\lr] \in \cPkt{\lp}} \lf_{\lG^{\theta}}(\lr, \x),
\]
where $\x = \a(s)$.
\end{proof}

\begin{remark}
We have not shown the uniqueness of $\cPkt{\lp}$ here. In fact that will follow from the character relation (see Theorem~\ref{thm: twisted character relation for elliptic parameter}). 
\end{remark}

The key issue in proving the main local theorem is to find a candidate for the stable distribution associated with any lift $\lp$ of $\p \in \cPbd{G}$. As one can see from Lemma~\ref{lemma: refined L-packet for non-discrete parameter}, the critical case is when $\p \in \cPdt{G}$. The way to find such a distribution is to lift $\p$ to a global parameter $\dot{\p}$ and use the global stable distribution $\Sdt{\lif{\dot{G}}}{, \dot{\p}}$ in the stabilized trace formula. Under some assumptions on this lifted global parameter $\dot{\p}$, we can obtain the local stable distribution associated with $\lp$ using an argument based on stability (cf. Corollary~\ref{cor: refined L-packet}). Let us write $S_{\infty}$ for the set of archimedean places of a global field $\dot{F}$, and $S_{\infty}(u) = S_{\infty} \cup \{u\}$ for any nonarchimedean place $u$. Suppose $F = \dot{F}_{u}$ and $\dot{G}_{u} = G$. Let $\dot{X} = \Hom(\lif{\dot{G}}(\A_{\dot{F}})/Z_{\lif{\dot{G}}}(\A_{\dot{F}})\dot{G}(\A_{\dot{F}}), \C^{\times})$.

\begin{theorem}
\label{thm: standard argument on stability}
For $\p \in \cPdt{G}$, suppose $\dot{\p} \in \cPdt{\dot{G}}$ is a global lift of $\p$ with $\dot{\p}_{u} = \p$ and it also satisfies the following additional conditions:
\begin{enumerate}
\item $\dot{\p}_{v} \in \cuP{\dot{G}_{v}} - \cPel{\dot{G}_{v}} \text{ for all } v \notin S_{\infty}(u)$; 
\item $\S{\lif{\dot{\p}}} = 1$;
\item $\Sigma_{0}$-strong multiplicity one holds for $\lif{\dot{\p}}$.
\end{enumerate}
Then one can assign an L-packet $\cPkt{\lp}$ to any lift $\lp$ of $\p$ satisfying (1) and (2) of the main local theorem (Theorem~\ref{thm: refined L-packet}).
\end{theorem}

\begin{proof}
In view of Lemma~\ref{lemma: refined L-packet for non-discrete parameter}, the first condition of our global lift $\dot{\p}$ just means that the main local theorem (except for the $(\theta_{0}, \x)$-twisted character relation in the even orthogonal case) holds for all $\lif{\dot{\p}}_{v}$ $(v \neq u)$. 
The second condition means that 
\[
tr R^{\lif{\dot{G}}}_{disc, \dot{\p}}(\lif{\dot{f}}) = I^{\lif{\dot{G}}}_{disc, \dot{\p}}(\lif{\dot{f}}) = S^{\lif{\dot{G}}}_{disc, \dot{\p}}(\lif{\dot{f}}) \neq 0
\]
for $\lif{\dot{f}} \in \sH(\lif{\dot{G}}, \lif{\dot{\chi}})$, which follows from Corollary~\ref{cor: stability} and the fact that $\dot{\p} \in \cPdt{\dot{G}}$ (cf. \eqref{eq: discrete part vs discrete spectrum}). 
It follows from Proposition~\ref{prop: discrete spectrum} that
\begin{align}
I^{\lif{\dot{G}}}_{disc, \dot{\p}}(\lif{\dot{f}}) = m_{\dot{\p}} \sum_{\dot{\x} \in \dot{Y} / \a(\S{\dot{\p}})} \sum_{\lif{\dot{\r}}} \lif{\dot{f}}_{\lif{\dot{G}}}(\lif{\dot{\r}} \otimes \dot{\x}) \label{eq: standard argument on stability 1}
\end{align}
for $\lif{\dot{f}} \in \sH(\lif{\dot{G}}, \lif{\dot{\chi}})$, where the sum of $\lif{\dot{\r}}$ is taken over representatives of $\lif{\bar{\Pi}}_{\dot{\p}, \lif{\dot{\zeta}}} / \dot{X}$ inside $\mathcal{A}_{2}(\lif{\dot{G}})$. Here we will always view representations of $\lif{\dot{G}}(\A_{F})$ as $\sH(\lif{\dot{G}}, \lif{\dot{\chi}})$-modules. Since $I^{\lif{\dot{G}}}_{disc, \dot{\p}}(\lif{\dot{f}})$ is stable, it is stable at every place. If we take $\lif{\dot{f}} = \bigotimes_{w}\lif{\dot{f}}_{w}$ and fix $\bigotimes_{w \neq v}\lif{\dot{f}}_{w}$ for $v \neq u$, then by Corollary~\ref{cor: refined L-packet} the coefficients of $\lif{\dot{f}}_{v}(\lif{\dot{\r}}_{v})$ in $I^{\lif{\dot{G}}}_{disc, \dot{\p}}(\lif{\dot{f}})$ must be the same for all $[\lif{\dot{\r}}_{v}] \in \cPkt{\lif{\dot{\p}}_{v}}$. Moreover, if we fix a representation $\lif{\dot{\r}} \in \mathcal{A}_{2}(\lif{\dot{G}})$, by varying $\bigotimes_{w \neq v}\lif{\dot{f}}_{w}$ and the linear independence of characters of $\bigotimes_{w \neq v}\bar{\mathcal{H}}(\lif{\dot{G}}_{w}, \lif{\dot{\chi}}_{w})$-modules, we will observe that for $v \neq u$
\[
[\lif{\dot{\r}}_{v}] \bigotimes (\bigotimes_{w \neq v} [\lif{\dot{\r}}_{w}] ) \in \lif{\bar{\Pi}}_{\dot{\p}, \lif{\dot{\zeta}}} 
\]
contributes to \eqref{eq: standard argument on stability 1} if and only if
\[
[\lif{\dot{\r}}'_{v}] \bigotimes (\bigotimes_{w \neq v} [\lif{\dot{\r}}_{w}] ) \in \lif{\bar{\Pi}}_{\dot{\p}, \lif{\dot{\zeta}}}
\]
also contributes to \eqref{eq: standard argument on stability 1} for all $[\lif{\dot{\r}}'_{v}] \in \cPkt{\lif{\dot{\p}}_{v}}$, where $\cPkt{\lif{\dot{\p}}_{v}}$ contains $[\lif{\dot{\r}}_{v}]$. We still fix $\lif{\dot{\r}}$ and hence $\cPkt{\lif{\dot{\p}}_{v}}$ for all $v \neq u$. Then \eqref{eq: standard argument on stability 1} will contain $\sH(\lif{\dot{G}}, \lif{\dot{\chi}})$-modules of the form 
\[
[\lif{\dot{\r}}_{u}] \bigotimes (\bigotimes_{v \neq u} [\lif{\dot{\r}}_{v}])
\]
where $[\lif{\dot{\r}}_{v}]$ ranges over $\cPkt{\lif{\dot{\p}}_{v}}$ for all $v \neq u$. Suppose there is a distinct $\sH(\lif{\dot{G}}, \lif{\dot{\chi}})$-module 
\[
[\lif{\dot{\r}}'_{u}] \bigotimes (\bigotimes_{v \neq u} [\lif{\dot{\r}}_{v}])
\]
in \eqref{eq: standard argument on stability 1} such that $[\lif{\dot{\r}}_{v}] \in \cPkt{\lif{\dot{\p}}_{v}}$ for all $v \neq u$, then $[\lif{\dot{\r}}'_{u}] \neq [\lif{\dot{\r}}_{u}] \otimes \x$ for any character $\x \in X$. Otherwise, there will exist $\dot{\x} \in \dot{Y}$ such that $[\lif{\dot{\r}}_{u}] \otimes \dot{\x}_{u} = [\lif{\dot{\r}}'_{u}] \neq [\lif{\dot{\r}}_{u}]$ and $\cPkt{\lif{\dot{\p}}_{v}} = \cPkt{\lif{\dot{\p}}_{v}} \otimes \dot{\x}_{v}$ for all $v \neq u$. This is impossible because of the third condition, i.e. $\Sigma_{0}$-strong multiplicity one holds for $\lif{\dot{\p}}$. Therefore if we consider all $[\lr] \in \lif{\bar{\Pi}}_{\p, \lif{\zeta}}$ such that 
\[
[\lr] \bigotimes (\bigotimes_{v \neq u} \cPkt{\lif{\dot{\p}}_{v}})
\] 
is contained in \eqref{eq: standard argument on stability 1}, this gives a non-empty set $\cPkt{\lp}$ of representatives of 
\(
\lif{\bar{\Pi}}_{\p, \lif{\zeta}} / X
\)
in $\lif{\bar{\Pi}}_{\p, \lif{\zeta}}$. To see why this gives all the representatives, one just needs to take the test function $\lif{\dot{f}} = \otimes_{v}\lif{\dot{f}}_{v}$ such that $\lif{\dot{f}}_{u}$ is supported on $\lif{\dot{Z}}_{\dot{F}_{u}}\dot{G}(\dot{F}_{u})$, then it is the same to consider representations of 
\[
\lif{\dot{Z}}_{\dot{F}_{u}} \dot{G}(\dot{F}_{u}) \times \prod_{v \neq u} \lif{\dot{G}}(\dot{F}_{v}).
\]
By the same reasoning using stability, one can conclude that 
\[
\cPkt{\p} \bigotimes (\bigotimes_{v \neq u} \cPkt{\lif{\dot{\p}}_{v}})
\]
is contained in \eqref{eq: standard argument on stability 1}, therefore $\cPkt{\lp}$ must contains all representatives of $\lif{\bar{\Pi}}_{\p, \lif{\zeta}} / X$ in $\lif{\bar{\Pi}}_{\p, \lif{\zeta}}$. Moreover, it follows again from $\Sigma_{0}$-strong multiplicity one and stability of \eqref{eq: standard argument on stability 1} that 
\[
\lf(\lp) := \sum_{[\lr] \in \cPkt{\lp}} \lf_{\lG}(\lr)
\]
is stable. This shows the packet $\cPkt{\lp}$ satisfies the property (1) and (2) of Theorem~\ref{thm: refined L-packet}.

\end{proof}

\begin{remark}
\label{rk: standard argument on stability}

\begin{enumerate}

\item
Following the proof, we can rewrite \eqref{eq: standard argument on stability 1} as 
\[
I^{\lif{\dot{G}}}_{disc, \dot{\p}}(\lif{\dot{f}}) = m_{\dot{\p}} \sum_{\dot{\x} \in \dot{Y} / \a(\S{\dot{\p}})} \sum_{[\lif{\dot{\r}}] \in \cPkt{\lif{\dot{\p}}} \otimes \dot{\x}} \lif{\dot{f}}_{\lif{\dot{G}}}(\lif{\dot{\r}} ) 
\]
where 
\[
\cPkt{\lif{\dot{\p}}} =  \cPkt{\lp} \bigotimes (\bigotimes_{v \neq u} \cPkt{\lif{\dot{\p}}_{v}} ) .
\]
If we define 
\[
\lif{\dot{f}}(\lif{\dot{\p}}) := \prod_{v} \lif{\dot{f}}_{v}(\lif{\dot{\p}}_{v}) , 
\]
then we get the stable multiplicity formula for our lift $\lif{\dot{\p}}$
\[
S^{\lif{\dot{G}}}_{disc, \dot{\p}}(\lif{\dot{f}}) = I^{\lif{\dot{G}}}_{disc, \dot{\p}}(\lif{\dot{f}}) = m_{\dot{\p}} \sum_{\dot{\x} \in \dot{Y} / \a(\S{\dot{\p}})} \lif{\dot{f}}(\lif{\dot{\p}} \otimes \dot{\x}).
\]
This identity will be used in the proof of Theorem~\ref{thm: refined L-packet for discrete parameter}.

\item 
The statement of this theorem indicates that we need a lifting result for the existence of such $\dot{\p}$. In fact, there is a standard argument using the simple invariant trace formula which provides a global lift so that one is allowed to impose some local conditions. That argument is carried out in quite detail in (\cite{Arthur:2013}, Section 6.2 and 6.3), and the local conditions that Arthur imposes already take care of our first additional condition in most cases. Even though the global lift which Arthur uses does not necessarily satisfy the other two conditions, his argument is still flexible enough to leave us a lot of room to manipulate. In fact, it is not hard to impose the second condition after we give a combinatorial description of the exact sequence 
\begin{align*}
\xymatrix{1 \ar[r] &  \S{\lif{\dot{\p}}} \ar[r]^{\iota} & \S{\dot{\p}} \ar[r]^{\a \quad \quad \quad \quad \quad \quad }  & \Hom(\lif{\dot{G}}(\A_{\dot{F}})/\lif{\dot{G}}(\dot{F})\dot{G}(\A_{\dot{F}}), \C^{\times}).}               
\end{align*}
However, for technical reasons the third condition is not so easy to satisfy, and it seems that we have asked something too strong. By tracking the argument in our proof carefully, one will observe that it is enough to have 
\[
\prod^{aut}_{v} \a(\S{\dot{\p}_{v}}^{\Sigma_{0}}) = \prod^{aut}_{v \neq u} \a(\S{\dot{\p}_{v}}^{\Sigma_{0}}).
\] 
This condition can also be interpreted in terms of strong multiplicity one, which means that for any $\dot{\lr} \in \mathcal{A}(\dot{\lG})$ such that $[\dot{\lr}_{v}] \in \cPkt{\lif{\dot{\p}}_{v}}$ for all $v \neq u$, $[\dot{\lr}]$ must be also in $\cPkt{\lif{\dot{\p}}}$. If this condition is satisfied, we say $\Sigma_{0}$-strong multiplicity one holds for $\lif{\dot{\p}}$ {\bf at the place} $u$. And it is the most technical part of this paper to establish this property.

\end{enumerate}

\end{remark}

\subsection{A combinatorial description of $\S{\p}$} 
\label{sec: combinatorial description}

Now we will give a combinatorial description of the exact sequences \eqref{eq: local twisted endoscopic sequence} and \eqref{eq: global twisted endoscopic sequence}. We assume $F$ is either local or global. Suppose $G = G(n)$, $\p \in \cP{G}$ if $F$ is global or $\cPbd{G}$ if $F$ is local, and 
\[
\p = l_{1} \p_{1} \# \cdots \# l_{r} \p_{r},
\]
where $\p_{i} \in \Psm{N_{i}}$ for $1 \leqslant i \leqslant r$. From the discussion of Section~\ref{sec: Arthur's theory}, the set of indices can be written as a disjoint union of 
\[
I_{\p, O} \sqcup I_{\p, S} \sqcup J_{\p} \sqcup J_{\p}^{\vee}
\]
where $I_{\p, O}$ ($I_{\p, S}$) is the set of indices that index self-dual parameters of orthogonal (symplectic) type. In particular, since we are considering $G$ to be either a special even orthogonal group or a symplectic group, $\D{G}$ will always be orthogonal and hence the multiplicities $l_{i}$ must be even for $i \in I_{\p, S}$. On the other hand, let us denote 
\begin{align*}
I^{odd}_{\p, O} & = \{ i \in I_{\p, O} : l_{i} \text{ is odd }\}, \\
I^{even}_{\p, O} & = \{ i \in I_{\p, O} : l_{i} \text{ is even }\}.
\end{align*}  
Moreover, let $S$ and $T$ be subsets of $I^{odd}_{\p, O}$ and $I^{even}_{\p, O}$ respectively, with the condition that 
\(
\sum_{i \in S \cup T} N_{i}
\) 
is even if $G$ is special even orthogonal. And we allow $S$ and $T$ to be empty sets. Then the pair of such sets modulo the following equivalence relation gives us the combinatorial object that we need to substitute for $\S{\p}$, i.e.
\[
\mathcal{P}_{\p} = \{ (S, T) \} / (S, T) \sim (S^{c}, T)
\]
where $S^{c}$ is the complement of $S$ in $I^{odd}_{\p, O}$. There is a natural map from $\mathcal{P}_{\p}$ to $\Hom(\lG(F)/G(F), \C^{\times})$ if $F$ is local (resp. $\Hom(\lG(\A_{F})/ \lG(F)G(\A_{F}), \C^{\times})$ if $F$ is global), which sends
\[
(S, T) \longmapsto ( \prod_{i \in S \cup T} \eta_{\p_{i}} ) \circ \c
\]
where $\eta_{\p_{i}}$ is the central character of $\r_{\p_{i}}$. Let us denote this map by $\a_{\mathcal{P}}$ and the kernel of this map by $\mathcal{P}_{\lp}$, then we get a sequence 
\begin{align}
\label{eq: local combinatorial sequence}
\xymatrix{ 1  \ar[r]   &    \mathcal{P}_{\lp} \ar[r]      &  \mathcal{P}_{\p}   \ar[r]^{\a_{\mathcal{P}} \quad \quad \quad \quad }  &  \Hom(\lG(F)/G(F), \C^{\times}) }  
\end{align}
if $F$ is local, and
\begin{align}
\label{eq: global combinatorial sequence}
\xymatrix{ 1  \ar[r]   &    \mathcal{P}_{\lp} \ar[r]      &  \mathcal{P}_{\p}   \ar[r]^{\a_{\mathcal{P}} \quad \quad \quad \quad \quad \quad}  &  \Hom(\lG(\A_{F})/ \lG(F)G(\A_{F}), \C^{\times}) }  
\end{align}
if $F$ is global.

To compare these sequences with \eqref{eq: local twisted endoscopic sequence} and \eqref{eq: global twisted endoscopic sequence}, we need a map connecting $\S{\p}$ and $\mathcal{P}_{\p}$. To define such a map, we consider semisimple $s \in \cS{\p}$, and $(G_{s}', \p') \rightarrow (\p, s)$. In general, $G'_{s}$ may not be elliptic, but it will lie in $\End{ell}{M}$ for some Levi subgroup $M$ of $G$. Since $M$ is a product of general linear groups  with a group $G_{-}$ of the same type as $G$ with smaller rank, then $G'_{s}$ contains a factor $G'_{I} \times G'_{II} \in \End{ell}{G_{-}}$ and $\p'$ will decompose accordingly. Suppose $\p_{-}$ is the component of $\p$ contributing to $G_{-}$, and
\[
\p_{-}' = \p'_{I} \times \p'_{II}
\]
if $G$ is special even orthogonal, or
\[
\p_{-}' = (\p'_{I} \otimes \eta_{\p'_{I}}) \times \p'_{II} 
\]
if $G$ and $G'_{I}$ are symplectic. In either case, $\p'_{I}, \p'_{II}$ give a partition of simple parameters in $\p_{-}$. Let $S$ ($T$) be the subset of $I^{odd}_{\p, O}$ ($I^{even}_{\p, O}$) parametrizing simple parameters in $\p'_{I}$ with odd multiplicites. It is easy to see that $(S, T) \in \mathcal{P}_{\p}$, so we get a map $\bold{c}^{G} = \bold{c}: \cS{\p} \longrightarrow \mathcal{P}_{\p}$ in this way. Moreover we have the following lemma.

\begin{lemma}
\label{lemma: combinatorial description}
\begin{enumerate}
\item The map $\bold{c}$ defined above will factor through $\S{\p}$, and it gives a bijection between $\S{\p}$ and $\mathcal{P}_{\p}$.
\item If we denote the bijection in (1) still by $\bold{c}$, then we have a commutative diagram.
\[
\xymatrix{   1 \ar[r] \ar@{=}[d]    &  \S{\lp} \ar[d]  \ar[r]    &   \S{\p} \ar[d]^{\bold{c}}    \ar[r]^{\a \quad \quad \quad \quad}   &  \Hom(\lG(F)/G(F), \C^{\times})  \ar@{=}[d]   \\
1 \ar[r]     &    \mathcal{P}_{\lp}  \ar[r]    &  \mathcal{P}_{\p}   \ar[r]^{\a_{\mathcal{P}} \quad \quad \quad \quad}   & \Hom(\lG(F)/G(F), \C^{\times}) }
\]
if $F$ is local, or
\[
\xymatrix{   1 \ar[r] \ar@{=}[d]    &  \S{\lp} \ar[d]  \ar[r]    &   \S{\p} \ar[d]^{\bold{c}}    \ar[r]^{\a \quad \quad \quad \quad \quad \quad}   &  \Hom(\lG(\A_{F})/ \lG(F)G(\A_{F}), \C^{\times})  \ar@{=}[d]   \\
1 \ar[r]     &    \mathcal{P}_{\lp}  \ar[r]    &  \mathcal{P}_{\p}   \ar[r]^{\a_{\mathcal{P}} \quad \quad \quad \quad \quad \quad}   & \Hom(\lG(\A_{F})/ \lG(F)G(\A_{F}), \C^{\times}) }
\]
if $F$ is global.
\end{enumerate}
\end{lemma}

\begin{proof}
First we would like to show $\bold{c}$ factors though $\S{\p}$, i.e. for any $s \in \cS{\p}$, $\bold{c}(s)$ only depends on the image $x$ of $s$ in $\S{\p}$. Note that if $s$ is replaced by an $\com[0]{\cS{\p}}$-conjugate $s_{1}$, then the corresponding pair $(G'_{1}, \p'_{1})$ is isomorphic to $(G', \p')$. So by our definition 
\[
\bold{c}(s_{1}) = \bold{c}(s).
\]
Now if we fix a maximal torus $\bar{T}_{\p}$ of $\com[0]{\cS{\p}}$ and a Borel subgroup $\bar{B}_{\p}$ containing it, any automorphism of the complex reductive group $\com[0]{\cS{\p}}$ stabilizes a conjugate of $(\bar{T}_{\p}, \bar{B}_{\p})$. So we can choose a representative $s_{x}$ of $x$ in $\cS{\p}$ so that $\Int(s_{x})$ stabilizes $(\bar{T}_{\p}, \bar{B}_{\p})$, and such representatives are determined up to a $\bar{T}_{\p}$-translate. Moreover the complex torus
\[
\bar{T}_{\p, x} = \Cent(s_{x}, \bar{T}_{\p})^{0}
\]
in $\bar{T}_{\p}$ is uniquely determined by $x$. Note that $\bar{T}_{\p, x}$ is the connected component of the kernel of the following morphism
\[
\xymatrix{ \bar{T}_{\p} \ar[r]  & \bar{T}_{\p} \\
t \ar@{|->}[r]  & s_{x}^{-1}ts_{x}t^{-1}}
\]
So any point of $\bar{T}_{\p}$ can be written as $(s_{x}^{-1}ts_{x}t^{-1}) t_{x}$ for $t \in \bar{T}_{\p}$ and $t_{x} \in \bar{T}_{\p, x}$ (see \cite{Springer:2009}, Corollary 5.4.5), and hence any point in $s_{x} \bar{T}_{\p}$ can be written as
\[
s_{x} (s_{x}^{-1}ts_{x}t^{-1}) t_{x} = t s_{x} t^{-1} t_{x} = t s_{x} t_{x}  t^{-1}, \,\,\,\, t \in \bar{T}_{\p}, \,\,\, t_{x} \in \bar{T}_{\p, x}.
\]
Therefore it will be enough to show that 
\[
\bold{c}(s_{x}) = \bold{c}( s_{x} t_{x})
\]
for any $t_{x} \in \bar{T}_{\p, x}$.

The centralizer $\D{M}_{x}$ of $\bar{T}_{\p, x}$ in $\D{G}$ is a Levi subgroup of $\D{G}$, which is dual to a Levi subgroup $M_{x}$ of $G$. So $(\p, s_{x})$ is the image of a pair
\[
(\p_{M_{x}}, s_{M_{x}}), \,\,\,\, \p_{M_{x}} \in \cP{M_{x}}, s_{M_{x}} \in S_{\p_{M_{x}}},
\] 
attached to $M_{x}$ under the $L$-embedding $\L{M_{x}} \subseteq \L{G}$. And this pair is in turn the image of an endoscopic pair $(M'_{x}, \p'_{M_{x}})$. Note that $M_{x}$ has  a component $G_{x}$ of the same type as $G$ and $\cS{\p_{G_{x}}} \cong \cS{\p_{M_{x}}}$. So we can define a map $\bold{c}^{M_{x}}$ on $\cS{\p_{M_{x}}}$ by $\bold{c}^{G_{x}}$ with respect to the component $\p_{G_{x}}$ of $\p_{M_{x}}$. Since $M'_{x}$ can be identified with a Levi subgroup of $G'$, one can easily check that 
\[
\bold{c}^{M_{x}}(s_{M_{x}}) = \bold{c}(s_{x}).
\]  
Note that $\bold{c}^{M_{x}}(s_{M_{x}})$ is invariant under the translation of $s_{M_{x}}$ by $\bar{T}_{\p, x}$, so the same is true of $\bold{c}(s_{x})$.

Secondly we need to show that $\bold{c}$ is in fact a bijection between $\S{\p}$ and $\mathcal{P}_{\p}$. Note that we can actually compute $|\S{\p}|$ and $|\mathcal{P}_{\p}|$ explicitly. For $|\S{\p}|$, we have the description from Section~\ref{sec: Arthur's theory} that
\begin{align}
S_{\p} = (\prod_{i \in I_{\p, O}} O(l_{i}, \C))_{\p}^{+} \times (\prod_{i \in I_{\p, S}} Sp(l_{i}, \C)) \times (\prod_{j \in J_{\p}} GL(l_{j}, \C)),   \label{eq: combinatorial description 1}
\end{align}
where $(\prod_{i \in I_{\p, O}} O(l_{i}, \C))_{\p}^{+}$ is the kernel of the character 
\[
\xi_{\p}^{+} : \prod_{i} g_{i} \longrightarrow \prod_{i} (det \, g_{i})^{N_{i}}, \,\,\,\,\, g_{i} \in O(l_{i}, \C), i \in I_{\p, O}.
\]
If $G$ is symplectic or $G$ is special even orthogonal with $I^{odd}_{\p, O}$ being empty, then
\[
\S{\p} = \begin{cases}
                (\Two)^{|I_{\p, O}|} & \text{ if all $N_{i}$ are even for $i \in I_{\p, O}$}, \\
                (\Two)^{|I_{\p, O}|-1} & \text{ otherwise. }
                \end{cases}
\]
If $G$ is special even orthogonal and $I^{odd}_{\p, O}$ is not empty, then $Z(\D{G}) \notin \com[0]{\cS{\p}}$ and thus
\[
\S{\p} = \begin{cases}
                (\Two)^{|I_{\p, O}| - 1} & \text{ if all $N_{i}$ are even for $i \in I_{\p, O}$}, \\
                (\Two)^{|I_{\p, O}| - 2} & \text{ otherwise. }
                \end{cases}
\] 
For $|\mathcal{P}_{\p}|$, it is just a combinatorial computation. Suppose $G$ is symplectic, there is no condition on $N_{i}$, so $| \mathcal{P}_{\p} | =  2^{|I_{\p, O}|-1}$. Suppose $G$ is special even orthogonal, it again divides into two cases. If all $N_{i}$ are even for $i \in I_{\p, O}$, then the condition on  $N_{i}$ is automatically satisfied and hence 
\[
|\mathcal{P}_{\p}| = \begin{cases}
                                   2^{|I_{\p, O}|} & \text{ if $I^{odd}_{\p, O}$ is empty}, \\
                                   2^{|I_{\p, O}| - 1} & \text{ otherwise }. \\
                                   \end{cases}
\]
And if there exists $i \in I_{\p, O}$ such that $N_{i}$ is odd then 
\[
|\mathcal{P}_{\p}| = \begin{cases}
                                   2^{|I_{\p, O}| - 1} & \text{ if $I^{odd}_{\p, O}$ is empty}, \\
                                   2^{|I_{\p, O}| - 2} & \text{ otherwise }. \\
                                   \end{cases}
\]
Therefore one can conclude that $|\S{\p}| = |\mathcal{P}_{\p}|$. Now it suffices to show that $\bold{c}$ is surjective. In fact for any partition $(S, T)$, one can choose an element $s = (s_{k})_{k \in K_{\p}} \in S_{\p}$ according to the decomposition \eqref{eq: combinatorial description 1} such that it has the form
\[
s_{i} =   \begin{pmatrix} 
              -1   &&& \\
              &1   &&  \\
              &&   \ddots & \\
              &&& 1
              \end{pmatrix}
\text{ for } i \in S \cup T \text{, and } s_{k} = I \text{ otherwise }.
\]
When $G$ is symplectic, we can assume $\sum_{i \in S \cup T} N_{i}$ is odd by possibly changing $(S, T)$ to $(S^{c}, T)$. If $(G', \p') \rightarrow (\p, s)$, then $G'$ is elliptic and $\p' = \p'_{I} \times \p'_{II} \, (\text{or $\p_{-}' = (\p'_{I} \otimes \eta_{\p'_{I}}) \times \p'_{II} $ if $G'_{I}$ is symplectic})$ with the property that 
\[
\p'_{I} = \boxplus_{i \in S \cup T} \p_{i}.
\]
Hence by the definition $\bold{c}(s) = (S, T)$.

For the second part of the lemma, it is enough to show that 
\[
\a(s) = \a_{\mathcal{P}}(\bold{c}(s)),
\]
for $s \in S_{\p}$ being such representatives chosen above. Let 
\[
\eta' = \eta_{\p'_{I}} = \prod_{i \in S \cup T} \eta_{\p_{i}},
\]
then $\a_{\mathcal{P}}(\bold{c}(s)) = \eta'  \circ \c $. Moreover, $G'$ will be $Sp(2n_{1}) \times SO(2n_{2}, \eta')$ if $G = Sp(2n)$, and $SO(2n_{1}, \eta') \times SO(2n_{2}, \eta' \eta)$ if $G = SO(2n, \eta)$. As one can see from the table of twisted elliptic endoscopic groups in Section \ref{subsubsec: twisted endoscopy}, $G'$ can be lifted to $\lG' \in \tEnd{ell}{\lG}$ with $\x = \eta' \circ \c$. So $\x = \a_{\mathcal{P}}(\bold{c}(s))$. Finally we just need to notice $\a(s) = \x$ by Lemma~\ref{lemma: twisted character}, hence $\a(s) = \a_{\mathcal{P}}(\bold{c}(s))$. 

\end{proof}

\begin{corollary}
\label{cor: combinatorial description}
Suppose $\p = l_{1}\p_{1} \# \cdots \# l_{r}\p_{r} \in \cPbd{G}$ if $F$ is local (resp. $\cP{G}$ if F is global), and $S$, $T$ are subsets of $I^{odd}_{\p, O}$, $I^{even}_{\p, O}$ respectively. Suppose that
\[
( \prod_{i \in S \cup T}\eta_{\p_{i}} ) \circ \c \neq 1
\]
unless $T$ is empty and $S$ is either empty or equal to $I^{odd}_{\p, O}$. Then $\S{\lp} = 1$.
\end{corollary}

\begin{proof}
It follows from the definition that $| \mathcal{P}_{\lp} | = 1$. By Lemma~\ref{lemma: combinatorial description}, one has $| \S{\lp} | = | \mathcal{P}_{\lp} | = 1$. Hence $\S{\lp} = 1$.
\end{proof}

The following proposition will become useful in our later proofs.

\begin{proposition}
\label{prop: consistency on induction}
Suppose $\p \in \cPel{G^{\theta}}$ for $\theta \in \Sigma_{0}$ and $\S{\lp} = 1$. We assume one of the following condition is satisfied.
\begin{enumerate}

\item $G$ is symplectic,

\item $G$ is special even orthogonal with $\eta_{G} \neq 1$,

\item $G$ is sepcial even orthogonal with $\eta_{G} = 1$, and $I_{\p, O}^{odd}$ or $I_{\p, O}^{even}$ is empty.

\end{enumerate}
If $\p$ factors through $\p' \in \cP{G'}$, where $G' = G_{I} \times G_{II}$ is a twisted elliptic endoscopic group of $G$, let $\p' = \p_{I} \times \p_{II}$, where $\p_{i} \in \cP{G_{i}}$ for $i = I, II$. Then $S_{\lp_{I}} = S_{\lp_{II}} = 1$.

\end{proposition}

\begin{proof}

Since $\p \in \cPel{G^{\theta}}$, we can view 
\[
I_{\p_{I}, O}^{odd} \subseteq I_{\p, O} \text{ and } I_{\p_{I}, O}^{even} \subseteq I_{\p, O}^{even}
\] 
after possibly twisting $\p_{I}$ by $\eta_{\p_{I}}$. Suppose $\S{\lp_{I}} \neq 1$, we can represent any nontrivial element of $\S{\lp_{I}}$ by $(S', T')$ for 
\[
S' \subseteq I_{\p_{I}, O}^{odd} \text{ and } T' \subseteq I_{\p_{I}, O}^{even}
\] 
such that $S' \cup T'$ is nonempty, $\sum_{i \in S' \cup T'} N_{i}$ is even and 
\[
\prod_{i \in S' \cup T'} \eta_{\p_{i}} = 1.
\]
Then we want to show $\S{\lp} \neq 1$, which would lead to a contradiction. 

Let us define $(S, T)$ by 
\[
S = S' \cap I_{\p, O}^{odd} \text{ and } T = (S' \cup T') \cap I_{\p, O}^{even}.
\] 
Then $(S, T)$ corresponds to a nontrivial element in $\S{\lp}$ unless $T$ is empty and $S =  I_{\p, O}^{odd}$. In the exceptional cases, we have 
\[
\text{$T'$ is empty and $S' = S = I_{\p, O}^{odd}$}.
\] 
By our conditions on $(S', T')$, we see $G$ has to be special even orthogonal and $\eta_{G} = 1$. So we only need to consider condition (3). In particular, we can assume $I_{\p, O}^{even}$ is empty, i.e. $I_{\p, O} = I_{\p, O}^{odd}$. It follows $S' = I_{\p_{I}, O}^{odd}$. But this is impossible for $(S', T')$ should correspond to a nontrivial element in $\S{\lp_{I}}$ by our assumption. Therefore, we see $\S{\lp} \neq 1$. Similarly, if we assume $\S{\lp_{II}} \neq 1$, we can also deduce $\S{\lp} \neq 1$. This finishes the proof.

\end{proof}

In the case that $G$ is a special even orthogonal group and $\p \in \Pbd{\com{G}}$ if $F$ is local (resp. $\P{\com{G}}$, if $F$ is global), we can have a similar combinatorial description of the map
from $\S{\p}^{\Sigma_{0}}$ to $\Hom(\lG(F)/G(F), \C^{\times})$ if $F$ is local (resp. $\Hom(\lG(\A_{F})/ \lG(F)G(\A_{F}), \C^{\times})$ if $F$ is global). To do this we need to take a bigger set
\[
\mathcal{P}_{\p}^{\Sigma_{0}} = \{ (S, T) \} / S \sim S^{c} 
\]
by withdrawing the condition that 
\(
\sum_{i \in S \cup T} N_{i}
\)
must be even. It is easy to see that one can extend the map $\a_{\mathcal{P}}$ to $\mathcal{P}_{\p}^{\Sigma_{0}}$ and the map $\bold{c}$ to $\cS{\p}^{\Sigma_{0}}$ with its image in $\mathcal{P}_{\p}^{\Sigma_{0}}$. As a result, we have the following lemma which is an analogue of Lemma~\ref{lemma: combinatorial description}, and the proof is similar.

\begin{lemma}
\label{lemma: plus combinatorial description}
\begin{enumerate}
\item The extended map $\bold{c}$ will factor through $\S{\p}^{\Sigma_{0}}$, and it gives a bijection between $\S{\p}^{\Sigma_{0}}$ and $\mathcal{P}_{\p}^{\Sigma_{0}}$.
\item If we denote the bijection in (1) still by $\bold{c}$, then we have a commutative diagram.
\[
\xymatrix{   1 \ar[r] \ar@{=}[d]    &  \S{\lp}^{\Sigma_{0}} \ar[d]  \ar[r]    &   \S{\p}^{\Sigma_{0}} \ar[d]^{\bold{c}}    \ar[r]^{\a \quad \quad \quad \quad}   &  \Hom(\lG(F)/G(F), \C^{\times})  \ar@{=}[d]   \\
1 \ar[r]     &    \mathcal{P}^{\Sigma_{0}}_{\lp}  \ar[r]    &  \mathcal{P}^{\Sigma_{0}}_{\p}   \ar[r]^{\a_{\mathcal{P}} \quad \quad \quad \quad}   & \Hom(\lG(F)/G(F), \C^{\times}) }
\]
if $F$ is local, or
\[
\xymatrix{   1 \ar[r] \ar@{=}[d]    &  \S{\lp}^{\Sigma_{0}} \ar[d]  \ar[r]    &   \S{\p}^{\Sigma_{0}} \ar[d]^{\bold{c}}    \ar[r]^{\a \quad \quad \quad \quad \quad \quad}   &  \Hom(\lG(\A_{F})/ \lG(F)G(\A_{F}), \C^{\times})  \ar@{=}[d]   \\
1 \ar[r]     &    \mathcal{P}^{\Sigma_{0}}_{\lp}  \ar[r]    &  \mathcal{P}^{\Sigma_{0}}_{\p}   \ar[r]^{\a_{\mathcal{P}} \quad \quad \quad \quad \quad \quad}   & \Hom(\lG(\A_{F})/ \lG(F)G(\A_{F}), \C^{\times}) }
\]
if $F$ is global.
\end{enumerate}
\end{lemma}

\begin{remark}
\label{rk: plus combinatorial description}
It is a consequence of Lemma~\ref{lemma: combinatorial description} and \ref{lemma: plus combinatorial description} that
\[
\a(\S{\p}^{\Sigma_{0}}) = \a_{\mathcal{P}}(\mathcal{P}_{\p}^{\Sigma_{0}})
\]
\end{remark}

Here we give two applications of these combinatorial descriptions. The first one gives the refined $L$-packet in the archimedean case (cf. Remark~\ref{rk: refined L-packet}).

\begin{proposition}
\label{prop: refined L-packet Archimedean case}
Suppose $F$ is real, $\p \in \cPbd{G}$ and $\lr$ is an irreducible admissible representation of $\lG(F)$ whose restriction to $G(F)$ have irreducible constituents contained in $\cPkt{\p}$. If $\S{\p} \neq 1$, then $\lr \otimes \x \cong \lr$ for all characters $\x$ of $\lG(F) / \lZ(F)G(F)$. In particular, let $\lif{\zeta}$ be the central character of $\lr$, then we can define $\cPkt{\lp} = \clPkt{\p, \lif{\zeta}} $ if $\cPkt{\p}$ is not a singleton.
\end{proposition}

\begin{proof}
Notice that $\lG(\R) / \lZ(\R)G(\R)$ is either $1$ or $\R^{\times} / \R_{>0}$, and the only nontrivial character $\varepsilon$ of $\R^{\times} / \R_{>0}$ is given by the sign character. Since $GL(n, \R)$ has essentially discrete series only when $n \leqslant 2$, the set $I_{\p, O}$ only parametrizes quadratic characters of $F^{\times}$ and discrete series of $GL(2, \R)$ with central character $\varepsilon$. Now we suppose $\lr \otimes \varepsilon \ncong \lr$, then it is clear from Lemma~\ref{lemma: combinatorial description} that this is only possible when $I_{\p, O}$ parametrize quadratic characters of $F^{\times}$, i.e. $\varepsilon$ and the trivial character $\varepsilon_{0}$. Depending on which characters $I_{\p, O}^{odd}$ and $I_{\p, O}^{even}$ parametrize, we have eight cases and we can represent them as follows: $\varepsilon_{0}, 2\varepsilon_{0}; \varepsilon, 2\varepsilon; \varepsilon_{0} \+ 2\varepsilon, \varepsilon \+ 2\varepsilon_{0}; \varepsilon_{0} \+ \varepsilon, 2\varepsilon_{0} + 2\varepsilon$. One can see easily from Lemma~\ref{lemma: combinatorial description} that in these cases either $\S{\p} = 1$ or $\varepsilon \in \a(\S{\p})$. Therefore we get a contradiction. For the last point, one just needs to notice $\S{\p} \neq 1$ if $\cPkt{\p}$ is not a singleton.
\end{proof}

\begin{remark}
\label{rk: refined L-packet Archimedean case}
The proof of Proposition~\ref{prop: refined L-packet Archimedean case} also shows that $X(\lr) = X$ for discrete series representation $\lr$ of $\lG(\R)$.
\end{remark}

The second application is on the multiplicity problem that we have discussed in Section~\ref{sec: multiplicity formula}. Now let us assume $F$ is global again, and it turns out more convenient to ask when both $\Sigma_{0}$-multiplicity one and $\Sigma_{0}$-strong multiplicity one hold for $\lp$ together, i.e.
\[
\a(\S{\p}^{\Sigma_{0}}) = \prod^{aut}_{almost \, all \, v} \a(\S{\p_{v}}^{\Sigma_{0}}).
\]
By our Remark~\ref{rk: plus combinatorial description}, this is the same as 
\[
\a_{\mathcal{P}}(\mathcal{P}_{\p}^{\Sigma_{0}}) = \prod^{aut}_{almost \, all \, v} \a_{\mathcal{P}}(\mathcal{P}_{\p_{v}}^{\Sigma_{0}}).
\]

The next lemma gives an answer for the simplest type of parameters.

\begin{lemma}
\label{lemma: splitting parameters}
Suppose 
\[
\p = l_{1} \eta_{1} \# \cdots \# l_{r} \eta_{r} \in \cP{G},
\]
where $\eta_{i}$ are quadratic id\`ele class characters for $1 \leqslant i \leqslant r$. Then both $\Sigma_{0}$-multiplicity one and $\Sigma_{0}$-strong multiplicity one hold for $\lp$.
\end{lemma}

\begin{proof}
From Lemma~\ref{lemma: similitude character}, we can view $\eta_{i} \circ \c$ as quadratic id\`ele class characters of $F'$, where $F'$ is the extension of $F$ associated to the character 
\[
\eta_{\p} = \prod_{i=1}^{r} \eta_{\p_{i}}^{l_{i}}
\]
by class field theory. Let us denote $\eta_{i} \circ \c$ by $\eta'_{i}$. We are going to prove this lemma by induction on the number of nontrivial characters $\eta'_{i}$ for $1 \leqslant i \leqslant r$.  
Suppose there exists some quadratic id\`ele class character 
$\x$ of $F$ such that $\x' =  \x \circ \c$ is contained in 
\[
\prod^{aut}_{almost \, all \, v} \a_{\mathcal{P}}(\mathcal{P}_{\p_{v}}^{\Sigma_{0}}).
\]
If we assume $\eta'_{1}$ is nontrivial, and let $E$ be the quadratic extension of $F'$ associated to $\eta'_{1}$, then after a base change to $E$, we get
\[
\p_{E} =  l_{1} \eta_{E, 1} \# \cdots \# l_{r} \eta_{E, r}
\]
where $\eta_{E, i} = \eta'_{i} \circ \Nm_{E / F'}$. And we have 
\[
\x_{E} =  \x' \circ \Nm_{E / F'}  
\]
contained in  
\[
\prod^{aut}_{almost \, all \, v} \a_{\mathcal{P}}(\mathcal{P}_{\p_{E, v}}^{\Sigma_{0}}).
\]
Since $\eta_{E,1} = 1$, by induction the lemma is true for $\p_{E}$. Hence
\[
\x_{E} = \prod^{m}_{k=1} \eta_{E, i_{k}},
\]
for some $1 < i_{k} \leqslant r$ and $m < r$. This implies that 
\[
( \x' \cdot \prod^{m}_{ k=1} \eta'_{i_{k}} ) \circ \Nm_{E / F'} = 1.
\]
Since 
\(
|I_{F'} : \Nm_{E/F'} I_{E} | = 2,
\)
then
\[
\x' \cdot \prod^{m}_{ k=1 } \eta'_{i_{k}} = 1 \text{ or } \eta'_{1}.
\]
Hence $\x' \in \a_{\mathcal{P}}(\mathcal{P}_{\p}^{\Sigma_{0}})$.
\end{proof}

\subsection{Construction of global parameters}
\label{subsec: construction of global parameters}

In this section we are going to give the global lifting result needed in Theorem~\ref{thm: standard argument on stability}. Let us assume $F$ is a nonarchimedean local field and $\dot{F}$ is a totally real global field with $\dot{F}_{u} = F$ (cf. \cite{Arthur:2013}, Lemma 6.2.1). Let $\dot{G}$ be a quasisplit special even orthogonal group or symplectic group over $\dot{F}$ such that $\dot{G}_{u} = G$ and $\dot{G}_{v}$ has discrete series for $v \in S_{\infty}$ (cf. \cite{Arthur:2013}, Lemma 6.2.2). For any finite set $S$ of nonarchimedean places of $\dot{F}$, we denote the unitary dual of $\dot{G}(\dot{F}_{S}) = \prod_{v \in S}\dot{G}(\dot{F}_{v})$ by $\widehat{\dot{G}(\dot{F}_{S})}$, and the Plancherel measure on $\widehat{\dot{G}(\dot{F}_{S})}$ by $\D{\mu}^{pl}_{S}$. 

\begin{lemma}
\label{lemma: multiple global lifting}
For $\p \in \cPsm{G}$ and an open subset $\D{U}$ of tempered representations of $\dot{G}(\dot{F}_{S})$ such that $\D{\mu}^{pl}_{S}(\D{U}) > 0$, one can find $\dot{\p} \in \cPsm{\dot{G}}$ with the following properties.
\begin{enumerate}
\item $\dot{\p}_{u} = \p$, and $\otimes_{v \in S} \cPkt{\dot{\p}_{v}} \subseteq \D{U}$.
\item If $v \notin S_{\infty}(u) \cup S$, then $\dot{\p}_{v}$ is spherical. In particular, it can be written as a direct sum of quasicharacters of $F_{v}^{\times}$ with at most one ramified quasicharacter.
\item If $v \in S_{\infty}$, $\dot{\p}_{v} \in \cPdt{\dot{G}_{v}}$.
\end{enumerate}
\end{lemma}

\begin{proof}
This lemma is the consequence of (\cite{Shin:2012}, Theorem 5.8) and (\cite{Arthur:2013}, Lemma 6.2.2 and Corollary 6.2.4). 
As one can see in the proof of (\cite{Arthur:2013}, Lemma 6.2.2), $\dot{\p}_{v}$ has a ramified quasicharacter if and only if $\eta_{\dot{G}_{v}}$ is ramified. Also note that  (\cite{Shin:2012}, Theorem 5.8) requires $G$ to have trivial centre, however this condition can be removed by choosing suitable discrete series in the archimedean places as in the proof of (\cite{Arthur:2013}, Lemma 6.2.2). In fact, the main techniques in both proofs are the same, i.e., Arthur's simple invariant trace formula. The new input in (\cite{Shin:2012}, Theorem 5.8) is Harish-Chandra's Plancherel formula and Sauvageot's principle of density result (cf. \cite{Sauvageot:1997}, Theorem 7.3). 
\end{proof}

Lemma~\ref{lemma: multiple global lifting} serves as the building blocks of our global lifting result, and because we want to impose the condition of $\Sigma_{0}$-strong multiplicity one at one place for any global lift, it is important to consider the case of simple parameters first. We will begin with another description of $\Sigma_{0}$-strong multiplicity one, which is kind of dual to the original one.

\subsubsection{$\Sigma_{0}$-strong multiplicity one at one place}

Suppose $F$ is a global field, $G$ is special even orthogonal or symplectic group over $F$ and $\p \in \cP{G}$. We define

\begin{align*}
\iG{\A_{F}} & =  \lZ(F_{u})G(F_{u}) \times \prod_{v \neq u}\lG(F_{v}), \\
\iG{F} & = \lG(F) \cap \iG{\A_{F}}.
\end{align*}
Let 
\[
\lG(\r_{v}^{\Sigma_{0}}) = \{ g \in \lG(F_{v}) : \x_{v}(g) = 1 \text{ for all } \x_{v} \in \a(\S{\p_{v}}^{\Sigma_{0}}) \},
\]
for any $[\r_{v}] \in \cPkt{\p_{v}}$, and 
\[
\lG(\r^{\Sigma_{0}}) = \prod_{v} \lG(\r_{v}^{\Sigma_{0}}),
\]
for any $[\r] \in \cPkt{\p}$.
We define $\iG{}(\r^{\Sigma_{0}}) = \lG(\r^{\Sigma_{0}}) \cap \iG{\A_{F}}$. As a consequence we can have the following identities
\begin{align*}
\prod^{aut}_{v} \a(\S{\p_{v}}^{\Sigma_{0}}) & = ( \lG(\A_{F}) / \lG(F) \lG(\r^{\Sigma_{0}}) )^{*} \\
\prod^{aut}_{v \neq u} \a(\S{\p_{v}}^{\Sigma_{0}}) & = ( \iG{\A_{F}} / \iG{F} \iG{}(\r^{\Sigma_{0}}) )^{*} .
\end{align*}
By the approximation theory for number fields (cf. \cite{Neukirch:1999}), we have $\lG(\A_{F}) = \lG(F) \iG{\A_{F}}$. So
\[
( \lG(\A_{F}) / \lG(F) \lG(\r^{\Sigma_{0}}))^{*} = ( \iG{\A_{F}} / \iG{\A_{F}} \cap \lG(F) \lG(\r^{\Sigma_{0}}))^{*}
\]
Therefore the condition of $\Sigma_{0}$-strong multiplicity one at the place $u$ is equivalent to
\begin{align}
\label{eq: strong multiplicity one at one place 1}
|\iG{\A_{F}} \cap \lG(F) \lG(\r^{\Sigma_{0}}) : \iG{F} \iG{}(\r^{\Sigma_{0}})| = |\iG{\A_{F}} \cap \lG(F) \lG(\r^{\Sigma_{0}}) : \iG{F} ( \lG(\r^{\Sigma_{0}}) \cap \iG{\A_{F}} ) | = 1.
\end{align}
Let
\begin{align*}
A  &= \iG{\A_{F}} / G(\A_{F}), \quad
A_{F}  = \iG{F}G(\A_{F}) / G(\A_{F}) \\
B_{F}  &= \lG(F) / G(F),  \quad
\bar{B}(\r)  =  \lG(\r^{\Sigma_{0}}) / G(\A_{F})
\end{align*}
and $\bar{B}(\r_{v}) = \lG(\r_{v}^{\Sigma_{0}}) / G(F_{v})$, then we can rewrite \eqref{eq: strong multiplicity one at one place 1} as
\[
|A \cap \bar{B}(\r) B_{F} : ( A \cap \bar{B}(\r) ) A_{F}| = 1.
\]
In particular,
\[
( A \cap \bar{B}(\r) ) A_{F} = A \cap \bar{B}(\r) A_{F},
\]
so we has proved the following lemma.

\begin{lemma}
\label{lemma: strong multiplicity one at one place 1}
Suppose $\p \in \cP{G}$ and $[\r] \in \cPkt{\p}$. Then $\Sigma_{0}$-strong multiplicity one at the place $u$ holds for $\lp$ if and only if
\begin{align}
\label{eq: strong multiplicity one at one place 2}
|A \cap \bar{B}(\r) B_{F} : A \cap \bar{B}(\r)  A_{F}| = 1.        
\end{align}
\end{lemma}

Note that all these groups $A$, $\bar{B}(\r)$ and $A_{F}$, $B_{F}$ can be viewed as subgroups of $I_{F}$ and $F^{\times}$ respectively through the similitude character $\c$. Therefore we have the following equivalent statement for \eqref{eq: strong multiplicity one at one place 2}.

\begin{lemma}
\label{lemma: strong multiplicity one at one place 2}
$|A \cap \bar{B}(\r) B_{F} : A \cap \bar{B}(\r)  A_{F}| = 1$ if and only if for any $x \in \bar{B}(\r)$, there exists $y \in \bar{B}(\r) \cap B_{F}$ such that $xy \in A$. (i.e., for any $x_{u} \in \bar{B}(\r_{u}) / (F^{\times}_{u})^{2}$, there existis $y \in \bar{B}(\r) \cap B_{F}$ such that $y_{u} = x_{u}^{-1}$ mod $(F^{\times}_{u})^{2}$.)
\end{lemma}

\begin{proof}
Suppose $x \in \bar{B}(\r)$ and $z \in B_{F}$ such that $xz \in A$. First let us assume \eqref{eq: strong multiplicity one at one place 2}, then $xz \in A \cap \bar{B}(\r) A_{F}$ since $xz \in A \cap \bar{B}(\r)B_{F}$. So we can write $xz = u w$ where $u \in A_{F}$ and $w \in \bar{B}(\r) \cap A$. In particular $x w^{-1} = u z^{-1} \in \bar{B}(\r) \cap B_{F}$. Let us set $y = w x^{-1}$ which is also in $\bar{B}(\r) \cap B_{F}$, then one has $xy = w \in A$.

Conversely, let us take $y \in  \bar{B}(\r) \cap B_{F}$ such that $xy \in A$. Then we can write $xz = (xy)(y^{-1}z)$. Since both $xz$ and $xy$ lie in $A$, one has $y^{-1}z \in A$ and in particular, $y^{-1}z \in A\cap B_{F} = A_{F}$. Moreover, it clear that $xy \in \bar{B}(\r)$. Hence $xz \in A \cap \bar{B}(\r)  A_{F}$ and the rest of the lemma should be clear.
\end{proof}

\subsubsection{Global lift}

Now we are going to use Lemma~\ref{lemma: multiple global lifting}, Lemma~\ref{lemma: strong multiplicity one at one place 1} and Lemma~\ref{lemma: strong multiplicity one at one place 2} to produce a global lift with the intended property of $\Sigma_{0}$-strong multiplicity one at one place.

\begin{lemma}
\label{lemma: strong multiplicity one at one place}
Suppose $F$ is a nonarchimedean local field and $\p \in \cPsm{G}$, there exists a totally real global field $\dot{F}$ and a group $\dot{G}$ over $\dot{F}$ such that $\dot{F}_{u} = F$ and $\dot{G}_{u} = G$, and one can lift $\p$ to a global simple parameter $\dot{\p} \in \cPsm{G}$ satisfying the following properties.
\begin{enumerate}
\item $\dot{\p}_{u} = \p$, and $\dot{\p}_{v} \in \cPdt{\dot{G}_{v}}$ for $v \in S_{\infty}$.
\item If $v \notin S_{\infty}(u)$, $\dot{\p}_{v}$ is a direct sum of quasicharacters of $F_{v}^{\times}$ with at most one ramified quasicharacter.
\item $\Sigma_{0}$-strong multiplicity one holds for $\lif{\dot{\p}}$ at the place $u$.
\end{enumerate}
\end{lemma}

\begin{proof}
Let $\dot{G}$ be a quasisplit special even orthogonal group or symplectic group over $\dot{F}$ such that $\dot{G}_{u} = G$ and $\dot{G}_{v}$ has discrete series for $v \in S_{\infty}$, and let $\dot{\eta}_{\p} = \eta_{\dot{G}}$. In view of Lemma~\ref{lemma: multiple global lifting}, one would like to impose restrictions over a finite set $S$ of nonarchimedean places described by an open subset $\D{U}$ of tempered representations with $\D{\mu}^{pl}_{S} (\D{U}) > 0$, so that the global lift $\dot{\p}$ obtained from Lemma~\ref{lemma: multiple global lifting} has the property of $\Sigma_{0}$-strong multiplicity one at the place $u$. To determine $S$ and $\D{U}$, we need to use the equivalent characterization of the property of $\Sigma_{0}$-strong multiplicity one at one place given by Lemma~\ref{lemma: strong multiplicity one at one place 1} and Lemma~\ref{lemma: strong multiplicity one at one place 2}. 
First, let us take two distinct quadratic id\`ele class character $\dot{\eta}_{i}$ ($i = 1, 2$), so that $\dot{\eta}_{i, u} = 1$ and $\dot{\eta}_{i, v} = \varepsilon_{v}$ the sign character of $\mathbb{R}^{\times}$ for $v \in S_{\infty}$. This is possible for one can construct a quadratic extension of any number field with arbitrarily prescribed localizations at finitely many places.
Then we can consider the following composite parameter 
\[
\dot{\p}_{\eta} = \dot{\eta}_{\p} \# \dot{\eta}_{1} \# \dot{\eta}_{2} \# \dot{\eta}_{1}\dot{\eta}_{2} \in \cP{\dot{G}_{\eta}},
\] 
and let 
\begin{align*}
\bar{B}( \dot{\p}_{\eta, v} ) = \{ z \in  \lif{\dot{G}}_{\eta}(\dot{F}_{v}) / \dot{G}_{\eta}(\dot{F}_{v}): \x_{v}(z) = 1 \text{ for all } \x_{v} \in \a(\S{\dot{\p}_{\eta, v}}^{\Sigma_{0}})\}
\end{align*}
for any place $v$. Note that  $B_{\dot{F}} \cong \lif{\dot{G}}_{\eta}(\dot{F})/ \dot{G}_{\eta}(\dot{F})$ and $\lif{\dot{G}}_{\eta}(\dot{F}_{v}) / \dot{G}_{\eta}(\dot{F}_{v}) \cong \lif{\dot{G}}(\dot{F}_{v}) / \dot{G}(\dot{F}_{v})$. Moreover if $[\r] \in \cPkt{\p}$, then $\bar{B}( \r ) = \bar{B}(\dot{\p}_{\eta, u})$. And $\bar{B}( \dot{\p}_{\eta, v} ) = \R_{> 0}$ if $v \in S_{\infty}$. If we apply Lemma~\ref{lemma: splitting parameters} and Lemma~\ref{lemma: strong multiplicity one at one place 2} to $\dot{\p}_{\eta}$, we get for any $x \in \bar{B}( \r ) / (F^{\times})^2$, there exists $y \in B_{\dot{F}}$ such that $y_{u} = x^{-1}$ mod $( F^{\times} )^{2}$ and $y_{v} \in \bar{B}( \dot{\p}_{\eta, v} )$ for $v \neq u$. In particular, $y_{v} \in \mathbb{R}_{>0}$ if $v \in S_{\infty}$. Since we are going to use Lemma~\ref{lemma: multiple global lifting} to lift $\p$, by its properties we can conclude immediately that $y_{v} \in \bar{B}(\dot{\r}_{v})$ unless $v \neq u$ is nonarchimedean and $|y_{v}| \neq 1$. To emphasize the dependence of $y$ on $x$, we also write $y = y(x)$. Let $S_{y(x)}$ be those nonarchimedean places $v \neq u$ where $|y_{v}| \neq 1$, and since the group $\bar{B}( \r ) / (F^{\times})^2$ is finite, we can take the union of all such sets and get
\[
S = \bigcup_{x \in \bar{B}( \r ) / (F^{\times})^2} S_{y(x)}
\]
which is still finite. Note that $S$ depends on the choice of $y$ for each $x$. Now we can describe $\D{U} = \prod_{v \in S} \D{U}_{v}$. In fact for any $v \in S$, one just needs to take $\D{U}_{v}$ to be the union of tempered packets $\cPkt{\p_{v}}$ for spherical $\p_{v} \in \cPbd{\dot{G}_{v}}$ such that 
$\a(\S{\p_{v}}^{\Sigma_{0}}) = 1$. By Lemma~\ref{lemma: plus combinatorial description}, this condition is equivalent to requiring no finite products of unramified quasicharacters in $\p_{v}$ gives the nontrivial unramified quadratic character unless $\eta_{\dot{G}_{v}}$ is nontrivial and quadratic. Since this is an open condition, $\D{U}_{v}$ is open with positive Plancherel measure. Note that the condition $\a(\S{\p_{v}}^{\Sigma_{0}}) = 1$ also guarantees $\bar{B}( \dot{\p}_{\eta, v} ) \subseteq \bar{B}(\r_{v})$ for $[\r_{v}] \in \cPkt{\p_{v}} \subseteq \D{U}_{v}$.
Finally, we can use Lemma~\ref{lemma: multiple global lifting} to get a global lift $\dot{\p}$ such that $\cPkt{\dot{\p}_{v}} \subseteq \D{U}_{v}$ for $v \in S$. And it is clear that for any $x \in \bar{B}( \r ) / (F^{\times})^2$, the $y$ chosen above will lie in $\bar{B}(\dot{\r}) \cap B_{\dot{F}}$ for $[\dot{\r}] \in \cPkt{\dot{\p}}$. This finishes the proof.
\end{proof}

This lemma is the first step to overcome the lack of strong multiplicity one, next we will generalize this to composite parameters.


\begin{lemma}
\label{lemma: global lifting}
Suppose 
\[
\p = \p_{1} \+ \cdots \+ \p_{q} \+ 2\p_{q+1} \+ \cdots \+ 2\p_{r} \in \cPel{G^{\theta}},
\]
where $\p_{i}$ is simple for $1 \leqslant i \leqslant r$ and $\theta \in \Sigma_{0}$. We assume $\p$ factors through $\p_{M} \in \cPdt{M}$. Then one can choose a lift $(\dot{G}, \dot{M}, \dot{F}, \dot{\p})$ of $(G, M, F, \p)$ with the following properties:
\begin{enumerate}
\item $\dot{F}$ is a totally real field and there exists a place $u$ such $(\dot{G}_{u}, \dot{M}_{u}, \dot{F}_{u}, \dot{\p}_{u}) = (G, M, F, \p)$.
\item $\dot{\p} = \dot{\p}_{1} \# \cdots \# \dot{\p}_{q} \# 2\dot{\p}_{q+1} \# \cdots \# 2\dot{\p}_{r} \in \cPel{\dot{G}^{\theta}}$.
\item $\S{\lif{\dot{\p}}} = 1$. Moreover, $(\dot{G}, \dot{\p})$ satisfies the conditions in Proposition~\ref{prop: consistency on induction}.
\item If $v \notin S_{\infty}(u)$, the local Langlands parameter 
\[
\dot{\p}_{v} = \dot{\p}_{1, v} \+ \cdots \+ \dot{\p}_{q, v} \+ 2\dot{\p}_{q+1, v} \+ \cdots \+ 2\dot{\p}_{r, v}
\]
is a direct sum of quasicharacters of $F_{v}^{\times}$, and it contains at most one ramified quasicharacter counted without multiplicities modulo the unramified quasicharacters.
\item If $v \in S_{\infty}$, $\dot{\p}_{i, v} \in \cPdt{\dot{G}_{\p_{i, v}}}$.
\item $\Sigma_{0}$-strong multiplicity one holds for $\lif{\dot{\p}}$ at the place $u$.
\end{enumerate}
\end{lemma}

\begin{proof}
The idea is to apply Lemma~\ref{lemma: strong multiplicity one at one place} to each simple parameters. But one needs to be extra careful at the following points. 

The first point is about property (3) and (4), they require choosing those global characters $\dot{\eta}_{\p_{i}}$ in a consistent way so that the condition of Corollary~\ref{cor: combinatorial description} is satisfied, and also there is at most one ramified character in $\dot{\eta}_{\p_{1, v}} \+ \cdots \+ \dot{\eta}_{\p_{r, v}}$ at each place $v \notin S_{\infty}(u)$ counted without multiplicities modulo the unramified quasicharacters. But this can be done easily. In fact, one can fix nonarchimedean places $\{v_{1}, v_{2}, \cdots, v_{r}\}$ distinct from $u$. If $q = 0$, we require for $1 \leqslant i \leqslant r$ and $1 \leqslant j \leqslant r$ that
\begin{align}
\label{eq: global lifting}
\dot{\eta}_{\p_{i, v_{j}}} = \begin{cases}
                          \text{the unramified quadratic character of $\dot{F}^{\times}_{v_{j}}$, when } i = j, \\
                          1, \text{ otherwise. }
                          \end{cases}
\end{align}
If $q \neq 0$, we only impose \eqref{eq: global lifting} for $2 \leqslant i \leqslant r$ and $1 \leqslant j \leqslant r$, and require for $2 \leqslant j \leqslant r$ that
\[
\dot{\eta}_{\p_{1, v_{j}}} = \begin{cases} \prod_{1 < i \leqslant q}\dot{\eta}_{\p_{i, v_{j}}}, & \text{ if } q > 1, \\
1, & \text{ if } q = 1.
\end{cases}
\]
In this case, we also require $\dot{\eta}_{\p_{1, v_{1}}} \neq 1$ when $G$ is special even orthogonal. It is easy to see that these conditions will guarantee (3). Next we can choose $\dot{\eta}_{\p_{i}}$ satisfying these conditions. Moreover, we can choose them consecutively for $i$ decreasing from $r$ to $1$ such that $\dot{G}_{\p_{i}}$ has discrete series and $\dot{\eta}_{\p_{i}}$ is unramified over the ramified places of $\dot{\eta}_{\p_{j}}$ for $j > i$, except in the case $i =1$ and $G$ is symplectic, where we would like to assume $G_{\p_{1}}$ is also symplectic and take 
\[
\dot{\eta}_{\p_{1}} = \begin{cases} \prod_{1 < i \leqslant q}\dot{\eta}_{\p_{i}}, & \text{ if } q > 1, \\
1, & \text{ if } q = 1.
\end{cases}
\]

The second point is about choosing the set $S$ of nonarchimedean places as in the proof of Lemma~\ref{lemma: strong multiplicity one at one place}. It is tempting to think that it should be the union of all such sets defined in Lemma~\ref{lemma: strong multiplicity one at one place} for each simple parameter $\p_{i}$. In fact, one should choose this set $S$ for $\p$ as a whole. Let $\dot{\eta}_{1}$ and $\dot{\eta}_{2}$ again be two distinct quadratic id\`ele class characters defined as in Lemma~\ref{lemma: strong multiplicity one at one place}. And here we consider
\[
\dot{\p}_{\eta}  = \dot{\eta}_{\p_{1}} \# \cdots \# \dot{\eta}_{\p_{q}} \# 2\dot{\eta}_{\p_{q+1}} \# \cdots \# 2\dot{\eta}_{\p_{r}} \# \dot{\eta}_{1} \# \dot{\eta}_{2} \# \dot{\eta}_{1} \dot{\eta}_{2} \in \cP{\dot{G}_{\eta}}
\]
when $q$ is odd; or
\[
\dot{\p}_{\eta}  = \dot{\eta}_{\p_{1}} \# \cdots \# \dot{\eta}_{\p_{q}} \# 2\dot{\eta}_{\p_{q+1}} \# \cdots \# 2\dot{\eta}_{\p_{r}} \# 2\dot{\eta}_{1} \in \cP{\dot{G}_{\eta}}
\] 
when $q$ is even. By applying Lemma~\ref{lemma: splitting parameters} and Lemma~\ref{lemma: strong multiplicity one at one place 2} to $\dot{\p}_{\eta}$, we can get a set $S$ of nonarchimedean places using the same argument as in Lemma~\ref{lemma: strong multiplicity one at one place}. Finally, one can choose the open set of tempered representations $\D{U}_{i} = \prod_{v \in S} \D{U}_{i, v}$ for each simple parameter $\p_{i}$ as in Lemma~\ref{lemma: strong multiplicity one at one place} and make them smaller enough so that $\a(\S{\dot{\p}_{v}}^{\Sigma_{0}}) \subseteq \a(\S{(\dot{\p}_{\eta})_{v}}^{\Sigma_{0}})$. This is possible again by Lemma~\ref{lemma: plus combinatorial description}. Note that if $\p_{i}$ is a quadratic character, there is no need to impose any local conditions on $S$. This finishes the proof.
\end{proof}

\begin{remark}
\begin{enumerate}
\item
In view of Lemma~\ref{lemma: plus combinatorial description}, the second property about this global lift $\dot{\p}$ implies the natural inclusion $S_{\dot{\p}}^{\Sigma_{0}} \hookrightarrow S_{\p}^{\Sigma_{0}}$ is an isomorphism here. So we can identify $S_{\dot{\p}}^{\Sigma_{0}}$ with $S_{\p}^{\Sigma_{0}}$. 

\item In later proofs, we would like to apply the discussions in Section~\ref{subsec: comparison of trace formulas} to such global parameters $\dot{\p}$. In fact, by our induction assumption and Proposition~\ref{prop: consistency on induction} we can replace Conjecture~\ref{conj: global L-packet}, ~\ref{conj: stable multiplicity formula}, ~\ref{conj: compatible normalization} by Theorem~\ref{thm: main global} for the proper Levi subgroups and twisted endoscopic groups of $\lif{\dot{G}}$. It is then clear that Lemma~\ref{lemma: twisted spectral expansion} is still valid for $\dot{\p}$. Since we are only going to establish the stable multiplicity formula for discrete parameters in Theorem~\ref{thm: main global}, we need to require 
\[
\cS{\dot{\p}, ell}^{\theta}(\dot{\x}) = \cS{\dot{\p}, ss}^{\theta}(\dot{\x})
\]
when $\dot{\p}$ is not a discrete parameter in Lemma~\ref{lemma: twisted endoscopic expansion}. However, in our application this will always be satisfied by our choice of $\dot{\x}$ and the fact that $\S{\lif{\dot{\p}}} = 1$.

\end{enumerate}
\end{remark}


\subsection{Proof of main local theorem}
\label{subsec: proof of main local theorem}

With all these refined lifting results, we can start to prove the main local theorem. Let $F$ be a nonarchimedean local field, 
\begin{align}
\label{eq: elliptic parameter}
\p = \p_{1} \+ \cdots \+ \p_{q} \+ 2\p_{q+1} \+ \cdots \+ 2\p_{r} \in \cPel{G^{\theta}},
\end{align}
where $\p_{i}$ is simple for $1 \leqslant i \leqslant r$ and $\theta \in \Sigma_{0}$. The simplest cases are when $\p_{i}$ are quadratic characters $\eta_{i}$ for $1 \leqslant i \leqslant r$, and one can see not all of these cases will follow from induction. So we have to treat the exceptional cases differently. In fact regarding property $(4)$ of Lemma~\ref{lemma: global lifting}, one only needs to consider the cases when $r \leqslant 4$, i.e., the trivial character $\varepsilon_{0}$, the unramified quadratic character $\varepsilon$, a ramified quadratic character $\eta$, and also $\eta \cdot \varepsilon$. In fact, when $r = 4$ and $G$ is special even orthogonal, it can be further reduced to the case $r \leqslant 3$ by our induction argument as one can see from the proof of Lemma~\ref{lemma: global lifting}. 

\begin{lemma}
\label{lemma: character case}
For $\p$ shown in \eqref{eq: elliptic parameter}, if $\p_{i} = \eta_{i}$, and $r \leqslant 3$ (or $r \leqslant 4$ when $G$ is symplectic), then the main local theorem (Theorem~\ref{thm: refined L-packet}) holds for $\lp$.
\end{lemma}

\begin{proof}
There are two types of parameters which lead to nontrivial cases here. \\
Type I :
\[
\p = \eta_{1} \+ \eta_{2} \+ \eta_{3} \in \cPbd{G}
\]
 Type II:
\[
\p \in \cPel{G^{\theta}} - \cPdt{G}
\]

For type I, $\lG = GL(2)$ and $\S{\lp} = 1$ by Lemma~\ref{lemma: combinatorial description}, so the refined $L$-packet $\cPkt{\lp}$ is a singleton and hence determined by $\cPkt{\p}$. Since the character of any irreducible admissible representation of $GL(2, F)$ is already stable, the packet $\cPkt{\lp}$ is then stable. Therefore the only thing we need to prove is the twisted character relation \eqref{eq: theta twisted character relation} for $\x \in \a( \S{\p} )$ and $\theta = id$. To prove this, we use the stabilized $\dot{\x}$-twisted trace formula for $\dot{\x} \in \a( \S{\dot{\p}} )$ and some global lift $\dot{\p}$.

Assume $\dot{\p} = \dot{\eta}_{1} \# \dot{\eta}_{2} \# \dot{\eta}_{3}$ is a global lift of $\p$. Because the global $L$-packet for $GL(2)$ should also be a singleton and multiplicity one holds for $GL(2)$, the spectral side of the discrete part of the $\dot{\x}$-twisted trace formula becomes
\[
I^{(\lif{\dot{G}}, \dot{\x})}_{disc, \dot{\p}} ( \lif{\dot{f}} ) = tr R^{(\lif{\dot{G}}, \dot{\x})}_{disc, \dot{\p}} ( \lif{\dot{f}} ) =  \sum_{\dot{\x}' \in Y / \a(\S{\dot{\p}})} m(  \lif{\dot{\r}} \otimes \dot{\x}', \dot{\x} ) \lif{\dot{f}}_{\lif{\dot{G}}}( \lif{\dot{\r}} \otimes \dot{\x}', \dot{\x} ),
\]  
where $\lif{\dot{\r}}$ is taken to be any representation in $\mathcal{A}_{2}(\lif{\dot{G}})$, whose restriction to $\dot{G}(\A_{\dot{F}})$ are contained in $\cPkt{\dot{\p}}$, and 
\[
\lif{\dot{f}}_{\lif{\dot{G}}}( \lif{\dot{\r}} \otimes \dot{\x}', \dot{\x} ) = \prod_{v} \lif{\dot{f}}_{{\lif{\dot{G}}}_{v}}( \lif{\dot{\r}}_{v} \otimes \dot{\x}_{v}', \dot{\x}_{v} ),
\]
defined by \eqref{eq: theta twisted intertwining operator}. In particular, $m(  \lif{\dot{\r}} \otimes \dot{\x}', \dot{\x} ) = \pm 1$. For the endoscopic side, we can apply Lemma~\ref{lemma: twisted endoscopic expansion} and get
\[
I^{(\lif{\dot{G}}, \dot{\x})}_{disc, \dot{\p}} ( \lif{\dot{f}} ) =  \sum_{\dot{\x}' \in Y / \a(\S{\dot{\p}})}  \lif{\dot{f}}'_{\lif{\dot{G}}} (\lif{\dot{\p}} \otimes \dot{\x}', \dot{x}),
\]
where $\a( \dot{x} ) = \dot{\x}$ and $\S{\lif{\dot{\p}}} = 1$. Therefore we have an identity
\[
\sum_{\dot{\x}' \in Y / \a(\S{\dot{\p}})} m(  \lif{\dot{\r}} \otimes \dot{\x}', \dot{\x} ) \lif{\dot{f}}_{\lif{\dot{G}}}( \lif{\dot{\r}} \otimes \dot{\x}', \dot{\x} ) = \sum_{\dot{\x}' \in Y / \a(\S{\dot{\p}})}  \lif{\dot{f}}'_{\lif{\dot{G}}} (\lif{\dot{\p}} \otimes \dot{\x}', \dot{x}).
\]
Note that strong multiplicity one also holds for $\lif{\dot{\p}}$. This either can be seen from Lemma~\ref{lemma: splitting parameters} or from the fact that $\lG = GL(2)$ here. Then one can use the Satake parameters of representations of $\lif{\dot{G}}(\A_{\dot{F}})$ to distinguish the summands on both sides of the identity (cf. Lemma~\ref{lemma: trace formula component}). As a result, we get
\[
m(  \lif{\dot{\r}}, \dot{\x} ) \lif{\dot{f}}_{\lif{\dot{G}}}( \lif{\dot{\r}}, \dot{\x} ) = \lif{\dot{f}}'_{\lif{\dot{G}}} (\lif{\dot{\p}}, \dot{x} ),
\]
where we may need to change $\lif{\dot{\r}}$ by twisting with some $\dot{\x}' \in Y / \a(\S{\dot{\p}})$ to get this equality. Therefore for any place $v$ one has
\[
m(  \lif{\dot{\r}}_{v}, \dot{\x}_{v} ) \lif{\dot{f}}_{{\lif{\dot{G}}}_{v}}( \lif{\dot{\r}}_{v}, \dot{\x}_{v} ) = \lif{\dot{f}}'_{\lif{\dot{G}}_{v}} (\lif{\dot{\p}}_{v}, \dot{x}_{v} ),
\]
where $m(  \lif{\dot{\r}}_{v}, \dot{\x}_{v} )$ are some constants in $\C^{\times}$. If we take $\lif{\dot{f}}_{v}$ supported on $\lif{Z}_{\dot{F}_{v}}\dot{G}(\dot{F}_{v})$, then $ \lif{\dot{f}}_{{\lif{\dot{G}}}_{v}}( \lif{\dot{\r}}_{v}, \dot{\x}_{v} ) = \lif{\dot{f}}'_{\lif{\dot{G}}_{v}} (\lif{\dot{\p}}_{v}, \dot{x}_{v} )$ by character relation for $\dot{G}_{v}$, hence $m(  \lif{\dot{\r}}_{v}, \dot{\x}_{v} ) = 1$ and so is $m(  \lif{\dot{\r}}, \dot{\x} )$. In particular we have shown the twisted character relation for $\lp$ of type I.

For type II, it also suffices to show the twisted character relation regarding Lemma~\ref{lemma: refined L-packet for non-discrete parameter}. In fact its proof is similar to the proof of the twisted character relation for a general parameter
\[
\p = \p_{1} \+ \cdots \+ \p_{q} \+ 2\p_{q+1} \+ \cdots \+ 2\p_{r} \in \cPel{G^{\theta}} - \cPdt{G}.
\]
So here we will carry out the general strategy. First we apply Lemma~\ref{lemma: global lifting} to $\p$ and get a global lift $\dot{\p}$ such that $\dot{\p}_{u} = \p$, $\S{\lif{\dot{\p}}} = 1$, 
and $\Sigma_{0}$-strong multiplicity one holds for $\lif{\dot{\p}}$ at the place $u$. In particular, for the case of type II parameters considered here, it is true that both $\Sigma_{0}$-multiplicity one and $\Sigma_{0}$-strong multiplicity one hold according to Lemma~\ref{lemma: splitting parameters}. Next we would like to apply Lemma~\ref{lemma: twisted spectral expansion} and Lemma~\ref{lemma: twisted endoscopic expansion} to get an identity of the spectral expansion and endoscopic expansion of the stabilized $(\theta, \dot{\x})$-twisted trace formula. In view of Lemma~\ref{lemma: refined L-packet for non-discrete parameter}, we can assume semisimple $s \in \cS{\p}^{\theta}$ satisfies $|\cS{\p, s}^{0}| < \infty$. This implies $s \in \cS{\p, ell}^{\theta}$ and we denote its preimage in $\cS{\dot{\p}, ell}^{\theta}$ by $\dot{s}$. Let $\dot{x}$ be the image of $\dot{s}$ in $\S{\dot{\p}, ell}^{\theta}$. Since $\S{\lif{\dot{\p}}} = 1$, we have $\dot{\x} \neq 1$ and
\[
C_{\lif{\dot{\p}}} \sum_{ \dot{\x}' \in Y / \a(\S{\dot{\p}})} i^{\theta}_{\dot{\p}}(\dot{x}) \lif{\dot{f}}_{\lif{\dot{G}}^{\theta}}( \lif{\dot{\p}} \otimes \dot{\x}', \dot{x} ) = C_{\lif{\dot{\p}}} \sum_{ \dot{\x}' \in Y / \a(\S{\dot{\p}})} e'^{\theta}_{\dot{\p}}(\dot{x}) \lif{\dot{f}}'_{\lif{\dot{G}}^{\theta}}( \lif{\dot{\p}} \otimes \dot{\x}', \dot{x} ),
\]
It follows from Proposition~\ref{prop: endoscopy of complex group} and the fact that $\dot{s} \in \cS{\dot{\p}, ell}^{\theta}$, 
\[
i^{\theta}_{\dot{\p}}(\dot{x}) = e'^{\theta}_{\dot{\p}}(\dot{x}) \neq 0.
\]
Therefore 
\begin{align}
\label{eq: character case}
\sum_{ \dot{\x}' \in Y / \a(\S{\dot{\p}})}  \lif{\dot{f}}_{\lif{\dot{G}}^{\theta}}( \lif{\dot{\p}} \otimes \dot{\x}', \dot{x} ) = \sum_{ \dot{\x}' \in Y / \a(\S{\dot{\p}})} \lif{\dot{f}}'_{\lif{\dot{G}}^{\theta}}( \lif{\dot{\p}} \otimes \dot{\x}', \dot{x} ).
\end{align}
In case $\Sigma_{0}$-strong multiplicity one holds, we can again use the Satake parameters of admissible representations of $\lif{\dot{G}}(\A_{\dot{F}})$ to distinguish the summands on both sides of the identity. As a result, we get
\[
\label{eq: character case 2}
\lif{\dot{f}}_{\lif{\dot{G}}^{\theta}}( \lif{\dot{\p}}, \dot{x} ) = \lif{\dot{f}}'_{\lif{\dot{G}}^{\theta}}( \lif{\dot{\p}}, \dot{x} )
\]
and hence there exist constants $n_{v}$ for all places such that
\[
\lif{\dot{f}}_{\lif{\dot{G}}^{\theta}_{v}}( \lif{\dot{\p}}_{v}, \dot{x}_{v} ) = n_{v} \cdot \lif{\dot{f}}'_{\lif{\dot{G}}^{\theta}_{v}}( \lif{\dot{\p}}_{v}, \dot{x}_{v} ).
\]
If we take $\lif{\dot{f}}_{v}$ supported on $\lif{Z}_{\dot{F}_{v}}\dot{G}(\dot{F}_{v})$, then $\lif{\dot{f}}_{\lif{\dot{G}}^{\theta}_{v}}( \lif{\dot{\p}}_{v}, \dot{x}_{v} ) =  \lif{\dot{f}}'_{\lif{\dot{G}}^{\theta}_{v}}( \lif{\dot{\p}}_{v}, \dot{x}_{v} )$ by the twisted local intertwining relation for $\dot{G}_{v}$. Therefore $n_{v} = 1$ and we have shown the twisted local intertwining relation for $\lp$. By Lemma~\ref{lemma: twisted intertwining relation 1}, this implies the twisted character relation for $\lp$. So we have finished the proof for the parameters of type II. 

\end{proof}

\begin{theorem}
\label{thm: twisted character relation for elliptic parameter}
Suppose $F$ is a nonarchimedean local field and 
\(
\p \in \cPel{G^{\theta}} - \cPdt{G},
\)
then the twisted character relations \eqref{eq: theta twisted character relation}
holds for $\lp$.
\end{theorem}

\begin{proof}

 We first get \eqref{eq: character case} following the general strategy in the second part of the proof of Lemma~\ref{lemma: character case}. In this general case, we only know $\Sigma_{0}$-strong multiplicity one holds at the place $u$ for $\lif{\dot{\p}}$. But now we can assume the twisted character relation for all places except $u$. This is because of the property of our lift $\dot{\p}$ shown in Lemma~\ref{lemma: global lifting} and the fact
that we have shown the twisted character relation for the exceptional cases considered in Lemma~\ref{lemma: character case}. Under these assumptions, we can conclude from the linear independence of twisted characters of $\otimes_{v \neq u} \bar{\mathcal{H}}(\lif{\dot{G}}_{v}, \lif{\dot{\chi}}_{v})$-modules (see \cite{Lemaire:2016}, A.4.1) that 
\[
(\lif{\dot{f}}_{\lif{\dot{G}}^{\theta}_{u}}( \lif{\dot{\p}}_{u}, \dot{x}_{u} ) - \lif{\dot{f}}'_{\lif{\dot{G}}^{\theta}_{u}}( \lif{\dot{\p}}_{u}, \dot{x}_{u} )) \prod_{v} \lif{\dot{f}}_{\lif{\dot{G}}^{\theta}_{v}}( \lif{\dot{\p}}_{v}, \dot{x}_{v} ) = 0,
\]
and hence
\[
\lif{\dot{f}}_{\lif{\dot{G}}^{\theta}_{u}}( \lif{\dot{\p}}_{u}, \dot{x}_{u} ) = \lif{\dot{f}}'_{\lif{\dot{G}}^{\theta}_{u}}( \lif{\dot{\p}}_{u}, \dot{x}_{u} ).
\]
This proves the twisted local intertwining relation, which implies the twisted character relation according to Lemma~\ref{lemma: twisted intertwining relation 1}.
\end{proof}

Now we can deal with the discrete parameters $\p \in \cPdt{G}$. 


\begin{theorem}
\label{thm: refined L-packet for discrete parameter}
Suppose $F$ is a nonarchimedean local field and $ \p \in \cPdt{G}$, then the main local theorem (Theorem~\ref{thm: refined L-packet}) holds for $\lp$.
\end{theorem}

\begin{proof}
We can apply Lemma~\ref{lemma: global lifting} to $\p$, and because of Lemma~\ref{lemma: character case} and Theorem~\ref{thm: twisted character relation for elliptic parameter}, we can use the argument in Theorem~\ref{thm: standard argument on stability} to show part (1) and (2) of the main local theorem. At the same time we can deduce the stable multiplicity formula for the global lift $\lif{\dot{\p}}$ (cf. Remark~\ref{rk: standard argument on stability}), i.e.
\[
I^{\lif{\dot{G}}}_{disc, \dot{\p}} (\lif{\dot{f}}) = S^{\lif{\dot{G}}}_{disc, \dot{\p}} (\lif{\dot{f}}) = m_{\dot{\p}} \sum_{\dot{\x}' \in Y / \a(\S{\dot{\p}})} \lif{\dot{f}}^{\lif{\dot{G}}}(\lif{\dot{\p}} \otimes \dot{\x}').
\]
Hence the only thing left is to show the twisted character relations \eqref{eq: theta twisted character relation}. In order to deduce the $(\theta, \x)$-twisted character relation we use the stabilized $(\theta, \dot{\x})$-twisted trace formula. Note that 
\[
I^{(\lif{\dot{G}}^{\theta}, \dot{\x})}_{disc, \dot{\p}} (\lif{\dot{f}}) = tr R^{(\lif{\dot{G}}^{\theta}, \dot{\x})}_{disc, \dot{\p}} (\lif{\dot{f}}) =  \sum_{\dot{\x}'} \sum_{[\lif{\dot{\r}}] \in \cPkt{\lif{\dot{\p}} \otimes \dot{\x}' }} m(\lif{\dot{\r}}, \dot{\x}) \prod_{v} \lif{\dot{f}}_{\lif{\dot{G}}^{\theta}_{v}}( \lif{\dot{\r}}_{v}, \dot{\x}_{v}),
\]
where $\lif{\dot{f}}_{\lif{\dot{G}}^{\theta}_{v}}( \lif{\dot{\r}}_{v}, \dot{\x}_{v})$ is normalized according to \eqref{eq: theta twisted intertwining operator}, the sum of $\dot{\x}'$ is taken over 
\[
Y / \prod^{aut}_{v} \a(\S{\dot{\p}_{v}}^{\Sigma_{0}}),
\]
and $|m(\lif{\dot{\r}}, \dot{\x})|$ is some integer not larger than the multiplicity 
\[
m(\lif{\dot{\p}}) := m_{\dot{\p}} \, | \prod^{aut}_{v} \a(\S{\dot{\p}_{v}}^{\Sigma_{0}}) | \, |\a(\S{\dot{\p}}) |^{-1}
\]
of $\lif{\dot{\r}}$ as $\bar{\mathcal{H}}(\lif{\dot{G}}, \lif{\dot{\chi}})$-module. By Lemma~\ref{lemma: twisted endoscopic expansion}, we get
\[
I^{(\lif{\dot{G}}^{\theta}, \dot{\x})}_{disc, \dot{\p}} (\lif{\dot{f}}) = m_{\dot{\p}} \sum_{\dot{\x}' \in Y / \a(\S{\dot{\p}})} \lif{\dot{f}}'_{\lif{\dot{G}}^{\theta}}( \lif{\dot{\p}} \otimes \dot{\x}', \dot{x}),
\]
where $\a(\dot{x}) = \dot{\x}$. By Lemma~\ref{lemma: character case} and Theorem~\ref{thm: twisted character relation for elliptic parameter}, we can assume the twisted character relations for all places $v \neq u$. Then it follows from the linear independence of twisted characters of $\otimes_{v \neq u} \bar{\mathcal{H}}(\lif{\dot{G}}_{v}, \lif{\dot{\chi}}_{v})$-modules and $\Sigma_{0}$-strong multiplicity one for $\lif{\dot{\p}}$ at the place $u$ that
\begin{align}
\label{eq: refined L-packet for discrete parameter}
\sum_{[\lif{\dot{\r}}] \in \cPkt{\lif{\dot{\p}}}} m(\lif{\dot{\r}}, \dot{\x}) \lif{\dot{f}}_{\lif{\dot{G}}^{\theta}_{u}} ( \lif{\dot{\r}}_{u}, \dot{\x}_{u}) \prod_{v \neq u} \lif{\dot{f}}_{\lif{\dot{G}}^{\theta}_{v}}( \lif{\dot{\r}}_{v}, \dot{\x}_{v}) = 
m(\lif{\dot{\p}}) \lif{\dot{f}}'_{\lif{\dot{G}}^{\theta}_{u}} (\lif{\dot{\p}}_{u}, x) \prod_{v \neq u} ( \sum_{[\lif{\dot{\r}}_{v}] \in \cPkt{\lif{\dot{\p}}_{v}}} \lif{\dot{f}}_{\lif{\dot{G}}^{\theta}_{v}}( \lif{\dot{\r}}_{v}, \dot{\x}_{v}) ).
\end{align}
Now we choose $\lif{\dot{f}} = \otimes_{v}\lif{\dot{f}}_{v}$ with $\lif{\dot{f}}_{u}$ supported on $\lif{Z}_{F}G(F)$ and thus 
\[
\lif{\dot{f}}'_{\lif{\dot{G}}^{\theta}_{u}} (\lif{\dot{\p}}_{u}, \dot{x}_{u}) = \sum_{[\lif{\dot{\r}}_{u}] \in \cPkt{\lif{\dot{\p}}_{u}}} \lif{\dot{f}}_{\lif{\dot{G}}^{\theta}_{u}} ( \lif{\dot{\r}}_{u}, \dot{\x}_{u}).
\]
Substitute such test functions into the identity \eqref{eq: refined L-packet for discrete parameter}, one deduces that 
\[
m(\lif{\dot{\r}}, \dot{\x}) = m(\lif{\dot{\p}}).
\]
Therefore
\[
m(\lif{\dot{\p}}) ( \lif{\dot{f}}'_{\lif{\dot{G}}^{\theta}_{u}} (\lif{\dot{\p}}_{u}, x) - \sum_{[\lif{\dot{\r}}_{u}] \in \cPkt{\lif{\dot{\p}}_{u}}} \lif{\dot{f}}_{\lif{\dot{G}}^{\theta}_{u}} ( \lif{\dot{\r}}_{u}, \dot{\x}_{u})) \prod_{v \neq u} ( \sum_{[\lif{\dot{\r}}_{v}] \in \cPkt{\lif{\dot{\p}}_{v}}} \lif{\dot{f}}_{\lif{\dot{G}}^{\theta}_{v}}( \lif{\dot{\r}}_{v}, \dot{\x}_{v}) ) = 0
\]
for all $\lif{\dot{f}} \in \sH(\lif{\dot{G}}, \lif{\dot{\chi}})$. So we must have 
\[
\lif{\dot{f}}'_{\lif{\dot{G}}^{\theta}_{u}} (\lif{\dot{\p}}_{u}, x) = \sum_{[\lif{\dot{\r}}_{u}] \in \cPkt{\lif{\dot{\p}}_{u}}} \lif{\dot{f}}_{\lif{\dot{G}}^{\theta}_{u}} ( \lif{\dot{\r}}_{u}, \dot{\x}_{u}),
\]
and this finishes the proof of the $(\theta, \x)$-twisted character relation.
\end{proof}




At this point, we have proved our main local theorem (Theorem~\ref{thm: refined L-packet}) for $\lG = \lG(N)$, 
and the general case is just a corollary of that.

\begin{corollary}
\label{cor: refined L-packet for general group}
Suppose 
\[
G = G(n_{1}) \times G(n_{2}) \times \cdots \times G(n_{q}),
\]
with $n_{i} \leqslant N$ for $1 \leqslant i \leqslant q$ and $\p \in \cPbd{G}$, then the main local theorem (Theorem~\ref{thm: refined L-packet}) holds for $\lp$.
\end{corollary}

\begin{proof}
Let us write $\p = \p_{1} \times \p_{2} \times \cdots \times \p_{q}$ such that $\p_{i} \in \cPbd{G(n_{i})}$ for $1 \leqslant i \leqslant q$. Note that 
\[
\lG \subseteq \lG(n_{1}) \times \lG(n_{2}) \times \cdots \times \lG(n_{q})
\]
form a pair of groups satisfying the assumption in Section~\ref{subsubsec: notations}. Since $\cPkt{\lp_{i}}$ is well defined now, we can define $\cPkt{\lp}$ to be the restriction of $\bigotimes_{i = 1}^{q} \cPkt{\lp_{i}}$ to $\lG(F)$. It is clear that $\cPkt{\lp}$ satisfies part $(1)$ and $(2)$ of Theorem~\ref{thm: refined L-packet}. Moreover, the twisted endoscopic groups of $\lG$ lift to twisted endoscopic groups of $\lG(n_{1}) \times \lG(n_{2}) \times \cdots \times \lG(n_{q})$ by Proposition~\ref{prop: lifting endoscopic group}. Then it is a consequence of Corollary~\ref{cor: twisted endoscopic transfer} that the twisted character relations of $\lG$ follow from that of $\lG(n_{1}) \times \lG(n_{2}) \times \cdots \times \lG(n_{q})$ again by restriction. This completes the proof in the general case.
\end{proof}


\subsection{Proof of global theorem}
\label{subsec: proof of global theorems}

In this section, we are going to prove the global theorem, i.e., Theorem~\ref{thm: main global}. We will keep the notations as in Section~\ref{subsec: beginning of proofs}. Suppose $F$ is global, 
\[
G = G(n_{1}) \times G(n_{2}) \times \cdots \times G(n_{q}),
\]
with $\sum_{i =1}^{q} n_{i} = N$.
Now we can assume the main local theorem for $\lG$ thanks to Section~\ref{subsec: proof of main local theorem}. We will first prove the corresponding statement of Conjecture~\ref{conj: global L-packet} for $\lG$.

\begin{theorem}
\label{thm: global L-packet for discrete parameter}

We assume $\p \in \cP{G}$ satisfies the assumption in Theorem~\ref{thm: main global}.

\begin{enumerate}
\item One can associate a global packet $\cPkt{\lp}$ of irreducible admissible representations of $\lG(\A_{F})$ as $\sH(\lG)$-modules, satisfying the following properties:
           \begin{enumerate}
          \item$ \cPkt{\lp} = \bigotimes'_{v} \cPkt{\lp_{v}}$, where $\cPkt{\lp_{v}}$ is some lift of $\cPkt{\p_{v}}$ defined in Theorem~\ref{thm: refined L-packet}.
          \item there exists $[\lr] \in \cPkt{\lp}$, so that $\lr$ is isomorphic to an automorphic representation as $\sH(\lG)$-modules.
          \end{enumerate}
Moreover, $\cPkt{\lp}$ is unique up to twisting by characters of $\lG(\A_{F}) / \lG(F)G(\A_{F})$. And we can define a global character of $\S{\lp}$ by 
\[
<x, \lr> :=  \prod_{v} <x_{v}, \lr_{v}> \,\,\,\,\, \text{ for } \, [\lr] \in \cPkt{\lp} \text{ and } \, x \in \S{\lp}. 
\]
\item If $\p \in \cPdt{G}$, the $\p$-component of the $\lif{\zeta}$-equivariant discrete spectrum of $\lG(\A_{F})$ as $\sH(\lG)$-modules can be decomposed as follows
\[
L^{2}_{disc, \p}(\lG(F) \backslash \lG(\A_{F}), \lif{\zeta}) = m_{\p} \sum_{\x \in Y / \a(\S{\p})} \sum_{\substack{ [\lr] \in \cPkt{\lp} \otimes \x \\ <\cdot, \lr> = 1}} \lr,
\]
where $m_{\p}$ is defined as in Remark~\ref{rk: discrete spectrum}.
\end{enumerate}
\end{theorem}

\begin{proof}
If $\p$ factors through $\p_{M} \in \cPdt{M}$ for some proper Levi subgroup $M$ of $G$, then by our induction assumption, we have a global $L$-packet $\cPkt{\lp_{M}}$ for $\lM$, and we can define $\cPkt{\lp}$ to be the irreducible constituents induced from $\cPkt{\lp_{M}}$. So it is enough to consider the case $\p \in \cPdt{G}$. Note that one can always define a global packet $\cPkt{\lp}$ as follows. If $\lr \in \mathcal{A}_{2}(\lG)$ and its restriction to $G(\A_{F})$ have irreducible constituents contained in $\cPkt{\p}$, then we can take the local lift $\cPkt{\lp_{v}}$ of $\cPkt{\p_{v}}$ to be the one containing $[\lr_{v}]$. We form a global packet
\[
\cPkt{\lp} := \prod_{v} \cPkt{\lp_{v}},
\]
and the uniqueness property should follow from Corollary~\ref{cor: modular character} and the decomposition in Part (2). 

To show Part (2), 
we can apply Lemma~\ref{lemma: twisted endoscopic expansion} to get
\begin{align}
\label{eq: global L-packet for discrete parameter 1}
\Idt{\lG}{, \p}(\lf) = \Sdt{\lG}{, \p}(\lf) + C_{\lp} \sum_{\x \in Y / \a(\S{\p})} \sum_{x \in \S{\lp} - \{1\}} \lf'_{\lG}(\lp \otimes \x, x). 
\end{align}
By the local character relation, one can define a global packet $\cPkt{\lp_{x}}$ transferred from $\cPkt{\lp'}$ for any $x \in \S{\lp} - \{1\}$. Next we would like to add
\begin{align}
\label{eq: global L-packet for discrete parameter 2}
2 \cdot C_{\lp} \sum_{\x \in Y / \a(\S{\p})} \sum_{x \in \S{\lp} - \{1\}}  \sum_{ \substack{ [\lr] \in \cPkt{\lp_{x}} \\ <x, \lr> = -1}} \lf_{\lG}(\lr \otimes \x)
\end{align}
to both sides of \eqref{eq: global L-packet for discrete parameter 1}. Note that this sum does not include $x \in \S{\lp} - \{1\}$ such that $<x, \lr> = 1$ for all $[\lr] \in \cPkt{\lp_{x}}$. For those $x$ which does not contribute to \eqref{eq: global L-packet for discrete parameter 2}, we have  $\lf'_{\lG}(\lp \otimes \x, x) =\lf^{\lG}(\lp_{x} \otimes \x)$ which is defined by $\cPkt{\lp_{x}}$ and is stable. Then the right hand side becomes 
\[
\Sdt{\lG}{, \p}(\lf) + C_{\lp} \sum_{\x \in Y / \a(\S{\p})} \sum_{x \in \S{\lp} - \{1\}}  \lf^{\lG}(\lp_{x} \otimes \x), 
\]
which is again stable. So the left hand side
\begin{align}
\label{eq: global L-packet for discrete parameter}
\Idt{\lG}{, \p}(\lf) +  2 \cdot C_{\lp} \sum_{\x \in Y / \a(\S{\p})} \sum_{x \in \S{\lp} - \{1\}}  \sum_{ \substack{ [\lr] \in \cPkt{\lp_{x}} \\ <x, \lr> = -1}} \lf_{\lG}(\lr \otimes \x)  
\end{align}
is also stable. By \eqref{eq: discrete part vs discrete spectrum},
\[
\Idt{\lG}{, \p}(\lf) = tr R_{disc, \p}^{\lG}(\lf).
\] 
Let $\cPkt{\lp}$ be the global packet defined in the beginning with respect to some fixed $\lr^{0} \in \mathcal{A}_{2}(\lG)$. Since \eqref{eq: global L-packet for discrete parameter} is stable, it is stable at every place. So we can take $\lf = \otimes_{w} \lf_{w}$ and fix $\otimes_{w \neq v}\lf_{w}$ for any place $v$, then by Corollary~\ref{cor: refined L-packet} the coefficient of $\lf_{v}(\lr_{v})$ in \eqref{eq: global L-packet for discrete parameter} must be the same for all $\lr_{v} \in \cPkt{\lp_{v}}$. By varying $\otimes_{w \neq v}\lf_{w}$ and the linear independence of characters of $\otimes_{w \neq v} \sH(\lG_{w}, \lif{\chi}_{w})$-modules, we have that
\[
[\lr^{0}] = [\lr^{0}_{v}] \otimes (\otimes_{w \neq v} [\lr^{0}_{w}])
\]
contributes to \eqref{eq: global L-packet for discrete parameter} if and only if all elements in
\[
\cPkt{\lp_{v}} \otimes (\otimes_{w \neq v} [\lr^{0}_{w}])
\]
also contribute to \eqref{eq: global L-packet for discrete parameter}. By repeating this kind of argument, one can show all elements in $\cPkt{\lp}$ contribute to \eqref{eq: global L-packet for discrete parameter}. Note for any $[\lr] \in \cPkt{\lp}$ such that $<\cdot, \lr> = 1$, it can only contribute to $\Idt{\lG}{, \p}(\lf)$, which means it belongs to $\mathcal{A}_{2}(\lG)$. Then the decomposition in Part (2) will follow from Proposition~\ref{prop: discrete spectrum} immediately.

\end{proof}

\begin{remark}
\label{rk: global L-packet for discrete parameter}

Following the proof, we can also apply the same argument to elements in $\cPkt{\lp_{x}}$ which contributes to \eqref{eq: global L-packet for discrete parameter}. It follows all elements in $\cPkt{\lp_{x}}$ contributes to \eqref{eq: global L-packet for discrete parameter}. For $[\lr] \in \cPkt{\lp_{x}}$ such that $<\cdot, \lr> = 1$, it can only come from $\Idt{\lG}{, \p}(\lf)$. So $\cPkt{\lp_{x}} = \cPkt{\lp}$ up to twisting by some character in $Y$. Note this is only true for $x \in \S{\lp} - \{1\}$ in the sum \eqref{eq: global L-packet for discrete parameter 2}. As a result, \eqref{eq: global L-packet for discrete parameter} becomes
\[
m_{\p} \sum_{\x \in Y / \a(\S{\p})} \lf^{\lG}(\lp \otimes \x),
\]
where
\[
 \lf^{\lG} (\lp \otimes \x) = \prod_{v} \lf_{v}(\lp_{v} \otimes \x_{v}).
 \]
Moreover, we have    
\[
\Sdt{\lG}{, \p}(\lf) = m_{\p} \sum_{\x \in Y / \a(\S{\p})} \lf^{\lG}(\lp \otimes \x) - C_{\lp} \sum_{\x \in Y / \a(\S{\p})} \sum_{x \in \S{\lp} - \{1\}}  \lf^{\lG}(\lp_{x} \otimes \x).
\]
Suppose $\cPkt{\lp_{x}} = \cPkt{\lp}$ up to twisting by some character in $Y$ for all $x \in \S{\lp} - \{1\}$, then
\[
\Sdt{\lG}{, \p}(\lf) = m_{\p} \sum_{\x \in Y / \a(\S{\p})} |\S{\lp}|^{-1} \lf^{\lG}(\lp \otimes \x).
\]
This is the stable multiplicity formula in Conjecture~\ref{conj: stable multiplicity formula}. We will come back to this identity in Theorem~\ref{thm: stable multiplicity formula}.

\end{remark}

Next, we will prove the corresponding statement of Conjecture~\ref{conj: compatible normalization} for $\lG$.

\begin{theorem}
\label{thm: compatible normalization 1}
Suppose $\p \in \cPdt{G}$ satisfying the assumption in Theorem~\ref{thm: main global} and $x \in \S{\p}^{\theta}$ with $\a(x) = \x$ for $\theta \in \Sigma_{0}$ and some character $\x$ of $\lG(\A_{F})/\lG(F)G(\A_{F})$. For $[\lr] \in \cPkt{\lp}$ with $<\cdot, \lr> =1$, the canonical intertwining operator 
\[
R(\theta)^{-1} \circ R(\x)
\]
restricted to the $\lr$-isotypic component $I(\lr)$ is equal to the product of $m(\lr)$ and the local intertwining operators $A_{\lr_{v}}(\theta, \x_{v})$ determined by $x_{v}$ (see \eqref{eq: theta twisted intertwining operator}), i.e.
\begin{align}
\label{eq: compatible normalization 1}
\tIdt{\lG^{\theta}}{, \p}(\lf) = m_{\p} \sum_{\x' \in Y / \a(\S{\p})} \sum_{\substack{ [\lr] \in \cPkt{\lp} \otimes \x' \\ <\cdot, \lr> = 1}} \lf_{\lG^{\theta}}(\lr, \x),  \,\,\,\,\, \lf \in \sH(\lG, \lif{\chi}),
\end{align}
where $ \lf_{\lG^{\theta}}(\lr, \x) = \prod_{v} \lf_{\lG^{\theta}_{v}}(\lr_{v}, \x_{v})$, and it does not depend on $x$.
\end{theorem}

\begin{proof}
It follows from Theorem~\ref{thm: global L-packet for discrete parameter} that 
\[
\tIdt{\lG^{\theta}}{, \p}(\lf) = tr R^{(\lG^{\theta}. \x)}_{disc, \p}(\lf) = \sum_{\x'} \sum_{\substack{ [\lr] \in \cPkt{\lp} \otimes \x' \\ <\cdot, \lr> = 1}} m(\lr, \theta, \x) \lf_{\lG^{\theta}}(\lr, \x)
\]
where the sum of $\x'$ is taken over
\[
Y / \prod^{aut}_{v} \a(\S{\p_{v}}^{\Sigma_{0}}),
\]
and $|m(\lr, \theta, \x)|$ is some integer less than or equal to
\[
m(\lp) :=  m_{\p} \, |\prod^{aut}_{v} \a(\S{\p_{v}}^{\Sigma_{0}}) | \, | \a(\S{\p})|^{-1}.
\]
By Lemma~\ref{lemma: twisted endoscopic expansion}, we have
\[
\tIdt{\lG^{\theta}}{, \p}(\lf) = |\S{\lp}|^{-1} \sum_{\x'} \sum_{x' \in \S{\p}^{\theta}(\x)} m(\lp) \lf'_{\lG^{\theta}}(\lp \otimes \x', x'),
\]
where the sum of $\x'$ is again over
\[
Y / \prod^{aut}_{v} \a(\S{\p_{v}}^{\Sigma_{0}}).
\]
Therefore
\begin{align*}
\sum_{\x'} \sum_{\substack{ [\lr] \in \cPkt{\lp} \otimes \x' \\ <\cdot, \lr> = 1}} m(\lr, \theta, \x) \lf_{\lG^{\theta}}(\lr, \x) =   |\S{\lp}|^{-1} \sum_{\x'} \sum_{x' \in \S{\p}^{\theta}(\x)} m(\lp) \lf'_{\lG^{\theta}}(\lp \otimes \x', x'). 
\end{align*}
By the twisted character relation, one can define a global packet $\cPkt{\lp_{x'}}$ transferred from $\cPkt{\lp'}$ for any $x' \in \S{\p}^{\theta}(\x)$. Note $\S{\p}^{\theta}(\x) = x \cdot \S{\lp}$, then
\begin{align}
\label{eq: compatible normalization 2}
\sum_{\x'} \sum_{\substack{ [\lr] \in \cPkt{\lp} \otimes \x' \\ <\cdot, \lr> = 1}} m(\lr, \theta, \x) \lf_{\lG^{\theta}}(\lr, \x) =   |\S{\lp}|^{-1} \sum_{\x'} \sum_{y \in \S{\lp}} m(\lp)  \sum_{[\lr] \in \cPkt{\lp_{xy}} \otimes \x'} <y, \lr> \lf_{\lG^{\theta}}(\lr, \x),
\end{align}
where $\lf_{\lG^{\theta}}(\lr, \x)$ is normalized by $x$ (cf. \eqref{eq: theta twisted intertwining operator}). This implies 
\begin{align*}
\sum_{\x'} \sum_{\substack{ [\lr] \in \cPkt{\lp} \otimes \x' \\ <\cdot, \lr> = 1}} m(\lr, \theta, \x) \lf_{\lG^{\theta}}(\lr, \x) =   |\S{\lp}|^{-1}  \sum_{\x'} \sum_{y \in \S{\lp}} m(\lp)  \sum_{\substack{ [\lr] \in \cPkt{\lp_{xy}} \otimes \x' \\ <\cdot, \lr> = 1}} \lf_{\lG^{\theta}}(\lr, \x). 
\end{align*}
It follows from the linear independence of twisted characters of $\bar{\mathcal{H}}(\lG, \lif{\chi})$-modules that we can choose $\cPkt{\lp_{x'}} = \cPkt{\lp}$ for all $x' \in \S{\p}^{\theta}(\x)$. Then
\begin{align*}
\sum_{\substack{ [\lr] \in \cPkt{\lp} \\ <\cdot, \lr> = 1}} m(\lr, \theta, \x) \lf_{\lG^{\theta}}(\lr, \x) =   m(\lp)  \sum_{\substack{ [\lr] \in \cPkt{\lp} \\ <\cdot, \lr> = 1}} \lf_{\lG^{\theta}}(\lr, \x). 
\end{align*}
So $m(\lr, \theta, \x) = m(\lp)$. Hence
\begin{align*}
\tIdt{\lG^{\theta}}{, \p}(\lf) = m_{\p} \sum_{\x' \in Y / \a(\S{\p})} \sum_{\substack{ [\lr] \in \cPkt{\lp} \otimes \x' \\ <\cdot, \lr> = 1}} \lf_{\lG^{\theta}}(\lr, \x).
\end{align*}
\end{proof}

Now we are only left with proving the corresponding statement of Conjecture~\ref{conj: stable multiplicity formula}. From Remark~\ref{rk: global L-packet for discrete parameter}, we see the key is to prove the functoriality of endoscopic transfer. Here we would like to consider a more general notion of that, i.e., {\bf functoriality of twisted endoscopic transfer}, and we formulate it in our context as follows. For $\p \in \cP{G}$ and semisimple $s \in \cS{\p}$, let $(G', \p') \rightarrow (\p, s)$ and $\lG' \in \tEnd{}{\lG}$ be the lift of $G'$, the functoriality of twisted endoscopic transfer means the global $L$-packet $\cPkt{\lp'}$ transfers to a global $L$-packet $\cPkt{\lp}$ through the local twisted character relation \eqref{eq: theta twisted character relation}. By the same argument in the proof of Lemma~\ref{lemma: induced twisted character}, we see the transfer of $\cPkt{\lp'}$ only depends on the image $x$ of $s$ in $\S{\p}$. So we can denote the transfer image by $\cPkt{\lp_{x}}$.
\begin{lemma}
\label{lemma: functoriality for simple group}
Suppose $\lG = \lG(n)$ for $n \leqslant N$, $\p \in \cP{G}$ such that $\S{\lp} = 1$, then $\cPkt{\lp_{x}} = \cPkt{\lp}$ up to twisting by some character in $Y$ for any $x \in \S{\p}$.
\end{lemma}

\begin{proof}
For semisimple $s \in \cS{\p}$, let $(G', \p') \rightarrow (\p, s)$. Suppose $|\cS{\p, s}^{0}| = \infty$, let $T_{\p, s}$ be a maximal torus of $(S_{\p, s})^{0}$, then $\D{M}' = \Cent(T_{\p, s}, \D{G}')$ defines a proper Levi subgroup of $\D{G}'$ such that $\p'$ factors through $\p'_{M} \in \cPdt{M'}$. Moreover, $M' \in \End{ell}{M}$ for a proper Levi subgroup $M$ of $G$, which is determined by $\D{M} = \Cent(T_{\p, s}, \D{G})$. So $\p$ factors through $\p_{M} \in \cP{M}$, and we can reduce this case to $\lM$.

Next we assume $|\cS{\p, s}^{0}| < \infty$, then $\p \in \cPel{G}$. In particular, $s \in \cS{\p, ell}$ and we let $x$ be its image in $\S{\p}$. We can also assume $x \neq 1$, then $\S{\lp} = 1$ implies $\a(x) = \x \neq 1$. By Lemma~\ref{lemma: twisted endoscopic expansion}, we have 
\[
\tIdt{\lG}{, \p} (\lf) = C_{\lp} \sum_{\x' \in Y / \a(\S{\p})} e'_{\p}(x) \lf_{\lG}' (\lp \otimes \x', x).    
\]
By Lemma~\ref{lemma: twisted spectral expansion} and Theorem~\ref{thm: compatible normalization 1}, we have 
\[
\tIdt{\lG}{, \p} (\lf) = C_{\lp} \sum_{\x' \in Y / \a(\S{\p})} i_{\p}(x) \lf_{\lG}(\lp \otimes \x', x).         
\]
Note $e'_{\p}(x) = i_{\p}(x) \neq 0$. Then by the linear independence of twisted characters, we have $\cPkt{\lp_{x}} = \cPkt{\lp}$ up to twisting by some character in $Y$.
\end{proof}

It is not hard to extend this result to the general case.

\begin{lemma}
\label{lemma: functoriality}
Suppose $G = G(n_{1}) \times G(n_{2}) \times \cdots \times G(n_{q})$, and $\p = \p_{1} \times \p_{2} \times \cdots \times \p_{q} \in \cP{G}$ with $\p_{i} \in \cP{G(n_{i})}$ for $1 \leqslant i \leqslant q$. If $\S{\lp_{i}} = 1$ for all $i$, then $\cPkt{\lp_{x}} = \cPkt{\lp}$ up to twisting by some character in $Y$ for any $x \in \S{\p}$. 
\end{lemma}

\begin{proof}
If we write the image of $x$ in $\S{\p_{i}}$ by $x_{i}$, then by Lemma~\ref{lemma: functoriality for simple group}, $\bigotimes_{i = 1}^{q} \cPkt{\lp_{x_{i}}}$ is a global L-packet of $\lG(n_{1}) \times \lG(n_{2}) \times \cdots \times \lG(n_{q})$, and $\cPkt{\lp_{x}}$ is its restriction to $\lG$. Since the restriction of a global L-packet is again a global L-packet, then $\cPkt{\lp_{x}} = \cPkt{\lp}$ up to twisting by some character in $Y$.
\end{proof}

\begin{remark}
We would like to point out in the case of this lemma, $\S{\lp}$ can be nontrivial even though $\S{\lp_{i}} = 1$ for all $i$. For example, let $G = Sp(2n) \times Sp(2n)$ and $\p = (\p_{1} \# \p_{2}) \times (\p_{1} \# \p_{2}) \in \cPdt{G}$. We assume the central characters satisfy $\eta_{\p_{1}} = \eta_{\p_{2}} \neq 1$, then $\S{\lp} \cong \Two$.
\end{remark}

Now we can prove the corresponding statement of Conjecture~\ref{conj: stable multiplicity formula}.

\begin{theorem}
\label{thm: stable multiplicity formula}
Suppose $G = G(n_{1}) \times G(n_{2}) \times \cdots \times G(n_{q})$, and $\p = \p_{1} \times \p_{2} \times \cdots \times \p_{q} \in \cPdt{G}$ with $\p_{i} \in \cP{G(n_{i})}$ for $1 \leqslant i \leqslant q$. If $\S{\lp_{i}} = 1$ for all $i$, then 
\[
\Sdt{\lG}{, \p}(\lf) = m_{\p} \sum_{\x \in Y / \a(\S{\p})} |\S{\lp}|^{-1} \sigma( \com[0]{\cS{\p}}) \lf^{\lG} (\lp \otimes \x), \,\,\,\,\,\,\, \lf  \in \sH(\lG, \lif{\chi}).
\]
\end{theorem}

\begin{proof}
It follows from Remark~\ref{rk: global L-packet for discrete parameter} and Lemma~\ref{lemma: functoriality} that 
\[
\Sdt{\lG}{, \p}(\lf) = m_{\p} \sum_{\x \in Y / \a(\S{\p})} |\S{\lp}|^{-1} \lf^{\lG} (\lp \otimes \x).
\]
Note in this case $\sigma( \com[0]{\cS{\p}}) = 1$. 
This finishes the proof.

\end{proof}

Up to now, we have proved the local and global theorems for $\lG$ under our induction assumptions, when $G$ does not contain any factor of $SO(2N+2, \eta)$ (cf. Remark~\ref{rk: induction assumption}). By adding these results to our induction assumptions, we can prove the general case by repeating the previous arguments in Section~\ref{subsec: proof of main local theorem} and Section~\ref{subsec: proof of global theorems} without any change.


\appendix

\section{An irreducibility result}
\label{sec: irreducibility}

The aim of this appendix is to prove Proposition~\ref{prop: local constituents of cuspidal representations} and Proposition~\ref{prop: irreducibility of non-unitary induced representation}. We should point out that neither these results nor their proofs are new, and just for the convenience of the readers we would like to present their proofs here. In this section $G$ will always denote a symplectic group or special even orthogonal group. First, let us restate Proposition~\ref{prop: local constituents of cuspidal representations}, and for its proof we will follow the same argument in \cite{MullerSpeh:2004}.

\begin{proposition}
\label{prop: local constituents of cuspidal representations A}
If $F$ is global and $\p \in \Psm{N}$, then $\p_{v} \in \uP{N_{v}}$.
\end{proposition}  

\begin{proof}
Suppose $\r_{\p}$ is the unitary cuspidal automorphic representation of $GL(N)$ associated to $\p$. By \cite{Shalika:1974}, $\r_{\p, v}$ is generic for all places $v$. And one knows from \cite{JacquetShalika:1983} that any irreducible generic representation of $GL(N)$ over a local field is a fully induced representation
\[
\mathcal{I}_{P}(\nu^{a_{1}}\sigma_{1} \otimes \nu^{a_{2}}\sigma_{2} \otimes \cdots \otimes \nu^{a_{r}}\sigma_{r})
\]
where $P$ is some standard parabolic subgroup of $GL(N)$, $\sigma_{i}$ are unitary essentially discrete series representations and $a_{i} \in \R$ for $1 \leqslant i \leqslant r$. So for our $\p$, we have 
\[
\p_{v} = \nu^{a_{1}}\p_{v, 1} \+ \nu^{a_{2}}\p_{v, 2} \+ \cdots \+ \nu^{a_{s}}\p_{v, s}
\]
where $\p_{v, j} \in \Psm{N_{v, j}}$ for $1 \leqslant j \leqslant s$. Since $\r_{\p, v}$ is also unitary, by the classification of unitary dual of $GL(N)$ (archimedean case in \cite{Tadic:2009} and nonarchimedean case in \cite{Tadic:1986}), $\p_{v}$ must belongs to $\uP{N}$. In particular, $|a_{j}| < 1/2$ for $1 \leqslant j \leqslant s$.
\end{proof}

Next we restate Proposition~\ref{prop: irreducibility of non-unitary induced representation} as follows.

\begin{proposition}
\label{prop: irreducibility of non-unitary induced representation A}
Suppose $F$ is local, $\p \in \cuP{G}$, and $\p$ can be regarded as $\p_{M, \lambda}$ where $\p_{M} \in \cPbd{M}$ and $\lambda \in \mathfrak{a}^{*}_{M}$ lies in some open chamber of $P \supseteq M$. Then for any $[\r_{M}] \in \cPkt{\p_{M}}$, the induced representation $\mathcal{I}_{P} (\r_{M, \lambda}) $ is irreducible.
\end{proposition}

Before starting the proof, we want to introduce some notations for the parabolic induction. Suppose $\r_{1}$ and $\r_{2}$ are representations of $GL(N_{1})$ and $GL(N_{2})$ respectively, we will write $\r_{1} \times \r_{2}$ for the parabolic induction of $\r_{1} \otimes \r_{2}$ by viewing $GL(N_{1}) \times GL(N_{2})$ as the Levi component of a maximal parabolic subgroup in $GL(N_{1} + N_{2})$. And similarly, suppose $\r$ and $\sigma$ are representations of $GL(N)$ and $G$, we will write $\r \rtimes \sigma$ for the parabolic induction of $\r \otimes \sigma$ by viewing $GL(N) \times G$ as the Levi component of a maximal parabolic subgroup in $G_{+}$ which is of same type as $G$.

The proof of this proposition breaks down to several steps. First notice $\p \in \cuP{G}$, so we can write 
\[
\p = \p_{G_{-}} \+ (\nu^{a_{1}}\p_{1} \+ \nu^{-a_{1}}\p_{1}^{\vee}) \+ \cdots \+ (\nu^{a_{m}}\p_{m} \+ \nu^{-a_{m}}\p_{m}^{\vee}),
\]
where $G_{-}$ is of the same type as $G$, $\p_{G_{-}} \in \cPbd{G_{-}}$, and $\p_{i} \in \Psm{N_{i}}$ for $1 \leqslant i \leqslant m$ with $0 < a_{m} \leqslant \cdots \leqslant a_{1} < 1/2$. Then 
\[
\p_{M} = \p_{G_{-}} \+ \p_{1} \+ \cdots \+ \p_{m} \in \cPbd{M},
\]
and $\lambda = (a_{1}, a_{2}, \cdots, a_{m})$. Let $\r_{\p_{i}}$ be the corresponding essentially discrete series representation of $GL(N_{i})$ associated to $\p_{i}$ for $1 \leqslant i \leqslant m$, then any $\r_{M} \in \cPkt{\p_{M}}$ can be written as $\r_{M} = \r_{G_{-}} \otimes \r_{\p_{1}} \otimes \cdots \otimes \r_{\p_{m}}$, for some $\r_{G_{-}} \in \cPkt{\p_{G_{-}}}$. And 
\(
\mathcal{I}_{P} (\r_{M, \lambda}) = \nu^{a_{1}}\r_{\p_{1}} \times \cdots \nu^{a_{m}}\r_{\p_{m}} \rtimes \r_{G_{-}}.
\)
Next, we have two reduction steps. These steps are due to Speh and Vogan \cite{SpehVogan:1980}, and they are also presented very clearly in \cite{Muic:2005}, so we will state them without proofs. The first reduction is given by the following lemma.

\begin{lemma}
\label{lemma: first reduction}
$\mathcal{I}_{P} (\r_{M, \lambda})$ is irreducible if and only if $\nu^{a_{i}}\r_{\p_{i}} \times \nu^{a_{j}} \r_{\p_{j}}$, $\nu^{a_{i}}\r_{\p_{i}} \times \nu^{- a_{j}} \r_{\p_{j}}^{\vee}$ are irreducible for all $i \neq j$, and also $\nu^{a_{i}}\r_{\p_{i}} \rtimes \r_{G_{-}}$ is irreducible for all $i$.
\end{lemma}

To make the second reduction, we need to write $\p_{G_{-}} = \p_{G'_{-}} \+ (\p'_{1} \+ \p_{1}'^{\vee}) \+ \cdots \+ (\p'_{n} \+ \p_{n}'^{\vee})$, where $G'_{-}$ is of the same type as $G_{-}$ and $\p_{G'_{-}} \in \cPdt{G'_{-}}$, $\p'_{i} \in \Psm{N'_{i}}$. Then it is clear that $\r_{G_{-}}$ is a subrepresentation of $\r_{\p'_{1}} \times \cdots \times \r_{\p'_{n}} \rtimes \r_{G'_{-}}$ for some $\r_{G'_{-}} \in \cPkt{\p_{G'_{-}}}$. And we can state the second reduction as follows.

\begin{lemma}
\label{lemma: second reduction}
For $1 \leqslant i \leqslant m$, $\nu^{a_{i}}\r_{\p_{i}} \rtimes \r_{G_{-}}$ is irreducible if $\nu^{a_{i}} \r_{\p_{i}} \times \r_{\p'_{k}}$, $\nu^{a_{i}} \r_{\p_{i}} \times \r_{\p'_{k}}^{\vee}$ are irreducible for $1 \leqslant k \leqslant n$, and also $\nu^{a_{i}} \r_{\p_{i}} \rtimes \r_{G'_{-}}$ is irreducible.
\end{lemma}

Since $0< a_{i} < 1/2$ for $1 \leqslant i \leqslant m$, by the theory of Zelevinsky \cite{Zelevinsky:1980} and its archimedean analogue \cite{Tadic:2009}, we can see easily that $\nu^{a_{i}}\r_{\p_{i}} \times \nu^{a_{j}} \r_{\p_{j}}$, $\nu^{a_{i}}\r_{\p_{i}} \times \nu^{- a_{j}} \r_{\p_{j}}^{\vee}$ are irreducible for all $i \neq j$, and $\nu^{a_{i}} \r_{\p_{i}} \times \r_{\p'_{k}}$, $\nu^{a_{i}} \r_{\p_{i}} \times \r_{\p'_{k}}^{\vee}$ are irreducible for all $k$. So it reduces to show $\nu^{a_{i}} \r_{\p_{i}} \rtimes \r_{G'_{-}}$ is irreducible for $1 \leqslant i \leqslant m$. And this is the consequence of the following proposition.

\begin{proposition}
\label{prop: reducibility of generalized principal series; discrete series case}
Suppose $\r$ and $\sigma$ are discrete series representations of $GL(N)$ and $G$ respectively, then the induced representation $\nu^{a}\r \rtimes \sigma$ is irreducible if $0 < a < 1/2$.
\end{proposition}

The proof of this proposition is divided into two cases: archimedean and nonarchimedean. The archimedean case follows from the result of Speh and Vogan directly, and we refer the readers to \cite{SpehVogan:1980} for precise statement of their theorems. For the nonarchimedean case, the story is not that straightforward, and we will concentrate on this case. So now we assume $F$ is nonarchimedean, and we can write the essentially discrete series $\nu^{a}\r$ of $GL(N)$ as segments according to the classification theory of Zelevinsky \cite{Zelevinsky:1980}, i.e. $\nu^{a}\r = \delta(\nu^{-l_{1}}\rho, \nu^{l_{2}}\rho)$ where $\rho$ is a cuspidal representation of $GL(d_{\rho})$, $d_{\rho}$ is defined by $\rho$, and $l_{1} + l_{2} \in \mathbb{Z}_{\geqslant 0}$ with $a = (l_{2} - l_{1}) / 2$. Note that $2l_{2} + 1 \notin \mathbb{Z}$ for $0 < a < 1/2$. Then it is easy to see that the nonarchimedean case follows from the following theorem due to Tadi{\'c} (\cite{Muic:2004}, Theorem 2.2).

\begin{theorem}
\label{thm: a theorem of Tadic}
Suppose $F$ is nonarchimedean, $\delta = \delta(\nu^{-l_{1}}\rho, \nu^{l_{2}}\rho)$ and $\sigma$ is a discrete series of $G$. If $\rho \neq \rho^{\vee}$ or $2l_{2} + 1 \nsubseteq \mathbb{Z}$, then $\delta \rtimes \sigma$ is irreducible.
\end{theorem}

The proof of this theorem again involves several reduction steps. It is clear that there is no harm to assume $l_{2} > l_{1}$. The simplest case is when $\delta = \nu^{l}\rho$ and $\sigma$ is a cuspidal representation, and this has been settled by M{\oe}glin in \cite{Moeglin:2000}. In fact, M{\oe}glin has proved the following result.

\begin{proposition}
\label{prop: reducibility of generalized principal series; cuspidal case}
Suppose $F$ is nonarchimedean,  $\rho$ and $\sigma$ are cuspidal representations of $GL(N)$ and $G$ respectively. If $\rho = \rho^{\vee}$ and $\nu^{\alpha}\rho \rtimes \sigma$ is reducible, then $\alpha \in \frac{1}{2}\mathbb{Z}$, and $\nu^{\beta}\rho \rtimes \sigma$ is irreducible for all $\beta \in \R \setminus \{\pm \alpha\}$; If $\rho \neq \rho^{\vee}$, then $\nu^{l} \rtimes \sigma$ is irreducible for all $l \in \R$.
\end{proposition}

We should point out the original proof of Proposition~\ref{prop: reducibility of generalized principal series; cuspidal case} in \cite{Moeglin:2000} is based on the assumption of functorial lifting from classical groups to general linear groups. Now it is unconditional by Arthur's result \cite{Arthur:2013}. For the second reduction step of the proof of Theorem~\ref{thm: a theorem of Tadic}, Tadi\'c shows the theorem is true for $\sigma$ being cuspidal in \cite{Tadic:1998}, and he uses an induction argument which is quite different from the proof of Proposition~\ref{prop: reducibility of generalized principal series; cuspidal case}. Finally, we can consider the general case, i.e. $\sigma$ is a discrete series of $G$. Suppose $\sigma$ is not cuspidal, then by M{\oe}glin-Tadi\'cs' classification of discrete series of classical groups \cite{MoeglinTadic:2002}, one has an inclusion
\[
\xymatrix{\sigma \, \ar@{^{(}->}[r] \, & \nu^{\alpha_{1}}\r_{1} \times \nu^{\alpha_{2}}\r_{2} \times \cdots \times \nu^{\alpha_{m}}\r_{m} \rtimes \sigma'}
\] 
where $\r_{i}$ are self-dual essentially discrete series of $GL(N_{i})$, $\alpha_{i} \in \frac{1}{2}\mathbb{Z}_{>0}$, and $\sigma'$ is a discrete series of $G'$, where $G'$ is of the same type as $G$. So we can do induction on the rank of $G$, and assume $\delta \rtimes \sigma'$ is irreducible, i.e. the standard intertwining operator $\delta \rtimes \sigma' \xrightarrow{\simeq} \delta^{\vee} \rtimes \sigma'$ is a bijection by the theory of Langlands quotient. Note that $\delta \times \nu^{\alpha_{i}}\r_{i}$ and $\delta^{\vee} \times \nu^{\alpha_{i}}\r_{i}$ are irreducible, so the normalized intertwining operators
\[
\xymatrix{\delta \times \nu^{\alpha_{i}}\r_{i} \ar[r]^{\simeq}   &    \nu^{\alpha_{i}}\r_{i} \times \delta \\
\nu^{\alpha_{i}}\r_{i} \times \delta^{\vee} \ar[r]^{\simeq}   &    \delta^{\vee} \times \nu^{\alpha_{i}}\r_{i} }
\]
are bijections. Now consider the following composition of normalized intertwining operators.
\begin{align*}
\delta \rtimes \sigma      & \xhookrightarrow{}     \delta \times \nu^{\alpha_{1}}\r_{1} \times \nu^{\alpha_{2}}\r_{2} \times \cdots \times \nu^{\alpha_{m}}\r_{m} \rtimes \sigma'  
\xrightarrow{\simeq}        \nu^{\alpha_{1}}\r_{1} \times \delta \times \nu^{\alpha_{2}}\r_{2} \times \cdots \times \nu^{\alpha_{m}}\r_{m} \rtimes \sigma' \\
 & \xrightarrow{\simeq}   \cdots    \xrightarrow{\simeq}    \nu^{\alpha_{1}}\r_{1} \times  \nu^{\alpha_{2}}\r_{2} \times \cdots \times \nu^{\alpha_{m}}\r_{m} \rtimes  \delta \times  \sigma'  
\xrightarrow{\simeq}     \nu^{\alpha_{1}}\r_{1} \times  \nu^{\alpha_{2}}\r_{2} \times \cdots \times \nu^{\alpha_{m}}\r_{m} \rtimes  \delta^{\vee} \times  \sigma'    \\
& \xrightarrow{\simeq}     \nu^{\alpha_{1}}\r_{1} \times  \nu^{\alpha_{2}}\r_{2} \times \cdots \times \delta^{\vee} \times \nu^{\alpha_{m}}\r_{m} \rtimes    \sigma'  
\xrightarrow{\simeq}  \cdots  \xrightarrow{\simeq}  \delta^{\vee} \times \nu^{\alpha_{1}}\r_{1} \times  \nu^{\alpha_{2}}\r_{2} \times \cdots \times  \nu^{\alpha_{m}}\r_{m} \rtimes   \sigma' 
\end{align*}
This shows the standard intertwining operator 
\[
\delta \rtimes \sigma \xrightarrow{\simeq} \delta^{\vee} \rtimes \sigma
\]
is a bijection, and hence $\delta \rtimes \sigma$ is irreducible. This finishes the proof of Theorem~\ref{thm: a theorem of Tadic}.

\bibliographystyle{amsalpha}

\bibliography{reps}

\providecommand{\bysame}{\leavevmode\hbox to3em{\hrulefill}\thinspace}
\providecommand{\MR}{\relax\ifhmode\unskip\space\fi MR }
\providecommand{\MRhref}[2]{%
  \href{http://www.ams.org/mathscinet-getitem?mr=#1}{#2}
}
\providecommand{\href}[2]{#2}
\begin{thebibliography}{M{\oe}g09}

\bibitem[AP06]{AdlerPrasad:2006}
J.~D. Adler and D.~Prasad, \emph{On certain multiplicity one theorems}, Israel
  J. Math. \textbf{153} (2006), 221--245.

\bibitem[Art90]{Arthur:1990}
J.~Arthur, \emph{Unipotent automorphic representations: global motivation},
  Automorphic forms, {S}himura varieties, and {$L$}-functions, {V}ol.\ {I}
  ({A}nn {A}rbor, {MI}, 1988), Perspect. Math., vol.~10, Academic Press,
  Boston, MA, 1990, pp.~1--75.

\bibitem[Art96]{Arthur:1996}
\bysame, \emph{On local character relations}, Selecta Math. (N.S.) \textbf{2}
  (1996), no.~4, 501--579.

\bibitem[Art01]{Arthur:2001}
\bysame, \emph{A stable trace formula. {II}. {G}lobal descent}, Invent. Math.
  \textbf{143} (2001), no.~1, 157--220.

\bibitem[Art02]{Arthur:2002}
\bysame, \emph{A stable trace formula. {I}. {G}eneral expansions}, J. Inst.
  Math. Jussieu \textbf{1} (2002), no.~2, 175--277.

\bibitem[Art03]{Arthur:2003}
\bysame, \emph{A stable trace formula. {III}. {P}roof of the main theorems},
  Ann. of Math. (2) \textbf{158} (2003), no.~3, 769--873.

\bibitem[Art13]{Arthur:2013}
\bysame, \emph{The endoscopic classification of representations: orthogonal and
  symplectic groups}, Colloquium Publications, vol.~61, American Mathematical
  Society, 2013.

\bibitem[Bor79]{Borel:1979}
A.~Borel, \emph{Automorphic {$L$}-functions}, Automorphic forms,
  representations and {$L$}-functions ({P}roc. {S}ympos. {P}ure {M}ath.,
  {O}regon {S}tate {U}niv., {C}orvallis, {O}re., 1977), {P}art 2, Proc. Sympos.
  Pure Math., XXXIII, Amer. Math. Soc., Providence, R.I., 1979, pp.~27--61.

\bibitem[Bou87]{Bouaziz:1987}
A.~Bouaziz, \emph{Sur les caract\`eres des groupes de {L}ie r\'eductifs non
  connexes}, J. Funct. Anal. \textbf{70} (1987), no.~1, 1--79. \MR{870753
  (89c:22020)}

\bibitem[Clo87]{Clozel:1987}
L.~Clozel, \emph{Characters of nonconnected, reductive {$p$}-adic groups},
  Canad. J. Math. \textbf{39} (1987), no.~1, 149--167.

\bibitem[HC63]{H-C:1963}
Harish-Chandra, \emph{Invariant eigendistributions on semisimple {L}ie groups},
  Bull. Amer. Math. Soc. \textbf{69} (1963), 117--123.

\bibitem[HC75]{H-C:1975}
\bysame, \emph{Harmonic analysis on real reductive groups. {I}. {T}he theory of
  the constant term}, J. Functional Analysis \textbf{19} (1975), 104--204.

\bibitem[HC99]{H-C:1999}
\bysame, \emph{Admissible invariant distributions on reductive {$p$}-adic
  groups}, University Lecture Series, vol.~16, American Mathematical Society,
  Providence, RI, 1999, Preface and notes by Stephen DeBacker and Paul J.
  Sally, Jr.

\bibitem[Hen00]{Henniart:2000}
G.~Henniart, \emph{Une preuve simple des conjectures de {L}anglands pour {${\rm
  GL}(n)$} sur un corps {$p$}-adique}, Invent. Math. \textbf{139} (2000),
  no.~2, 439--455.

\bibitem[HS12]{HiragaSaito:2012}
K.~Hiraga and H.~Saito, \emph{On {$L$}-packets for inner forms of {$SL_n$}},
  Mem. Amer. Math. Soc. \textbf{215} (2012), no.~1013, vi+97.

\bibitem[HT01]{HarrisTaylor:2001}
M.~Harris and R.~Taylor, \emph{The geometry and cohomology of some simple
  {S}himura varieties}, Annals of Mathematics Studies, vol. 151, Princeton
  University Press, Princeton, NJ, 2001, With an appendix by Vladimir G.
  Berkovich.

\bibitem[JS83]{JacquetShalika:1983}
H.~Jacquet and J.~Shalika, \emph{The {W}hittaker models of induced
  representations}, Pacific J. Math. \textbf{109} (1983), no.~1, 107--120.

\bibitem[KS99]{KottwitzShelstad:1999}
R.~E. Kottwitz and D.~Shelstad, \emph{Foundations of twisted endoscopy},
  Ast\'erisque \textbf{55} (1999), no.~255, vi+190.

\bibitem[Lab85]{Labesse:1985}
J.-P. Labesse, \emph{Cohomologie, {$L$}-groupes et fonctorialit\'e}, Compositio
  Math. \textbf{55} (1985), no.~2, 163--184.

\bibitem[Lan89]{Langlands:1989}
R.~P. Langlands, \emph{On the classification of irreducible representations of
  real algebraic groups}, Representation theory and harmonic analysis on
  semisimple {L}ie groups, Math. Surveys Monogr., vol.~31, Amer. Math. Soc.,
  Providence, RI, 1989, pp.~101--170.

\bibitem[Lem16]{Lemaire:2016}
B.~Lemaire, \emph{Caract\`eres tordus des repr\'esentations admissibles}, To
  appear in Ast\'erisque (2016).

\bibitem[LL79]{LabesseLanglands:1979}
J.-P. Labesse and R.~P. Langlands, \emph{{$L$}-indistinguishability for {${\rm
  SL}(2)$}}, Canad. J. Math. \textbf{31} (1979), no.~4, 726--785.

\bibitem[M{\oe}g00]{Moeglin:2000}
C.~M{\oe}glin, \emph{Normalisation des op\'erateurs d'entrelacement et
  r\'eductibilit\'e des induites de cuspidales; le cas des groupes classiques
  {$p$}-adiques}, Ann. of Math. (2) \textbf{151} (2000), no.~2, 817--847.

\bibitem[M{\oe}g06]{Moeglin:2006}
\bysame, \emph{Sur certains paquets d'{A}rthur et involution
  d'{A}ubert-{S}chneider-{S}tuhler g\'en\'eralis\'ee}, Represent. Theory
  \textbf{10} (2006), 86--129.

\bibitem[M{\oe}g09]{Moeglin:2009}
\bysame, \emph{Paquets d'{A}rthur discrets pour un groupe classique
  {$p$}-adique}, Automorphic forms and {$L$}-functions {II}. {L}ocal aspects,
  Contemp. Math., vol. 489, Amer. Math. Soc., Providence, RI, 2009,
  pp.~179--257.

\bibitem[Mor11]{Morel:2011}
S.~Morel, \emph{Cohomologie d'intersection des vari\'et\'es modulaires de
  {S}iegel, suite}, Compos. Math. \textbf{147} (2011), no.~6, 1671--1740.

\bibitem[MS04]{MullerSpeh:2004}
W.~M{\"u}ller and B.~Speh, \emph{Absolute convergence of the spectral side of
  the {A}rthur trace formula for {${\rm GL}_n$}}, Geom. Funct. Anal.
  \textbf{14} (2004), no.~1, 58--93, With an appendix by E. M. Lapid.

\bibitem[MT02]{MoeglinTadic:2002}
C.~M{\oe}glin and M.~Tadi{\'c}, \emph{Construction of discrete series for
  classical {$p$}-adic groups}, J. Amer. Math. Soc. \textbf{15} (2002), no.~3,
  715--786 (electronic).

\bibitem[Mui04]{Muic:2004}
G.~Mui{\'c}, \emph{Composition series of generalized principal series; the case
  of strongly positive discrete series}, Israel J. Math. \textbf{140} (2004),
  157--202.

\bibitem[Mui05]{Muic:2005}
\bysame, \emph{Reducibility of standard representations}, Pacific J. Math.
  \textbf{222} (2005), no.~1, 133--168.

\bibitem[M{\"u}l89]{Muller:1989}
W.~M{\"u}ller, \emph{The trace class conjecture in the theory of automorphic
  forms}, Ann. of Math. (2) \textbf{130} (1989), no.~3, 473--529.

\bibitem[MW16]{MW:2016}
C.~M{\oe}glin and J.-L. Waldspurger, \emph{Stabilisation de la formule des
  traces tordue}, Progress in Mathematics, vol. 316/317, Birkh\"auser Basel,
  2016.

\bibitem[Neu99]{Neukirch:1999}
J.~Neukirch, \emph{Algebraic number theory}, Grundlehren der Mathematischen
  Wissenschaften [Fundamental Principles of Mathematical Sciences], vol. 322,
  Springer-Verlag, Berlin, 1999, Translated from the 1992 German original and
  with a note by Norbert Schappacher, With a foreword by G. Harder.

\bibitem[Ng{\^o}10]{Ngo:2010}
B.~C. Ng{\^o}, \emph{Le lemme fondamental pour les alg\`ebres de {L}ie}, Publ.
  Math. Inst. Hautes \'Etudes Sci. (2010), no.~111, 1--169.

\bibitem[Sau97]{Sauvageot:1997}
F.~Sauvageot, \emph{Principe de densit\'e pour les groupes r\'eductifs},
  Compositio Math. \textbf{108} (1997), no.~2, 151--184.

\bibitem[Sch13]{Scholze:2013}
P.~Scholze, \emph{The local {L}anglands correspondence for {${GL}_n$} over
  {$p$}-adic fields}, Invent. Math. \textbf{192} (2013), no.~3, 663--715.

\bibitem[Sha74]{Shalika:1974}
J.~A. Shalika, \emph{The multiplicity one theorem for {${\rm GL}_{n}$}}, Ann.
  of Math. (2) \textbf{100} (1974), 171--193.

\bibitem[Sha90]{Shahidi:1990}
F.~Shahidi, \emph{A proof of {L}anglands' conjecture on {P}lancherel measures;
  complementary series for {$p$}-adic groups}, Ann. of Math. (2) \textbf{132}
  (1990), no.~2, 273--330. \MR{1070599 (91m:11095)}

\bibitem[She79]{Shelstad:1979}
D.~Shelstad, \emph{Characters and inner forms of a quasi-split group over
  {${\bf R}$}}, Compositio Math. \textbf{39} (1979), no.~1, 11--45.

\bibitem[She12]{Shelstad:2012}
\bysame, \emph{On geometric transfer in real twisted endoscopy}, Ann. of Math.
  (2) \textbf{176} (2012), no.~3, 1919--1985.

\bibitem[Shi12]{Shin:2012}
S.-W. Shin, \emph{Automorphic {P}lancherel density theorem}, Israel J. Math.
  \textbf{192} (2012), no.~1, 83--120.

\bibitem[Spr09]{Springer:2009}
T.~A. Springer, \emph{Linear algebraic groups}, second ed., Modern Birkh\"auser
  Classics, Birkh\"auser Boston Inc., Boston, MA, 2009.

\bibitem[SV80]{SpehVogan:1980}
B.~Speh and D.~Vogan, \emph{Reducibility of generalized principal series
  representations}, Acta Math. \textbf{145} (1980), no.~3-4.

\bibitem[Tad86]{Tadic:1986}
M.~Tadi{\'c}, \emph{Classification of unitary representations in irreducible
  representations of general linear group (non-{A}rchimedean case)}, Ann. Sci.
  \'Ecole Norm. Sup. (4) \textbf{19} (1986), no.~3, 335--382.

\bibitem[Tad98]{Tadic:1998}
\bysame, \emph{On reducibility of parabolic induction}, Israel J. Math.
  \textbf{107} (1998), 29--91.

\bibitem[Tad09]{Tadic:2009}
\bysame, \emph{{${\rm GL}(n, \mathbb{C})$} and {${\rm GL}(n, \mathbb{R})$}},
  Automorphic forms and {$L$}-functions {II}. {L}ocal aspects, Contemp. Math.,
  vol. 489, Amer. Math. Soc., Providence, RI, 2009, pp.~285--313.

\bibitem[Vog86]{Vogan:1986}
D.~Vogan, \emph{The unitary dual of {${\rm GL}(n)$} over an {A}rchimedean
  field}, Invent. Math. \textbf{83} (1986), no.~3, 449--505.

\bibitem[Wal08]{Waldspurger:2008}
J.-L. Waldspurger, \emph{L'endoscopie tordue n'est pas si tordue}, Mem. Amer.
  Math. Soc. \textbf{194} (2008), no.~908, x+261.

\bibitem[Xu16]{Xu:2016}
B.~Xu, \emph{On a lifting problem of {L}-packets}, Compos. Math. \textbf{152}
  (2016), 1800--1850.

\bibitem[Zel80]{Zelevinsky:1980}
A.~V. Zelevinsky, \emph{Induced representations of reductive {p}-adic groups.
  {II}. {O}n irreducible representations of {${\rm GL}(n)$}}, Ann. Sci. \'Ecole
  Norm. Sup. (4) \textbf{13} (1980), no.~2, 165--210.

\end{thebibliography}

\end{document}